\documentclass[oneside, 11pt]{amsart}  
\newtheorem{theorem}{Theorem}[section] 
\newtheorem{corly}[theorem]{Corollary}
\newtheorem{lemma}[theorem]{Lemma}
\theoremstyle{definition}

\theoremstyle{remark}
\newtheorem{remark}[theorem]{Remark}
\newtheorem{example}[theorem]{Example}
\newcommand{\F}{\mathsf{F}}
\newcommand{\Ss}{\mathcal{S}}
\newcommand{\Hm}{\mathcal{H}}
\newcommand{\RR}{\mathbb{R}}
\newcommand{\mbf}{\mathbf}
\newcommand{\mcl}{\mathcal}

\numberwithin{equation}{section}
\usepackage{amssymb, latexsym, amscd}
\usepackage{epsfig}

\begin{document}

\title[S.P.\ Ellis, Approximation of Functions]{On the Approximation of a Function Continuous off a Closed Set by One Continuous Off a Polyhedron}
       
\author[S.~P.~Ellis]{Steven p.\ Ellis}
\date{May 3, 2011}

\address{Unit 42, NYSPI \\
1051 Riverside Dr. \\
New York, NY 10032 \\
U.S.A.}

\email{spe4@columbia.edu}


\thanks{This paper is dedicated to the memory of my mother.}

\thanks{This research is supported in part by United States PHS grant MH62185.}

\begin{abstract}  
Let $P$ be a finite simplicial complex (i.e., a finite collection of simplices that fit together nicely) and denote its underlying space (the union of the simplices in $P$) by $|P|$. Let $Q$ be a subcomplex of $P$ (e.g., $Q = P$). Let $a \geq 0$.  Then there exists $K < \infty$, \emph{depending only on $a$ and $Q$,} with the following property.  Let $\Ss \subset |P|$ be closed and suppose $\Phi$ is a continuous map of $|P| \setminus \Ss$ into some topological space $\F$. (``$\setminus$'' indicates set-theoretic subtraction.) Suppose $\dim (\tilde{\Ss} \cap |Q|) \leq a$, where ``$\dim$'' indicates Hausdorff dimension.  Then there exists $\tilde{\Ss} \subset |P|$ such that $\tilde{\Ss} \cap |Q|$ is the underlying space of a subcomplex of $Q$ and there is a continuous map $\tilde{\Phi}$ of $|P| \setminus \tilde{\Ss}$ into $\F$ such that
	\begin{itemize}
		\item $\Hm^{a} \bigl( \tilde{\Ss} \cap |Q| \bigr) \leq K \Hm^{a} \bigl( \Ss \cap |Q| \bigr)$, 
		        where $\Hm^{a}$ denotes 
		            $a$-dimensional Hausdorff measure;
		\item  if $x  \in \tilde{\Ss}$ then $x$ belongs to a simplex in $P$ intersecting $\Ss$;
		\item  if $x \in |P| \setminus \Ss$, $x \in \sigma \in P$, and $\sigma$ does not intersect 
		       any simplex in $Q$ whose simplicial interior intersects $\Ss$, then $\tilde{\Phi}(x)$ is 
		       defined and equals $\Phi(x)$; 
		\item  if $\sigma \in P$ then 
		       $\tilde{\Phi}(\sigma \setminus \tilde{\Ss}) \subset \Phi(\sigma \setminus \Ss)$;
		\item if $\F$ is a metric space and $\Phi$ is locally Lipschitz on $|P| \setminus \Ss$ then 
		            $\tilde{\Phi}$ is locally Lipschitz on $|P| \setminus \tilde{\Ss}$; and
		    \item $\dim (\tilde{\Ss} \cap |Q|)  \leq \dim (\Ss \cap |Q|)$ 
		                and $\dim \tilde{\Ss} \leq \dim \Ss$.
	\end{itemize}

Moreover, $P$ can be replaced by an arbitrarily fine subdivision without changing the constant $K$.  Consequently, modulo subdivision, if $\epsilon > 0$, we may assume $\tilde{\Phi}(x) = \Phi(x)$ if $dist(x, \Ss)$, the distance from $x$ to $\Ss$, exceeds $\epsilon$ and we may assume $\max \bigl\{ dist ( y, \Ss) : y \in \tilde{\Ss}  \bigr\} < \epsilon$.

Note that $\Ss$ can be any compact subset of $|P|$.  For example, no rectifiability assumptions on $\Ss$ are required. 
	But $\tilde{\Ss}$ is rectifiable and if $M$ is a compact $C^{1}$ manifold, $\mcl{T} \subset M$ is closed, and $\Phi : M \setminus \mcl{T} \to \F$ is continuous, then it is immediate from the preceding that $\Phi$ can be approximated by a continuous map $\tilde{\Phi} : M \setminus \tilde{\mcl{T}} \to \F$ where $\tilde{\mcl{T}}$ is closed and $a$-rectifiable.
\end{abstract}

\maketitle

\section{Introduction and main result} \label{S:intro}
Let $P$ be a finite simplicial complex of dimension $p$ and denote its underlying space by $|P|$.  (See appendix \ref{S:basics.of.simp.comps} for definitions and basic properties related to simplices and simplicial complexes.)  Suppose $\Ss \subset |P|$ is compact and $\Phi$ maps $|P| \setminus \Ss$ continuously into a topological space $\F$.  (``$\setminus$'' denotes set-theoretic subtraction)  One might be interested in properties of $\Ss$.  (E.g., see remark \ref{R:application} below.)  However, being a \emph{prima facie} arbitrary compact subset of $|P|$, $\Ss$ may be hard to analyze.  In this paper, we study the problem of deforming $\Ss$ and $\Phi$ a little so that $\Ss$ becomes a rather regular set, $\tilde{\Ss}$, specifically, the underlying space of a finite simplicial complex.  By deforming ``a little'' I mean that the region of $|P|$ in which the deformation takes place can be confined to an arbitrary open neighborhood of $\Ss$ and the volume of $\tilde{\Ss}$ is controlled as well.  (\cite{spE.PolyhedralApprox} is a less detailed version of this paper.)

The idea of deforming a set into a union of cells is reminiscent of the ``Deformation Theorem'' in geometric measure theory (Federer \cite[pp.\ 401 -- 408]{hF69}, 
Simon \cite[29.1, p.\ 163 and 29.4, p.\ 166]{lS83.GMT}, 
Hardt and Simon \cite[Hardt's Lecture 3, pp.\ 83--93]{rHlS86.GMT}
and Giaquinta \emph{et al} 
\cite[Lemma 2, p.\ 495; Theorem 1, p.\ 498; and Theorem 2, p.\ 503, Volume I]{mGgMjS98.cart.currents}).  In our case it is not just $\Ss$ that gets deformed.  $\Phi$ does as well.  As explained in remark
 \ref{E:compare.def.thm.w/.my.thm} below, it seems that our main theorem, 
 theorem \ref{T:main.theorem}, does not easily follow from the Deformation Theorem nor will the method of proof of the Deformation Theorem work for the problem considered here.

If $a \geq 0$, let $\Hm^{a}$ denote $a$-dimensional Hausdorff measure.  (See subsection \ref{S:Lip.Haus.meas.dim} in the appendix for background on Hausdorff measure and dimension.)  
Let $\lfloor a \rfloor$ denote the integer part of $a$, i.e., $\lfloor a \rfloor$ is the largest integer $\leq a$.  The (nonconstructive) proof of the following can be found in section \ref{S:Proof.of.main.thm} (with the most technical aspects of the proof relegated to appendix \ref{S:misc.proofs}).  If $X$ is a metric space with metric $d$, $x \in X$, and $A \subset X$, let 
	\[
			dist(x, A) = \inf \bigl\{ d(x, y) :  y \in A \bigr\}.
	\]
Also recall that the ``diameter'' of $A$ is defined to be
	\[
		diam(A) = \sup \bigl\{ d(x,y) \geq 0: x, y \in A \bigr\}.
	\]
 (See Munroe \cite[p.\ 12]{meM71.meas.thy} .)  Given $A$, the function $x \mapsto dist(x,A)$ is obviously continuous, in fact, Lipschitz 
(appendix \ref{S:Lip.Haus.meas.dim}), in $x \in  X$. (See Munroe \cite[p.\ 12]{meM71.meas.thy}.) 

In the following ``$\dim$'' denotes Hausdorff dimension and, for $a \geq 0$,  $\Hm^{a}$ denotes $a$-dimensional Hausdorff measure. (See appendix \ref{S:Lip.Haus.meas.dim}.)
   \begin{theorem}   \label{T:main.theorem}
     Let $P$ be a finite  
     simplicial complex lying in a Euclidean space.  
  	Let $|P|$ be the polytope or underlying space of $P$.  Use the metric on $|P|$ that it inherits 
	from the ambient Euclidean space.  Let $\Ss \subset |P|$ be closed.
     Let $\F$  be a topological space and suppose $\Phi : |P| \setminus \Ss \to \F$ is continuous.
        Let $Q$ be a subcomplex of $P$ (e.g., $Q = P$), let $a \geq 0$, and suppose 
          $\dim \bigl( \Ss \cap |Q| \bigr) \leq a$.  Then there is a closed set, $\tilde{\Ss} \subset |P|$ 
          and a continuous map $\tilde{\Phi} : |P| \setminus \tilde{\Ss} \to \F$ such that
            \begin{enumerate} 
              \item If $\F$ is a metric space and $\Phi$ is locally Lipschitz off $\Ss$ then $\tilde{\Phi}$ 
              is locally Lipschitz off $\tilde{\Ss}$.                    \label{I:Phi.tilde.locally.Lip.if.Phi.is}      
              \item $\dim (\tilde{\Ss} \cap |Q|)  \leq \dim (\Ss \cap |Q|)$ and $\dim \tilde{\Ss} \leq \dim \Ss$.
              	            \label{I:dim.Stilde.no.bggr.thn.dim.S}
              \item $\tilde{\Ss} \cap |Q|$ is either empty or the underlying space of a subcomplex 
              of the $\lfloor a \rfloor$-skeleton of $Q$.                \label{I:S.tilde.subcomp}
              \item Suppose $\tau \in P$ has the following property.  If $\rho \in Q$ and 
               $(\text{Int} \, \rho) \cap \Ss \ne \varnothing$ then $\tau \cap \rho = \varnothing$.  
              Then $\tilde{\Ss} \cap \tau = \Ss \cap \tau$ and $\tilde{\Phi}$ and $\Phi$ agree 
              on $\tau \setminus \Ss$.   
                              \label{I:no.change.off.nbhd.of.S.in.Q}  
              \item Let $\rho \in P \setminus Q$. (But $\rho \cap |Q| \neq \varnothing$ is possible.)  
              Then for every $s \geq 0$, if $\Hm^{s} \bigl( \Ss \cap (\text{Int} \, \rho) \bigr) = 0$ then 
              $\Hm^{s} \bigl( \tilde{\Ss} \cap (\text{Int} \, \rho) \bigr)  = 0$.  
              In particular, $\dim \bigl( \tilde{\Ss} \cap (\text{Int} \, \rho) \bigr) 
                       \leq \dim \bigl( \Ss \cap (\text{Int} \, \rho) \bigr)$.   
                               \label{I:no.change.off.Q}
              \item If $\tau \in Q$ and $\Hm^{\lfloor a \rfloor} \bigl( \tilde{\Ss} \cap (\text{Int} \, \tau) \bigr) > 0$, 
              then $\tau$ is an $\lfloor a \rfloor$-simplex and 
                 $\Hm^{\lfloor a \rfloor} \bigl[ \Ss \cap (\text{Int} \, \sigma) \bigr] > 0$ for some simplex
                 $\sigma$ of $Q$ having $\tau$ as a face.  ($\sigma = \tau$ is possible.)  
                    \label{I:Ha.tau.=0.if.Ha.all.nbrs.0}
              \item 
              If $y \in \tilde{\Ss}$ then there exists $\sigma \in P$ such that $y \in \sigma$ 
              and $\sigma \cap \Ss \ne \varnothing$.  Thus, $dist(y, \Ss)$ is not greater than the largest  of 
              the diameters of the simplices in $P$.   
	              \label{I:dist.tween.S.tilde.and.S}
	    \item If $\sigma \in P$ then 
	          $\tilde{\Phi}(\sigma \setminus \tilde{\Ss}) \subset \Phi(\sigma \setminus \Ss)$.
	              \label{I:Phi.tilde.sigma.subset.Phi.sigma}
              \item There is a constant $K < \infty$ depending only on $a$ and $Q$ such that
	                 \begin{equation}  \label{E:polyhedral.volume.magnification.factor}
	                      \Hm^{a} \bigl( \tilde{\Ss} \cap |Q| \bigr) \leq K \Hm^{a} \bigl( \Ss \cap |Q| \bigr).
	                 \end{equation}    \label{I:bound.on.S.tilde.vol}
             \item There is a constant $K < \infty$ depending only on $a$ and $P$ with the following property. For 
             every $\epsilon > 0$ there is a subdivision, $P'$, of $P$ such that $diam(\zeta) < \epsilon$ 
             for every $\zeta \in P'$ and parts 
      (\ref{I:Phi.tilde.locally.Lip.if.Phi.is}) through (\ref{I:Phi.tilde.sigma.subset.Phi.sigma}) above and 
         \eqref{E:polyhedral.volume.magnification.factor} 
             hold when $P$ is replaced 
      by $P'$ and $Q$ is replaced by the corresponding subcomplex of $P'$ (subdivision of $Q$).
                \label{I:arb.fine.subdivision}
            \end{enumerate}             
      \end{theorem}

The assumption that $P$ is finite can be replaced by a regularity property trivially satisfied by finite complexes.  The proof of the following is given in appendix \ref{S:misc.proofs}. (See appendix \ref{S:basics.of.simp.comps} for discussion of the topology of polytopes.)

\begin{corly}  \label{C:drop.finiteness.of.P}
Let $P$ be a, not necessarily finite, simplicial complex s.t.\ $|P| \subset \RR^{N}$, where $N < \infty$.  Suppose 
	\begin{multline}  \label{E:intersect.finitely.many}
		\text{Every } x \in |P| \text{ has a neighborhood, open in } \RR^{N}, \\
		            \text{ intersecting only finitely many simplices in } P. 
	\end{multline}
Then $|P|$ is a locally compact subspace of $\RR^{N}$.  (I.e., $|P|$'s polytope topology and relative topologies coincide and are locally compact.)  Put on $|P|$ the restriction of the usual metric on $\RR^{N}$. Let $\Ss \subset |P|$ be closed. Let $\F$  be a topological space and suppose $\Phi : |P| \setminus \Ss \to \F$ is continuous. Let $Q$ be a finite subcomplex of $P$, let $a \geq 0$, and suppose $\dim \bigl( \Ss \cap |Q| \bigr) \leq a$.  Then there is a closed set, $\tilde{\Ss} \subset |P|$ and a continuous map $\tilde{\Phi} : |P| \setminus \tilde{\Ss} \to \F$ such that parts (\ref{I:Phi.tilde.locally.Lip.if.Phi.is}) through (\ref{I:bound.on.S.tilde.vol}) of theorem \ref{T:main.theorem} hold. 

Let $P_{Q}$ be the simplicial complex consisting of all simplices in $P$ that intersect $|Q|$ and the faces of all such simplices.  
Then $P_{Q}$ is finite. The following replacement for part (\ref{I:arb.fine.subdivision}) of the theorem holds.

\begin{enumerate}
\item[(\ref{I:arb.fine.subdivision}')] There is a constant $K < \infty$ depending only on $a$ and $P_{Q}$ with the following property. For 
             every $\epsilon > 0$ there is a subdivision, $P'$, of $P$ such that $diam(\zeta) < \epsilon$ 
             for every $\zeta \in P'$ with $\zeta \subset |P_{Q}|$ and parts 
      (\ref{I:Phi.tilde.locally.Lip.if.Phi.is}) through (\ref{I:Phi.tilde.sigma.subset.Phi.sigma}) of theorem \ref{T:main.theorem} and \eqref{E:polyhedral.volume.magnification.factor} hold when $P$ is replaced 
      by $P'$ and $Q$ is replaced by the corresponding subcomplex of $P'$ (subdivision of $Q$).
\end{enumerate}
\end{corly}

\begin{remark}  \label{R:when.a.rectble}
Suppose $a \geq 0$ is an integer. The theorem tells us that we may assume that $\Phi$ is continuous off an $\Hm^{a}$-measurable countably $a$-rectifiable set (Giaquinta \emph{et al} 
\cite[pp.\ 90--91, Volume I]{mGgMjS98.cart.currents}, Hardt and Simon \cite[p.\ 20]{rHlS86.GMT}), and still have some control over its volume.  In fact, by \eqref{E:Lip.magnification.of.Hm} 
in appendix \ref{S:Lip.Haus.meas.dim}, trivially the same thing holds if $|P|$ is replaced by any compact space with a bi-Lipschitz triangulation (i.e., a triangulation that is Lipschitz and has a Lipschitz inverse), e.g., a compact $C^{1}$ manifold (Munkres \cite[Theorem 10.6, pp.\ 103--104]{jrM66}). (By Munkres \cite[Lemma 2.5, p.\ 10]{jrM84}, any simplicial complex $P$ with $|P|$ compact must be finite.)
\end{remark}

\begin{remark}  \label{R:miscellaneous.observations}
We make the following simple observations.
\begin{enumerate}
\item From appendix \ref{S:Lip.Haus.meas.dim}, we see that if we rescale $P$ by multiplying by a constant $\lambda > 0$, then the constant $K$ in \eqref{E:polyhedral.volume.magnification.factor} is multiplied by $\lambda^{a}$.
\item Part (\ref{I:Phi.tilde.sigma.subset.Phi.sigma}) of the theorem does not imply that $\tilde{\Phi}(x)$ and $\Phi(x)$ are close (when defined) because if $\sigma \in P$, $\Phi(\sigma \setminus \Ss)$ may be big.
\item Suppose $a$ is an integer.  If $\Hm^{a} \bigl( \Ss \cap |Q| \bigr)$ is very small we must have $\dim ( \tilde{\Ss} \cap |Q| \bigr) < a$.  For suppose 
$\dim ( \tilde{\Ss} \cap |Q| \bigr) = a$ then from theorem \ref{T:main.theorem} parts (\ref{I:S.tilde.subcomp} and \ref{I:bound.on.S.tilde.vol}), one can conclude
		\begin{equation}  \label{E:sing.set.vol.ineq}
			 \Hm^{a}\bigl( \Ss \cap |Q| \bigr) 	
			    \geq K^{-1} \Hm^{a}\bigl( \tilde{\Ss} \cap |Q| \bigr) 
				\geq K^{-1} \Hm^{a}( \text{smallest $a$-simplex in } Q) > 0.
		  \end{equation}  
	        \label{I:S.small.then.dim.S.tilde.<a}
\item Let $P$ be a finite simplicial complex and let $\Ss \subset |P|$ be compact and have empty interior (in particular $\Ss \ne |P|$).  One can easily construct a continuous function 
$\Phi : |P| \setminus \Ss \to \RR$ such that $\Phi$ cannot be continuously extended to any set larger than 
$|P| \setminus \Ss$.  Just take $\F = \RR$ and
	\[
		\Phi(x) = \sin \left( \frac{1}{dist(x, \Ss)} \right), \quad x \in P \setminus \Ss. 
	\]
(\emph{Proof:} It is easy to see that, since $|P|$ is locally arcwise connected (appendix \ref{S:basics.of.simp.comps}), arbitrarily close to $x \in \Ss$ there are points $y, y' \in |P| \setminus \Ss$ such that (s.t.) $\sin \bigl[ 1/dist(y, \Ss) \bigr] = +1$ and $\sin \bigl[ 1/dist(y', \Ss) \bigr] = -1$.)
\end{enumerate}
\end{remark}

A potentially useful corollary is the following.  See appendix \ref{S:misc.proofs} for the proof.

\begin{corly}  \label{C:poly.sing.set.near}
Let $P$ be a finite simplicial complex and let $a \geq 0$.  Then there exists $K < \infty$ (depending only on $a$ and $P$) with the following property.  Let $\Ss \subset |P|$ be compact and suppose $\Phi$ is a continuous map of $|P| \setminus \Ss$ into some topological space $\F$.  Suppose $\dim \Ss \leq a$.  Let $\epsilon > 0$.  Then there exists $\tilde{\Ss} \subset |P|$ and there is a continuous map 
$\tilde{\Phi}$ of $|P| \setminus \tilde{\Ss}$ into $\F$ such that  
	\begin{enumerate}
	           \item $\tilde{\Ss}$ is the underlying space of a subcomplex of a subdivision of $P$.
	                \label{I:S.tilde.is.subcmplx.of.subdiv}
		\item $\max \bigl\{ dist ( x, \Ss) : x \in \tilde{\Ss}  \bigr\} < \epsilon$, 
		     \label{I:Stilde.within.eps.ofS}
		\item if $x \in  |P|$ and $dist ( x, \Ss) \geq \epsilon$ then $\tilde{\Phi}(x) = \Phi(x)$, 
		     \label{I:same.far.away}
		     \item If $\sigma \in P$ then 
	                $\tilde{\Phi}(\sigma \setminus \tilde{\Ss}) \subset \Phi(\sigma \setminus \Ss)$, and
		     \label{I:Phi.tilde.sigma.subset.Phi.sigma.in.corly}
		\item $\Hm^{a}(\tilde{\Ss}) < K \Hm^{a}(\Ss)$.
		     \label{I:Hma.meas.Stilde.<.K.x.that.of.S}
	\end{enumerate}
\end{corly}

\begin{remark}
A ``cell'' is a closed, bounded region of some Euclidean space defined by finitely many linear equalities and inequalities.  Theorem \ref{T:main.theorem} probably can be generalized to general finite ``cell complexes'' (Munkres \cite[Definition 7.6, p.\ 74]{jrM66}), i.e., complexes consisting of cells that fit together nicely.  A ``cubical set'' (Kaczynski \emph{et al} \cite[Definition 2.9, p.\ 43]{tKkMmM04.CompHomol}) is an example.  Since any cell complex has a finite simplicial subdivision 
(Munkres \cite[Lemma 7.8, p.\ 75]{jrM66}), corollary \ref{C:poly.sing.set.near} certainly extends immediately to finite cell complexes.
\end{remark}

\begin{remark}  \label{R:application}
The set of applications of theorem \ref{T:main.theorem} is nonempty.  It turns out that multivariate statistical procedures often have ``singularities'', i.e., data sets at which the procedure, regarded as a function, does not have a limit.  
(See e.g., \cite{spE91.top.direct.axis,spE.3.or.4,spE02.nonlin,spE.fact.anal,
spE.SingsOfDataMaps}.)  Let $\Ss' \subset |P|$ be the singular set (set of singularities) of a data analytic procedure $\Psi : |P| \setminus \Ss' \to \F$.  $\Ss'$ may not be closed, so we cannot apply the theorem with $\Ss = \Ss'$.  However, it turns out that $\Psi$ can often be replaced by another procedure $\Phi : |P| \setminus \Ss \to \F$, where $\Ss$ is a closed subset of $\Ss'$ consisting of the most ``severe'', and therefore most interesting, singularities of $\Phi$.  It frequently turns out that $\dim \Ss$ is bounded below by some integer $a$ depending on general features of the statistical problem.  If $\dim \Ss > a$, then $\Hm^{a}(\Ss') \geq \Hm^{a}(\Ss) = \infty$.  Assume $\dim \Ss = a$ and apply theorem \ref{T:main.theorem} to $\Phi$.  It turns out that the set $\tilde{\Ss}$ must also have Hausdorff dimension $a$.  Hence, by remark \ref{R:miscellaneous.observations}(\ref{I:S.small.then.dim.S.tilde.<a}), the $\Hm^{a}$-volume of $\tilde{\Ss}$, and, hence, of $\Ss$ and $\Ss'$, is bounded below. The paper \cite{spE.SingsOfDataMaps} (in preparation) will develop and refine this idea.
\end{remark}

\begin{remark}
We make no effort to compute the best, or even a good, constant $K$ in \eqref{E:polyhedral.volume.magnification.factor}.  In principle, for a given $P$ and $a$ one could follow the proof and compute some value of $K$, but it would probably be very large. Let $K_{P}(a)$ denote the best, i.e., smallest, possible value of $K$ in \eqref{E:polyhedral.volume.magnification.factor}.  It would be helpful and interesting to know something about the relationship between $K_{P}$ and the structure of $P$.
\end{remark}

Now we briefly discuss the key ideas of the proof.  Let $\sigma \in Q$ be a simplex of dimension 
$> a$.  Let $\mcl{A} = \Ss \cap \sigma$.  If $\text{Int} \, \sigma \subset \mcl{A}$ or 
$(\text{Int} \, \sigma) \cap \mcl{A} = \varnothing$, then nothing much has to be done in $\text{Int} \, \sigma$.  So suppose $\varnothing \ne (\text{Int} \, \sigma) \cap \mcl{A} \ne \text{Int} \, \sigma$.  Call such a $\sigma \in Q$ a ``partial simplex'' (of $\Ss$).  The process of deforming $\Ss$ so that it becomes a subcomplex involves ``pushing $\mcl{A}$ out'' of $\text{Int} \, \sigma$ 
from a point $z \in (\text{Int} \, \sigma) \setminus \mcl{A}$.  Figure \ref{F:SimplexPush} illustrates this.  Given a point $z \in (\text{Int} \, \sigma) \setminus \mcl{A}$, move every point, $y \in \mcl{A}$ along the ray emanating from $z$ out to a point $\bar{h}_{z, \sigma}(y) \in \text{Bd} \, \sigma$.  This results in a new $\Ss$, call it $\Ss'$.  If the dimension of $\sigma$ is maximal among all partial simplices in $Q$, we have
	\begin{equation}  \label{E:description.of.S'}
		\Ss' = (\Ss \setminus \sigma) \cup \bar{h}_{z, \sigma}(\mcl{A})
	\end{equation}
(lemma \ref{L:big.lemma}(\ref{I:S'.cap.sigma.=.C.cap.Bd}, \ref{I:simps.in.S.are.also.in.S'}, \ref{I:dont.change.sings.on.othr.n.simps}, \ref{I:dont.increase.sing.dim.in.big.simps})).

The $\Hm^{a}$-volume of the image of $(\text{Int} \, \sigma) \cap \mcl{A}$ as it is flattened against the sides of the simplex can be larger or smaller than $\Hm^{a}((\text{Int} \, \sigma) \cap \mcl{A})$ (figure \ref{F:SimplexPush}).  One can easily imagine a subset $\mcl{A}$ of the simplex, $\sigma$, in figure \ref{F:SimplexPush} with the property that the image of $\mcl{A}$ after pushing out from \emph{any} interior point $z$ of $\sigma$ would have large 1-dimensional Hausdorff measure.  But in order to have this property the set $\mcl{A}$ itself must have large 1-dimensional Hausdorff measure.  So that observation does not lead to a counter example to the theorem.

Call the ratio of image volume, $\Hm^{a} \bigl[ \bar{h}_{z, \sigma}( \mcl{A}) \bigr]$, to input volume, 
$\Hm^{a}(\mcl{A})$, the ``$(\sigma, \mcl{A}, z)$-magnification factor.''  In order to prove \eqref{E:polyhedral.volume.magnification.factor} we need to always choose 
$z = z_{0}(\sigma, \Ss) \in \text{Int} \, \sigma$ so that the magnification factor is bounded independently of $\mcl{A}$.  The existence of such an ($\mcl{A}$-dependent) $z$ is shown by demonstrating that, averaged over $z \in \text{Int} \, \sigma$, the magnification factor is bounded independently of $\mcl{A}$.  (Actually, we average over $z$ in a concentric simplex sitting in the interior of $\sigma$.) The bound turns out to not depend on the size of $\sigma$, but only on its shape, specifically, its ``thickness'' (appendix \ref{S:basics.of.simp.comps}).  That is important because one can arbitrarily finely subdivide a finite simplicial complex all the while controlling the thickness of the simplices.

The operation of pushing $\Ss$ out of simplices is performed recursively. 
Let $\Ss^{0} = \Ss$ and $\Phi_{0} = \Phi$. Push $\Ss^{0}$ out of a partial simplex $\sigma \in Q$ of highest dimension.  The pushing out operation results in a new set, \emph{viz.}, $\Ss'$ as defined by \eqref{E:description.of.S'} (with $\Ss^{0}$ in place of $\Ss$).  A new function, call it $\Phi'$, must be defined that is continuous off $\Ss'$.  

Let $\rho \in P$ and suppose $\sigma$ is a face of $\rho$ (e.g., $\rho = \sigma$).  If $\dim \rho > \dim \sigma$, then by maximality of $\dim \sigma$, the simplex $\rho$ will not be a partial simplex of $\Ss^{0}$.  Thus, either $(\text{Int} \, \rho) \cap \Ss^{0} = \varnothing$ or $\text{Int} \, \rho \subset \Ss^{0}$. 
If $\text{Int} \, \rho \subset \Ss^{0}$ then $\rho \subset \Ss^{0}$, since $\Ss^{0}$ is closed.  Therefore, $\sigma \subset \Ss^{0}$, since $\sigma \subset \rho$.  But $\sigma$ is a partial simplex of $\Ss$. Therefore, $(\text{Int} \, \rho) \cap \Ss^{0} = \varnothing$.  The map $\Phi$ has to be deformed in $\rho$ (by composing it on the right with a locally Lipschitz map $g : \rho \setminus \bar{h}_{z, \sigma}(\mcl{A}) \to \rho$) so that the resulting map, call it $\Phi'$, is still defined and continuous in $\text{Int} \, \rho$.   Now let $\Ss^{0} = \Ss'$ and $\Phi_{0} = \Phi'$. Repeat until no partial simplices remain.

\begin{remark} 
$\tilde{\Phi}$ and $\tilde{\Ss}$ are probably not constructible unless the set $\Ss$ is ``decidable'' (Blum \emph{et al} \cite[Definition 2, p.\ 47]{lBfCmSsS98.realcompute}).
\end{remark}
	
	\begin{remark}[``Tomography'']  \label{R:tomography}
	By part (\ref{I:bound.on.S.tilde.vol}) of theorem \ref{T:main.theorem},  $\Hm^{a} \bigl( \tilde{\Ss} \cap |Q| \bigr)$ gives some information about $\Hm^{a} \bigl( \Ss \cap |Q| \bigr)$. Here we show that  $\tilde{\Ss}$ is a deformation retract of a set determinted by $\Ss$.

	Recall that $\tilde{\Ss}$ is constructed by recursively pushing the compact sets $\Ss^{0}$ from the simplicial interiors of its partial simplices of maximal dimension.
	Track the recursive construction of $\tilde{\Ss}$ by letting $\Ss_{i}$ denote the set obtained after $i$ pushing operations.  Thus, $\Ss_{0} = \Ss$.  For some value $i = m$, the set $\Ss_{i} \cap |Q|$ will be the underlying space of a subcomplex of $Q$.  Then $\tilde{\Ss} = \Ss_{m}$.  The recursive construction thus stops after $m$ steps.  Unless otherwise specified let $i = 0, \ldots, m$.  

	If $\sigma \in Q$ then for some values of $i$, the simplex $\sigma$ may be a partial simplex of $\Ss_{i}$ but for $i$ sufficiently large, $\sigma$ will not be a partial simplex (lemma \ref{L:big.lemma}, points (\ref{I:S'.cap.sigma.=.C.cap.Bd}, \ref{I:simps.in.S.are.also.in.S'}, \ref{I:dont.increase.sing.dim.in.big.simps})).  Let $\sigma_{i}$ be the partial simplex of $\Ss_{i}$ from which $\Ss_{i}$ will be pushed to produce $\Ss_{i+1}$.  Note that for any $\sigma \in P$, there is at most one $i$ s.t.\ $\sigma = \sigma_{i}$.  There is a point 
		\begin{equation}   \label{E:z0i.notin.Si}
			z_{0, i} \in \text{Int} \, \sigma_{i} \setminus \Ss_{i}
		\end{equation}
	from which $\Ss_{i}$ will be pushed.  If $0 \leq i < m$ and $y \in \text{Int} \, \sigma_{i}$, let $h_{i}(y) = \bar{h}_{z_{0,i}, \sigma_{i}}(y)$.  If $y \in |Q| \setminus (\text{Int} \, \sigma_{i})$ or if $y \in |Q|$ and $i \geq m$, let $h_{i}(y) = y$.  Note that, by \eqref{E:description.of.S'},
		\begin{equation}    \label{E:hi.maps.Si.onto.Si+1}
			h_{i}(\Ss_{i}) = \Ss_{i+1}.
		\end{equation}

	For each $i = 0, \ldots, m$, the function $h_{i}$ is continuous on $|Q| \setminus \{ z_{0,i} \}$.  This is because, first, $h_{i}$ is trivially continuous on $|Q| \setminus (\text{Int} \, \sigma_{i})$.  Second, by lemma \ref{L:props.of.s} part (\ref{I:hbar.locally.Lip.b.Lip}), $h_{i}$ is continuous on $\sigma_{i} \setminus \{ z_{0,i} \}$.  Finally, by \eqref{E:hbar.is.idendt.on.Bd.sigma}, $\bar{h}_{z_{0,i}}(y) = y = h_{i}(y)$ for $y \in  \text{Bd} \, \sigma_{i}$.

	For $i = 0, \ldots, m-1$, let $L_{i}(y)$ be the, possibly trivial, closed line segment joining $y \in \Ss_{i}$ to $h_{i}(y) \in \Ss_{i+1}$.  So $L_{i}(y) \subset |Q|$.
	Notice that for $y \in \Ss_{i}$, 
		\begin{equation}   \label{E:Li.cap.Bd.sigma.i}
			L_{i}(y) \cap (\text{Bd} \, \sigma_{i}) = \{ h_{i}(y) \} 
				\text{ or } L_{i}(y) \cap (\text{Bd} \, \sigma_{i}) = \varnothing.
		\end{equation}
	Observe also 
		\begin{align} \label{E:geometry.of.Li.y}
			&\text{If } y \in \Ss_{i} \text{ then } z_{0,i} \notin L_{i}(y).  \\
			&\text{If } y \in \Ss_{i} \text{ and } y' \in L_{i}(y) \text{ then } h_{i}(y') = h_{i}(y). \notag
		\end{align}	

		Now let
		\[
				\mcl{D}_{i} := \bigcup_{y \in \Ss_{i}} L_{i}(y) \text{ and }
				\mcl{E}_{i} := \bigcup_{j=i}^{m} \mcl{D}_{j}, \quad i = 0, \ldots, m-1.
				\text{ Let } \mcl{E}_{m} = \Ss_{m} = \tilde{\Ss}.
		\]
	Note that $\Ss \subset \mcl{E}_{0}$.
	\emph{Claim:} For $i = 0, \ldots, m-1$, the set $\mcl{E}_{i+1}$ is a deformation retract of $\mcl{E}_{i}$.  Let $F_{i} : \mcl{E}_{i} \times [0,1] \to \mcl{E}_{i}$ be defined by
		\[
			F_{i}(y,t) = 
				\begin{cases}
					y, &\text{ if } y \in \mcl{E}_{i+1},  \\
					(1-t) y + t \, h_{i}(y), &\text{ if } y \in \mcl{E}_{i} \setminus \mcl{E}_{i+1} \; .
				\end{cases}
		\]
	We show that $F_{i}$ is a deformation retraction. First, we show that 
	$F_{i} \bigl( \mcl{E}_{i} \times [0,1] \bigr) \subset \mcl{E}_{i}$.  It suffices to show that 
		\begin{equation}   \label{E:Fi.maps.into.Ei}
			F_{i} \Bigl( \bigl[  \mcl{E}_{i} \setminus \mcl{E}_{i+1} \bigr] \times [0,1] \Bigr) 
			   \subset \mcl{E}_{i}.
		\end{equation}
	Let $t \in [0,1]$ and $y \in \mcl{E}_{i} \setminus \mcl{E}_{i+1}$.  Then there exists $y' \in \Ss_{i}$ s.t.\ $y \in L_{i}(y')$.  By \eqref{E:geometry.of.Li.y}, we have $h_{i}(y) = h_{i}(y')$.  Therefore, $(1-t) y + t \, h_{i}(y) \in L_{i}(y') \subset \mcl{E}_{i+1}$.  The equation $h_{i}(y) = h_{i}(y')$ and \eqref{E:hi.maps.Si.onto.Si+1} also implies that 
		\begin{equation}   \label{E:hi.Ei.in.Ei+1}
			h_{i}(\mcl{E}_{i}) \subset \mcl{E}_{i+1}.
		\end{equation} 
	This completes the proof of \eqref{E:Fi.maps.into.Ei}. 
	Next, observe that $F_{i}$ is continuous.  By \eqref{E:hi.Ei.in.Ei+1}, we have $F_{i}(y,1) \in \mcl{E}_{i+1}$ for every $y \in \mcl{E}_{i}$. Finally, obviously $F_{i}(y,0) = y$ and $F_{i}(y, t) = y$ for every $y \in \mcl{E}_{i+1}$.

	Obviously, by first retracting $\mcl{E}_{0}$ onto $\mcl{E}_{1}$ then $\mcl{E}_{1}$ onto $\mcl{E}_{2}$ and so forth, finally retracting $\mcl{E}_{m-1}$ onto $\mcl{E}_{m}$, the net result is a deformation retraction of $\mcl{E}_{0}$ onto $\mcl{E}_{m} = \tilde{\Ss}$.  Hence, the homology of $\tilde{\Ss}$ is the same as that of $\mcl{E}_{1}$. Thus, if one is willing to accept $\mcl{E}_{1}$ as some sort of approximation to $\Ss$ then the homology of $\tilde{\Ss}$ should be considered an approximation to the homology of $\Ss$.  However, in fact, the relationship between the homologies of $\Ss$ and $\tilde{\Ss}$ is very loose at best.
	
	It follows from the preceding that $\Ss$ is deformable in $|P|$ into (onto, actually) $\tilde{\Ss}$ (Spanier \cite[p.\ 29]{ehS66}).
	
	Morever, note that the approximation $\mcl{E}_{1}$ is not natural since it depends on the somewhat arbitrary choices of the $z_{0,i}$'s.
	\end{remark}

             \begin{figure}  
      \includegraphics[width=6.2in, height = 5in, angle = -90]{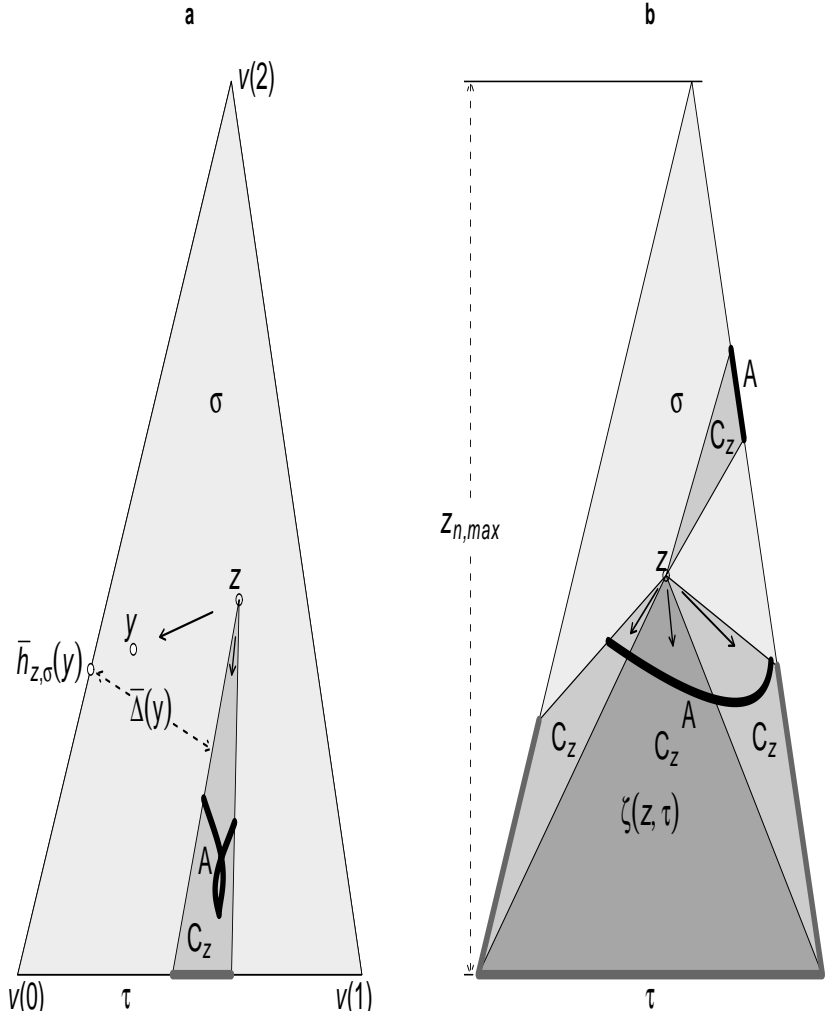}  %
       \caption{Pushing $\Ss$ out of a simplex, $\sigma$ (light grey triangles).  Heavy black curves constitute possible $\mcl{A}$'s.  (But $\mcl{A}$ can be any compact subset of $\sigma$.)  Heavy dark grey lines are possible images, $\bar{h}_{z, \sigma}(\mcl{A})$, of $\mcl{A}$ after pushing it out of $\sigma$ from $z$.  (a) Sometimes pushing reduces volume.  Medium grey triangle is the set $\mcl{C}_{z}$.  $v(0)$, $v(1)$, and $v(2)$ are the vertices of $\sigma$.  $\tau$ is the 1-simplex that is the ``bottom'' face of $\sigma$.  $y$ is a generic point of $\sigma$.  Its image under $\bar{h}_{z, \sigma}$ is obtained by pushing it along the ray emanating from $z$ out to the boundary.  $\bar{\Delta}(y)$ is the distance from that image to $\mcl{C}_{z}$.  (b) The worrisome case is when pushing increases volume.  $\zeta(z, \tau)$ is the medium dark grey triangle. $\mcl{C}_{z}$ is the union of the medium grey regions and $\zeta(z, \tau)$.  $z_{n,max}$ is the height of $\sigma$ from the plane spanned by $\tau$. Points of $\mcl{A}$ on $\text{Bd} \, \sigma$ are unaffected by pushing.}  
       \label{F:SimplexPush}
       \end{figure}

\begin{remark}  \label{E:compare.def.thm.w/.my.thm}
Theorem \ref{T:main.theorem} is reminiscent of the Deformation Theorem 
(Federer \cite[pp.\ 401 -- 408]{hF69}, 
Simon \cite[29.1, p.\ 163 and 29.4, p.\ 166]{lS83.GMT}, 
Hardt and Simon \cite[Hardt's Lecture 3, pp.\ 83--93]{rHlS86.GMT}
and Giaquinta \emph{et al} \cite[Lemma 2, p.\ 495; Theorem 1, p.\ 498; and Theorem 2, p.\ 503, Volume I]{mGgMjS98.cart.currents}) in geometric measure theory. E.g., in the proof of the Deformation Theorem as in the proof of theorem \ref{T:main.theorem}, a set, $\Ss$, also gets deformed by pushing it out of the interior of cells onto the boundaries.  In theorem \ref{T:main.theorem} the cells are simplicial while in the Deformation Theorem the cells are cubical, but that is not important. 

Might theorem \ref{T:main.theorem} follow from the Deformation Theorem?  Here I argue that deducing theorem \ref{T:main.theorem} from the Deformation Theorem would not be straight forward.  The Deformation Theorem shows that a ``normal current'' can be approximated by one supported by a cubical complex.  (See preceding references for relevant terminology from geometric measure theory.) How might we show that, at least, a ``cubical'' version of theorem \ref{T:main.theorem} follows from the Deformation Theorem?  Consider the case where $a \in (0,p)$ is an integer.  
We may assume that $P$ sits in a finite dimensional Euclidean space, $\RR^{N}$. 

In order to apply the Deformation Theorem to $T$ we must assume $T$ is a normal current, i.e., $\mathbf{M}(T) + \mathbf{M}(\partial T) < \infty$, where $\mathbf{M}$ denotes ``mass'' (Giaquinta \emph{et al} \cite[p.\ 125, Volume I]{mGgMjS98.cart.currents}). A natural way to make $\Ss$ into a normal current, $T$, is to let $T$ be a current of the form
	\[
		T(\omega) 
		    = \int_{\Ss} \bigl\langle \omega(x) , \vec{T}(x) \bigr\rangle \, d\mu(x) 
			\quad \bigl( \omega \in \mathcal{D}^{a}(\RR^{N} \bigr),
	\]
where 
$\mu$ is a Radon measure on $\RR^{N}$ (e.g., $\mu = \Hm^{a}$), $\vec{T}$ is an $\Hm^{a}$-measurable function taking values in the space of  $a$-vectors of length 1, $\mathbb{D}^{a}(\RR^{n})$ is the space of $C^{\infty}$ differential $a$-forms on $\RR^{N}$ with compact support, and ``$\langle \cdot , \cdot \rangle$'' indicates evaluation of the first argument at the second.  (See Giaquinta \emph{et al} \cite[Theorem 1, p.\ 126, Volume I]{mGgMjS98.cart.currents} 
and Hardt and Simon \cite[p.\ 67]{rHlS86.GMT}.) 
To avoid trivialities assume $\mu(\Ss) > 0$.  Then
	\[
		T \text{ is not 0}.
	\]

For concreteness, suppose $P$ is a 2-simplex, $\sigma$, (and its faces and vertices) sitting in $\RR^{2}$  and $a = 1$.  Suppose $\Ss \subset \text{Int} \, \sigma$ with $\dim \Ss = a = 1$.  (In particular, $\Ss$ has empty interior as a subset of $P$.) 
Since $\mathbf{M}(T) + \mathbf{M}(\partial T) < \infty$, 
by Giaquinta \emph{et al} \cite[Theorem 2, p.\ 129, Volume I]{mGgMjS98.cart.currents}, 
Simon \cite[Lemma 26.29, pp.\ 143]{lS83.GMT}, 
or Hardt and Simon \cite[Theorem 2.4, p.\ 78]{rHlS86.GMT}, if $L$ and $L^{\perp}$ are two perpendicular one-dimensional subspaces of $\RR^{2}$, then, since $T \ne 0$, the projections 
of $\Ss$ onto $L$ and $L^{\perp}$ cannot both have 0 Lebesgue measure.  

But suppose $\Ss \subset \RR^{2}$ is a compact ``Cantor dust'' of dimension 1 and having positive $\Hm^{1}$ measure (Falconer \cite[pp.\ xvi, 31]{kF90}).  But there exists perpendicular lines $L$ and $L^{\perp}$ s.t.\ the projections of $\Ss$ onto them each have Lebesgue measure 0. (Falconer \cite[Example 6.7, p.\ 87]{kF90}).  So the Deformation Theorem does not apply to $T$. 
  
If theorem \ref{T:main.theorem} is not a consequence of the Deformation Theorem itself, perhaps the method of proof of the Deformation Theorem might be used to prove 
theorem \ref{T:main.theorem}.  But at least some proofs of the Deformation Theorem rely on ``slicing'' (Giaquinta \emph{et al} \cite[pp.\ 151--152, Volume I]{mGgMjS98.cart.currents}).  It is not clear how to create a useful analogue of slicing that applies to general compact sets.  
\end{remark}

\section{Proof of theorem \ref{T:main.theorem}}   \label{S:Proof.of.main.thm}
\subsection{Pushing out} \label{SS:pushing.out}
We may assume that $P$ sits in a finite dimensional Euclidean space, $\RR^{N}$.  The metric on $|P|$ is the one it inherits from the Euclidean space. Let $\sigma \in Q$ be an $n$-simplex with $n > 0$.   By definition of simplex (appendix \ref{S:basics.of.simp.comps}), $\sigma$ does not lie in any $(n-1)$-dimensional affine plane. Let $\Ss \subset |P|$ be closed and suppose 
$\dim \bigl( \Ss \cap |Q| \bigr) \leq a$. Suppose 
	\begin{equation}  \label{E:sigma.intrscts.S}
		\text{Int} \, \sigma \nsubseteq \Ss \text{ but } (\text{Int} \, \sigma) \cap \Ss \neq \varnothing.
	\end{equation}
I.e., $\sigma$ is a ``partial simplex'' of $\Ss$. We wish to redefine $\Phi$ so that it is continuous in $\text{Int} \, \sigma$. (Of course, it may be possible to continuously extend the \emph{restriction}, $\Phi \vert _{ \text{Int} \, \sigma}$, of $\Phi$ to all of $\text{Int} \, \sigma$.  But it may not be possible to extend $\Phi$ itself to be continuous on $\text{Int} \, \sigma$ because of the behavior of $\Phi$ on $|P| \setminus \sigma$.)  Consider the following construction.  Let 
   \begin{equation}  \label{E:defn.of.set.A}
      \mcl{A} = \Ss \cap \sigma.
   \end{equation}
So by \eqref{E:sigma.intrscts.S} and compactness of $\Ss$,
	\begin{equation}  \label{E:A.intrscts.not.contains.sigma.int}
		\mcl{A} \text{ is compact and } 
		\varnothing \ne \mcl{A} \cap (\text{Int} \, \sigma) \ne \text{Int} \, \sigma.
	\end{equation}
Let  
	\begin{equation}  \label{E:z.notin.A}
		z \in (\text{Int} \, \sigma) \setminus \mcl{A} = (\text{Int} \, \sigma) \setminus \Ss
	\end{equation} 
be fixed.   Let $\mcl{C} \subset \sigma$ be  the set
	\begin{equation}  \label{E:defn.of.cone.C}
	   \mcl{C} := \mcl{C}_{z}  := \mcl{C}_{z}(\sigma) := \mcl{C}_{z}(\sigma, \Ss) 
	        := \{ \alpha x + (1 - \alpha) z \in \sigma : x \in \mcl{A} 
	      \text{ and } \alpha \geq 0 \}.
	\end{equation}
(Thus, $\alpha > 1$ is possible. See figure \ref{F:SimplexPush}.)  $\mcl{C}$ consists of the intersection of $\sigma$ with the union of all rays emanating from $z$ and passing through some point of $\mcl{A}$.  \emph{Claim:} 
	\begin{equation}  \label{E:C.is.compact.contains.A}
		\mcl{C} \text{ is compact, } \Ss \cap \sigma = \mcl{A} \subset \mcl{C}, \text{ and }
			z \in \mcl{C} \setminus \mcl{A}.
	\end{equation}
It suffices to show that $\mcl{C}$ is closed.  Let $\{ y_{m} \} \subset \mcl{C}$ and suppose $y_{m} \to y \in \sigma$.  Write
	\begin{equation}   \label{E:ym.xm.alpham}
		y_{m} = \alpha_{m} x_{m} + (1 - \alpha_{m}) z = \alpha_{m} (x_{m} - z) + z, 
		\quad \alpha_{m} \geq 0, \;  x_{m} \in \mcl{A},   \text{ and } m = 1, 2, \ldots.
	\end{equation}
Now, $z \notin \mcl{A}$ and $\mcl{A}$ is compact.  Therefore, there exists $\delta > 0$ s.t.\ $|x_{m} - z| \geq \delta$.  Hence, by \eqref{E:ym.xm.alpham},  
	\[
		0 \leq \alpha_{m} \leq \frac{diam(\sigma)}{\delta} < \infty.
	\]
Therefore, by compactness, extracting a subsequence if necessary, we have $\alpha_{m} \to \alpha \geq 0$ and $x_{m} \to x \in \mcl{A}$ so
	\[
		y_{m} = \alpha_{m} x_{m} + (1 - \alpha_{m}) z \to \alpha x + (1 - \alpha) z \in \mcl{C}.
	\]
This proves the claim \eqref{E:C.is.compact.contains.A}.

If  $x \in \sigma$, choose $b(x) = b_{z, \sigma}(x) \in [0,1]$ and 
$\bar{h}(x) = \bar{h}_{\sigma}(x) = \bar{h}_{z, \sigma}(x) \in \text{Bd} \, \sigma$ s.t.\ 
   \begin{equation}  \label{E:b.hbar.defn}
      \bar{h}(x) + \bigl( 1-b(x) \bigr) \bigl( z - \bar{h}(x) \bigr) = b(x) \bigl( \bar{h}(x) - z \bigr) + z 
          = b(x) \bar{h}(x) + \bigl( 1 - b(x) \bigr) z = x.
   \end{equation}
(See figure \ref{F:SimplexPush}.) If $x \neq z$ then $b(x)$ and $\bar{h}(x)$ are unique (lemma \ref{L:rays.intersect.bndry.in.one.pt};
define $\bar{h}(z) \in \text{Bd} \, \sigma$ arbitrarily).  If $x = z$, then $b(x)$ is still unique because 
$z \in  \text{Int} \, \sigma$.  Note that 
   \begin{equation}  \label{E:hbar.is.idendt.on.Bd.sigma}
      \bar{h}(x) = x \text{ if and only if } x \in \text{Bd} \, \sigma.
   \end{equation}
and
	\begin{equation}  \label{E:when.bx.=1.or.0}
		b(x) = 1 \text{ if and only if } x \in \text{Bd} \, \sigma  
			\text{ and } b(x) = 0 \text{ if and only if } x = z.
	\end{equation}
	
\emph{Claim:} 
   \begin{equation}   \label{E:hbar.maps.A.to.C.cap.Bd.sigma}
      \mcl{C} \cap (\text{Bd} \, \sigma) = \bar{h}(\mcl{A}).
   \end{equation}
Suppose $x \in \sigma \setminus \{ z \}$.  Then 
	\[
		b(x) \bar{h}(x) + \bigl( 1 - b(x) \bigr) z = x \text{ and } b(x) \in (0,1].
	\]
Then
	\begin{align*}
		\bar{h}(x) &= b(x)^{-1} x + b(x)^{-1} \bigl( b(x) - 1\bigr) z \\
		    &= b(x)^{-1} x + \bigl( 1 - b(x)^{-1} \bigr) z.
	\end{align*}
Hence, if $x \in \mcl{A}$, then $\bar{h}(x) \in \mcl{C} \cap (\text{Bd} \, \sigma)$ by definition.  

Conversely, suppose $x \in \mcl{C} \cap (\text{Bd} \, \sigma)$.  Then there exists $\alpha > 0$ and $y \in \mcl{A}$ (so by \eqref{E:z.notin.A} $y \ne z$) s.t.\ 
	\begin{equation}  \label{E:x.in.terms.of.alpha.y.z}
		x = \alpha y + (1-\alpha) z.
	\end{equation}
Consider the ray
	\[
		R(t) := t y + (1-t) z = y + (t-1)(y-z) , \quad t \geq 0.
	\]
Then $R(0) = z$ and $R(1) = y \in \sigma$ and for $t > 1$ sufficiently large we have 
$R(t) \notin \sigma$.  By lemma \ref{L:rays.intersect.bndry.in.one.pt}
there exists a unique $t_{0} > 0$ s.t.\ $R(t_{0}) \in \text{Bd} \, \sigma$.  But if $0 < t < 1$, 
then $R(t) \in \text{Int} \, \sigma$.  
Therefore, we must have $\alpha = t_{0} \geq 1$.  From \eqref{E:x.in.terms.of.alpha.y.z} we get
	\[
		y = \alpha^{-1} x + (1 - \alpha^{-1}) z.  
	\]
I.e., $x = \bar{h}(y)$ and $b(y) = \alpha^{-1} \in [0,1]$.  
This proves the claim \eqref{E:hbar.maps.A.to.C.cap.Bd.sigma}.

Moreover, we \emph{claim}
   \begin{equation}  \label{E:y.in.C.iff.hbar.y.in.C}
      \text{For } y \in \sigma \setminus \{ z \} \text{ we have } y \in \mcl{C} \text{ if and only if } 
           \bar{h}(y) \in \mcl{C}.
   \end{equation}
To see this, first assume $\bar{h}(y) \in \mcl{C}$.  Then, by \eqref{E:hbar.maps.A.to.C.cap.Bd.sigma},  
$\bar{h}(y) = \bar{h}(x)$ for some $x \in \mcl{A}$.  Now, $x \ne z$ (by \eqref{E:z.notin.A}) so $b(x) > 0$ by \eqref{E:when.bx.=1.or.0}. Hence, letting $\alpha = 1/b(x)$, we have
	\[
	x = \alpha^{-1} \bar{h}(y) + (1 - \alpha^{-1}) z \text{ and } y = b(y) \bar{h}(y) + \bigl[ 1 - b(y) \bigr] z.
	\]
Thus, $\bar{h}(y) = \alpha x + (1 - \alpha) z$ so 
	\[
		y = b(y) \alpha x + b(y) (1 - \alpha) z + \bigl[ 1 - b(y) \bigr] z = b(y) \alpha x 
		     + \bigl[ 1 - b(y) \alpha \bigr] z.
	\]
Thus, $y \in \mcl{C}$.

Conversely, suppose $y \in \mcl{C} \setminus \{ z \}$.  Then $b(y) > 0$ by \eqref{E:when.bx.=1.or.0} and there exists $x \in \mcl{A}$ and $\alpha \geq 0$ s.t.\ 
	\[
		b(y) \bar{h}(y) + \bigl[ 1 - b(y) \bigr] z = y = \alpha x + (1 - \alpha) z.
	\]
Thus, 
	\[
		\bar{h}(y) = b(y)^{-1} \alpha x + \bigl( 1 - b(y)^{-1} \alpha  \bigr) z.
	\]
Thus, $\bar{h}(y) \in \mcl{C}$.  This completes the proof of the claim \eqref{E:y.in.C.iff.hbar.y.in.C}.

Here and below we find it convenient to restrict ourselves to the situation when $\bar{h}(x)$ belongs to a specified proper face of $\sigma$. Let $\tau$ be a face of $\sigma$ of dimension $n-1$.    Without loss of generality (WLOG) we may 
assume 
	\begin{equation}  \label{E:sigma.in.R.n.tau.in.R.n-1}
		\sigma \subset \RR^{n} \text{ and } \tau \subset \RR^{n-1}, 
	\end{equation}
where we identify $\RR^{n-1}$ 
with $\bigl\{ (y_{1}, \ldots, y_{n-1}, 0) \in \RR^{n} : (y_{1}, \ldots, y_{n-1}) \in \RR^{n-1} \bigr\}$.  (Recall $n = \dim \sigma$.) Thus, the $n^{th}$ coordinate of any $y \in \sigma \setminus \tau$ is non-zero.  
If $y \in \sigma \setminus \{ z \}$ let
   \begin{multline}  \label{E:b.h.tau.defns}
      b = b(y) = b(y; z, \tau) \text{ and } h = h_{z, \tau}(y) = h(y; z, \tau)
         \text{ be the solution of } b h + ( 1 - b ) z = y    \\
          \text{ with } b \in (0,1] \text{ and } h \in \tau, \text{ if such $b(y)$ and $h_{z, \tau}(y)$ exist.}
   \end{multline}
Otherwise, $b(y; z, \tau)$ and $h(y; z, \tau)$ are undefined.  Thus, 
	\begin{equation}  \label{E:b.sigma.b.tau.h.hbar}
		b(y; z,\tau) = b_{z, \sigma}(y) \text{  and } h(y; z,\tau) = \bar{h}_{z, \sigma}(y), 
			\text{  if  } y \in \sigma \setminus \{z\} \text{ and } \bar{h}_{z, \tau}(y) \in \tau.
	\end{equation}

The set of $y \in \sigma$ for which $b(y)$ and $h_{z, \tau}(y)$ exist uniquely is obviously precisely the set
   \begin{equation}  \label{E:zeta.z.defn}
      \zeta(z) = \zeta(z; \tau) = \{ b x + (1 - b) z \in \sigma : x \in \tau, \; 0 < b \leq 1 \}.
   \end{equation}
(See figure \ref{F:SimplexPush}.) We have
	\begin{equation}  \label{E:barh.and.h.agree.on.zetaz}
		\bar{h}_{z,\sigma}(y) = h_{z,\tau}(y), \quad \text{ for } y \in  \zeta(z).
	\end{equation}
In particular, 
	\begin{equation}  \label{E:when.hbar.and.h.agree}
		\bar{h}_{z, \sigma} \bigl( \mcl{A} \cap (\text{Int} \, \sigma) \bigr) \cap \tau 
			= h_{z,\tau} \Bigl( \mcl{A} \cap \bigl[ (\text{Int} \, \sigma) \cap \zeta(z; \tau) \bigr] \Bigr) 
				\text{ for } z \in (\text{Int} \, \sigma) \setminus \mcl{A}.
	\end{equation}
In particular, if $\tau'$ is another $(n-1)$-face of $\sigma$ that intersects $\tau$ 
then \eqref{E:when.hbar.and.h.agree} implies
	\begin{equation}  \label{E:h.indep.of.tau.face}
	     \zeta(z; \tau) \cap \zeta(z; \tau') \ne \varnothing \text{ and } 
		h_{z,\tau}(y) = h_{z,\tau'}(y) \in \tau \cap \tau' \text{ for } y \in \zeta(z; \tau) \cap \zeta(z; \tau').
	\end{equation}  
Notice further that
	\begin{equation}  \label{E:h.is.identity.on.tau}
	      h_{z, \tau}(y) = y, \text{ if } y \in \tau.
	\end{equation}

Note that $z \notin \zeta(z; \tau)$.  We \emph{claim}
	\begin{equation}  \label{E:sigma.less.z.in.zeta.union}
	      \sigma \setminus \{ z \} = \bigcup_{\omega} \zeta(z; \omega),
	\end{equation}
where the union is taken over all $(n-1)$-faces, $\omega$, of $\sigma$.  For let $y \in \sigma \setminus \{z\}$ and write
	\[
		y = \sum_{i=0}^{n} \beta_{i} v(i) \text{ and } z = \sum_{i=0}^{n} \alpha_{i} v(i),
	\]
where $v(0), \ldots, v(n) \in \RR^{n}$ are the vertices of $\sigma$.  Then all the $\beta_{i}$'s 
($\alpha_{i}$'s) are nonnegative and sum to 1.  Since $z \in \text{Int} \, \sigma$ the coordinates $\alpha_{0}, \ldots, \alpha_{n}$ are all strictly positive.  
Let $i = j$ maximize $(\alpha_{i} - \beta_{i})/\alpha_{i}$.  
(Define $(\alpha_{i} - \beta_{i})/\alpha_{i} = - \infty$ if $\alpha_{i} = 0$.) Since 
$\sum_{i} \alpha_{i} = 1 = \sum_{i} \beta_{i}$, but 
$(\alpha_{0}, \ldots, \alpha_{n}) \neq (\beta_{0}, \ldots, \beta_{n})$ we have $\alpha_{i} > \beta_{i}$ for at least one $i$.  Therefore, we have $\alpha_{j} > \beta_{j} > 0$ and
$b := (\alpha_{j} - \beta_{j})/\alpha_{j} > 0$.  For $i=0, \ldots, n$, let
	\[
		\gamma_{i} = \frac{ \alpha_{j} \beta_{i} - \alpha_{i} \beta_{j} }{ \alpha_{j} - \beta_{j} }.
	\]
Now for $i=0, \ldots, n$,
	\[
		(\alpha_{j} - \beta_{j}) \gamma_{i} = \alpha_{i} (\alpha_{j} - \beta_{j}) 
			- \alpha_{j} (\alpha_{i} - \beta_{i})  \geq 0
	\]
by choice of $j$.  I.e., since $\alpha_{j} > \beta_{j}$, we have $\gamma_{i} \geq 0$ for $i = 0, \ldots, n$.  Moreover, $\sum_{i} \gamma_{i} = 1$.  Hence, if $x = \sum_{i=0}^{n} \gamma_{i} v(i)$, then $x \in \sigma$.  But $\gamma_{j} = 0$, so if $\omega$ is the $(n-1)$-face of $\sigma$ opposite $v(j)$, then actually $x \in \omega$.  But $b \in (0, 1]$ and  it is easy to see that $y = b x + (1-b) z$.  I.e., $y \in \zeta(z; \omega)$.  This proves the claim \eqref{E:sigma.less.z.in.zeta.union}.  

(Thus, $b(y) = (\alpha_{j} - \beta_{j})/\alpha_{j} > 0$ if $y \ne z$.  But $b(y) = (\alpha_{j} - \beta_{j})/\alpha_{j} = 0$ still works if $y = z$.  Hence, it follows from lemma \ref{L:bary.coords.are.Lip}, \eqref{E:comp.of.Lips.is.Lip}, and the fact that ``$\max$'' is Lipschitz that $b(y)$ is Lipschitz in $y \in \sigma$.)

If $y \in \zeta(z)$ then obviously
   \begin{equation}  \label{E:hx.in.terms.of.b.y.z}
      h_{z, \tau}(y) = b(y)^{-1} (y - z) + z.
   \end{equation}
If $y \in \RR^{n}$, let $y_{j}$ be the $j^{th}$ coordinate of $y$ 
($j=1, \ldots, n$).  Since $\tau \subset \RR^{n-1}$, $x_{n} = 0$ if $x \in \tau$. If $v(n)$ is the vertex 
of $\sigma$ opposite $\tau$ then its $n^{th}$ coordinate, $v_{n}(n)$, is not 0.  WLOG we may assume $v_{n}(n) > 0$.  Then $y \in \sigma$ implies $y_{n} \geq 0$.  Since $z \in \text{Int} \, \sigma$, we have $z_{n} > 0$.  Therefore,
	\begin{equation}  \label{E:yn.<.zn.if.y.in.zeta}
		y \in \zeta(z) \text{ implies } 0 \leq y_{n} < z_{n}. 
	\end{equation}
Then \eqref{E:hx.in.terms.of.b.y.z} implies
   \[
      0 = b(y)^{-1} (y_{n} - z_{n}) + z_{n}.
   \]
Thus, 
  \begin{equation}   \label{E:b.in.terms.of.y.z}
     b(y) = z_{n}^{-1} (z_{n} - y_{n}).
  \end{equation}
In particular, by \eqref{E:b.h.tau.defns},
   \begin{equation}  \label{E:z.is.vertex.of.zeta.z}
      y_{n} \uparrow z_{n} \text{ with } y \in \zeta(z) \text{ implies } b(y) \to 0 \text{, which implies }  y \to z.
   \end{equation}
Substituting \eqref{E:b.in.terms.of.y.z} into \eqref{E:hx.in.terms.of.b.y.z} we get
   \begin{equation}   \label{E:h.in.terms.of.y.z}
      h_{z, \tau}(y) = \frac{z_{n}}{z_{n} - y_{n}} (y - z) + z 
         = \frac{z_{n}}{z_{n} - y_{n}} \left( y - \frac{y_{n}}{z_{n}} z \right), \quad y \in \zeta(z).
   \end{equation}
Thus, $h_{z, \tau}$ is continuous on $\zeta(z)$.  Conversely, we have
	\begin{equation}  \label{E:condition.for.y.in.zeta}
		\text{If } y \in \sigma, \, y_{n} < z_{n}, \text{ and } \frac{z_{n}}{z_{n} - y_{n}} (y - z) + z \in \tau,
			\text{ then } y \in \zeta(z).
	\end{equation}

Since $\mcl{A}$ is compact and $z \notin \mcl{A}$, by \eqref{E:z.notin.A} and \eqref{E:z.is.vertex.of.zeta.z}, the difference $z_{n} - y_{n} > 0$, $y \in \zeta(z) \cap \mcl{A}$, is bounded away from 0.  It follows from corollary \ref{C:cont.diff.=.loc.Lip} in appendix \ref{S:Lip.Haus.meas.dim} that
	\begin{equation}  \label{E:h.z.tau.is.loc.Lip.on.zeta}
		h_{z, \tau} \text{ is locally Lipschitz on } \zeta(z) 
		       \text{ and Lipschitz on } \mcl{A} \cap \zeta(z).
	\end{equation}
(Lemma \ref{L:props.of.s}, part (\ref{I:hbar.locally.Lip.b.Lip}), will tell us that $\bar{h} : \sigma \setminus \{ z \} \to \text{Bd} \, \sigma$ is  locally Lipschitz on $\sigma \setminus \{ z \}$ and Lipschitz on $\mcl{A}$.)

For $y \in \sigma \setminus \{ z \}$, let
	\begin{equation}   \label{E:Delta.bar.defn}
	      \bar{\Delta}(y) := \text{dist} \bigl( \bar{h}(y), \mcl{C} \bigr).
	\end{equation}
(See figure \ref{F:SimplexPush}.)  Note that $\bar{\Delta}(y) \in \bigl[ 0, diam(\sigma) \bigr]$ for all $y \in \sigma \setminus \{ z \}$. 
Define.  
	\begin{equation}    \label{E:Delta.bar.z.=diam.sig}
		\bar{\Delta}(z) := diam(\sigma) > 0
	\end{equation}
and define $s : \sigma \times [0,1] \to \sigma$ as follows
 
Let $k$ be the function
   \begin{equation}   \label{E:defn.of.k}
      k(\beta, \delta, t) := 
	      \begin{cases} 
		      \exp \left\{ - \frac{\delta + t}{1 - \beta} + \delta + t \right\}
		       = \exp \left\{ - \frac{\beta (\delta + t )}{1 - \beta} \right\}, 
		              &\text{ if } \delta, t \in \RR, \beta < 1, \\
		      0,  &\text{ if } \delta + t > 0 \text{ and } \beta = 1.      
	      \end{cases}  
   \end{equation}
Let $y \in \sigma, \; 0 \leq t \leq 1$, write $b(y) = b_{z,\sigma}(y)$, and let
	\begin{equation}  \label{E:f.defn}
		f(y,t) := k \bigl[ b(y), \bar{\Delta}(y), t \bigr], 
		  \text{ for } (y,t) \in B_{z} := \bigl( \sigma \times [0,1] \bigr)
		   \setminus \Bigl[ \bigl( \mcl{C}_{z} \cap [\text{Bd} \, \sigma] \bigr)  \times \{0\} \Bigr].
	\end{equation}
Finally, let
   \begin{equation}  \label{E:syt.defn}
      s(y,t) :=  s_{z}(y,t) :=  
         \begin{cases}
            y, &\text{ if } b(y) = 1, \text{ i.e.\ } y \in \text{Bd} \, \sigma, \\
                  \bigl(1 - f(y,t) \bigr) \bar{h}(y)  
                   + f(y,t) \, z, 
                &\text{ if } 0 \leq b(y) < 1, \text{ i.e., } y \in \text{Int} \, \sigma. 
         \end{cases}.
   \end{equation}
   
If $y \in \overline{\text{St}} \, \sigma$ (see appendix \ref{S:basics.of.simp.comps}), we can write
	\[
		y = \mu(y) \sigma(y) + \bigl( 1- \mu(y) \bigr) w(y),
	\]
where $\mu(y) \in [0,1]$, $\sigma(y) \in \sigma$, and $w(y) \in \text{Lk} \, \sigma$.  
(See proof of lemma \ref{L:big.lemma} below for details.)  

Next, we define $g : |P| \setminus  \bigl[ \mcl{C} \cap (\text{Bd} \, \sigma) \bigr] \to |P|$.  
If $\overline{\text{St}} \, \sigma \ne \sigma$ define
	\begin{multline}  \label{E:g.defn.when.St.sigma.ne.sigma}
		g(y) := g_{z}(y) :=  \\
			\begin{cases}
				\mu(y) \, s_{z} \bigl( \sigma(y), 1 - \mu(y) \bigr) + \bigl( 1 - \mu(y) \bigr) w(y) 
				                  \in \overline{\text{St}} \, \sigma, 
				   &\text{ if } y \in \overline{\text{St}} \, \sigma 
				             \setminus \bigl( \mcl{C} \cap (\text{Bd} \, \sigma) \bigr), \\
				y, &\text{ if } y \in |P| \setminus (\overline{\text{St}} \, \sigma).
			\end{cases}
	\end{multline}
If  $\overline{\text{St}} \, \sigma = \sigma$ define
	\begin{equation}  \label{E:g.defn.when.St.=.sigma}
	g(y) := g_{z}(y) := 
		\begin{cases}
			s(y,0), &\text{ if } y \in \sigma \setminus \bigl[ \mcl{C} \cap (\text{Bd} \, \sigma) \bigr], \\
			y, &\text{ if } y \in |P| \setminus \sigma.
		\end{cases}
	\end{equation}

We will see (\eqref{E:g.is.cont.off.C.meet.Bd.sigma}) that $g$ is a continuous map 
of $|P| \setminus \bigl[ \mcl{C} \cap (\text{Bd} \, \sigma) \bigr]$ into itself. Define 
	   \begin{equation}   \label{E:Phi'.defn}
		  \Phi'(y) := \Phi_{z}'(y) := \Phi \circ g_{z}(y) \in \F, \quad y \in |P|,
	    \end{equation}
whenever the right hand side is defined.  Define also
	\begin{equation} \label{E:S'.defn}
		\Ss' =\Ss'_{z} =  g_{z}^{-1} \bigl[ \Ss \setminus \sigma \bigr] 
		   \cup \bigl[ \mcl{C}_{z} \cap (\text{Bd} \, \sigma) \bigr].
	\end{equation}

The following gives some properties of $\Phi'$ and $\Ss'$.  Its proof, in subsection \ref{S:misc.proofs} in the appendix, develops the properties of $\sigma(\cdot)$ and $g$.

   \begin{lemma}  \label{L:big.lemma}
      Let $\Ss$ be a nonempty closed subset of $|P|$.  Let $\sigma \in Qù$ be an $n$-simplex ($n > 0$), let $\mcl{A} = \Ss \cap \sigma$, and suppose \eqref{E:A.intrscts.not.contains.sigma.int}
      holds.
      Let $z \in (\text{Int} \, \sigma) \setminus \Ss$.  Then $\Phi'$ and $\Ss'$, defined by \eqref{E:Phi'.defn} and \eqref{E:S'.defn} have the following properties.
         \begin{enumerate}
           \renewcommand{\theenumi}{\alph{enumi}}
           \item $\Ss' \cap \sigma = \mcl{C}_{z} \cap (\text{Bd} \, \sigma)$.  
           In particular, by \eqref{E:C.is.compact.contains.A}, 
           $\Ss \cap (\text{Bd} \, \sigma) \subset \Ss' \cap (\text{Bd} \, \sigma) $ 
           and $\Ss' \cap (\text{Int} \, \sigma) = \varnothing$.  Moreover, 
	           \[
		           \text{for every $s \geq 0$, if }  \Hm^{s}(\mcl{A}) = 0 
			           \text{ then } \Hm^{s}\bigl[ \mcl{C}_{z} \cap (\text{Bd} \, \sigma) \bigr] = 0.
	           \]
           In particular, 
         	       \begin{equation}   \label{E:dim.S'.cap.sigma.le.dim.S.cap.sigma}
	           \dim (\Ss' \cap \sigma) = \dim \bigl[ \mcl{C}_{z} \cap (\text{Bd} \, \sigma) \bigr] 
	                 \leq \dim \mcl{A} = \dim (\Ss \cap \sigma).
	       \end{equation}
                 			\label{I:S'.cap.sigma.=.C.cap.Bd}
	 \item If $\rho \in P$ is not a face of $\sigma$ then $g$ is one-to-one 
	 	on $\text{Int} \, \rho$, $g(\text{Int} \, \rho) = \text{Int} \, \rho = g^{-1}(\text{Int} \, \rho)$, and
		the restriction, $g \vert _{ \text{Int} \, \rho}$,  has a locally Lipschitz inverse on
		  $\text{Int} \, \rho$.
			\label{I:g.has.locally.Lip.invrs} 
           \item $\Ss'$ is closed and $\Phi'$ is defined and continuous at{E:z.notin.A} every point 
              of $|P| \setminus \Ss'$.  If $\F$ is a metric space and $\Phi$ is locally Lipschitz 
              on $|P| \setminus \Ss$ then $\Phi'$ is locally Lipschitz on $|P| \setminus \Ss'$.
                 \label{I:S'.closed.Phi'.cont.off.S'}
           \item If $\rho \in P$ and $\text{Int} \, \rho \subset \Ss$ 
               then $\text{Int} \, \rho \subset \Ss'$.  
              \label{I:simps.in.S.are.also.in.S'}
           \item If $\tau \in P$ and $\tau \cap \sigma = \varnothing$ then $\Ss' \cap \tau = \Ss \cap \tau$ 
           and if $y \in \tau \setminus \Ss$, then $\Phi'(y) = \Phi(y)$.
              In particular, $\Ss' \cap (\text{Lk} \, \sigma) = \Ss \cap (\text{Lk} \, \sigma)$. 
                \label{I:Phi'.and.Phi.agree.on.Lk.sigma}
           \item If $\tau \neq \sigma$ is a simplex in $P$ of dimension no greater than $n := \dim \sigma$, 
           then  $\Ss' \cap (\tau \setminus \sigma) = \Ss \cap (\tau \setminus \sigma)$ and $\Phi' = \Phi$ on $\tau \setminus (\sigma \cup \Ss)$.
                 \label{I:dont.change.sings.on.othr.n.simps}
           \item Suppose $\tau \in P$ and $\dim \tau < n := \dim \sigma$.  
           Then 
	           \[
		           (\text{Int} \, \tau) \cap \Ss' 
		           = (\text{Int} \, \tau) \cap 
		                      \Bigl( \Ss \cup \bigl[ \mcl{C}_{z} \cap (\text{Bd} \, \sigma) \bigr] \Bigr)
		           = (\text{Int} \, \tau) \cap \bigl[ \Ss \cup \bar{h}_{z, \sigma}(\Ss \cap \sigma) \bigr].
	           \]
                 \label{I:S'.bggr.S.on.n-1.simps}
           \item If $\rho \in P$ is a simplex of dimension at least $n$ then for every $s \geq 0$, 
           if $\Hm^{s} \bigl[ \Ss \cap (\text{Int} \, \rho) \bigr] = 0$ 
                  then $\Hm^{s} \bigl[ \Ss' \cap (\text{Int} \, \rho) \bigr] = 0$. In particular, 
           $\dim \bigl[ \Ss' \cap (\text{Int} \, \rho) \bigr]) \leq \dim \bigl[ \Ss \cap (\text{Int} \, \rho) \bigr]$.  
           In particular as well, by \eqref{E:A.empty.iff.H0.A.=.0}, 
           if $(\text{Int} \, \rho) \cap \Ss = \varnothing$, then $(\text{Int} \, \rho) \cap \Ss' = \varnothing$. 
                \label{I:dont.increase.sing.dim.in.big.simps}
           \item If $s \geq 0$, $\tau \in P$, and $\Hm^{s} \bigl[ \Ss' \cap (\text{Int} \, \tau) \bigr] > 0$, then either $\Hm^{s} \bigl[ \Ss \cap (\text{Int} \, \tau) \bigr] > 0$ or $\tau$ is a proper face of $\sigma$ 
           and $\Hm^{s} \bigl[ \Ss \cap (\text{Int} \, \sigma) \bigr] > 0$.  Taking $s = 0$, we conclude: 
           If $\tau \in P$ and $(\text{Int} \, \tau) \cap \Ss' \ne \varnothing$ then either 
           $(\text{Int} \, \tau) \cap \Ss \ne \varnothing$ or $\tau$ is a proper face of $\sigma$.  
                                    \label{I:Int.rho.cut.S'.then.cuts.S.or.rho.is.sigma.face}
           \item We have
	           \[
		           \dim (\Ss' \cap |Q|)  \leq \dim (\Ss \cap |Q|)  \text{ and }  \dim \Ss' \leq \dim \Ss.
	           \]
	       \label{I:dim.S'.leq.dim.S}
	  \item If $\rho \in P$ then 
	  \begin{equation}  \label{E:Phi'.rho.subset.Phi.rho}
		   \Phi'(\rho \setminus \Ss') \subset \Phi(\rho \setminus \Ss).
	  \end{equation}
	       \label{I:Phi'.rho.subset.Phi.rho}
           \item If $\overline{\text{St}} \, \sigma = \sigma$ 
           then $\Ss' \setminus \sigma = \Ss \setminus \sigma$, 
           $|P| \setminus \bigl[ (\text{Int} \, \sigma) \cup \Ss' \bigr] 
              \subset |P| \setminus \bigl[ (\text{Int} \, \sigma) \cup \Ss \bigr]$, and $\Phi' = \Phi$ 
           off $(\text{Int} \, \sigma) \cup \Ss'$.    
              \label{I:Phi'.Phi.can.agree.off.sigma}
        \end{enumerate}
   \end{lemma}
(Point \ref{I:Phi'.Phi.can.agree.off.sigma} of the lemma does not seem to be used anywhere.) By point (\ref{I:S'.cap.sigma.=.C.cap.Bd}) of the lemma $\Ss' \cap (\text{Int} \, \sigma) = \varnothing$.  This is the main goal of the pushing out operation.  (See figure \ref{F:SimplexPush}.) Call the operation of replacing $\Phi$ by $\Phi'$ and $\Ss$ by $\Ss'$  ``pushing $\Ss$ out of $\text{Int} \, \sigma$ (from $z$)''.  Note that $\Ss$ might not only be pushed onto $\text{Int} \, \tau$ for $(n-1)$-faces, $\tau$, of $\sigma$, but also possibly onto faces of $\sigma$ of any dimension $< n$.

\subsection{Constructing $\tilde{\Ss}$ and $\tilde{\Phi}$}  \label{SS:constructing.tilde.S.Phi}
Say that a simplex $\sigma \in P$ is a ``partial simplex of $\Ss$'' (or just ``partial simplex'' if $\Ss$ is understood) if $(\text{Int} \, \sigma) \cap \Ss \ne \varnothing$ and $(\text{Int} \, \sigma) \setminus \Ss \ne \varnothing$.  Thus, it is for partial simplices that the pushing out operation was defined.  (See \eqref{E:sigma.intrscts.S}.)  Note that if $\sigma$ is a partial simplex of $\Ss$ and $\sigma$ is a proper face of a simplex $\rho$, then, since $\Ss$ is closed, either $(\text{Int} \, \rho) \cap \Ss = \varnothing$ or $\rho$ is also a partial simplex of $\Ss$.  

Let $q = \dim Q$ and, for $j = 1, \ldots, q$, let $M_{j}$ be the number of  $j$-dimensional simplices in $Q$ and let $\mbf{M}$ be the set of all $q$-tuples $(m_{q}, \ldots, m_{1})$ with $0 \leq m_{j} \leq M_{j}$ 
($j = 1, \ldots, q$).  Order $\mbf{M}$ by lexicographic ordering.  I.e., define 
$(m_{q1}, \ldots, m_{11}) >_{\mbf{M}} (m_{q2}, \ldots, m_{12})$ if and only if the following holds.  
$(m_{q1}, \ldots, m_{11}) \ne (m_{q2}, \ldots, m_{12})$ and if $j = i$ is the largest $i = 1, \ldots, q$ s.t.\ $m_{i1} \ne m_{i2}$ then $m_{j1} > m_{j2}$. The relation $>_{\mbf{M}}$ total orders $\mbf{M}$.

Describe $\Ss$'s by elements of $\mbf{M}$ as follows.  Note that there is no such thing as a partial 0-dimensional simplex.  Let $\boldsymbol{\mu}(\Ss) = \bigl( m_{q}(\Ss), \ldots, m_{1}(\Ss) \bigr) \in \mbf{M}$, where, for $j = 1, \ldots, q$, $m_{j}(\Ss)$ is the number of $j$-dimensional partial simplices of $\Ss$ in $Q$.   Note that $\boldsymbol{\mu}$ maps the collection of compact subsets of $|P|$ onto $\mbf{M}$.  To see this, let $B \subset Q$ (but $B$ does not have to be a subcomplex).  For every $\sigma \in B$, pick a point $x_{\sigma} \in \text{Int} \, \sigma$.  Then $\Ss = \{ x_{\sigma} \in |Q| : \sigma \in B \}$ is a closed set whose set of partial simplices is $B$. So $\boldsymbol{\mu}$ is surjective, but it is not injective.  E.g., $\boldsymbol{\mu}$ maps both $|P|$ and $\varnothing$ to $(0, \ldots, 0)$.  Since $>_{\mbf{M}}$ is a total ordering of $\mbf{M}$, each $\mbf{m} \in \mbf{M}$ has a unique rank.  Thus, $(0, \ldots, 0)$ is the unique element of $\mbf{M}$ with rank 1, etc.  If $\Ss \subset |P|$ is closed, let $\mbf{rank}(\Ss)$ be the rank of $\boldsymbol{\mu}(\Ss)$.

If there are no partial simplices of $\Ss$ in $Q$ we are done:  Just take $\tilde{\Ss} = \Ss$ and 
$\tilde{\Phi} = \Phi$. So assume there is at least one partial simplex in $Q$.  Thus, $\mbf{rank}(\Ss) > 1$. Since $Q$ is a finite complex, there is at least one partial simplex in $Q$ of maximal dimension.  If $j = 1, \ldots, q$ is that maximal dimension, then $m_{q}(\Ss) = \cdots m_{j+1}(\Ss) = 0$. Suppose $\sigma \in Q$ is a $j$-dimensional partial simplex and let $\tau \in Q$ have dimension at least $j$ and \emph{not} be a partial simplex of $\Ss$. By  
lemma \ref{L:big.lemma}(\ref{I:S'.cap.sigma.=.C.cap.Bd},\ref{I:simps.in.S.are.also.in.S'},\ref{I:dont.change.sings.on.othr.n.simps},\ref{I:dont.increase.sing.dim.in.big.simps}), the simplex $\sigma$ will not be a partial simplex of $\Ss'$, but neither will $\tau$.  I.e., pushing $\Ss$ out of a partial simplex in $Q$ of highest dimension among the partial simplices of $\Ss$ in $Q$ always results in a $\Ss' \subset |P|$ s.t.\  
	\begin{equation}  \label{E:pushing.reduces.rank}
		\mbf{rank}(\Ss') < \mbf{rank}(\Ss).
	\end{equation} 

Let $\Ss \subset |P|$ be a closed set and $\Phi : |P| \setminus \Ss \to \F$ ($\F$ an arbitrary topological space) be continuous.  Suppose $\dim (\Ss \cap |Q|) \leq a$, where $a \geq 0$. The set $\tilde{\Ss}$ posited in theorem \ref{T:main.theorem} is obtained from $\Ss$ as follows.  If $\Ss$ has no partial simplices in $Q$ then take $\tilde{\Ss} = \Ss$.  Otherwise, begin with a partial simplex, $\sigma \in Q$, of $\Ss$ of maximal dimension and push $\Ss$ out of it from an appropriate $z \in \text{Int} \, \sigma$.  (Defining ``appropriate'' and showing that an appropriate $z$ exists is the business of subsection \ref{SS:bound.magnification.fact}.)  Then by lemma \ref{L:big.lemma}(\ref{I:S'.closed.Phi'.cont.off.S'})
we obtain a closed subset, $\Ss'$ of $|P|$. Moreover, also by point (\ref{I:S'.closed.Phi'.cont.off.S'}) of the lemma, the new map $\Phi'$ is continuous off $\Ss'$ and, by point (\ref{I:dim.S'.leq.dim.S}) of lemma \ref{L:big.lemma}, $\dim (\Ss' \cap |Q|)  \leq \dim (\Ss \cap |Q|) \leq a$.  Now replace $\Ss$ by $\Ss'$ and $\Phi$ by $\Phi'$ and repeat:  Find a partial simplex $\sigma \in Q$ of $\Ss$ of maximal dimension, etc. By \eqref{E:pushing.reduces.rank} this procedure terminates after a finite number of steps resulting in a pair $(\tilde{\Ss}, \tilde{\Phi})$ s.t.\  $\tilde{\Ss}$ has no partial simplices.  Thus, $\mbf{rank}(\tilde{\Ss}) = 1$.  Hence, $\tilde{\Ss} \cap |Q|$ is either empty or is a subcomplex of $|Q|$.  

We prove all but parts (\ref{I:bound.on.S.tilde.vol} and \ref{I:arb.fine.subdivision}) of the theorem by induction on $\mbf{rank}(\Ss)$.  If $\mbf{rank}(\Ss) = 1$, i.e., there are no partial simplices of $\Ss$ 
in $Q$ then the theorem holds with $\tilde{\Ss} = \Ss$ and $\tilde{\Phi} = \Phi$.  Let $r \geq 1$ and assume parts (\ref{I:Phi.tilde.locally.Lip.if.Phi.is}  through \ref{I:Phi.tilde.sigma.subset.Phi.sigma}) of the theorem hold whenever $\mbf{rank}(\Ss) \leq r$.  Suppose $\mbf{rank}(\Ss) = r+1$.  Choose a partial simplex, $\sigma \in Q$, of $\Ss$ having maximal dimension among all partial simplices in $Q$ and push $\Ss$ out 
of $\sigma$ to obtain a new pair ($\Ss'$, $\Phi'$) as in \eqref{E:Phi'.defn} and \eqref{E:S'.defn}.  
By points (\ref{I:S'.closed.Phi'.cont.off.S'} and \ref{I:dim.S'.leq.dim.S}) of lemma \ref{L:big.lemma}, $\Ss = \Ss'$ and $\Phi = \Phi'$ satisfy the hypotheses of theorem \ref{T:main.theorem}.  But by \eqref{E:pushing.reduces.rank}, parts (\ref{I:Phi.tilde.locally.Lip.if.Phi.is}  through \ref{I:Phi.tilde.sigma.subset.Phi.sigma}) of the theorem hold for $\Ss = \Ss'$ and $\Phi = \Phi'$.  Let $\tilde{\Ss}$ and $\tilde{\Phi}$ in the theorem be the corresponding set and map.   

By point (\ref{I:S'.closed.Phi'.cont.off.S'}) of lemma \ref{L:big.lemma}, $\Phi'$ is locally Lipschitz off $\Ss'$ if $\Phi$ is locally Lipschitz off $\Ss$.  Part (\ref{I:Phi.tilde.locally.Lip.if.Phi.is}) of the theorem follows by induction.

Then part (\ref{I:dim.Stilde.no.bggr.thn.dim.S}) of the theorem is immediate from  point (\ref{I:dim.S'.leq.dim.S}) 
of lemma \ref{L:big.lemma} and the induction hypothesis.  Part (\ref{I:S.tilde.subcomp}) is immediate from the induction hypothesis.

Next, we prove part (\ref{I:no.change.off.nbhd.of.S.in.Q}) of theorem \ref{T:main.theorem}.  Let $\sigma \in Q$ be a partial simplex of $\Ss$ of maximal dimension in $Q$ that will be the simplex in $Q$ from which $\Ss$ will be pushed to produce $(\Ss', \Phi')$.  First, we prove that part (\ref{I:no.change.off.nbhd.of.S.in.Q}) holds with $(\Ss', \Phi')$ in place of $(\tilde{\Ss}, \tilde{\Phi})$.  Suppose $\tau \in P$ is s.t.\ if $\rho \in Q$ and $(\text{Int} \, \rho) \cap \Ss \ne \varnothing$ then $\tau \cap \rho = \varnothing$. Now $(\text{Int} \, \sigma) \cap \Ss \ne \varnothing$ because $\sigma$ is a partial simplex of $\Ss$. Therefore, in particular, $\tau \cap \sigma = \varnothing$.  So by point (\ref{I:Phi'.and.Phi.agree.on.Lk.sigma}) of lemma \ref{L:big.lemma}, part (\ref{I:no.change.off.nbhd.of.S.in.Q}) holds with $\tilde{\Phi} = \Phi'$ and $\tilde{\Ss} = \Ss'$.  

Next, observe that if $\rho \in Q$ and 
$(\text{Int} \, \rho) \cap \Ss' \ne \varnothing$ then $\tau \cap \rho = \varnothing$.  For by point (\ref{I:Int.rho.cut.S'.then.cuts.S.or.rho.is.sigma.face}) of the lemma, either $(\text{Int} \, \rho) \cap \Ss \ne \varnothing$ or $\rho$ is a proper face of $\sigma$.  In the former case, $\tau \cap \rho = \varnothing$ by assumption on $\tau$.  In the latter case, 
if $\tau \cap \rho \ne \varnothing$ then $\tau \cap \sigma \ne \varnothing$, which again contradicts the assumption on $\tau$.  Part (\ref{I:no.change.off.nbhd.of.S.in.Q}) of the theorem now follows by induction.
  
Let $\rho \in P \setminus Q$, let $s \geq 0$, and suppose 
$\Hm^{s}\bigl[ \Ss \cap (\text{Int} \, \rho) \bigr] = 0$.  Note that, by \eqref{E:Int.rho.cuts.sigma.then.rho.in.sigma}, $(\text{Int} \, \rho) \cap |Q| = \varnothing$.  Hence, by points (\ref{I:dont.change.sings.on.othr.n.simps} and \ref{I:dont.increase.sing.dim.in.big.simps}) 
of lemma \ref{L:big.lemma} (and \eqref{E:Int.rho.cuts.sigma.then.rho.in.sigma} again), $\Hm^{s}\bigl[ \Ss' \cap (\text{Int} \, \rho) \bigr] = 0$.  Part (\ref{I:no.change.off.Q}) of theorem \ref{T:main.theorem} is immediate from the induction hypothesis.

Notice that if $\dim \tau = a$, then 
$\Hm^{\lfloor a \rfloor}(\Ss' \cap \tau) = \Hm^{\lfloor a \rfloor} \bigl[ \Ss' \cap (\text{Int} \, \tau) \bigr]$.  Part (\ref{I:Ha.tau.=0.if.Ha.all.nbrs.0}) of the theorem is immediate from point (\ref{I:Int.rho.cut.S'.then.cuts.S.or.rho.is.sigma.face}) of the lemma, the induction hypothesis, and the fact that a face of a face of a simplex $\tau$ is a face of $\tau$ itself.  Similarly, part (\ref{I:dist.tween.S.tilde.and.S}) of the theorem is immediate from point (\ref{I:Int.rho.cut.S'.then.cuts.S.or.rho.is.sigma.face}) of lemma \ref{L:big.lemma}.

By the induction hypothesis, part (\ref{I:Phi.tilde.sigma.subset.Phi.sigma}) of the theorem, and point (\ref{I:Phi'.rho.subset.Phi.rho}) of lemma \ref{L:big.lemma}, if $\rho \in P$, we have 
	\[
		\tilde{\Phi}(\sigma \setminus \tilde{\Ss}) \subset \Phi(\sigma \setminus \Ss')  
			\subset \Phi(\rho \setminus \Ss).
	\]
This proves part (\ref{I:Phi.tilde.sigma.subset.Phi.sigma}) of the theorem.
 
In summary of the preceding discussion, parts (\ref{I:Phi.tilde.locally.Lip.if.Phi.is}  through \ref{I:Phi.tilde.sigma.subset.Phi.sigma}) of the theorem are proved.

\subsection{Magnification in one simplex}  \label{SS:magnification.in.1.simp}
The remainder of the proof is taken up with  proving parts (\ref{I:bound.on.S.tilde.vol} 
and \ref{I:arb.fine.subdivision}) of the theorem.  Let $q = \dim Q$.  If $q < a$ then $\Hm^{a}(|Q|) = 0$, so \eqref{E:polyhedral.volume.magnification.factor} holds with any nonnegative $K$.  So suppose 
$0 \leq a \leq q$. By assumption, $\dim (\Ss \cap |Q|) \leq a$.  Suppose $a = 0$.  
Then if $\Hm^{a}(\Ss \cap |Q|) = 0$, then $\Ss \cap |Q|$ is empty, by \eqref{E:A.empty.iff.H0.A.=.0}, and we are done: By subsection \ref{SS:constructing.tilde.S.Phi}, $\tilde{\Phi} = \Phi$ and 
$\tilde{\Ss} = \Ss$.  

Assume then that $\Ss \cap |Q| \ne \varnothing$.  With $a = 0$ we have $\dim (\Ss \cap |Q|) = 0$.  
Let $\sigma \in P$ be a partial simplex of $\Ss$ of maximal dimension, $n$, say. By 
lemma \ref{L:big.lemma}(\ref{I:S'.cap.sigma.=.C.cap.Bd}), 
\eqref{E:hbar.maps.A.to.C.cap.Bd.sigma}, lemma \ref{L:props.of.s} 
(part \ref{I:hbar.locally.Lip.b.Lip}), and lemma \ref{L:loc.Lip.image.of.null.set.is.null}, we have 
$\Hm^{0}(\Ss' \cap \sigma) \leq \Hm^{0}(\Ss' \cap \sigma)$.  If $\tau \in P$ and $\dim \tau > n$, then, by maximality of $n$, $(\text{Int} \, \tau ) \cap \Ss = \varnothing$.  Hence, by lemma \ref{L:big.lemma}(\ref{I:dont.increase.sing.dim.in.big.simps}), we have $(\text{Int} \, \tau ) \cap \Ss = \varnothing$.
If $\tau \in P$ and $\dim \tau \leq n$, then by lemma \ref{L:big.lemma}(\ref{I:dont.change.sings.on.othr.n.simps}), we have 
$\Hm^{0} \bigl[ \Ss' \cap (\tau \setminus \sigma) \bigr] 
= \Hm^{0} \bigl[ \Ss \cap (\tau \setminus \sigma) \bigr]$.  Summing up, we have 
$\Hm^{0} (\Ss') \leq \Hm^{0} (\Ss)$.  Therefore, as we recursively construct $\tilde{\Ss}$ from $\Ss$ as described in subsection \ref{SS:constructing.tilde.S.Phi}, at each stage the $\Hm^{a}$-measure 
(= $\Hm^{0}$-measure) of $\Ss'$ is not increased, so \eqref{E:polyhedral.volume.magnification.factor} holds with any $K \geq 1$. So assume 
	\begin{equation}  \label{E:a.is.positive}
		a > 0.
	\end{equation}

Suppose $a = q = \dim Q$.  
First, suppose that some $q$-simplex of $Q$ lies in $\Ss$.  
Then \eqref{E:polyhedral.volume.magnification.factor} holds with 
   \begin{equation}  \label{E:K.for.a=q}
      K = K_{1} 
         = \frac{\Hm^{q}(|Q|)}{\min \bigl\{ \Hm^{q}(\omega) : 
             \omega \text{ is a $q$-simplex of $Q$} \bigr\} }
	      \geq 1.
   \end{equation}
Suppose no $q$-simplex of $Q$ lies in $\Ss$.  Then when we recursively push $\Ss$ out of all its partial simplices as described in subsection \ref{SS:constructing.tilde.S.Phi}, the intersection of the resulting $\tilde{\Ss}$ with $|Q|$ will lie 
in $|Q^{(a-1)}|$.  Then $\Hm^{q}(\tilde{\Ss} \cap |Q|) = \Hm^{q}(\tilde{\Ss} \cap |Q^{(a-1)}|) = 0$ so we can still use any nonnegative $K$.  Thus, we may assume 
	\begin{equation}  \label{E:0.<.a.<.q}
		0 < a < q.
	\end{equation}

We may assume $a$ is an integer: Recall that $\lfloor a \rfloor$ is the largest integer less than or equal to $a$.  Suppose $\lfloor a \rfloor < a$.  Then if $\sigma \in Q$ and $\dim \sigma >  \lfloor a \rfloor$ then either $\Ss \cap \sigma = \varnothing$ or $\sigma$ is a partial simplex of $\Ss$.  Hence, by recursively applying the pushing out operation we eventually get $\tilde{\Ss}$ in the $\lfloor a \rfloor$-skeleton of $P$.  But the $\Hm^{a}$ measure of this skeleton is  0 so \eqref{E:polyhedral.volume.magnification.factor} holds for any $K \geq 0$.  So assume $a$ is an integer.

If $\Hm^{a}(\Ss \cap |Q|) = \infty$ then \eqref{E:polyhedral.volume.magnification.factor} holds 
for any $K > 0$.  So assume $\Hm^{a}(\Ss \cap |Q|) < \infty$.  Recall (see \eqref{E:defn.of.set.A}) 
$\mcl{A} = \Ss \cap \sigma$. Since $\Hm^{a}(\Ss \cap |Q|) < \infty$ we have
	\begin{equation}  \label{E:A.has.finite.Ha.measure}
		\Hm^{a}(\mcl{A}) < \infty.
	\end{equation}

\emph{Claim:} If $\Hm^{a}\bigl[ \Ss \cap (\text{Int} \, \sigma) \bigr] = 0$, 
then $\Hm^{a} \bigl( \Ss' \cap |Q| \bigr) \leq \Hm^{a} \bigl( \Ss \cap |Q| \bigr)$.  To see this, assume 
$\Hm^{a}\bigl[ \Ss \cap (\text{Int} \, \sigma) \bigr] = 0$ and note that, by points (\ref{I:simps.in.S.are.also.in.S'},\ref{I:dont.increase.sing.dim.in.big.simps}) of lemma \ref{L:big.lemma} and the fact that $\sigma$ has maximal dimension ($n$) among all partial simplices, we have
	\[
		\Hm^{a} \Bigl( \bigl[ \Ss' \cap |Q| \bigr] \setminus |Q^{(n)}| \Bigr) 
			= \Hm^{a} \Bigl( \bigl[ \Ss \cap |Q| \bigr] \setminus |Q^{(n)}| \Bigr).
	\]
Hence, it suffices to show $\Hm^{a}\bigl( \Ss' \cap |Q^{(n)}| \bigr) \leq \Hm^{a}\bigl( \Ss' \cap |Q^{(n)}| \bigr)$.  But point (\ref{I:dont.change.sings.on.othr.n.simps}) of lemma \ref{L:big.lemma} then implies that it suffices to show $\Hm^{a}( \Ss' \cap \sigma) \leq \Hm^{a}(\Ss \cap \sigma)$. By point (\ref{I:S'.cap.sigma.=.C.cap.Bd}) of lemma \ref{L:big.lemma}, \eqref{E:hbar.maps.A.to.C.cap.Bd.sigma}, \eqref{E:hbar.is.idendt.on.Bd.sigma}, 
lemma \ref{L:props.of.s} part (\ref{I:hbar.locally.Lip.b.Lip}), and lemma \ref{L:loc.Lip.image.of.null.set.is.null} we have
	\begin{align*}
		\Hm^{a} ( \Ss' \cap \sigma) 
		  &\leq \Hm^{a} \bigl[ \Ss \cap (\text{Bd} \, \sigma) \bigr]  
			+ \Hm^{a} \Bigl( \bar{h} \bigl[ \Ss \cap (\text{Int} \, \sigma) \bigr] \Bigr) \\
		  &= \Hm^{a} \bigl[ \Ss \cap (\text{Bd} \, \sigma) \bigr] \\
		  &\leq \Hm^{a} ( \Ss \cap \sigma). 
	\end{align*}
The claim follows. Therefore, assume
	\begin{equation}  \label{E:Ha.S.cap.Int.sigma.poz}
		\Hm^{a}\bigl[ \Ss \cap (\text{Int} \, \sigma) \bigr] > 0; 
		   \text{ in particular, } n \geq a.
	\end{equation}

In our construction of $\tilde{\Phi}$, $\tilde{\Ss}$ in subsection \ref{SS:constructing.tilde.S.Phi}, we only applied the pushing operation in partial simplices of maximal dimension.  Assume, therefore, that $n = \dim \sigma$ is no smaller than the dimension of any partial simplex of $\Ss$ in $Q$. 
Let $x \in |Q| \setminus \sigma$. Then by \eqref{E:x.in.exctly.1.simplex.intrr}, there exists $\rho \in Q$ s.t.\ $x \in \text{Int} \, \rho$.    
If $\dim \rho > n$, then, by maximality of $n$, either $(\text{Int} \, \rho) \cap \Ss = \varnothing$ or 
$(\text{Int} \, \rho) \cap \Ss = \text{Int} \, \rho$.  Since $x \notin \sigma$, $\rho$ is not a face of $\sigma$. Hence, by \eqref{E:Int.rho.cuts.sigma.then.rho.in.sigma}, $\text{Int} \, \rho \subset \rho \setminus \sigma$. 
But by lemma \ref{L:big.lemma}(\ref{I:simps.in.S.are.also.in.S'},\ref{I:dont.change.sings.on.othr.n.simps},\ref{I:dont.increase.sing.dim.in.big.simps}),  we have
	\[
		\text{ If } \rho \in Q \setminus \{ \sigma \} \text{ is not a face of } \sigma  
		         \text{ then } (\text{Int} \, \rho) \cap \Ss' = (\text{Int} \, \rho) \cap \Ss.
	\]  
Therefore, $(\Ss' \setminus \sigma) \cap |Q| = (\Ss \setminus \sigma) \cap |Q|$. So to compare 
$\Hm^{a} \bigl(\Ss' \cap |Q| \bigr)$ to $\Hm^{a} \bigl( \Ss \cap |Q| \bigr)$ it suffices to compare
$\Hm^{a} \bigl( \Ss' \cap \sigma \bigr)$ to $\Hm^{a} \bigl( \Ss \cap \sigma \bigr)$.
But $\Hm^{a} \bigl[ \Ss' \cap (\text{Int} \, \sigma) \bigr] = 0$ by lemma \ref{L:big.lemma}(\ref{I:S'.cap.sigma.=.C.cap.Bd}).  So we only need consider the impact that pushing $\Ss$ out 
of $\sigma$ has on $\text{Bd} \, \sigma$. 

If $n=a$ then for any $z_{0} \in \text{Int} \, \sigma$ we have 
	\[
		\Ss' \cap (\text{Bd} \, \sigma)
			\subset |Q|^{(a-1)}
	\]
so $\Hm^{a} \bigl[ \Ss' \cap  (\text{Bd} \, \sigma ) \bigr] = 0$.  Hence, after pushing, 
$\Hm^{a} \bigl( \Ss' \cap |Q| \bigr) \leq \Hm^{a} \bigl( \Ss \cap |Q| \bigr)$.  
So assume $n > a$. Thus, by \eqref{E:a.is.positive} 
   \begin{equation}  \label{E:n.>.a.>.0}
      a \text{ is an integer and } q \geq n > a > 0.
   \end{equation} 

We will find a number $\phi < \infty$, that depends only on $a$ and $Q$ but not on $\Ss$ s.t.\ we can always find a $z_{0} = z_{0}(\sigma, \Ss)  \in \text{Int} \, \sigma$ satisfying
	\begin{multline}  \label{E:meas.of.push.out.from.z0.is.bdd.above}
		\Hm^{a} \Bigl( \bar{h}_{z_{0}, \sigma} \bigl[ \mcl{A} \cap (\text{Int} \, \sigma) \bigr] 
			\cap \tau \Bigr)
		   = \Hm^{a} \Bigl( h \bigl[\mcl{A} \cap \zeta(z;\tau) \cap (\text{Int} \, \sigma); 
		                   z, \tau \bigr] \Bigr)  \\
		     \leq \phi \, \Hm^{a} \bigl[ \mcl{A} \cap (\text{Int} \, \sigma) \bigr]
		       \text{ for every $(n-1)$-face, $\tau$, of } \sigma.
	\end{multline}
(See \eqref{E:b.hbar.defn}, \eqref{E:b.h.tau.defns},  \eqref{E:zeta.z.defn}, and \eqref{E:when.hbar.and.h.agree}.)  
Observe that \eqref{E:meas.of.push.out.from.z0.is.bdd.above} implies that 
	\[
		\Hm^{a} \Bigl( \bar{h}_{z_{0}, \sigma} \bigl[ \mcl{A} \cap (\text{Int} \, \sigma) \bigr] 
			\cap \tau \Bigr)
			     \leq \phi \, \Hm^{a} \bigl[ \mcl{A} \cap (\text{Int} \, \sigma) \bigr]
	\]
holds for any proper face, $\tau$, of  $\sigma$.  However, since every proper face of $\sigma$ lies in an $(n-1)$-face, we need only consider $(n-1)$-faces $\tau$.

First, we bound above $\Hm^{a} \Bigl( h \bigl[\mcl{A} \cap \zeta(z;\tau); z, \tau \bigr] \Bigr)$ 
for $z \in \text{Int} \, \sigma$. Since by \eqref{E:h.z.tau.is.loc.Lip.on.zeta} $h_{z,\tau}$ is Lipschitz on 
$\mcl{A} \cap \zeta(z)$, by lemma \ref{L:loc.Lip.image.of.null.set.is.null}, we have
$\Hm^{a} \Bigl( h \bigl[\mcl{A} \cap \zeta(z;\tau) \cap (\text{Int} \, \sigma); \; z, \tau \bigr] \Bigr) = 0$ 
if $\Hm^{a} \bigl[ \mcl{A} \cap (\text{Int} \, \sigma) \bigr]  = 0$.  Thus, by \eqref{E:h.is.identity.on.tau}
   \begin{equation}   \label{E:when.Ha.hA.=.Ha.A}
      \Hm^{a} \Bigl( h \bigl[\mcl{A} \cap \zeta(z;\tau); \; z, \tau \bigr] \Bigr) 
         = \Hm^{a}(\mcl{A}), \quad \text{ if } \Hm^{a} \bigl[ \mcl{A} \cap (\text{Int} \, \sigma) \bigr] = 0.
   \end{equation}
I.e., if $\Hm^{a} \bigl[ \mcl{A} \cap (\text{Int} \, \sigma) \bigr]  = 0$ then \eqref{E:meas.of.push.out.from.z0.is.bdd.above} holds for any $\phi > 0$. So assume $\Hm^{a} \bigl[ \mcl{A} \cap (\text{Int} \, \sigma) \bigr]  > 0$.  In summary, by \eqref{E:A.has.finite.Ha.measure}, we may assume
   \begin{equation}   \label{E:Ha.A.cap.sig.circ.poz.not.infte}
      0 < \Hm^{a} \bigl[ \mcl{A} \cap (\text{Int} \, \sigma) \bigr]  < \infty.
   \end{equation}
(See \eqref{E:Ha.S.cap.Int.sigma.poz}.)

We apply lemma \ref{L:bound.on.Haus.meas.h.A} to $h = h_{z,\tau}$.  Note that $n > n-1 \geq a > 0$ by \eqref{E:n.>.a.>.0}.  (In particular, $n \geq 2$.) If $y = (y_{1}, \ldots, y_{n})$, write $y^{n} = (y_{1}, \ldots, y_{n-1})$, the $(n-1)$-dimensional row vector obtained from $y$ by dropping the last coordinate. Moreover, by \eqref{E:yn.<.zn.if.y.in.zeta} if $y \in \zeta(z)$ then $y_{n} < z_{n}$.  
Interpreting $h_{z,\tau}$ as a map into $\RR^{n-1}$ (see \eqref{E:sigma.in.R.n.tau.in.R.n-1}), 
\eqref{E:h.in.terms.of.y.z} becomes
   \begin{equation}   \label{E:exprssn.for.pushed.y.in.R.n-1}
      h_{z,\tau}(y) = \frac{z_{n}}{z_{n} - y_{n}} (y^{n} - z^{n}) + z^{n}.
   \end{equation}
The formula \eqref{E:exprssn.for.pushed.y.in.R.n-1} defines a point of $\RR^{n-1} \subset \RR^{n}$ for $y$ in the open superset, $\mcl{U} := \{ w \in \text{Int} \, \sigma : w_{n} < z_{n} \}$, 
of $\zeta(z) \cap (\text{Int} \, \sigma)$. On $\mcl{U}$ we have
   \[ñ
      Dh_{z,\tau}(y) = \left( \frac{z_{n}}{z_{n} - y_{n}} I_{n-1}, 
         \quad \frac{z_{n}}{(z_{n} - y_{n})^{2}} (y^{n} - z^{n})^{T} \right)^{(n-1) \times n},
   \]
where $I_{m}$ is the $m \times m$ identity matrix ($m = 1, 2, \ldots$) and ``${}^{T}$'' indicates matrix transposition.    Therefore,
	\begin{align*}
	Dh_{z,\tau}(y)^{T} Dh_{z,\tau}(y) &=
		\begin{pmatrix}
		\frac{z_{n}^{2}}{(z_{n} - y_{n})^{2}} I_{n-1} & \frac{z_{n}^{2}}{(z_{n} - y_{n})^{3}} 
		          (y^{n} - z^{n})^{T} \\
		 \frac{z_{n}^{2}}{(z_{n} - y_{n})^{3}} (y^{n} - z^{n}) &  \frac{z_{n}^{2}}{(z_{n} - y_{n})^{4}} 
		                                                                                                                  |y^{n} - z^{n}|^{2}
		\end{pmatrix}   \\
		&= \frac{z_{n}^{2}}{(z_{n} - y_{n})^{2}}
		\begin{pmatrix}
		 I_{n-1} & (z_{n} - y_{n})^{-1} (y^{n} - z^{n})^{T} \\
		 (z_{n} - y_{n})^{-1} (y^{n} - z^{n}) & (z_{n} - y_{n})^{-2} |y^{n} - z^{n}|^{2}
		\end{pmatrix}
	.
	\end{align*}
The vector $\bigl(-(y^{n} - z^{n}), (z_{n} - y_{n}) \bigr)^{T}$ is an eigenvector 
of $Dh_{z,\tau}(y)^{T} Dh_{z,\tau}(y)$ with eigenvalue 0.  
The vector $\bigl( y^{n} - z^{n}, (z_{n} - y_{n})^{-1} |y^{n} - z^{n}|^{2} \bigr)^{T} \in \RR^{n}$ is also an eigenvector with 
eigenvalue 
   \[
      \lambda(y;z)^{2} := \frac{|z - y|^{2} z_{n}^{2}}{ (z_{n} - y_{n})^{4}}.
   \]

If $n=2$ then we have accounted for all eigenvalues of $Dh_{z,\tau}(y)^{T} Dh_{z,\tau}(y)$.  If $n > 2$ and $v \in \RR^{n-1}$ is any non-zero row vector in $\RR^{n-1}$ that is perpendicular 
to $y^{n} - z^{n}$, then $(v, 0)^{T}$ is an eigenvector of $Dh(y;z,\tau)^{T} \, Dh(y;z,\tau)$ with eigenvalue $z_{n}^{2}/ (z_{n} - y_{n})^{2}$. Since the space of all $(n-1)$-vectors that are perpendicular to $y^{n} - z^{n}$ is $(n-2)$-dimensional, we have again accounted for all eigenvalues of $Dh_{z,\tau}(y)^{T} Dh_{z,\tau}(y)$.

If $y \in \zeta(z)$, then, by \eqref{E:zeta.z.defn}, $y = b x + (1 - b) z$ for some 
$x \in \tau, \; 0 < b \leq 1$.  Since $\tau \subset \RR^{n-1}$, we have $x_{n} = 0$.  Therefore, 
   \[
      z - y = b \, (z-x) = b \, (z^{n} - x^{n}, z_{n}).
   \]
Hence, $z_{n} - y_{n} = b z_{n}$.  Moreover, $|z-x|^{2} = |z^{n} - x^{n}|^{2} + z_{n}^{2} \geq z_{n}^{2}$.  Thus, whether $n > 2$ or not, if $y \in \zeta(z)$
   \begin{equation}  \label{E:lambda.yz.facts}
      \lambda(y;z)^{2} = \frac{|z - y|^{2} z_{n}^{2}}{ (z_{n} - y_{n})^{4}}
      	     = \frac{ b^{2} |z - x|^{2} z_{n}^{2} }{ (z_{n} - y_{n})^{2} (z_{n} - y_{n})^{2} } 
            = \frac{b^{2} |z - x|^{2} z_{n}^{2}}{ b^{2} z_{n}^{2} (z_{n} - y_{n})^{2}}
            = \frac{ |z - x|^{2} }{(z_{n} - y_{n})^{2}}
            \geq \frac{ z_{n}^{2} }{(z_{n} - y_{n})^{2}}.
   \end{equation}
Thus, the largest eigenvalue of $Dh(y;z,\tau)^{T} \, Dh(y;z,\tau)$ for $y \in \zeta(z)$ is 
$\lambda(y;z)^{2}$.   
From \eqref{E:lambda.yz.facts} we also conclude
   \begin{equation}   \label{E:bnd.on.bggst.eigval.of.h}
      \lambda(y;z) = \frac{ |z - x| }{z_{n} - y_{n}}
         \leq \frac{diam(\sigma)}{ z_{n} - y_{n}}, \quad \quad y \in \zeta(z; \tau).
   \end{equation}

Recall that the domain of $h_{z,\tau} = h( \cdot; z, \tau)$ is $\zeta(z; \tau)$.  
(See \eqref{E:b.h.tau.defns} and \eqref{E:zeta.z.defn}.)  Therefore, by \eqref{E:Ha.A.cap.sig.circ.poz.not.infte}, lemma \ref{L:bound.on.Haus.meas.h.A}, and \eqref{E:bnd.on.bggst.eigval.of.h}
	\begin{equation}  \label{E:integral.bnd.on.Hm.h.A}
	     \Hm^{a} \Bigl( h \bigl[ \mcl{A} \cap (\text{Int} \, \sigma) \cap \zeta(z; \tau) ; z, \tau \bigr] \Bigr)   
		     \leq \int_{\mcl{A} \cap (\text{Int} \, \sigma) \cap \zeta(z; \tau)}
		        \left( \frac{diam(\sigma)}{ z_{n} - y_{n}} \right)^{a} \, \Hm^{a}(dy),
		            \quad z \in \text{Int} \, \sigma.
	\end{equation}
(The integral in the preceding and other integrals we will encounter in subsection \ref{SS:bound.magnification.fact} look potential theoretic (Falconer \cite[section 4.3]{kF90} and Hayman and Kennedy \cite[Section 5.4.1, pp.\ 225--229]{wkHpbK76.SubharmFns}), but we only need elementary methods.)
 
\subsection{Bound on average magnification factor and existence of good point from which to push}
	\label{SS:bound.magnification.fact}
Let $v(0), \ldots, v(n) \in \RR^{n}$ be the vertices of $\sigma$.  By \eqref{E:barycenter.of.sigma}
	\begin{equation}  \label{E:barycent.of.sig}
      		\hat{\sigma} = \frac{1}{n+1} \sum_{j=0}^{n} v(j)
	\end{equation}
is the barycenter of $\sigma$.  By \eqref{E:n.>.a.>.0} we have $n+1 \geq 3$.
Let 
   \begin{equation}  \label{E:intrvl.of.defn.of.gamma}
      \gamma \in \left( \tfrac{1}{2(n+1)}, 1 \right)
   \end{equation}
be a constant (to be determined later) and define a simplex lying inside and ``concentric'' 
with $\sigma$ as follows.
	\begin{equation}  \label{E:formla.for.sigma.gamma}
	      \sigma_{\gamma} = \bigl\{ \gamma x + (1 - \gamma) \hat{\sigma} : x \in \sigma \bigr\}
	          \subset \text{Int} \, \sigma,
	\end{equation}
by \eqref{E:criterion.for.int.simp}. The vertices of $\sigma_{\gamma}$ are just 
$\gamma v(j) + (1 - \gamma) \hat{\sigma}$ ($j = 0, \ldots, n$).  If
$y = \sum_{j=0}^{n} \beta_{j}(y) v(j) \in \sigma_{\gamma}$, where the $\beta_{j}(y)$'s 
are nonnegative and sum to 1, then for some nonnegative $\hat{\beta}_{0}, \ldots, \hat{\beta}_{n}$ summing to 1 we have 
	\[
		\beta_{j} = \gamma \hat{\beta}_{j} + \frac{1-\gamma}{n+1}, \quad (j = 0, \ldots, n).
	\]
Therefore, 
   \begin{equation}    \label{E:beta.bnds.for.sigma.gamma}
      \frac{1 - \gamma}{n+1} \leq \beta_{j}(y) \leq  \gamma + \frac{1-\gamma}{n+1}, 
         \quad j = 0, \ldots, n.  \quad \text{ (For each $j$ these inequalities are tight.)}
   \end{equation}

We compute the average, over $z \in \sigma_{\gamma}$, of the right hand side (RHS)
in \eqref{E:integral.bnd.on.Hm.h.A}.  Let $\mcl{L}^{k}$ denote $k$-dimensional Lebesgue measure $(k = 1, 2, \ldots)$.  In the following calculation we employ the product measure theorem and Fubini's theorem (Ash \cite[Section 2.6]{rbA72}).  This is justified because we employ either Lebesgue measure or the product measure 
$\bigl( \Hm^{a} \vert_{\mcl{A} \cap (\text{Int} \, \sigma)} \bigr) \times \mcl{L}^{n}$ and the restriction 
$ \Hm^{a} \vert_{\mcl{A} \cap (\text{Int} \, \sigma)}$ is a finite measure by \eqref{E:A.has.finite.Ha.measure}.  

First, we show that certain subsets of $\RR^{2n}$ are Borel measurable. 
Recall that if $z \in \RR^{n}$ then we write 
$z = (z^{n}, z_{n})$ with $z^{n} \in \RR^{n-1}$ and $z_{n} \in \RR$.  
Let $G : (\text{Int} \, \sigma) \times \sigma \to \sigma \times \RR \times \RR^{n} \times \sigma$ be defined by
	\begin{equation}  \label{E:Gzzy.defn}
	     G(z,y) = G(z^{n}, z_{n}, y) = \left( (z^{n}, z_{n}), \frac{z_{n} - y_{n}}{z_{n}}, 
	        \frac{z_{n}}{z_{n} - y_{n}} \left( y - \frac{y_{n}}{z_{n}} (z^{n}, z_{n}) \right), y \right).
	\end{equation}
(Recall that, by \eqref{E:yn.<.zn.if.y.in.zeta}, $z \in  \text{Int} \, \sigma$ implies $z_{n} > 0$.  
Define $\frac{z_{n}}{z_{n} - y_{n}} \left( y - \frac{y_{n}}{z_{n}} z \right) = 0 \in \RR^{n}$ 
if $z_{n} - y_{n} = 0$.)  Then $G$ is Borel measurable. Hence, the set
	\[
		\bigl\{ (z, y) \in \sigma_{\gamma} \times ( (\text{Int} \, \sigma) \cap \mcl{A}) 
		    : y \in \zeta(z; \tau) \bigr\} 
			= G^{-1} \bigl[ \sigma_{\gamma} \times (0,1) \times \tau \times \mcl{A} \bigr]
	\]
is Borel measurable.  (See \eqref{E:yn.<.zn.if.y.in.zeta}, \eqref{E:zeta.z.defn}, \eqref{E:b.in.terms.of.y.z}, \eqref{E:h.in.terms.of.y.z}, and \eqref{E:condition.for.y.in.zeta}.) If $S$ is a set define the ``indicator'', $1_{S}$, to be the function
	\[
		1_{S}(x) =
			\begin{cases}
				1, \text{ if } x \in S, \\
				0, \text{ otherwise.}
			\end{cases}
	\]
(An indicator function is often called a ``characteristic function''.)
Thus, 
	\begin{equation}  \label{E:Borel.indicator}
		\text{the indicator function } 
			1_{\bigl\{ (z, y) \in \sigma_{\gamma} \times ( (\text{Int} \, \sigma) \cap \mcl{A}) 
				    : y \in \zeta(z; \tau) \bigr\} } \text{ is Borel measurable.}
	\end{equation}

Similarly, if $y \in \sigma$ and $z_{n} > 0$, the set 
   \begin{equation}  \label{E:measurability.of.2.zeta.sets}
         Z(y,z_{n}) := \bigl\{ z^{n} \in \RR^{n-1}: 
         (z^{n}, z_{n}) \in \sigma_{\gamma} \text{ and } y \in \zeta(z^{n}, z_{n}) \bigr\} 
   \end{equation}
is just the $z_{n}$-section (followed by projection onto the first factor) of the Borel set
	\[
		G^{-1} \bigl[ \sigma_{\gamma} \times (0,1] \times \tau \times \{ y \} \bigr].
	\]
(For some choices of $y, z_{n}$ we have $Z(y,z_{n}) = \varnothing$.) Therefore, for $y \in \text{Int} \, \sigma$ fixed, $\mcl{L}^{n-1}  \bigl[ Z(y,z_{n}) \bigr]$ is Borel measurable in $z_{n} \in \RR$.

Let $z \in \text{Int} \, \sigma$. Recall that $v(n)$ is the vertex of $\sigma$ opposite $\tau$.  Let 
   \begin{equation}   \label{E:zn,max.defn}
      z_{n,max} = z_{n,max}(\sigma, \tau) = \max \{ z_{n} \geq 0: z \in \sigma \}.
   \end{equation} 
(See figure \ref{F:SimplexPush}.) This is just the $n^{th}$ coordinate of $v(n)$.  
I.e., $z_{n,max} = v_{n}(n)$.
Since $\tau \subset \RR^{n-1} \subset \RR^{n}$, the $n^{th}$ coordinates of $v(0), \ldots, v(n-1)$ are 0. Therefore, $z_{n} = \beta_{n}(z) z_{n,max}$.  Let $z \in \sigma_{\gamma}$. Then multiplying \eqref{E:beta.bnds.for.sigma.gamma}, with $y = z$ and $j=n$, through by $z_{n,max}$ we get the following tight inequalities.
   \begin{equation}   \label{E:lwr.uppr.bnds.on.zn}
      \frac{1 - \gamma}{n+1} z_{n,max} \leq z_{n} \leq \left( \gamma + \frac{1-\gamma}{n+1} \right) 
          z_{n,max} = \frac{n \gamma + 1}{n+1} z_{n,max},
          \text{ for every } z \in \sigma_{\gamma}.
   \end{equation}
   
Recalling \eqref{E:Ha.A.cap.sig.circ.poz.not.infte}, \eqref{E:lwr.uppr.bnds.on.zn}, and \eqref{E:yn.<.zn.if.y.in.zeta}, 
we can bound the integrated right hand side of \eqref{E:integral.bnd.on.Hm.h.A} as follows.
   \begin{multline}  \label{E:int.Ha.h.A.bnd.ovr.sigma.gamma}
      \int_{\sigma_{\gamma}} \int_{\mcl{A} \cap (\text{Int} \, \sigma) \cap \zeta(z; \tau)}
		        \left( \frac{diam(\sigma)}{ z_{n} - y_{n}} \right)^{a} \, \Hm^{a}(dy) \, dz  \\
      \begin{aligned}
         {} 
            &= \int_{\mcl{A} \cap (\text{Int} \, \sigma)} \int_{ \{ z \in \sigma_{\gamma} : y \in \zeta(z) \} } 
               \left( \frac{diam(\sigma)}{ z_{n} - y_{n} } \right)^{a} \, dz \, \Hm^{a}(dy)  \\
            &= diam(\sigma)^{a} 
                \int_{\mcl{A} \cap (\text{Int} \, \sigma)} \int_{ \{ z \in \sigma_{\gamma} : y \in \zeta(z) \} } 
                       (z_{n} - y_{n})^{-a} \, dz \, \Hm^{a}(dy)  \\
            &= diam(\sigma)^{a} \int_{\mcl{A} \cap (\text{Int} \, \sigma)} 
               \int_{ \max \left\{ \frac{1 - \gamma}{n+1} z_{n,max} , y_{n} \right\} }
                   ^{\frac{n \gamma + 1}{n+1} z_{n,max} } 
                  \int_{ Z(y,z_{n}) } (z_{n} - y_{n})^{-a} \, dz^{n} \, dz_{n} \, \Hm^{a}(dy)  \\
            &\leq diam(\sigma)^{a} \int_{\mcl{A} \cap (\text{Int} \, \sigma)} 
               \int_{y_{n}}^{z_{n,max}} 
                  \int_{ Z(y,z_{n}) } (z_{n} - y_{n})^{-a} \, dz^{n} \, dz_{n} \, \Hm^{a}(dy)  \\
            &= diam(\sigma)^{a} \int_{\mcl{A} \cap (\text{Int} \, \sigma)} \int_{y_{n}}^{z_{n,max}} 
                \mcl{L}^{n-1} \bigl[ Z(y,z_{n}) \bigr] (z_{n} - y_{n})^{-a} \, dz_{n} \, \Hm^{a}(dy).
      \end{aligned}
   \end{multline}

If $z = (z^{n}, z_{n}) \in \RR^{n}$ with $z_{n} > 0$, let
	\begin{equation}  \label{E:zeta.star.defn}
		\zeta^{\ast}(z) = \left\{ y \in \RR^{n} : 0 \leq y_{n} < z_{n} 
		     \text{ and } \frac{z_{n}}{z_{n} - y_{n}} \left( y - \frac{y_{n}}{z_{n}} z \right) \in \tau \right\}.
	\end{equation}
Notice that definition \eqref{E:zeta.z.defn} makes sense for any $z = (z^{n}, z_{n}) \in \RR^{n}$ with $z_{n} > 0$.  Therefore, by \eqref{E:h.in.terms.of.y.z}, \eqref{E:condition.for.y.in.zeta}, and \eqref{E:yn.<.zn.if.y.in.zeta}, 
	\[
		\zeta(z) = \zeta^{\ast}(z) \cap \sigma.
	\]
Let $y \in \text{Int} \, \sigma$ and $z_{n} \in (y_{n}, z_{n,max})$.  For $z^{n}$ to be in $Z(y,z_{n})$ two things must happen, \emph{viz.} $y \in \zeta(z^{n}, z_{n})$ and $(z^{n}, z_{n}) \in \sigma_{\gamma}$.  Thus, we can bound $\mcl{L}^{n-1} \bigl[ Z(y,z_{n}) \bigr]$ above by the volume 
of the set $\bigl\{ z^{n} \in \RR^{n-1} : y \in  \zeta^{\ast}(z^{n}, z_{n}) \bigr\}$. We can also bound it above by the volume of 
$\bigl\{ z^{n} \in \RR^{n-1} : (z^{n}, z_{n}) \in \sigma_{\gamma} \bigr\}$.  In summary,
	\begin{multline}  \label{E:L.n-1.Z.leq.min}
		\mcl{L}^{n-1} \bigl[ Z(y,z_{n}) \bigr]   \\
			\leq \min \biggl\{ \mcl{L}^{n-1} \Bigl( \bigl\{ z^{n} \in \RR^{n-1} : 
			     y \in  \zeta^{\ast}(z^{n}, z_{n}) \bigr\} \Bigr), 
			\mcl{L}^{n-1} \Bigl(  \bigl\{ z^{n} \in \RR^{n-1} : 
			    (z^{n}, z_{n}) \in \sigma_{\gamma} \bigr\} \Bigr) \biggr\}.
	\end{multline}
 
From \eqref{E:zeta.star.defn} and interpreting $\tau$ as a subset of $\RR^{n-1}$, we see that 
if $y \in \zeta^{\ast}(z^{n}, z_{n}) \cap (\text{Int} \, \sigma)$ then 
	\[
		z_{n} > y_{n} > 0 \text{ and } 
		     \frac{z_{n}}{z_{n} - y_{n}} \left( y^{n} - \frac{y_{n}}{z_{n}} z^{n} \right) \in \tau .
	\]
I.e., 
   \begin{equation}   \label{E:for.what.z.is.x.in.zeta.z}
      z_{n} > y_{n} > 0 \text{ and } z^{n} \in 
          \frac{z_{n}}{y_{n}} y^{n} - \frac{z_{n} - y_{n}}{y_{n}} \tau .
   \end{equation}
By \eqref{E:Leb.meas.of.affine.trans} in appendix \ref{S:Lip.Haus.meas.dim} the volume of the set defined by the right end of the above is
   \begin{equation}  \label{E:H.n-1.mult.shift.tau}
	      \mcl{L}^{n-1} \left( 
	                               \frac{z_{n}}{y_{n}} y^{n} - \frac{z_{n} - y_{n}}{y_{n}} \tau
                 \right)  
        = \left( \frac{z_{n} - y_{n}}{y_{n}} \right)^{n-1} \mcl{L}^{n-1}( \tau ).
   \end{equation}
Thus, if $y \in \text{Int} \, \sigma$ and $z_{n} \in (y_{n}, z_{n,max})$ we have 
by \eqref{E:L.n-1.Z.leq.min}
	\begin{equation}  \label{E:L.n-1.Z.leq.zn.yn.frac}
		\mcl{L}^{n-1} \bigl[ Z(y,z_{n}) \bigr] 
			\leq \left( \frac{z_{n} - y_{n}}{y_{n}} \right)^{n-1} \mcl{L}^{n-1}( \tau ).
	\end{equation}

By \eqref{E:formla.for.sigma.gamma}, the $(n-1)$-dimensional ``base'' of $\sigma_{\gamma}$ (i.e., the $(n-1)$-face of $\sigma_{\gamma}$ corresponding to $z_{n} = (1-\gamma) z_{n,max}/(n+1)$) 
in \eqref{E:lwr.uppr.bnds.on.zn} is $\gamma \tau + (1-\gamma) \hat{\sigma}$. 
Call this face $\tau_{\gamma}$. Then, by \eqref{E:Leb.meas.of.affine.trans} again, 
$\mcl{L}^{n-1} (\tau_{\gamma}) = \gamma^{n-1} \mcl{L}^{n-1}( \tau )$. 
If $z_{n}$ satisfies \eqref{E:lwr.uppr.bnds.on.zn} the corresponding cross-section, 
$\bigl\{ w \in \RR^{n-1} : (w, z_{n}) \in \sigma_{\gamma} \bigr\}$, 
of $\sigma_{\gamma}$ is a convex combination of the base, $\tau_{\gamma}$, 
of $\sigma_{\gamma}$ and the ``top'' vertex of $\sigma_{\gamma}$, 
which is $\gamma v(n) + (1-\gamma) \hat{\sigma}$.  The cross section volume is given by 
   \begin{align}  \label{E:H.n-1.cross.sect.sigma.gamma}
      \mcl{L}^{n-1} \Bigl( \bigl\{ w \in \RR^{n-1} : (w, z_{n}) \in \sigma_{\gamma} \bigr\} \Bigr) 
         &= \left[ \frac{\frac{n \gamma + 1}{n+1} z_{n,max} - z_{n}}
          {\left( \gamma + \frac{1-\gamma}{n+1} - \frac{1-\gamma}{n+1} \right) z_{n,max}} \right]^{n-1}
          \gamma^{n-1} \mcl{L}^{n-1}( \tau )  \\
        &= \left( \frac{ \frac{n \gamma + 1}{n+1} z_{n,max} - z_{n} } { z_{n,max}} \right)^{n-1}
            \, \mcl{L}^{n-1}( \tau ).  \notag
   \end{align}
Thus, if $y \in \text{Int} \, \sigma$, $y_{n} < z_{n}$, and \eqref{E:lwr.uppr.bnds.on.zn} holds, we have by \eqref{E:L.n-1.Z.leq.min}
	\begin{equation}  \label{E:Ln-1.Z.bnd}
		\mcl{L}^{n-1} \bigl[ Z(y,z_{n}) \bigr] \leq 
		    \left( \frac{ \frac{n \gamma + 1}{n+1} z_{n,max} - z_{n} } { z_{n,max}} \right)^{n-1}
		            \, \mcl{L}^{n-1}( \tau ).
	\end{equation}

Suppose $z \in \sigma_{\gamma}$ and $y \in \text{Int} \, \sigma$ satisfy 
$z_{n} > y_{n} \geq \tfrac{(1 - \gamma) z_{n,max}}{2(n+1)}$.  
Then by \eqref{E:lwr.uppr.bnds.on.zn} and \eqref{E:intrvl.of.defn.of.gamma} we have 
	\begin{align*}
		z_{n} - y_{n} &\leq \left[ \gamma + \frac{1-\gamma}{n+1} - \frac{1-\gamma}{2(n+1)} \right] 
		                    z_{n,max}   \\
			  &= \left[ \gamma + \frac{1-\gamma}{2(n+1)} \right] z_{n,max}  \\
			  &= \gamma \left[ 1 + \frac{1-\gamma}{\gamma} \frac{1}{2(n+1)} \right] z_{n,max} \\
			  &< \gamma \left[ 1 + \frac{1 - \frac{1}{2(n+1)}}{\frac{1}{2(n+1)}} \frac{1}{2(n+1)} \right] 
			        z_{n,max} \\
			  & = \gamma \left( 2 - \frac{1}{2(n+1)} \right) z_{n,max} \\
			  &< 2 \gamma \, z_{n,max}.
	\end{align*}
Therefore, by \eqref{E:n.>.a.>.0} and \eqref{E:L.n-1.Z.leq.zn.yn.frac}, if
$z \in \sigma_{\gamma}$ and $y \in \text{Int} \, \sigma$ satisfy 
$z_{n} > y_{n} \geq \tfrac{(1 - \gamma) z_{n,max}}{2(n+1)}$, then
   \begin{align}  \label{E:H.n-1.bnd.for.large.yn}
      \mcl{L}^{n-1} \bigl[ Z(y,z_{n}) \bigr] (z_{n} - y_{n})^{-a}   
               &\leq \left( \frac{z_{n} -y_{n}}{y_{n}} \right)^{n-1} 
                          (z_{n} - y_{n})^{-a} \, \mcl{L}^{n-1}( \tau )   \notag \\
               &\leq \left( \frac{z_{n} -y_{n}}{\frac{(1 - \gamma) z_{n,max}}{2(n+1)}} \right)^{n-1} 
                          (z_{n} - y_{n})^{-a} \, \mcl{L}^{n-1}( \tau )   \notag \\
               &\leq \frac{(2 \gamma z_{n,max})^{n-a-1}}
                   { \left( \frac{(1 - \gamma) z_{n,max}}{2(n+1)} \right)^{n-1}} \mcl{L}^{n-1}( \tau) \\
               &= \frac{2^{2n-a-2} (n+1)^{n-1} \gamma^{n-a-1} (z_{n,max})^{-a}}{(1-\gamma)^{n-1}} \, \mcl{L}^{n-1}( \tau).    \notag
   \end{align}  

Now suppose $z \in \sigma_{\gamma}$ and 
$0 < y_{n} < \tfrac{(1 - \gamma) z_{n,max}}{2(n+1)} < \tfrac{(1 - \gamma) z_{n,max}}{n+1}$. Then we have by \eqref{E:Ln-1.Z.bnd}, \eqref{E:n.>.a.>.0}, and \eqref{E:lwr.uppr.bnds.on.zn}, 
   \begin{multline}  \label{E:H.n-1.bnd.for.small.yn}
      \mcl{L}^{n-1} \bigl[ Z(y,z_{n}) \bigr] (z_{n} - y_{n})^{-a}  \\
         \begin{aligned}
             {}  &\leq \frac{\left( \frac{n \gamma + 1}{n+1} z_{n,max} - z_{n} \right)^{n-1}}{(z_{n,max})^{n-1}} 
                    \left[ z_{n} - \frac{(1 - \gamma) z_{n,max}}{2(n+1)} \right]^{-a} \,\mcl{L}^{n-1}( \tau )  \\
                  &\leq \frac{\left( \gamma + \frac{1-\gamma}{n+1} -  \frac{1-\gamma}{n+1} \right)^{n-1}
                    (z_{n,max})^{n-1}}{(z_{n,max})^{n-1}} 
                      \left[ \frac{1-\gamma}{n+1} - \frac{1-\gamma}{2(n+1)} \right]^{-a} (z_{n,max})^{-a} 
                      \mcl{L}^{n-1}( \tau )      \\
                 &= \frac{2^{a} (n+1)^{a} \gamma^{n-1}}{(1-\gamma)^{a}} (z_{n,max})^{-a} 
                      \mcl{L}^{n-1}( \tau ).
         \end{aligned}
   \end{multline}
By \eqref{E:n.>.a.>.0} and \eqref{E:intrvl.of.defn.of.gamma}, for any $y \in \text{Int} \, \sigma$ with 
$y_{n} <  z_{n}$,  the final value in \eqref{E:H.n-1.bnd.for.large.yn} is no smaller than that 
in \eqref{E:H.n-1.bnd.for.small.yn}.  Thus, if \eqref{E:lwr.uppr.bnds.on.zn} holds then 
for every $y \in \text{Int} \, \sigma$ with $0 < y_{n} < z_{n}$ we have the following regardless 
if $y_{n}$ is larger or smaller than $\tfrac{(1 - \gamma) z_{n,max}}{2(n+1)}$.
   \begin{equation}   \label{E:smple.bnd.on.H.n-1.elgble.zn}
        \mcl{L}^{n-1} \bigl[ Z(y,z_{n}) \bigr] (z_{n} - y_{n})^{-a}   
                   \leq \frac{2^{2n-a-2} (n+1)^{n-1} \gamma^{n-a-1} (z_{n,max})^{-a}}{(1-\gamma)^{n-1}} \, \mcl{L}^{n-1}( \tau).
   \end{equation}
   
Then from \eqref{E:int.Ha.h.A.bnd.ovr.sigma.gamma} and \eqref{E:smple.bnd.on.H.n-1.elgble.zn} we get 
   \begin{multline}     \label{E:int.Ha.h.A.ovr.sigma.gamma.bound}
       \int_{\sigma_{\gamma}} \int_{\mcl{A} \cap (\text{Int} \, \sigma) \cap \zeta(z; \tau)}
		        \left( \frac{diam(\sigma)}{ z_{n} - y_{n}} \right)^{a} \, \Hm^{a}(dy) dz   \\
             \leq \left( \frac{diam(\sigma)}{z_{n,max}} \right)^{a} z_{n,max} 
                \frac{2^{2n-a-2} (n+1)^{n-1} \gamma^{n} }{(1-\gamma)^{n-1} \gamma^{a+1}} 
                    \mcl{L}^{n-1}(\tau) \, \Hm^{a} \bigl[ \mcl{A} \cap (\text{Int} \, \sigma) \bigr] .
   \end{multline}

Let $r(\sigma)$ be the radius and $t(\sigma)$ the thickness of $\sigma$ (appendix \ref{S:basics.of.simp.comps}).  \emph{Claim:} $r(\sigma) \leq z_{n, max}/(n+1)$. 
By \eqref{E:sigma.in.R.n.tau.in.R.n-1}, \eqref{E:barycent.of.sig},
and \eqref{E:zn,max.defn}, $z_{n, max}/(n+1)$ is the perpendicular distance from the barycenter, $\hat{\sigma}$, 
to $\RR^{n-1}$, thought of as the subspace of $\RR^{n}$ containing $\tau$.  A perpendicular dropped from $\hat{\sigma}$ to $\RR^{n-1}$ intersects $\RR^{n-1}$ at a point $\hat{\sigma}_{0}$, say.  
If $\hat{\sigma}_{0} \in \tau$, then $z_{n, max}/(n+1)$ is the distance from $\hat{\sigma}$ to $\tau$.  If $\hat{\sigma} \notin \tau$ then the line perpendicular to $\RR^{n-1}$ and joining $\hat{\sigma}$ to $\hat{\sigma}_{0}$ must intersect some other face, $\tau'$, of $\sigma$.  The distance from $\hat{\sigma}$ to $\tau'$ 
will be $< z_{n, max}/(n+1)$.  That proves the claim.

It follows that
   \begin{equation}   \label{E:diam.dn.thickness.bnd}
      \frac{ diam(\sigma) }{z_{n,max}} \leq \frac{1}{(n+1) \, t(\sigma)}.
   \end{equation}
Combining \eqref{E:int.Ha.h.A.ovr.sigma.gamma.bound} and \eqref{E:diam.dn.thickness.bnd} we get
   \begin{multline}  \label{E:nicer.int.Ha.h.A.ovr.sigma.gamma.bound}
       \int_{\sigma_{\gamma}} \int_{\mcl{A} \cap (\text{Int} \, \sigma) \cap \zeta(z; \tau)}
		        \left( \frac{diam(\sigma)}{ z_{n} - y_{n}} \right)^{a} \, \Hm^{a}(dy) dz   \\
         \leq \frac{2^{2n-a-2} z_{n,max} (n+1)^{n-a-1} \gamma^{n} }
                   {t(\sigma)^{a} (1-\gamma)^{n-1} \gamma^{a+1}}   
                      \mcl{L}^{n-1}(\tau) \Hm^{a} \bigl[ \mcl{A} \cap (\text{Int} \, \sigma) \bigr] .
   \end{multline}

Let
   \begin{align*}
      X = X(\tau) &= X(\tau, \sigma)    \\
         &:= \biggl\{ z \in \sigma_{\gamma} :  
                \int_{\mcl{A} \cap (\text{Int} \, \sigma) \cap \zeta(z; \tau)} 
		        \left( \frac{diam(\sigma)}{ z_{n} - y_{n}} \right)^{a} \, \Hm^{a}(dy)  \\
		        & \qquad \qquad \qquad \leq 
	              \frac{2^{2n-a-2} n (n+1)^{n-a-1} (n+2)}{(1-\gamma)^{n-1} \gamma^{a+1} \, t(\sigma)^{a}}
	                        \Hm^{a} \bigl[ \mcl{A} \cap (\text{Int} \, \sigma) \bigr]  \biggr\}.
   \end{align*} 
By \eqref{E:Borel.indicator} and \eqref{E:A.has.finite.Ha.measure}, Fubini 
(Ash \cite[Theorem 2.6.4, p.\ 101]{rbA72}) tells us $X$ is Borel measurable.  Suppose
   \begin{equation}   \label{E:Ln.X.<.Ln.sigma}
      \mcl{L}^{n}(X) < \frac{n+1}{n+2} \mcl{L}^{n} ( \sigma_{\gamma} ).
   \end{equation}
   
Now, by \eqref{E:lwr.uppr.bnds.on.zn} and \eqref{E:H.n-1.cross.sect.sigma.gamma},
   \begin{multline}  \label{E:volume.of.sigma.gamma}
      \mcl{L}^{n} ( \sigma_{\gamma} ) 
      = \int_{\tfrac{1-\gamma}{n+1} z_{n, max}}^{ \frac{n \gamma + 1}{n+1} z_{n, max}} 
        \left( \frac{ \frac{n \gamma + 1}{n+1} z_{n,max} - z_{n} } { z_{n,max}} \right)^{n-1}
            \, \mcl{L}^{n-1}( \tau ) \, dz_{n}   \\
      = \gamma^{n} z_{n,max} \, \mcl{L}^{n-1}( \tau )/n > 0.
   \end{multline}
Hence,
	\begin{equation}  \label{E:Leb.meas.sigma.gamma.frac}
		\frac{\gamma^{n} z_{n,max} \, \mcl{L}^{n-1}( \tau )}{n \, \mcl{L}^{n} ( \sigma_{\gamma} )} 
		            = 1.
	\end{equation}
Then by \eqref{E:nicer.int.Ha.h.A.ovr.sigma.gamma.bound}, \eqref{E:Leb.meas.sigma.gamma.frac}, 
\eqref{E:Ln.X.<.Ln.sigma},
\eqref{E:Ha.A.cap.sig.circ.poz.not.infte}, and the definition of $X$
   \begin{multline*}
      \frac{2^{2n-a-2} z_{n,max} (n+1)^{n-a-1} \gamma^{n}}{(1-\gamma)^{n-1} \gamma^{a+1} \, 
                t(\sigma)^{a}}  \mcl{L}^{n-1}(\tau)  \Hm^{a} \bigl[ \mcl{A} \cap (\text{Int} \, \sigma) \bigr]   \\
         \begin{aligned} 
            {} \qquad &\geq \int_{\sigma_{\gamma}} 
              \int_{\mcl{A} \cap (\text{Int} \, \sigma) \cap \zeta(z; \tau)}
		        \left( \frac{diam(\sigma)}{ z_{n} - y_{n}} \right)^{a} \, \Hm^{a}(dy) \, dz    \\
             &\geq \mcl{L}^{n} (\sigma_{\gamma} \setminus X) 
                     \times  \frac{2^{2n-a-2} n (n+1)^{n-a-1} (n+2)}{(1-\gamma)^{n-1} \gamma^{a+1} \, 
                                   t(\sigma)^{a}} \Hm^{a} \bigl[ \mcl{A} \cap (\text{Int} \, \sigma) \bigr]   \\
             &= \mcl{L}^{n} (\sigma_{\gamma} \setminus X) \times 
                (n+2) \frac{2^{2n-a-2} n (n+1)^{n-a-1} }{(1-\gamma)^{n-1} \gamma^{a+1} \, t(\sigma)^{a}}
                   \Hm^{a} \bigl[ \mcl{A} \cap (\text{Int} \, \sigma) \bigr]   \\
             &= \frac{\gamma^{n} z_{n,max} \, \mcl{L}^{n-1}( \tau )}{n \, \mcl{L}^{n} ( \sigma_{\gamma} )}
                \times \mcl{L}^{n} (\sigma_{\gamma} \setminus X) \times 
                   (n+2)  \frac{2^{2n-a-2} n (n+1)^{n-a-1} }{(1-\gamma)^{n-1} \gamma^{a+1} \, t(\sigma)^{a}} 
                     \Hm^{a} \bigl[ \mcl{A} \cap (\text{Int} \, \sigma) \bigr]   \\
             &= \mcl{L}^{n} (\sigma_{\gamma} \setminus X) \times (n+2) \mcl{L}^{n} 
                    ( \sigma_{\gamma} )^{-1} 
                        \gamma^{n} z_{n,max}  \\ 
               & \qquad \qquad \qquad \qquad \qquad \qquad              
                     \frac{2^{2n-a-2} (n+1)^{n-a-1} }{(1-\gamma)^{n-1} \gamma^{a+1} \, t(\sigma)^{a}}
                                 \mcl{L}^{n-1}( \tau ) \Hm^{a} \bigl[ \mcl{A} \cap (\text{Int} \, \sigma) \bigr]   \\
               &> 
             \left( 1 - \frac{n+1}{n+2} \right) \mcl{L}^{n} ( \sigma_{\gamma} ) \times (n+2)
               \mcl{L}^{n} ( \sigma_{\gamma} )^{-1} 
               \gamma^{n} z_{n,max}   \\
               & \qquad \qquad \qquad \qquad  
                    \times  \frac{2^{2n-a-2} (n+1)^{n-a-1} }
                                               {(1-\gamma)^{n-1} \gamma^{a+1} \, t(\sigma)^{a}} \mcl{L}^{n-1}( \tau ) 
                          \Hm^{a} \bigl[ \mcl{A} \cap (\text{Int} \, \sigma) \bigr]   \\
             &= \gamma^{n} z_{n,max} 
                \frac{2^{2n-a-2} (n+1)^{n-a-1} }{(1-\gamma)^{n-1} \gamma^{a+1} \, t(\sigma)^{a}} 
                    \mcl{L}^{n-1}( \tau ) \Hm^{a} \bigl[ \mcl{A} \cap (\text{Int} \, \sigma) \bigr] .
         \end{aligned}
   \end{multline*}
So the extremes of the preceding strict inequality are equal.  Contradiction.  We conclude that \eqref{E:Ln.X.<.Ln.sigma} is false.  I.e., 
   \begin{equation}  \label{E:vol.X.geq.mult.vol.sigma}
      \mcl{L}^{n} \bigl( X(\tau) \bigr) \geq \frac{n+1}{n+2} \mcl{L}^{n} ( \sigma_{\gamma} ) > 0.
   \end{equation}

For $r = a+1, a+2, \ldots$ and $t > 0$, let
   \begin{equation}     \label{E:phi.PR.defn}
      \tilde{\phi}(a,r,t,\gamma) 
         = \frac{2^{2r-a-2} r (r+1)^{r-a-1} (r+2)}{(1-\gamma)^{r-1} \gamma^{a+1} \, t^{a}}.
   \end{equation}
Thus,
   \begin{multline}   \label{E:X.tau.in.terms.of.phi.tilde}
      X(\tau) 
         = \biggl\{ z \in \sigma_{\gamma} :  
                \int_{\mcl{A} \cap (\text{Int} \, \sigma) \cap \zeta(z; \tau)} 
		        \left( \frac{diam(\sigma)}{ z_{n} - y_{n}} \right)^{a} \, \Hm^{a}(dy)  \\
	              \leq \tilde{\phi}(a, n, t(\sigma), \gamma)
	                        \Hm^{a} \bigl[ \mcl{A} \cap (\text{Int} \, \sigma) \bigr]  \biggr\}.
   \end{multline} 

Define $g: \gamma \mapsto (1 - \gamma)^{r-1} \gamma^{a+1} \; $ ($0 \leq \gamma\leq 1$).
Suppose $a$ and $n=r$ satisfy \eqref{E:n.>.a.>.0}.  I.e., suppose 
	\begin{equation}  \label{E:r>a>0}
		r \geq a+1 \text{ and } a \geq 1.
	\end{equation}
Now, $g$ is nonnegative on $[0,1]$ and $g(0) = g(1) = 0$.  Therefore, $g$ must have a local maximum in $(0,1)$. The maximum value of $g$ is achieved at
$\gamma = \tfrac{a+1}{r+a} \in \left( \tfrac{1}{2(r+1)}, 1 \right)$. The maximum value\footnote{
\emph{Claim:}  We have $c_{ra} \geq 2^{-a-r}$. 
The claim is true when $r = 2$ (which means $a = 1$ by \eqref{E:r>a>0}),  the smallest possible value of $r$.  Next, suppose $r \geq 3$ and $r$ and $a$ satisfy \eqref{E:r>a>0}.  Take the logarithm 
of $c_{ra}/2^{-a-r}$ and  divide by $a+r$.
   \[
      \frac{r-1}{a+r} \log(r-1) +  \frac{a+1}{a+r} \log(a+1) - \log (a+r) + \log 2.
   \]
Differentiating this expression w.r.t.\ $r$ we find that it is minimized when $r = a+2$.  The claim is proved.} 
of $g$ is
   \begin{equation*}  
      c_{ra} := g \left( \tfrac{a+1}{r+a} \right) = \frac{(r-1)^{r-1} (a+1)^{a+1}}{(r+a)^{r+a}}.
   \end{equation*} 
Let 
	\begin{equation}   \label{E:phi.a.r.t.defn}
		\phi(a,r,t) = \tilde{\phi} \left( a,r,t, \tfrac{a+1}{r+a} \right).
	  \end{equation}
So 
	\begin{equation}  \label{E:phi.leq.phi.tilde}
		\phi(a,r,t) \leq \tilde{\phi}(a,r,t,\gamma) 
		           \text{ for all } \gamma \in \left( \tfrac{1}{2(n+1)}, 1 \right).  
	\end{equation}
Moreover, it is easy to see that $c_{ra}$ is decreasing in $r$ and, hence,  
	\begin{equation}  \label{E:phi.increases.in.r}
		\phi(a,r,t) \text{ is increasing in } r > 1.
	\end{equation}
  
Let $q = \dim Q$, let 
	\begin{equation}   \label{E:t.min.Q.defn}
		t_{min}(Q) =  \min \bigl\{ t(\sigma') :  
		          \sigma' \text{ is an $m$-simplex of $Q$ with } m > a \bigr\}
	\end{equation}
and let
	\begin{equation}  \label{E:phi.Q.defn}
		\phi = \phi(Q) = \phi \bigl( a, q, t_{min}(Q) \bigr).
	\end{equation}
Note that by \eqref{E:phi.PR.defn} 
	\begin{multline}  \label{E:phi.depends.on.a.q.t}
		\phi \text{ depends on } a, \text{ it only depends on } \sigma 
		\text{ through } Q,  \\
		           \text{ and only depends on $Q$ through $q$ and $t_{min}(Q)$.}
	\end{multline}

If $\tau'$ is an $(n-1)$-face of $\sigma$, let 
      \[
      Y(\tau') = Y(\tau', \sigma) = 
         \left\{ z \in \sigma_{\gamma} : 
        \int_{\mcl{A} \cap (\text{Int} \, \sigma) \cap \zeta(z; \tau')}
		        \left( \frac{diam(\sigma)}{ z_{n} - y_{n}} \right)^{a} \, \Hm^{a}(dy) 
         \leq \phi \, \Hm^{a} \bigl[ \mcl{A} \cap (\text{Int} \, \sigma) \bigr]  \right\}.
   \]
As was the case with the set $X$, the set $Y(\tau')$ is Borel. Now, by \eqref{E:X.tau.in.terms.of.phi.tilde}, 
\eqref{E:phi.leq.phi.tilde}, and \eqref{E:phi.increases.in.r}, 
$X(\tau') \subset Y(\tau')$ so, by \eqref{E:vol.X.geq.mult.vol.sigma}, for every $(n-1)$-face, $\tau'$, of $\sigma$, 
   \begin{equation}  \label{E:Ln.z.for.whch.Ha.pushed.is.smll}
      \mcl{L}^{n} \bigl( Y(\tau') \bigr) 
            \geq \frac{n+1}{n+2} \mcl{L}^{n} ( \sigma_{\gamma} ) > 0.
   \end{equation}
Let $F_{n-1}(\sigma)$ denote the collection of $(n-1)$-faces of $\sigma$.  Then $F_{n-1}(\sigma)$ has $n+1$ elements.  It follows from \eqref{E:Ln.z.for.whch.Ha.pushed.is.smll}
   \begin{align}
      \mcl{L}^{n} \left( \bigcap_{\tau' \in F_{n-1}(\sigma)} Y(\tau') \right) 
               &= \mcl{L}^{n} ( \sigma_{\gamma} ) - \mcl{L}^{n} 
                  \left( \bigcup_{\tau' \in F_{n-1}(\sigma)} \bigl[ \sigma_{\gamma} \setminus Y(\tau') \bigr] 
                   \right)  \notag \\
               &\geq \mcl{L}^{n} ( \sigma_{\gamma} )
                 -  \sum_{\tau' \in F_{n-1}(\sigma)} 
                     \mcl{L}^{n} \bigl[ \sigma_{\gamma} \setminus Y(\tau') \bigr]  \\
               &\geq \mcl{L}^{n} ( \sigma_{\gamma} ) \left[1 - (n+1) \left( 1 - \frac{n+1}{n+2} \right) \right]   
                      \notag \\
               &= \frac{1}{n+2} \mcl{L}^{n} ( \sigma_{\gamma} )  > 0.   \notag
   \end{align}
In particular 
   \[
      \bigcap_{\tau' \in F_{n-1}(\sigma)} Y(\tau') \neq \varnothing.
   \]
Let
	\begin{equation}  \label{E:z0.choice}
	      z_{0} \in \bigcap_{\tau' \in F_{n-1}(\sigma)} Y(\tau').
	\end{equation}
Then by the definition of the $Y(\tau)$'s and \eqref{E:integral.bnd.on.Hm.h.A}
we have that \eqref{E:meas.of.push.out.from.z0.is.bdd.above} holds.
Thus, no matter what $\mcl{A}$ is, so long as it has properties 
\eqref{E:A.intrscts.not.contains.sigma.int} and \eqref{E:Ha.A.cap.sig.circ.poz.not.infte}, there exists $z_{0} \in \text{Int} \, \sigma$ s.t.\ the magnification factor of $z_{0}$ and $\mcl{A}$ is not greater than $(n+1) \phi$ and $\phi$ does not depend on $\mcl{A}$.  (This is what is meant by ``appropriate'' in subsection \ref{SS:constructing.tilde.S.Phi}.)

\subsection{Recursion and totaling}  \label{SS:recur.and.total}
Let $\Ss \subset |P|$ be closed and suppose $\dim \bigl( \Ss \cap |Q| \bigr) \leq a$. 
Let $\sigma \in Q$ be a partial simplex of $\Ss$.  Define $\Phi' = \Phi'_{z}$ ($\Ss' = \Ss'_{z}$) as 
in \eqref{E:Phi'.defn} (resp.\ \eqref{E:S'.defn}) with $z = z_{0}$ as in \eqref{E:z0.choice}.
Let $n = \dim \sigma$.  Let $\tau$ be a proper face of $\sigma$.  
By lemma \ref{L:big.lemma} (point \ref{I:S'.cap.sigma.=.C.cap.Bd}), 
\eqref{E:hbar.maps.A.to.C.cap.Bd.sigma}, \eqref{E:hbar.is.idendt.on.Bd.sigma}, 
and \eqref{E:meas.of.push.out.from.z0.is.bdd.above} and recalling that, by \eqref{E:defn.of.set.A},  $\mcl{A} = \Ss \cap \sigma$, we have
	\begin{align}  \label{E:bound.on.Hm.a.S'.cap.sigma.face}
	   \Hm^{a} \bigl[ \Ss' \cap (\text{Int} \, \tau) \bigr] 
	          &= \Hm^{a} \bigl[ \mcl{C} \cap (\text{Bd} \, \sigma) \cap (\text{Int} \, \tau) \bigr] \notag  \\
	          &= \Hm^{a} \Bigl( \bigl[ \bar{h}( \mcl{A} \cap (\text{Bd} \, \sigma) ) \cup 
		         \bar{h}( \mcl{A} \cap (\text{Int} \, \sigma)) \bigr] \cap (\text{Int} \, \tau) \Bigr)  
		             \notag  \\
	         &\leq \Hm^{a} \Bigl( \bigl[ \mcl{A} \cap (\text{Int} \, \tau) \bigr] \cup 
		         \bigl[ \bar{h}( \mcl{A} \cap (\text{Int} \, \sigma)) \cap (\text{Int} \, \tau) \bigr] \Bigr) \\
	         &\leq \Hm^{a} ( \mcl{A} \cap (\text{Int} \, \tau)) 
	                  +  \Hm^{a} 
	                    \bigl[  \bar{h}( \mcl{A} \cap (\text{Int} \, \sigma)) \cap (\text{Int} \, \tau) \bigr] 
	                          \notag  \\
	         &\leq \Hm^{a} ( \mcl{A} \cap (\text{Int} \, \tau)) 
	              +  \phi \Hm^{a} ( \mcl{A} \cap (\text{Int} \, \sigma) ).
					         \notag  
	\end{align}
Summing the inequality \eqref{E:bound.on.Hm.a.S'.cap.sigma.face} over all $2^{n+1} - 2$ (nonempty) proper faces, $\tau$, of $\sigma$, applying \eqref{E:x.in.exctly.1.simplex.intrr}, and recalling that, by lemma \ref{L:big.lemma}(\ref{I:S'.cap.sigma.=.C.cap.Bd}), 
$\Ss' \cap (\text{Int} \, \sigma) = \varnothing$, we get
	\begin{equation}  \label{E:bound.on.Hm.a.S'.cap.sigma}
		\Hm^{a}(\Ss' \cap \sigma) 
		      \leq \Hm^{a} \bigl[ \mcl{A} \cap (\text{Bd} \, \sigma) \bigr]
		            +  (2^{n+1} - 2) \phi \Hm^{a} \bigl[ \mcl{A} \cap (\text{Int} \, \sigma) \bigr].
	\end{equation}
Next, we develop two similar approaches to proving part (\ref{I:bound.on.S.tilde.vol}) of theorem \ref{T:main.theorem}.  The first approach uses \eqref{E:bound.on.Hm.a.S'.cap.sigma} to quickly shows the existence of a $K < \infty$ s.t.\ \eqref{E:polyhedral.volume.magnification.factor} holds.  The second approach computes a (probably wildly too big) expression for such a $K$.  The second approach is useful for proving part (\ref{I:arb.fine.subdivision}) of the theorem.

\subsubsection{Induction on $\mbf{rank}$}
To prove part (\ref{I:bound.on.S.tilde.vol}) of theorem \ref{T:main.theorem} we use induction as in subsection \ref{SS:constructing.tilde.S.Phi}. If $\mbf{rank}(\Ss) = 1$, i.e., there are no partial simplices of $\Ss$ in $Q$ then theorem \ref{T:main.theorem} holds with $\tilde{\Ss} = \Ss$, $\tilde{\Phi} = \Phi$, and any $K \geq 1$.  Let $r \geq 1$ and assume parts (\ref{I:Phi.tilde.locally.Lip.if.Phi.is} through \ref{I:bound.on.S.tilde.vol}) of theorem \ref{T:main.theorem} hold whenever $\mbf{rank}(\Ss) \leq r$.  In particular, whenever $\mbf{rank}(\Ss) \leq r$ there is a constant $K_{r} < \infty$ s.t.\  \eqref{E:polyhedral.volume.magnification.factor} holds.  Suppose $\mbf{rank}(\Ss) = r+1$.  

Choose a partial simplex, $\sigma \in Q$, of $\Ss$ having maximal dimension.  
Let $n = \dim \sigma$. Then there are no partial simplices of $\Ss$ in $Q$ of dimension $> n$.  
Hence, if $\rho \in Q$ has dimension greater than $n$, then either $\text{Int} \,  \rho = \varnothing$ or $\text{Int} \, \rho \subset \Ss$.  Push $\Ss$ out of $\sigma$ obtaining a new pair ($\Ss'$, $\Phi'$) as in lemma \ref{L:big.lemma} so that \eqref{E:bound.on.Hm.a.S'.cap.sigma} holds.  Then by \eqref{E:pushing.reduces.rank} $\mbf{rank}(\Ss) \leq r$.  Hence, by the induction hypothesis there is a compact subset $\tilde{\Ss} = \widetilde{(\Ss')} \subset |P|$ (and a corresponding map $\tilde{\Phi} = \widetilde{(\Phi')}$) satisfying 
parts (\ref{I:Phi.tilde.locally.Lip.if.Phi.is} through \ref{I:bound.on.S.tilde.vol}) 
of theorem \ref{T:main.theorem} 
(including \eqref{E:polyhedral.volume.magnification.factor} 
with $K = K_{r} < \infty$) with $\Ss$ and $\Phi$ replaced by $\Ss'$ and $\Phi'$, resp.  
In subsection \ref{SS:constructing.tilde.S.Phi} we showed that 
$(\tilde{\Phi}, \tilde{\Ss})$ satisfies parts (\ref{I:Phi.tilde.locally.Lip.if.Phi.is} 
-- \ref{I:Phi.tilde.sigma.subset.Phi.sigma}) of theorem \ref{T:main.theorem}.  Moreover, by \eqref{E:Int.rho.cuts.sigma.then.rho.in.sigma}, lemma \ref{L:big.lemma}
(\ref{I:simps.in.S.are.also.in.S'}, \ref{I:dont.change.sings.on.othr.n.simps}, \ref{I:dont.increase.sing.dim.in.big.simps}), and \eqref{E:bound.on.Hm.a.S'.cap.sigma}, we have
	\begin{align*}
		\Hm^{a} ( \tilde{\Ss} \cap |Q| ) &\leq K_{r} \Hm^{a} (\Ss' \cap |Q| ) \\
		  &= K_{r} \Hm^{a} \bigl[ (\Ss' \setminus \sigma) \cap |Q|  \bigr] 
		       + K_{r} \Hm^{a} (\Ss' \cap \sigma )   \\
		       &= K_{r} \sum_{\rho \in Q, \, \rho \nsubseteq \sigma} \Hm^{a}( (\text{Int} \, \rho) \cap \Ss')
		            + K_{r} \Hm^{a} ( \Ss' \cap \sigma) \\
		       &= K_{r} \sum_{\rho \in Q, \, \rho \nsubseteq \sigma} \Hm^{a}( (\text{Int} \, \rho) \cap \Ss)
		            + K_{r} \Hm^{a} ( \Ss' \cap \sigma) \\
		       &= K_{r} \Hm^{a} \bigl[ ( \Ss \setminus \sigma ) \cap |Q| \bigr]
		            + K_{r} \Hm^{a} ( \Ss' \cap \sigma) \\
		  &\leq K_{r} \Hm^{a} \bigl[ (\Ss \setminus \sigma) \cap |Q|  \bigr] 
		       + K_{r} \Hm^{a} ( \Ss \cap (\text{Bd} \, \sigma)) 
				    +  (2^{n+1} - 2) K_{r} \phi \Hm^{a} \bigl[ \Ss \cap (\text{Int} \, \sigma) \bigr]   \\
		  &\leq (2^{n+1} - 2) K_{r} (\phi+1) \, \Hm^{a} ( \Ss \cap |Q|).
	\end{align*}
Thus, \eqref{E:polyhedral.volume.magnification.factor} holds for compact $\Ss \subset |P|$ 
with $\mbf{rank}(\Ss) \leq r+1$.  So now we have proved parts (\ref{I:Phi.tilde.locally.Lip.if.Phi.is} through \ref{I:bound.on.S.tilde.vol}) of the theorem.  But to prove part (\ref{I:arb.fine.subdivision}) of the theorem we need a more explicit expression for $K$.

\subsubsection{More explicit expression for $K$}  \label{SSS:computational.approach}
Let $q = \dim Q$.  Let $k = 0, 1, 2, \ldots, q$. Recall that $Q^{(k)}$ denotes the $k$-skeleton of $Q$, i.e., the collection of all simplices in $Q$ of dimension no greater than $k$. Then $|Q^{(k)}|$ is the polytope, i.e., union of these simplices. Let $(Q^{(k)})^{c}$ denote the collection of simplices of $Q$ of dimension strictly greater than $k$. Let $\partial Q^{(k)}$ denote the collection of simplices of $Q$ of dimension exactly $k$. If $\tau$ is a proper face of a simplex $\sigma$ write 
$\tau \prec \sigma$ (Munkres \cite[p.\ 86]{jrM84}). If $\tau \prec \sigma$ or $\tau = \sigma$ 
write $\tau \preccurlyeq \sigma$.  We apply the recursive pushing procedure described in section \ref{SS:constructing.tilde.S.Phi}.  Extend the idea of pushing in a trivial way as follows. 
If $\sigma \in Q$ is not a partial simplex of $\Ss$, define \emph{pushing $\Ss$ out of} $\sigma$ to mean making no change to either $\Ss$ or $\Phi$. To be more explicit, this convention amounts to this extension of \eqref{E:S'.defn}.
	\begin{equation}  \tag{\ref{E:S'.defn}'}
		\Ss' = \Ss \text{ if } \sigma \text{ is not a partial simplex of } \Ss.
	\end{equation}
Moreover, if $\sigma \in Q$ is not a partial simplex of $\Ss$ and $z \in \text{Int} \, \sigma$, define 
$\bar{h} = \bar{h}_{z, \sigma}$ (see \eqref{E:b.hbar.defn}) to just be the identity on $\sigma$:
	\begin{equation}   \label{E:when.hbar.is.identity}
		\bar{h}(x) := \bar{h}_{\sigma}(x) := \bar{h}_{z, \sigma}(x) := x, \quad x \in \sigma, 
		     \text{ if $\sigma \in Q$ is not a partial simplex of $\Ss$.}
	\end{equation}
(In particular, in this case $\bar{h}(x) \notin \text{Bd} \, \sigma$ in general.)  In this way we can apply the pushing procedure to all simplices of $Q$.  In addition, with these extensions, lemma \ref{L:big.lemma} 
(points \ref{I:dont.change.sings.on.othr.n.simps}, \ref{I:S'.bggr.S.on.n-1.simps}) continue to hold, because if $\sigma$ is not a partial simplex of $\Ss$ then 
$\bar{h}(\Ss \cap \sigma) = \mcl{C}_{z} \cap (\text{Bd} \, \sigma) \subset \Ss$. 
For if $\text{Int} \, \sigma \subset \Ss$ then, since $\Ss$ is closed, $\text{Bd} \, \sigma \subset \Ss$.  On the other hand, if $(\text{Int} \, \sigma) \cap \Ss = \varnothing$, then, by \eqref{E:C.is.compact.contains.A}, \eqref{E:hbar.maps.A.to.C.cap.Bd.sigma}, and\eqref{E:hbar.is.idendt.on.Bd.sigma}, we have 
	\[
		 \Ss \cap (\text{Bd} \, \sigma) \subset \mcl{C}_{z} \cap (\text{Bd} \, \sigma)  
			 = \bar{h}(\Ss \cap \sigma) 
			 = \bar{h} \bigl[ \Ss \cap (\text{Bd} \, \sigma) \bigr] = \Ss \cap (\text{Bd} \, \sigma).
	\]

Let $(\Ss'_{q, 0}, \Phi'_{q,0}) := (\tilde{\Ss}_{q}, \tilde{\Phi}_{q}) := (\Ss, \Phi)$.  Let $a < d \leq q$ and let $\tilde{\Ss}_{d}$ be a compact subset of $|P|$ and let $\tilde{\Phi}_{d} : |P| \setminus \tilde{\Ss}_{d} \to \F$ be continuous.  The pair $(\tilde{\Ss}_{d}, \tilde{\Phi}_{d})$ is obtained through repeated pushing.  
Suppose $\dim \bigl( \tilde{\Ss}_{d} \cap |Q| \bigr) \leq a$ and no partial simplex of $\tilde{\Ss}_{d}$ has dimension higher than $d$. Then $\tilde{\Ss}_{d} \cap |Q| \subset |Q^{(d)}|$.  The reason for this is as follows.  Let $\sigma \in Q$ and suppose $\dim \sigma > d$.  Then by assumption, $\sigma$ is not a partial simplex of $\tilde{\Ss}_{d}$.  And we cannot have $\text{Int} \, \sigma \subset \tilde{\Ss}_{d}$ either because otherwise $\dim \bigl( \tilde{\Ss}_{d} \cap |Q| \bigr) \geq d > a$.  Therefore, 
$\text{Int} \, \sigma \subset \tilde{\Ss}_{d} = \varnothing$. So if $x \in \tilde{\Ss}_{d} \cap |Q|$ then, by \eqref{E:x.in.exctly.1.simplex.intrr}, $x$ lies in the interior of some simplex of $Q$ of dimension no greater than $d$.  I.e., $\tilde{\Ss}_{d} \cap |Q| \subset |Q^{(d)}|$, as desired. 
   
Write $(\Ss'_{d, 0}, \Phi'_{d,0}) := (\tilde{\Ss}_{d}, \tilde{\Phi}_{d})$. Push $\Ss'_{d,0}$ out of some $d$-simplex in $Q$.  Call the result $(\Ss'_{d, 1}, \Phi'_{d,1})$.  If there is another $d$-simplex in $Q$ 
then push $\Ss'_{d,1}$ out of that one, producing $(\Ss'_{d, 2}, \Phi'_{d,2})$.  Repeat this process producing a sequence $(\Ss'_{d, 1}, \Phi'_{d,1}), \ldots, (\Ss'_{d, M_{d}}, \Phi'_{d,M_{d}})$, where $M_{d}$ is the number of $d$-simplices in $Q$.  By lemma \ref{L:big.lemma}(\ref{I:dim.S'.leq.dim.S},\ref{I:dont.increase.sing.dim.in.big.simps}), for $m = 0, \ldots, M_{d}$, we have $\dim \bigl( \Ss'_{d, m} \cap |Q| \bigr) \leq a$ and there are no partial simplices of $\Ss'_{d, m}$ of dimension greater than $d$.  For the same reason and point (\ref{I:S'.cap.sigma.=.C.cap.Bd}) of lemma \ref{L:big.lemma}, there are no partial simplices of $\Ss'_{d, M_{d}}$ of dimension greater than $d-1$.
Let $(\tilde{\Ss}_{d-1}, \tilde{\Phi}_{d-1}) := (\Ss'_{d, M_{d}}, \Phi'_{d,M_{d}})$.   
So as before $\tilde{\Ss}_{d-1} \cap |Q| \subset |Q^{(d-1)}|$. 
Let $(\Ss'_{q, 0}, \Phi'_{q,0}) := (\tilde{\Ss}_{q}, \tilde{\Phi}_{q}) := (\Ss, \Phi)$.  Apply this operation recursively and let $(\tilde{\Ss}, \tilde{\Phi}) := (\tilde{\Ss}_{0}, \tilde{\Phi}_{0})$.

There are $M_{q}$ simplices of dimension $q$ in $Q$.  Denote them by $\sigma_{1}, \ldots, \sigma_{M_{q}}$, where the ordering is chosen so that $\Ss'_{q, i-1}$ is pushed out of $\sigma_{i}$ which produces $\Ss'_{q, i}$ ($i= 1, \ldots, M_{q}$; recall that $\Ss'_{q, 0} = \Ss$).   
Let $\bar{h}^{i} := \bar{h}_{\sigma_{i}}$ be the map 
$\bar{h}$ in the simplex $\sigma_{i}$. (See \eqref{E:b.hbar.defn} and \eqref{E:when.hbar.is.identity}.) We use the following fact.  Let $\tau \in Q^{(q-1)}$ and $r=1, \ldots, M_{q}$. Then,
	\begin{equation}   \label{E:tildeS.q-1.cap.tau.breakdown}
		\Ss'_{q, r} \cap (\text{Int} \, \tau) 
			\subset \bigl[ \Ss \cap (\text{Int} \, \tau) \bigr]  
			             \cup \bigcup_{i=1}^{r} 
			                   \bigl[ \bar{h}^{i} \bigl[ \Ss \cap (\text{Int} \, \sigma_{i}) \bigr]
			                   \cap (\text{Int} \, \tau) \bigr].
	\end{equation}
To prove this, first suppose $r = 1$.  Then by lemma 
\ref{L:big.lemma}(\ref{I:S'.bggr.S.on.n-1.simps}) (as extended above and using the modified definition \eqref{E:when.hbar.is.identity} of $\bar{h}^{1}$ if appropriate),
	\[
		\Ss'_{q, 1} \cap (\text{Int} \, \tau) 
		   = (\text{Int} \, \tau) \cap \bigl[ \Ss \cup \bar{h}^{1}(\Ss \cap \sigma_{1}) \bigr],
	\]
which is \eqref{E:tildeS.q-1.cap.tau.breakdown} in the case $r = 1$.  Let $m \geq 1$ and suppose \eqref{E:tildeS.q-1.cap.tau.breakdown} holds for any $r \leq m$.  Now suppose $r = m+1$.  Then by lemma \ref{L:big.lemma}(\ref{I:S'.bggr.S.on.n-1.simps}) again and the induction hypothesis,
	\begin{align}   \label{E:S'.a.Mq.induct.step}
		\Ss'_{q, r} \cap (\text{Int} \, \tau) &= \Ss'_{q, m+1} \cap (\text{Int} \, \tau) \notag \\
		   &= (\text{Int} \, \tau) \cap \bigl[ \Ss'_{q, m} 
		      \cup \bar{h}^{m+1}(\Ss'_{q, m}  \cap \sigma_{m+1}) \bigr]  \notag \\
		   &= \bigl[ (\text{Int} \, \tau) \cap \Ss'_{q, m} \bigr]  
		      \cup \bigl[ (\text{Int} \, \tau) \cap \bar{h}^{m+1}(\Ss'_{q, m}  \cap \sigma_{m+1}) \bigr]  \\
	        &\subset \left( \bigl[ \Ss \cap (\text{Int} \, \tau) \bigr]  
	             \cup \bigcup_{i=1}^{m} \Bigl[ \bar{h}^{i} \bigl( \Ss \cap (\text{Int} \, \sigma_{i}) \bigr) 
	                         \cap (\text{Int} \, \tau) \Bigr]  \right)   \notag \\
                 & \qquad \qquad \qquad \cup \bigl[ (\text{Int} \, \tau) 
	                        \cap \bar{h}^{m+1}(\Ss'_{q, m} \cap \sigma_{m+1}) \bigr].  \notag
	\end{align}
Now, 
	\begin{equation}  \label{E:divide.hbar.into.Int.Bd}
		\bar{h}^{m+1}(\Ss'_{q, m}  \cap \sigma_{m+1}) 
			= \bar{h}^{m+1} \bigl[ \Ss'_{q, m}  \cap (\text{Int} \, \sigma_{m+1})  \bigr]
			    \cup \bar{h}^{m+1}  \bigl[ \Ss'_{q, m}  \cap (\text{Bd} \, \sigma_{m+1})  \bigr].
	\end{equation}
Now by lemma \ref{L:big.lemma}(\ref{I:dont.change.sings.on.othr.n.simps}) (as extended above), and noting that if $j < m+1$, then $\Ss_{q,j}$ is obtained by pushing $\Ss'_{q,j-1}$ out of $\sigma_{j-1} \ne \sigma_{m+1}$,
	\begin{multline}   \label{E:S'.qm.cap.sigma.m+1.=.S.cap.sigma.m+1}
		\Ss'_{q, m}  \cap (\text{Int} \, \sigma_{m+1})
		   = \Ss'_{q, m-1}  \cap (\text{Int} \, \sigma_{m+1}) \\
		\cdots = \Ss'_{q,1}  \cap (\text{Int} \, \sigma_{m+1}) 
		   = \Ss'_{q,0}  \cap (\text{Int} \, \sigma_{m+1})
		      = \Ss \cap (\text{Int} \, \sigma_{m+1}).
	\end{multline}
Also, by \eqref{E:hbar.is.idendt.on.Bd.sigma}
	\begin{equation}  \label{E:hbar.m+1.is.ident.on.Bd}
		\bar{h}^{m+1}  \bigl[ \Ss'_{q, m}  \cap (\text{Bd} \, \sigma_{m+1})  \bigr]
			= \Ss'_{q, m}  \cap (\text{Bd} \, \sigma_{m+1}).
	\end{equation}
Substituting \eqref{E:S'.qm.cap.sigma.m+1.=.S.cap.sigma.m+1} and \eqref{E:hbar.m+1.is.ident.on.Bd} into \eqref{E:divide.hbar.into.Int.Bd} and recalling that $r = m+1$, we get
\begin{multline}   \label{E:hbar.m+1.cap.tau.decomp}
(\text{Int} \, \tau) \cap \bar{h}^{m+1}(\Ss'_{q, m}  \cap \sigma_{m+1}) 
	= \Bigl( \bar{h}^{m+1} \bigl[ \Ss  \cap (\text{Int} \, \sigma_{m+1})  \bigr] 
	      \cap (\text{Int} \, \tau) \Bigr)
	    \cup \Bigl( \bigl[ \Ss'_{q, m}  \cap (\text{Bd} \, \sigma_{m+1})  \bigr] 
	         \cap (\text{Int} \, \tau) \Bigr) \\
	\subset \Bigl( \bar{h}^{r} 
	        \bigl[ \Ss  \cap (\text{Int} \, \sigma_{r})  \bigr] \cap (\text{Int} \, \tau) \Bigr)
	    \cup \Bigl( \Ss'_{q, m}  \cap (\text{Int} \, \tau) \Bigr)
\end{multline}
Applying the induction hypothesis to $\Ss'_{q, m}  \cap (\text{Int} \, \tau)$ and then substituting \eqref{E:hbar.m+1.cap.tau.decomp} into \eqref{E:S'.a.Mq.induct.step} proves \eqref{E:tildeS.q-1.cap.tau.breakdown}.

By \eqref{E:hbar.maps.A.to.C.cap.Bd.sigma} and \eqref{E:when.hbar.is.identity}, 
for each $i = 1, \ldots, r$, we have 
$\bar{h}^{i} \bigl[ \Ss \cap (\text{Int} \, \sigma_{i}) \bigr] \subset \sigma_{i}$. Therefore, by \eqref{E:Int.rho.cuts.sigma.then.rho.in.sigma}, \eqref{E:tildeS.q-1.cap.tau.breakdown} implies
	\begin{equation}   \label{E:tildeS.q-1.cap.tau.brkdwn.for.bndrys}
		\Ss'_{q, r} \cap (\text{Int} \, \tau) 
			\subset \bigl[ \Ss \cap (\text{Int} \, \tau) \bigr]  
			             \cup \bigcup_{1 \leq i \leq r, \, \tau \prec \sigma_{i}} 
			                   \Bigl[ \bar{h}^{i} \bigl( \Ss \cap (\text{Int} \, \sigma_{i}) \bigr) 
			                   \cap (\text{Int} \, \tau) \Bigr].
	\end{equation}

\emph{Claim:}  If $\ell = 1, \ldots, q-a$ then, letting $m=q-\ell$, for every $\tau \in Q^{(q-\ell)}$ we have
   \begin{equation}  \label{E:volume.bound.inductive.statement}
      \Hm^{a} \bigl[ \tilde{\Ss}_{q-\ell} \cap (\text{Int} \, \tau) \bigr]  
          \leq \sum_{j=1}^{\ell} \; 
           \sum_{\tau \prec \tau_{1} \prec \cdots \prec \tau_{j} \in Q, \, \tau_{1} \in (Q^{(q-\ell)})^{c}} \;
             \phi^{j} \, \Hm^{a} \bigl[ \Ss \cap (\text{Int} \, \tau_{j}) \bigr] 
                + \Hm^{a} \bigl[ \Ss \cap (\text{Int} \, \tau) \bigr].
   \end{equation}

First, consider the case $\ell=1$.  Let $\tau \in Q^{(q-1)}$. Then by 
\eqref{E:tildeS.q-1.cap.tau.brkdwn.for.bndrys} and \eqref{E:meas.of.push.out.from.z0.is.bdd.above}, we have 
	\begin{align}  \label{E:bound.on.Hm.a.S'.cap.q-1.simplex}
	   \Hm^{a} \bigl[ \tilde{\Ss}_{q-1} \cap (\text{Int} \, \tau) \bigr] 
	          &= \Hm^{a} \bigl[ \Ss'_{q, M_{q}} \cap (\text{Int} \, \tau) \bigr]   \notag \\
	          &\leq \Hm^{a} \left[  \bigl( \Ss \cap (\text{Int} \, \tau) \bigr)  
	             \cup \left( \bigcup_{1 \leq i \leq M_{q}, \, \tau \prec \sigma_{i}} 
	                \Bigl[ \bar{h}^{i} \bigl( \Ss \cap (\text{Int} \, \sigma_{i}) \bigr) 
			                   \cap (\text{Int} \, \tau) \Bigr]
	                     \right) \right]   \\
	         &\leq \Hm^{a} \bigl[ \Ss \cap (\text{Int} \, \tau) \bigr] 
	                  +  \sum_{1 \leq i \leq M_{q}, \, \tau \prec \sigma_{i}} 
	                    \Hm^{a} \Bigl[ \bar{h}^{i} \bigl( \Ss \cap (\text{Int} \, \sigma_{i}) \bigr) 
			                   \cap (\text{Int} \, \tau) \Bigr]
	                              \notag  \\
	         &\leq \Hm^{a} \bigl[ \Ss \cap (\text{Int} \, \tau) \bigr] 
	                + \phi \sum_{1 \leq i \leq M_{q}, \, \tau \prec \sigma_{i}} 
	                    \Hm^{a} \bigl[ \Ss \cap (\text{Int} \, \sigma_{i}) \bigr].
					         \notag  
	\end{align}
This proves \eqref{E:volume.bound.inductive.statement} in the case $\ell = 1$.

Inductively, let $k = 1, \ldots, q-a$ and suppose \eqref{E:volume.bound.inductive.statement} holds for $\ell = 1, \ldots, k$.  We show that it holds with $\ell = k+1$. First apply \eqref{E:volume.bound.inductive.statement} with the $(q-k)$-skeleton, $Q^{(q-k)}$, in place of $Q$, $\tilde{\Ss}_{q-k}$ in place of $\Ss$, and $\ell = 1$.  Let $\tau \in Q^{(q-k-1)}$. Then by \eqref{E:bound.on.Hm.a.S'.cap.q-1.simplex},
   \begin{equation*}   
      \Hm^{a}\bigl[ \tilde{\Ss}_{q-k-1} \cap (\text{Int} \, \tau) \bigr]
         \leq \sum_{\tau \prec \tau_{1} 
               \in \partial Q^{(q-k)}} \phi \, \Hm^{a}(\tilde{\Ss}_{q-k} \cap (\text{Int} \, \tau_{1})) 
            + \Hm^{a} \bigl[ \tilde{\Ss}_{q-k} \cap (\text{Int} \, \tau) \bigr]  \notag  
  \end{equation*} 
Apply the induction hypothesis with $\ell = k$, but re-index the $\tau_{j}$'s as follows:  $\tau \to \tau_{1}$, $\tau_{1} \to \tau_{2}, \ldots$, and $\tau_{j} \to \tau_{j+1}$.  Then let $i = j+1$.
         \begin{align*}
             \Hm^{a}\bigl[ \tilde{\Ss}_{q-k-1} \cap (\text{Int} \, \tau) \bigr] 
                &\leq \sum_{\tau \prec \tau_{1} \in \partial Q^{(q-k)}} \phi \Biggl[ 
               \sum_{i=2}^{k+1} \, \sum_{\tau_{1} \prec \cdots \prec \tau_{i} \in Q, \, \tau_{2} \in 
                   (Q^{(q-k)})^{c}} \;
               \phi^{i-1} \Hm^{a} (\Ss \cap (\text{Int} \, \tau_{i})) \\
            & \qquad \qquad \qquad \qquad \qquad \qquad \qquad \qquad \qquad \qquad
                   + \Hm^{a} (\Ss \cap (\text{Int} \, \tau_{1})) \Biggr]   \\
            & \qquad \qquad \qquad \qquad \qquad 
                   + \Hm^{a} \bigl[ \tilde{\Ss}_{q-k} \cap (\text{Int} \, \tau) \bigr] .
         \end{align*}    \\
Apply the induction hypothesis to the last term with $\ell = k$ and note that 
$\tau_{1} \in  \partial Q^{(q-k)}$ and $\tau_{1} \prec  \tau_{2}$ automatically implies 
$\tau_{2} \in (Q^{(q-k)})^{c}$.  
         \begin{align}  \label{E:Ha.S.q-k-1.cap.tau.bound}
            \Hm^{a} \bigl[ \tilde{\Ss}_{q-k-1} \cap (\text{Int} \, \tau) \bigr] 
              &\leq \sum_{\tau \prec \tau_{1} \in \partial Q^{(q-k)}}  
                  \Biggl[ \sum_{i=2}^{k+1} \,\sum_{\tau_{1} \prec \cdots \prec \tau_{i} \in Q} \;
                   \phi^{i} \Hm^{a} (\Ss \cap (\text{Int} \, \tau_{i}))  \notag  \\
              & \qquad \qquad \qquad \qquad \qquad \qquad \qquad \qquad \qquad \qquad
                  + \phi \Hm^{a} (\Ss \cap (\text{Int} \, \tau_{1}) \Biggr]    \\
              & \qquad \qquad \qquad + \sum_{j=1}^{k} \; 
               \sum_{\tau \prec \tau_{1} \prec \cdots \prec \tau_{j} \in Q; \, \tau_{1} \in (Q^{(q-k)})^{c}} \;
                    \phi^{j} \Hm^{a} (\Ss \cap (\text{Int} \, \tau_{j}))  \notag \\
              & \qquad \qquad \qquad \qquad \qquad \qquad \qquad \qquad \qquad \qquad \qquad
                + \Hm^{a} (\Ss \cap (\text{Int} \, \tau)). \notag
        \end{align}

Suppose $\tau \prec \tau_{1} \prec \cdots \prec \tau_{j} \in Q$, $j=k+1$, and 
$\tau_{1} \in (Q^{(q-k)})^{c}$.  Then $q \geq \dim \tau_{j} = \dim \tau_{k+1} \geq \dim \tau_{1} + k$.  Hence, $q-k \geq \dim \tau_{1} \geq q-k+1$, an impossibility.  Therefore, if $j = k+1$, the following sum is empty:
	\[
		\sum_{\tau \prec \tau_{1} \prec \cdots \prec \tau_{j} \in Q; \, \tau_{1} \in (Q^{(q-k)})^{c}} \;
	                    \phi^{j} \Hm^{a} (\Ss \cap (\text{Int} \, \tau_{j})) = 0.
	\]
Hence, by \eqref{E:Ha.S.q-k-1.cap.tau.bound},
	\begin{multline*}
		\Hm^{a}\bigl[ \tilde{\Ss}_{q-k-1} \cap (\text{Int} \, \tau) \bigr] \leq 
	            \sum_{i=1}^{k+1} \,
	            \sum_{\tau \prec \tau_{1} \prec \cdots \prec \tau_{i} \in Q;  \tau_{1} \in \partial Q^{(q-k)}} 
	                \; \phi^{i} \Hm^{a} (\Ss \cap (\text{Int} \, \tau_{i}))    \\
	                + \sum_{j=1}^{k+1} \; 
	           \sum_{\tau \prec \tau_{1} \prec \cdots \prec \tau_{j} \in Q; \, \tau_{1} \in (Q^{(q-k)})^{c}} 
	              \; \phi^{j} \Hm^{a} (\Ss \cap (\text{Int} \, \tau_{j})) 
	                   + \Hm^{a} \bigl[ \Ss \cap (\text{Int} \, \tau) \bigr].
	\end{multline*}
Since $\partial Q^{(q-k)} \cap (Q^{(q-k)})^{c} = \varnothing$ and 
$\partial Q^{(q-k)} \cup (Q^{(q-k)})^{c} = (Q^{(q-k-1)})^{c}$ it follows that 
\eqref{E:volume.bound.inductive.statement} holds with $\ell = k+1$. 
By induction \eqref{E:volume.bound.inductive.statement} holds in general and 
the claim \eqref{E:volume.bound.inductive.statement} is proven.
  
Let $\tau_{0} \in \partial Q^{(a)}$.  Then $\tau_{0} \prec \tau_{1} \in Q$ automatically implies $\tau_{1} \in (Q^{(a)})^{c}$.  Thus,
applying \eqref{E:volume.bound.inductive.statement} with $\ell=q-a$ and $\tau = \tau_{0}$, 
we can start the summation in \eqref{E:volume.bound.inductive.statement} at $j=0$ and get for every $a$-simplex $\tau_{0} \in \partial Q^{(a)}$
   \begin{equation}   \label{E:S.measr.bound.in.a-simp}
      \Hm^{a}\bigl[ \tilde{\Ss}_{a} \cap (\text{Int} \, \tau_{0}) \bigr]
         \leq \sum_{j=0}^{q-a} \; \sum_{\tau_{0} \prec \tau_{1} \prec \cdots \prec \tau_{j} \in Q} \;
           \phi^{j} \Hm^{a} (\Ss \cap (\text{Int} \, \tau_{j})), \quad \tau_{0} \in \partial Q^{(a)}.
   \end{equation}

Next, let $\sigma \in (Q^{(a)})^{c}$ and let $\tau_{0} \preccurlyeq \sigma$ be an $a$-simplex.  We count the number of chains $\tau_{0} \prec \tau_{1} \prec \cdots  \prec \tau_{j-1} \prec \tau_{j} = \sigma$.
If $i = a$ let $j = 0$.  For $i = a+1, a+2, \dots$ let $j$ be any element of $\{1, \ldots, i-a \}$.  
Let  $V_{0} = \{ 0, \ldots, a \}$.  So $V_{0}$ has $a+1$ elements.  Let $N_{j}^{i}$ denote the number of filtrations $V_{0} \subsetneqq \cdots \subsetneqq V_{j} := \{ 0, \ldots, i \}$.  
Note that for $k=1,\ldots, j-1$, the elements of $V_{k}$ do not have to be consecutive integers.  Thus, $V_{k}$ will often contain fewer than $\max V_{k}$ elements and if $0 < k_{1} < k_{2} < j$ we may have $\max V_{k_{1}} > k_{2}$. We have $N_{0}^{a} =  1$ and for $i \geq a+1$ we also have $N_{1}^{i} = 1$. Moreover, $N_{j}^{i} = 0$ if $j > i-a$. 
If $a+1 = t_{0}  < t_{1} < \cdots < t_{j-1} < t_{j} = i+1$, the number of such filtrations s.t.\ the cardinality of $V_{k} = t_{k}$ ($k=0, \ldots, j$) is clearly the multinomial coefficient 
$\binom{i-a}{t_{1} - t_{0}, \; t_{2} - t_{1}, \; \cdots, \; t_{j} - t_{j-1}}$. Thus, if $i > a$ (so $j \geq 1$),
	\begin{equation}  \label{E:Nji.formula}
             N_{j}^{i} = \sum_{a+1 < t_{1} < \cdots < t_{j-1} < i+1} 
                  \binom{i-a}{t_{1} - a - 1, \; t_{2} - t_{1}, \; \cdots, \; i - t_{j-1}+1} < j^{i-a}.
	\end{equation}
Let $\tau_{0}, \sigma \in Q$ with $\tau_{0} \preccurlyeq \sigma$ and $\dim \tau_{0} = a$.  Then $N_{j}^{\dim \sigma}$ is the number of chains 
$\tau_{0} \prec \tau_{1} \prec \cdots  \prec \tau_{j-1} \prec \tau_{j} = \sigma$.

Let $m = 0, 1, \ldots$.  Let $\psi_{m} := 0$ if $m < a$, let $\psi_{a} := 1$, and for $m > a$ let
   \begin{equation}  \label{E:psi.m.defn}
      \psi_{m} := \psi_{m}(\phi) := \sum_{j=0}^{m-a} N_{j}^{m} \, \phi^{j} \text{ and } 
           \psi = \psi(\phi, q) := \max \{ \psi_{1}, \ldots, \psi_{q} \} \geq 1.
   \end{equation}
Note that by \eqref{E:phi.depends.on.a.q.t} and \eqref{E:Nji.formula}, $\psi$ only depends on $a$, 
$q = \dim Q$, and $t_{min}(Q)$ (see \eqref{E:t.min.Q.defn}).

Now $\tilde{\Ss}_{a} \cap |Q| \subset |Q^{(a)}|$ and $\Hm^{a}(|Q_{a-1}|) = 0$.  But
	\[
		(\tilde{\Ss}_{a} \cap |Q|) \setminus \left( \bigcup_{\tau_{0} \in \partial Q^{(a)}} 
		   \text{Int} \, \tau_{0} \right) \; \subset \; |Q_{a-1}|.
	\]
Moreover, by (\ref{E:intersection.of.simps}'), 
the sets $\text{Int} \, \tau_{0}$ ($\tau_{0} \in \partial Q^{(a)}$) are disjoint.  Therefore,
	\[
		\Hm^{a} (\tilde{\Ss}_{a} \cap |Q|) 
		   = \sum_{\tau_{0} \in \partial Q^{(a)}}  
		       \Hm^{a}\bigl[ \tilde{\Ss}_{a} \cap (\text{Int} \, \tau_{0}) \bigr].
	\]
Hence, from \eqref{E:S.measr.bound.in.a-simp} we get
   \begin{align*}  
      \Hm^{a}(\tilde{\Ss}_{a} \cap |Q|) &\leq \sum_{\tau_{0} \in \partial Q^{(a)}} \sum_{j=0}^{q-a} \; 
        \sum_{\tau_{0} \prec \tau_{1} \prec \cdots \prec \tau_{j} \in Q} \;
           \phi^{j} \Hm^{a} (\Ss \cap (\text{Int} \, \tau_{j}))  \\
        &= \sum_{\tau_{0} \in \partial Q^{(a)}} \sum_{j=0}^{q-a} \; \sum_{\tau_{j} \in Q} 
        \phi^{j} \Hm^{a} (\Ss \cap (\text{Int} \, \tau_{j})) \;
           \sum_{\tau_{0} \prec \tau_{1} \prec \cdots  \prec \tau_{j-1} \prec \tau_{j}} 1.
                   \notag
   \end{align*}
Now let $\sigma = \tau_{j} \in Q$ and recall that $N_{j}^{\dim \sigma} = 0$
if $j > \dim \sigma - a$.  Hence
   \begin{align*}  
     \Hm^{a}(\tilde{\Ss}_{a} \cap |Q|)  &\leq 
                   \sum_{\tau_{0} \in \partial Q^{(a)}} \sum_{\tau_{0} \preccurlyeq \sigma \in Q} 
                      \sum_{j=0}^{\dim \sigma -a} \phi^{j} N_{j}^{\dim \sigma } \;
                         \Hm^{a} (\Ss \cap (\text{Int} \, \sigma))  \notag  \\
        &= \sum_{\tau_{0} \in \partial Q^{(a)}} \sum_{\tau_{0} \preccurlyeq \sigma \in Q} 
              \psi_{\dim \sigma }  \; \Hm^{a} (\Ss \cap (\text{Int} \, \sigma))  \notag  \\
        &\leq \sum_{\tau \in \partial Q^{(a)}} \sum_{\tau \preccurlyeq \sigma \in Q} \; 
          \psi \Hm^{a} (\Ss \cap (\text{Int} \, \sigma))  \\
        &= \psi \sum_{\sigma \in Q; \, \dim \sigma \geq a} \; 
           \sum_{\tau \in \partial Q^{(a)}; \, \tau \preccurlyeq \sigma}
              \Hm^{a} (\Ss \cap (\text{Int} \, \sigma)).  \notag
   \end{align*}
Now, if $\dim \sigma \geq a$, the number of $a$-faces of $\sigma$ is $\binom{\dim \sigma + 1}{a + 1}$ and 
if $\dim \sigma < a$ then $\Hm^{a}( \Ss \cap (\text{Int} \, \sigma)) = 0$.  
Therefore, by \eqref{E:x.in.exctly.1.simplex.intrr},
   \begin{align}  \label{E:Hm.tilde.Sa.leq.mult.Hm.S.cap.Q}
     \Hm^{a}(\tilde{\Ss}_{a} \cap |Q|)    &\leq \psi \sum_{\sigma \in Q; \, \dim \sigma \geq a} 
                \binom{\dim \sigma + 1}{a + 1} \Hm^{a} (\Ss \cap (\text{Int} \, \sigma))    \notag  \\
        &\leq  \binom{q + 1}{a + 1} \psi \sum_{\sigma \in Q; \, \dim \sigma \geq a} 
             \Hm^{a} (\Ss \cap (\text{Int} \, \sigma))  \\
        &= \binom{q + 1}{a + 1} \psi \sum_{\sigma \in Q} \Hm^{a} (\Ss \cap (\text{Int} \, \sigma))  
             \notag  \\
        &= \binom{q + 1}{a + 1} \psi \, \Hm^{a} (\Ss \cap |Q|).   \notag
   \end{align}   

Now, $\tilde{\Ss}_{a} \subset |Q^{(a)}|$, but $\tilde{\Ss}_{a}$ may not be a subcomplex of $Q^{(a)}$.  I.e., there may exist partial simplices of $\Ss$ in $Q^{(a)}$.  If this happens continue recursively pushing $\tilde{\Ss}_{a}$ out of such simplices.
	Recall $\Ss'_{a,0} = \tilde{\Ss}_{a}$ and $M_{a} =$ the number of $a$-simplices in $Q$.  Let $i = 0, \ldots, M_{a} -1$, suppose one has constructed $\Ss'_{a,i}$, and let $\omega \in \partial Q^{(a)}$ be an $a$-simplex that has not yet benefited from the pushing operation, but suppose it is next in line.  If $\omega$ is not a partial simplex of $\Ss'_{a,i}$ then $\Ss'_{a,i+1} = \Ss'_{a,i}$ so $\Hm^{a}(\Ss'_{a,i+1}) = \Hm^{a}(\Ss'_{a,i})$.  If $\omega$ is a partial simplex of $\Ss'_{a,i}$ then by (points \ref{I:S'.cap.sigma.=.C.cap.Bd}, \ref{I:dont.change.sings.on.othr.n.simps}, \ref{I:dont.increase.sing.dim.in.big.simps})
of lemma \ref{L:big.lemma} we have $\Hm^{a}\bigl[ (\text{Int} \, \rho) \cap \Ss'_{a,i+1} \bigr] \leq \Hm^{a}\bigl[ (\text{Int} \, \rho) \cap \Ss'_{a,i} \bigr]$ for any 
$\rho \in \bigl( Q^{(a-1)} \bigr)^{c}$.   
As for $\Hm^{a}(\rho \cap \Ss'_{a,i+1})$ with $\rho \in Q^{(a-1)}$, of course $\Hm^{a}\bigl[ (\text{Int} \, \rho) \cap \Ss'_{a,i+1} \bigr] = \Hm^{a}\bigl[ (\text{Int} \, \rho) \cap \Ss'_{a,i} \bigr] = 0$.  In summary, further pushing from $\tilde{\Ss}_{a}$ can only reduce the $\Hm^{a}$-measure. Hence, 
   \begin{equation}  \label{E:Hm.tilde.S0.leq.mult.Hm.S.cap.Q}
     \Hm^{a}(\tilde{\Ss} \cap |Q|) = \Hm^{a}(\tilde{\Ss}_{0} \cap |Q|) 
        \leq \binom{q + 1}{a + 1} \psi \, \Hm^{a} (\Ss \cap |Q|) 
           = K_{2} \bigl[ \phi(Q), q \bigr] \Hm^{a} (\Ss \cap |Q|),
   \end{equation}
where 
	\begin{equation}   \label{E:K2.phi.q.defn}
		K_{2} = K_{2} \bigl[ \phi(Q), q \bigr] = \binom{q + 1}{a + 1} \psi \bigl[ \phi(Q), q \bigr].
	\end{equation}
(See \eqref{E:phi.Q.defn} and \eqref{E:psi.m.defn}.) By \eqref{E:psi.m.defn} 
	\begin{equation}   \label{E:K2.no.smaller.than.1}
		K_{2} \geq 1.
	\end{equation}
Then \eqref{E:polyhedral.volume.magnification.factor} holds with $K = K_{2}$ if $0 < a < q$.  (See \eqref{E:0.<.a.<.q}.)
\eqref{E:polyhedral.volume.magnification.factor} holds for general $a$ if we take
$K = \max \left\{ K_{1}, K_{2} \right\}$, where $K_{1}$ is defined in \eqref{E:K.for.a=q}.  

\subsection{Proof of part (\ref{I:arb.fine.subdivision}) of theorem \ref{T:main.theorem}}
First of all, a subdivision of $P$ is also a finite simplicial complex so parts  \ref{I:Phi.tilde.locally.Lip.if.Phi.is}  through \ref{I:Phi.tilde.sigma.subset.Phi.sigma} of theorem \ref{T:main.theorem} automatically hold for any subdivision of $P$.  The question is whether we can choose arbitrarily fine subdivisions of $P$ and a single constant $K < \infty$ in \eqref{E:polyhedral.volume.magnification.factor} that will work for all those subdivisions.

Let $p = \dim P$.  Suppose $0 < a < q$. By Munkres \cite[Lemma 9.4, p.\ 92]{jrM66}, there exists 
$t_{0} \in \bigl( 0, t_{min}(Q) \bigr]$ (see \eqref{E:t.min.Q.defn}) depending only on $P$ with the following property. 
For every $\epsilon > 0$ there exists a subdivision, $P' = P'(\epsilon)$, of $P$ s.t.\ the diameter of the largest simplex in $P'(\epsilon)$ is no greater than $\epsilon$
and the thickness of every simplex (of positive dimension) in $P'(\epsilon)$ is at least $t_{0}$.  By \eqref{E:phi.PR.defn}, \eqref{E:phi.a.r.t.defn}, \eqref{E:phi.increases.in.r}, \eqref{E:phi.Q.defn}, \eqref{E:psi.m.defn}, and \eqref{E:K2.phi.q.defn}, we have 
$K_{2} \bigl[ \phi(a,p, t_{0}), p \bigr] \geq K_{2} \bigl[ \phi(Q), q \bigr]$.  Hence, by \eqref{E:Hm.tilde.S0.leq.mult.Hm.S.cap.Q}, if we replace 
$K = K_{2} \bigl[ \phi(Q), q \bigr]$ by $K = K_{2} \bigl[ \phi(a,p, t_{0}), p \bigr]$ then \eqref{E:polyhedral.volume.magnification.factor} will hold if $P$ is replaced by $P'(\epsilon)$.  
  
As for the case $a = q$, all partial simplices in $Q$ can be ignored because the pushing out operation repaces them by ($\Hm^{a} = \Hm^{q}$)-null sets.  Only $q$-simplices lying 
in $\Ss$ matter.  By points (\ref{I:simps.in.S.are.also.in.S'},\ref{I:dont.increase.sing.dim.in.big.simps}) 
of lemma \ref{L:big.lemma}, pushing out and subdivision does not affect the total volume of these simplices.  

Finally, consider the case $a = 0$, so $\dim (\Ss \cap |Q|) = 0$. The argument given at the beginning of subsection \ref{SS:magnification.in.1.simp} shows that \eqref{E:polyhedral.volume.magnification.factor} holds for any $K \geq 1$ for any complex so, in particular, it holds for any $P'(\epsilon)$.

Hence, part (\ref{I:arb.fine.subdivision}) of theorem \ref{T:main.theorem} holds 
with $K = \max \left\{ K_{1}, K_{2} \right\}$ (see \eqref{E:K.for.a=q}), providing we use the updated version of $K_{2}$ defined above. That concludes the proof of theorem \ref{T:main.theorem}.

\clearpage

\setcounter{section}{0}
\renewcommand{\thesection}{\Alph{section}}


\begin{center}
\begin{Large}
APPENDICES
\end{Large}
\end{center}


\section{Miscellaneous proofs}  \label{S:misc.proofs}

\begin{proof}[Proof of corollary \ref{C:drop.finiteness.of.P}] Suppose $P$ is a simplicial complex s.t.\ $|P| \subset \RR^{N}$ and \eqref{E:intersect.finitely.many} holds.  Then, by lemma \ref{L:local.finiteness.and.compactness}(\ref{I:P.locally.finite},\ref{I:local.compactness}), $|P|$ is a locally compact subspace of $\RR^{N}$ and $P$ is locally finite.  
We prove that parts (\ref{I:Phi.tilde.locally.Lip.if.Phi.is}) through (\ref{I:bound.on.S.tilde.vol}) of theorem \ref{T:main.theorem} continue to hold and assertion (\ref{I:arb.fine.subdivision}') also holds.  
Since $P$ is locally finite and $Q$ is finite, the complex $P_{Q}$ is finite.  The idea is to apply the theorem, as stated, to $P_{Q}$.  Let $(\tilde{\Ss}_{Q}, \tilde{\Phi}_{Q})$ be the set-map pair whose existence is asserted by the theorem (applied to $P_{Q}$ and $\bigl( \Ss \cap |P_{Q}|, \Phi \vert_{|P_{Q}| \setminus \Ss} \bigr)$). (Here, $\Phi \vert_{|P_{Q}| \setminus \Ss}$ is the restriction of $\Phi$ to $|P_{Q}| \setminus \Ss$.)  Thus, $\tilde{\Ss}_{Q}$ is a closed subset of $|P_{Q}|$ and $\tilde{\Phi}_{Q} : |P_{Q}| \setminus \tilde{\Ss}_{Q} \to \F$ is continuous.  Now define
	\begin{equation}  \label{E:tildeS.tildePhi.defn}
		\tilde{\Ss} := \bigl( \Ss \setminus |P_{Q}| \bigr) \cup \tilde{\Ss}_{Q} \text{ and }
		\tilde{\Phi}(x) = 
			\begin{cases}
				\Phi(x), &\text{ if } x \in |P| \setminus |P_{Q}|, \\
				\tilde{\Phi}_{Q}(x), &\text{ if } x \in |P_{Q}|.
			\end{cases}
	\end{equation}
Obviously, by theorem \ref{T:main.theorem}, parts (\ref{I:dim.Stilde.no.bggr.thn.dim.S}), (\ref{I:S.tilde.subcomp}), and (\ref{I:no.change.off.Q}) through (\ref{I:bound.on.S.tilde.vol}) of theorem \ref{T:main.theorem} still hold with this definition because $P_{Q}$ is finite.  (Proof of part (\ref{I:dim.Stilde.no.bggr.thn.dim.S}) uses \eqref{E:dim.of.whole.=.max.dim.of.parts}.)  

First we show
	\begin{equation} \label{E:Stilde.=.S.off. PQ}
		\text{If } \sigma \in P \setminus P_{Q} \text{ then } \tilde{\Ss} \cap \sigma = \Ss \cap \sigma 
		 \text{ and } \tilde{\Phi} \vert_{\sigma \setminus \Ss} = \Phi \vert_{\sigma \setminus \Ss}.
	\end{equation}
By \eqref{E:tildeS.tildePhi.defn}, this is obvious if $\sigma \in P \setminus P_{Q}$ with $\sigma \cap |P_{Q}| = \varnothing$.  So let $\sigma \in P\setminus P_{Q}$ satisfy $\sigma \cap |P_{Q}| \ne \varnothing$. By \eqref{E:tildeS.tildePhi.defn} again, $\tilde{\Ss} \cap (\sigma \setminus |P_{Q}|) = \Ss \cap (\sigma \setminus |P_{Q}|)$ and 
		$\tilde{\Phi} \vert_{\sigma \setminus (\Ss \cup |P_{Q}|)} 
		                    = \Phi \vert_{\sigma \setminus (\Ss \cup |P_{Q}|)}$.
As for $\tilde{\Ss} \cap \sigma \cap |P_{Q}|$, note that by \eqref{E:intersection.of.simps}, $\sigma \cap |P_{Q}|$ is the union of simplices in $P_{Q}$ that are faces of $\sigma$. Let $\tau \in P_{Q}$ be such a simplex. If $\tau \cap |Q| \ne \varnothing$, then $\sigma \cap |Q| \ne \varnothing$.  But this contradicts the assumption that $\sigma \notin P_{Q}$.  Therefore, $\tau \cap |Q| = \varnothing$.  Hence, by part (\ref{I:no.change.off.nbhd.of.S.in.Q}) of theorem \ref{T:main.theorem}, $\tilde{\Ss} \cap \sigma \cap |P_{Q}| = \Ss \cap \sigma \cap |P_{Q}|$ and 
		$\tilde{\Phi} \vert_{(\sigma \cap |P_{Q}|) \setminus \Ss} 
		                    = \Phi \vert_{(\sigma \cap |P_{Q}|) \setminus \Ss}$.
This completes the proof of \eqref{E:Stilde.=.S.off. PQ}.

That part (\ref{I:no.change.off.nbhd.of.S.in.Q}) of the theorem still holds for $P$ is an obvious consequence of \eqref{E:Stilde.=.S.off. PQ}.

We still need to check that $\tilde{\Ss}$ is closed and $\tilde{\Phi}$ is continuous 
on $|P| \setminus \tilde{\Ss}$, prove that part (\ref{I:Phi.tilde.locally.Lip.if.Phi.is}) of the theorem still holds, and verify that assertion (\ref{I:arb.fine.subdivision}') holds.  By definition of the topology on $|P|$ (see appendix \ref{S:basics.of.simp.comps}), to show that $\tilde{\Ss}$ is closed it suffices to show that $\tilde{\Ss} \cap \sigma$ is closed in $\sigma$ for every $\sigma \in P$.  But by theorem \ref{T:main.theorem}, \eqref{E:Stilde.=.S.off. PQ}, and \eqref{E:tildeS.tildePhi.defn} whether $\sigma \in P_{Q}$ or $\sigma \in P \setminus P_{Q}$ since $\Ss$ is closed, $\tilde{\Ss} \cap \sigma$ is closed in $\sigma$. Thus, $\tilde{\Ss}$ is closed.

In a similar way we show that $\tilde{\Phi}$ is continuous on $|P| \setminus \tilde{\Ss}$.  Let $U \subset \F$ be open.  We must show that $\tilde{\Phi}^{-1}(U)$ is open in $|P| \setminus \tilde{\Ss}$.  But $\tilde{\Ss}$ is closed so it suffices to show that $\tilde{\Phi}^{-1}(U)$ is open in $|P|$.  Therefore, it suffices to show that $G(\sigma) := \tilde{\Phi}^{-1}(U) \cap \sigma$ is open in $\sigma$ for every $\sigma \in P$.  But this follows from theorem \ref{T:main.theorem}, \eqref{E:Stilde.=.S.off. PQ}, 
and the continuity of $\Phi$ on $|P| \setminus \Ss$. 

\emph{Next, we prove that part (\ref{I:Phi.tilde.locally.Lip.if.Phi.is}) of theorem \ref{T:main.theorem} still holds for $P$.}  Assume $\F$ has a metric, $d$, and $\Phi$ is locally Lipschitz on $|P| \setminus \Ss$.  Let $x \in |P| \setminus \tilde{\Ss}$ be arbitrary.  We need to find a neighborhood of $x$ on which $\tilde{\Phi}$ is Lipschitz.  
Let $P^{Q}$ be the subcomplex of $P$ consisting of all simplices in $P$ that are faces of simplices 
in $P \setminus P_{Q}$.  So $P \setminus P_{Q} \subset P^{Q}$.  Then, by \eqref{E:Stilde.=.S.off. PQ} we have $\tilde{\Ss} \cap |P^{Q}| = \Ss \cap |P^{Q}|$ and that $\tilde{\Phi}$ and $\Phi$ are both defined and agree on $|P^{Q}|  \setminus \Ss$. Thus, if $x \in |P| \setminus \bigl( |P_{Q}| \cup \tilde{\Ss} \bigr)$, $\tilde{\Phi}$ is Lipschitz in a neighborhood of $x$.  By part (\ref{I:Phi.tilde.locally.Lip.if.Phi.is}) of the theorem, the same thing is true if $x \in |P| \setminus \bigl( |P^{Q}| \cup \tilde{\Ss} \bigr)$.
  
So assume $x \in \bigl( |P_{Q}| \cap |P^{Q}| \bigr) \setminus \tilde{\Ss}$. By \eqref{E:tildeS.tildePhi.defn} and assumption on $\Phi$, there exists $K' < \infty$ and an open set $U \subset |P| \setminus \tilde{\Ss}$ s.t.\ $x \in U$ and $\tilde{\Phi}$ is Lipschitz on $U \cap |P^{Q}|$ with Lipschitz constant $K' < \infty$, say. By theorem \ref{T:main.theorem} part (\ref{I:Phi.tilde.locally.Lip.if.Phi.is}), we may assume $\tilde{\Phi}$ is Lipschitz on $U \cap |P_{Q}|$ and that the same Lipschitz constant $K'$ works for $\tilde{\Phi} \vert_{U \cap |P_{Q}|}$. 

By \eqref{E:x.in.exctly.1.simplex.intrr}, there exists a unique simplex $\sigma \in P$, s.t.\ $x \in  \text{Int} \, \sigma$.  By \eqref{E:St.sigma.is.open}, $\text{St} \, \sigma$ is an open neighborhood of $x$.
Therefore, for $t \in (0,1)$ the following set is also an open neighborhood of $x$.  
	\[
		V_{t,x} := t \bigl( (\text{St} \, \sigma) - x \bigr) + x.
	\]
Here, the vector operations are performed point-wise.  Pick $t > 0$ sufficiently small that $V_{t,x} \subset U$.  

We show that $\tilde{\Phi}$ is Lipschitz on $V_{t,x}$. If $\omega \in P$, let
	\begin{equation}   \label{E:rho.t.x.defn}
		\omega_{t,x} := t ( \omega -x) + x,
		         \quad t \in (0,1).
	\end{equation}
Then, $\omega_{t,x}$ is a simplex and
	\[
		\bigl\{ \omega_{t,x} \subset \RR^{N} : \omega \text{ is a face of some } \rho \in P \text{ s.t.\ } 
		                 \sigma \subset \rho \bigr\}
	\]
is a finite simplicial complex, call it $P_{t,x}$.  ($P_{t,x}$ is finite since $P$ is locally finite.)  If $\omega \in P$, then $\text{Int} \, \omega \subset \text{St} \, \sigma$ implies $\omega_{t,x} \in P_{t,x}$.   

Let $y, z \in V_{t,x}$. Then
there exist $\rho, \zeta \in P$ s.t.\ $\sigma \subset \rho \cap \zeta$ and $y \in \text{Int} \, \rho_{t,x}$ and $z \in \text{Int} \, \zeta_{x,t}$.  The simplex $\rho$ belongs to $P_{Q}$ or $P^{Q}$ and the same for $\zeta$.  Now by \eqref{E:Stilde.=.S.off. PQ} and assumption on $\Phi$, we have that $\tilde{\Phi}$ is Lipschitz on $V_{t,x} \cap |P^{Q}|$ with Lipschitz constant $K' < \infty$.  And by theorem \ref{T:main.theorem} part (\ref{I:Phi.tilde.locally.Lip.if.Phi.is}), $\tilde{\Phi}$ is Lipschitz on $V_{t,x} \cap |P^{Q}|$ also with Lipschitz constant $K'$. Hence, if both $\rho, \zeta \in P_{Q}$ or both $\rho, \zeta \in P^{Q}$ then
	\begin{equation}  \label{E:Lip.for.y.z}
		d \bigl[ \tilde{\Phi}(y), \tilde{\Phi}(z) \bigr] \leq K' | y - z |.
	\end{equation}

Now $P = P_{Q} \cup P^{Q}$ so, in particular, \eqref{E:Lip.for.y.z} holds if $\rho \subset \zeta$ 
or $\zeta \subset \rho$. Without loss of generality (WLOG) we may assume 
$\rho \in P_{Q}$, $\zeta \in P^{Q}$, $\rho \nsubseteq \zeta$, and $\zeta \nsubseteq \rho$.  
We have $\rho_{t,x}\cap \zeta_{t,x}\ne \varnothing$ because 
$\sigma_{t,x}\subset \rho_{t,x}\cap \zeta_{x,t}$.  Therefore, by \eqref{E:complex.contains.all.faces} and \eqref{E:intersection.of.simps}, $\rho_{t,x}\cap \zeta_{x,t}$ is a simplex in $P_{x,t}$.  
Moreover, $\rho_{t,x} \cap \zeta_{t,x} \subset P_{Q} \cap P^{Q}$. Hence, by \eqref{E:Lip.for.y.z} and corollary \ref{C:reverse.triangle.ineq.in.simp.cmplxs}, there exists $K < \infty$ depending only 
on $\rho_{x,t}$ and $\zeta_{x,t}$ and 
$\tilde{x}, \tilde{y} \in  \text{Int} \, (\rho_{t,x}\cap \zeta_{x,t}) \subset |P_{Q}| \cap |P^{Q}|$ s.t.\
	\begin{align*}
		d \bigl[ \tilde{\Phi}(y), \tilde{\Phi}(z) \bigr] 
		         &\leq d \bigl[ \tilde{\Phi}(y), \tilde{\Phi}(\tilde{y}) \bigr]  
			  + d \bigl[ \tilde{\Phi}(\tilde{y}), \tilde{\Phi}(\tilde{z}) \bigr]  
				  + d \bigl[ \tilde{\Phi}(\tilde{z}), \tilde{\Phi}(z)\bigr] \\
		    & \leq K' \bigl( |y - \tilde{y}| + |\tilde{y} - \tilde{z}| + |\tilde{z} - z| \bigr) \\
		    & \leq K K' |y - z|.
	\end{align*}
Since the complex $P_{t,x}$ is finite, we can assume $K$ works for any pair $\rho_{t,x}, \zeta_{t,x} \in P_{t,x}$ with $\text{Int} \, \rho_{t,x}, \, \text{Int} \, \zeta_{t,x} \subset V_{t,x}$. Thus, $\tilde{\Phi}$ is Lipschitz in $V_{t,x}$ and so part (\ref{I:Phi.tilde.locally.Lip.if.Phi.is}) of theorem \ref{T:main.theorem}, still holds for $P$.

Finally, we prove that (\ref{I:arb.fine.subdivision}') holds.  The only issue is whether the subdivision of $P_{Q}$ whose existence is asserted by part (\ref{I:arb.fine.subdivision}) of the theorem extends to a subdivision of all of $P$.  But one can always use a ``standard extension'' (Munkres \cite[Definition 7.12, pp. 76--77]{jrM66}) for this purpose. (Munkres' \cite[Definition 7.1, p. 69]{jrM66} definition of simplicial complex includes the requirement \eqref{E:intersect.finitely.many}.)
\end{proof}

\begin{proof}[Proof of corollary \ref{C:poly.sing.set.near}]
Let $\epsilon > 0$.  Define 
	\[
		A = \{ x \in |P| : dist(x, \Ss) \geq \epsilon \}.
	\]
Using subdivisions permitted by theorem \ref{T:main.theorem}, 
part (\ref{I:arb.fine.subdivision}), we may assume all simplices of $P$ have diameter $< \epsilon/2$.  Let $Q$ be the subcomplex of $P$ consisting of all simplices that intersect $\Ss$, and all faces of all such simplices.  Let $\sigma \in P$ have nonempty intersection with $A$.  
Suppose there exists $\tau \in Q$ s.t.\ $\sigma \cap \tau \ne \varnothing$.  By definition of $Q$ we may assume $\tau \cap \Ss \neq \varnothing$. Let $x_{1} \in \Ss \cap \tau$, $x_{2} \in  \sigma \cap \tau$, and $x_{3} \in \sigma \cap A$
Then
	\[
		\epsilon \leq |x_{1} - x_{3}| 
		  \leq |x_{1} -  x_{2}| + |x_{2} - x_{3}| < \tfrac{1}{2} \epsilon + \tfrac{1}{2} \epsilon = \epsilon,
	\]
A contradiction.  I.e.,
	\begin{equation*}  
		\text{No simplex intersecting $A$ is a face of any simplex intersecting } |Q|.
	\end{equation*}

Apply theorem \ref{T:main.theorem} to $Q$.  Point \ref{I:S.tilde.is.subcmplx.of.subdiv} of the corollary follows from parts (\ref{I:arb.fine.subdivision} and \ref{I:S.tilde.subcomp}) of the theorem. By part (\ref{I:dist.tween.S.tilde.and.S}) of theorem \ref{T:main.theorem}, point (\ref{I:Stilde.within.eps.ofS}) of the corollary holds.  By part (\ref{I:no.change.off.nbhd.of.S.in.Q}) of theorem \ref{T:main.theorem}, $\tilde{\Phi} = \Phi$ on $A$.  I.e., point \ref{I:same.far.away} of the corollary holds.  Point \ref{I:Phi.tilde.sigma.subset.Phi.sigma.in.corly} of the corollary follows from parts (\ref{I:arb.fine.subdivision} and \ref{I:Phi.tilde.sigma.subset.Phi.sigma}) of the theorem. Finally, point \ref{I:Hma.meas.Stilde.<.K.x.that.of.S} of the corollary follows from part (\ref{I:arb.fine.subdivision}) of theorem \ref{T:main.theorem}.
\end{proof}

\begin{proof}[Proof of lemma \ref{L:big.lemma}]
The following technical lemma will come in handy.
	\begin{lemma}   \label{L:k.theory}
Recall the definition, \eqref{E:defn.of.k}, of $k$.  We have
	\begin{equation}  \label{E:k.1-beta.bnd}
		k(\beta, \delta, t) \leq (1-\beta)^{\delta + t}, \text{ for } \beta \in [0,1], \; \delta + t > 0.
	\end{equation}
The gradient of $k$ is given by
	\begin{equation}  \label{E:gradient.of.k}
		   \nabla k(\beta, \delta, t) =
		   - \left( \frac{\delta + t}{(1 - \beta)^{2}} , \frac{\beta}{1 - \beta}  , \frac{\beta}{1 - \beta}\right) 
		   \exp \left\{ - \frac{\beta( \delta + t )}{1 - \beta} \right\}, \; 
		      \text{ for } \beta < 1.
	\end{equation}
Both $k$ and its first partial derivatives are continuous 
on $(\beta, \delta, t) \in (-\infty, 1) \times \RR \times \RR$. 
Let $ \ell = k$ or $ \ell = \partial k/ \partial \beta$ or $ \ell = \partial k/ \partial \delta$ 
or $ \ell = \partial k/ \partial t$. 
If $\delta > 0$ and $t' \geq 0$ then 
	\[
		\lim_{\beta \uparrow 1, t \downarrow 0}  \ell(\beta, \delta, t) 
			= \lim_{\beta \uparrow 1}  \ell(\beta, \delta, t') = 0.
	\]
Similarly, if $\delta' \geq 0$ and $t > 0$ then 
	\[
		\lim_{\beta \uparrow 1, \delta \downarrow 0}  \ell(\beta, \delta, t) 
			= \lim_{\beta \uparrow 1}  \ell(\beta', \delta, t) = 0.
	\]
It follows that $k$ is locally Lipschitz on 
	\begin{equation}   \label{E:defn.of.T}
		T := \bigl[ (-\infty, 1] \times [0, \infty) \times [0, \infty) \bigr]
		    \; \setminus \; \bigl\{ (1, 0, 0) \bigr\}.
	\end{equation}
Note that $k(\beta, \delta, t) \in [0,1]$ if $(\beta, \delta, t) \in T$ and $\beta \in [0,1]$.  Note further that 
$T$ is convex.
	\end{lemma}
Note that only one of conditions $\beta < 1$, $ \delta > 0$, $t > 0$ need hold in order 
that $(\beta, \delta, t) \in (-\infty, 1] \times [0, \infty) \times [0, \infty)$ is actually in $T$. 
	\begin{proof}[Proof of lemma \ref{L:k.theory}]
By definition of $k$, \eqref{E:k.1-beta.bnd} holds with $\beta = 1$.  So assume $0 \leq \beta < 1$.  Since 
	\[
		k(\beta, \delta, t) = \exp \left\{ - \frac{\beta}{1 - \beta} \right\}^{ \delta + t },
	\]
WLOG we may take $\delta + t = 1$.  Let 
	\[
		f(\beta) = \log(1-\beta) - \log k( \beta, \delta, t) 
		  = \log(1-\beta) + \frac{\beta}{1-\beta}.
	\]  
Then
	\[
		f(0) = 0 \text{ and } f'(\beta) = - \frac{1}{1-\beta} + \frac{1}{(1-\beta)^{2}}.
	\]
But the $f'$ is nonnegative if $0 \leq \beta < 1$.  \eqref{E:k.1-beta.bnd} follows.

The existence and values of the limits follows from the fact that for $r \geq 0$
	\[
		x^{r} e^{-x} \to 0 \text{ as } x \to + \infty.
	\]

By corollary \ref{C:cont.diff.=.loc.Lip}, $k$ is locally Lipschitz 
on $(\beta, \delta, t) \in (-\infty, 1) \times \RR \times \RR$, but we also want to allow $\beta = 1$ so another approach is required.  Let $x_{1} = (\beta_{1} , \delta_{1}, t_{1}) \in T$ and 
$x_{2} = (\beta_{2} , \delta_{2}, t_{2}) \in T$.  Let $r > |x_{2} - x_{1}|$.   Since $x_{1}, x_{2} \in T$, there exists $\epsilon > 0$ s.t.\ 
	\[
		\min \bigl\{ (1-\beta_{1}) + \delta_{1} + t_{1}, \;  
		  (1-\beta_{2}) + \delta_{2} + t_{2} \bigr\} > \epsilon.
	\]
Let
	\[
		K = \bigl\{ x = (\beta, \delta, t) \in T : |x - x_{1}| \leq r 
		   \text{ and } (1-\beta) + \delta + t \geq \epsilon \bigr\}.
	\]
It is clear that $T$ is convex. Therefore, $K$ is convex, being the intersection of three convex sets.  $K$ is also compact and contains $x_{1}$ and $x_{2}$.  $k$ is continuous on $K$ and its derivatives are continuous in the interior, $K^{\circ}$, of $K$.  Moreover, the derivatives of $k$ can be extended continuously to all of $K$ with finite values on the boundary of $K$.  Hence, since $K$ is compact, the derivatives of $k$ are bounded in the interior of $K$.  Let $D < \infty$ be an upper bound on the length of the gradient of $k$ in $K^{\circ}$.  Suppose $\beta_{1} = \beta_{2} = 1$.  Then, $\bigl| k(x_{2}) - k(x_{1}) \bigr| = 0 \leq D |x_{2} - x_{1}|$. Suppose for definiteness that $\beta_{1} < 1$.  Consider the line segment, $L$, joining $x_{1}$ and $x_{2}$.  It lies in $K$. Let $f$ denote the restriction, $k \vert_{D}$, of $k$ to $L$ parametrized by arc length.  Then $f$ is differentiable in the interior of $L$ (this is true even if $x_{1}$ and $x_{2}$ lie on the boundary of $K$, since $\beta_{1} < 1$) and derivative of $f$ is less than $D$ in absolute value.  Applying the Mean Value Theorem (Apostol \cite[Theorem 5--10, p.\ 93]{tmA57.Apostol}) we see that $\bigl| k(x_{2}) - k(x_{1}) \bigr| \leq D |x_{2} - x_{1}|$.
       \end{proof}

\emph{(Proof of lemma \ref{L:big.lemma} continued.)} We have
	\begin{equation}  \label{E:s.in.terms.of.k}
	      s(y,t) = \bar{h}(y) + k \bigl( b(y), \bar{\Delta}(y), t \bigr) \bigl[ z - \bar{h}(y) \bigr],
	         \quad y \in \text{Int} \, \sigma, \; t \in [0,1].
	\end{equation}
Let $y \in \sigma$ and $t \in [0,1]$.  By \eqref{E:defn.of.k} 
$0 \leq k \bigl( b(y), \bar{\Delta}(y), t \bigr) \leq 1$ so by \eqref{E:syt.defn},
	\begin{equation} \label{E:s.maps.sigma.into.sigma}
		s(y,t) \in \sigma, \quad \text{ if } 
			y \in \sigma, \; t \in [0,1].
	\end{equation}

The following describes some properties of the function $s$ defined in \eqref{E:syt.defn}, $\bar{h}$ defined in \eqref{E:b.hbar.defn}, and $f$ defined in \eqref{E:f.defn}.  
\begin{lemma}   \label{L:props.of.s}
Except where noted, $y \in \sigma$ and $t \in [0,1]$ are arbitrary.
\begin{enumerate}
      \item $s(y,0) = z$ for every $y \in \mcl{C} \setminus (\text{Bd} \, \sigma)$.   
          \label{I:s.y0.=z.on.C.less.Bd}
      \item $s(z,t) = z$.  If  $s(y,t) = z$ and $t \in (0,1]$ then $y = z$.   
           \label{I:z.is.fixed.pt.of.s}
      \item $s(y,t) \in \text{Bd} \, \sigma$ if and only if $y \in \text{Bd} \, \sigma$ 
         (in which case $s(y,t) = y$).
              \label{I:s.in.Bd.iff.y.in.Bd}
      \item If $s(y,t) \ne z$, then $y \ne z$ and $\bar{h} \bigl[ s(y,t) \bigr] = \bar{h}(y)$.   
           \label{I:hbar.s.=.hbar.y}
      \item $s(y,t) \in \mcl{C}$ if and only if $y \in \mcl{C}$.     	\label{I:s(y,t).in.C.iff.y.in.C}
      \item There exists $K'' = K''(z) < \infty$ depending only on $\sigma$ and $z$, and continuous in $z$, s.t.\ if $y_{1}, y_{2} \in \sigma$
	\begin{equation}  \label{E:|yi-z||h2-h1|.Lip.condtn}
		|y_{i} - z|  \bigl| \bar{h}(y_{2}) - \bar{h}(y_{1}) \bigr|  \leq K'' |y_{2} - y_{1}|, \quad i=1,2.
	\end{equation}
	    \label{I:y-z.h2-h1.bound}
      \item $\bar{h}(y)$ is locally Lipschitz in $y \in \sigma \setminus \{ z \}$, $\bar{h}$ is Lipschitz 
      on $\mcl{A}$, and
         \begin{equation}   \label{E:b.is.Lip.on.sigma}
		b(y) \text{ is Lipschitz on } \sigma.
	\end{equation}
	      \label{I:hbar.locally.Lip.b.Lip} 
      \item The function $f : (y,t) \mapsto k \bigl( b(y), \bar{\Delta}(y), t \bigr)$ is a locally Lipschitz map on	\begin{equation}   \label{E:Bz.defn}
	      B_{z} := \bigl( \sigma \times [0,1] \bigr)
	         \setminus \Bigl( \bigl[ \mcl{C} \cap (\text{Bd} \, \sigma) \bigr]  \times \{0\} \Bigr).
	\end{equation}
	\label{I:f=k(b,delta,t).loc.Lip}
      \item  $s$ is locally Lipschitz on $B_{z}$.          \label{I:s.locly.Lip.off.C.cap.Bd}
   \end{enumerate}
\end{lemma}
See below for the proof. The following is also proved below. 
  \begin{lemma}  \label{L:s.hat.has.Lip.invrs}
     For any $t \in (0,1]$ the map $y \mapsto s(y,t)$ is a one-to-one map of $\text{Int} \, \sigma$ onto itself.  Moreover, 
		   \begin{equation}  \label{E:s.t.has.locally.Lip.inv}
		      \text{the map } \hat{s} : (y,t) \mapsto \bigl( s(y,t), t \bigr) \text{ on } 
		            (\text{Int} \, \sigma) \times (0,1] \text{ has a locally Lipschitz inverse.}
		   \end{equation}
  \end{lemma}
  
Let $(\text{Lk} \, \sigma)^{(0)}$ denote the set of vertices in $\text{Lk} \, \sigma$.  Thus, every vertex in $\overline{\text{St}} \, \sigma$ is either in $\sigma^{(0)}$ or $(\text{Lk} \, \sigma)^{(0)}$. Let $x \in \overline{\text{St}} \, \sigma$.  Unless $x \in \text{Lk} \, \sigma$, some of the $\beta_{v}(x)$'s are positive at $v \in \sigma^{(0)}$.  In that case, define 
	\begin{equation}  \label{E:sigma.x.defn}
	      \sigma(x) = \left( \sum_{v \in \sigma^{(0)}} \beta_{v}(x) \right)^{-1}
	        \sum_{v \in \sigma^{(0)}} \beta_{v}(x) \, v \quad \in \; \sigma, 
	        \qquad x \in (\overline{\text{St}} \, \sigma) \setminus (\text{Lk} \, \sigma).
	\end{equation}
($\sigma(\cdot)$ will denote the function.  $\sigma$ without the parentheses will mean the simplex.) Then 
	\begin{equation}  \label{E:sigma.ident.on.sigma}
		\sigma(x) = x \text{  if } x \in \sigma
	\end{equation}
 and by lemma \ref{L:bary.coords.are.Lip}, corollary \ref{C:cont.diff.=.loc.Lip}, and \eqref{E:comp.of.Lips.is.Lip}, we have
	\begin{equation}  \label{E:sigma.(.).is.loc.Lip.off.Lk}
		\sigma(\cdot) \text{ is locally Lipschitz on } 
		    (\overline{\text{St}} \, \sigma) \setminus (\text{Lk} \, \sigma).
	\end{equation}
We have
	\begin{equation}  \label{E:x.in.rho.intrr.iff.sigma.x.in.sigma.intrr}
		\text{If } \rho \in P \text{ has } \sigma \text{ as a face (so }
		      \rho \subset \overline{\text{St}} \, \sigma \text{) and } 
		    x \in \text{Int} \, \rho, \text{ then } \sigma(x) \in \text{Int} \, \sigma
	\end{equation}
because $x \in \text{Int} \, \rho$ if and only if $\beta_{v}(x) > 0$ for every vertex $v$ of $\rho$.  In particular, $\beta_{v}(x) > 0$ for every vertex $v$ of $\sigma$.  

Partially conversely, suppose $\rho \in P$, $\rho \subset \overline{\text{St}} \, \sigma$, and $\rho \cap \sigma \ne \varnothing$ but $\sigma$ is not a face of $\rho$.  (In particular, $\rho \ne \sigma$. Note that the case in which $\rho$ is a proper face of $\sigma$ is included.)  Suppose  $x \in \rho$ but $x \notin \text{Lk} \, \sigma$.  Then $\sum_{v \in \sigma^{(0)}} \beta_{v}(x) > 0$ but since $\sigma$ is not a face of $\rho$, we have $\beta_{v}(x) = 0$ for some vertex $v$ of $\sigma$.  Say, $\beta_{v(j)}(x) = 0$.  Then by \eqref{E:sigma.x.defn} $\beta_{v(j)} \bigl[ \sigma(x) \bigr] = 0$.  To sum up, we have
	\begin{multline}  \label{E:when.sigma.x.in.Bd.sigma}
		\text{\emph{If} } \rho \in P, \; \rho \subset  \overline{\text{St}} \, \sigma, 
		       \text{  and  } \rho \cap \sigma \ne \varnothing 
		    \text{  but  } \sigma  \text{ is not a face of  } \rho  \\
			\text{  \emph{then}  } x \in \rho \setminus (\text{Lk} \, \sigma) 
			  \text{ implies } \sigma(x) \in \text{Bd} \, \sigma.
	\end{multline}

Suppose $\overline{\text{St}} \, \sigma \ne \sigma$.  (I.e., $\sigma$ is a proper subset 
of $\overline{\text{St}} \, \sigma$.)  Then there exists 
$\rho \in P$ with $\rho \subset \overline{\text{St}} \, \sigma$ having $\sigma$ as a \emph{proper} face.  Let $\omega \subset \rho$ be the face of $\rho$ ``opposite'' $\sigma$ (appendix \ref{S:basics.of.simp.comps}).  If $y \in \omega$ define $\sigma(y)$ to be an arbitrary point of $\sigma$. Then if $y \in \rho$ we can write 
	\begin{equation}  \label{E:y.in.terms.of.sigma.w}
		y = \mu \sigma(y) + (1 - \mu) w
	\end{equation}
where $\mu \in [0,1]$ and $w \in \omega$.  Now, suppose $y = \nu x + (1-\nu) u$, 
where $\nu \in [0,1]$, $x \in \sigma$, and $u \in \omega$.  Then, by the geometric independence (appendix \ref{S:basics.of.simp.comps}) of the vertices of $\rho$ we have $\nu = \mu$.  
If $y \notin \text{Lk} \, \sigma$ then we also have $x = \sigma(y)$.  If $y \notin \sigma$, then we also have $u = w(y)$.  (In particular, $\mu$ and $w$ are uniquely determined by $y \in \text{Int} \, \rho$.)  So we can define $\mu(y) = \mu$ and, providing $y \notin \sigma$, we can define $w(y) = w$.  Adopt the following conventions.
   \begin{small}
	        \begin{align} \label{E:w.on.sigma,sigma(.).on.Lk}
		    w(y)& = 
		       \begin{cases}
			\text{ an arbitrary fixed point of } \text{Lk} \, \sigma,  
			        & \text{ if } y \in \sigma \text{ and } \overline{\text{St}} \, \sigma \ne \sigma, \\
			\text{ and arbitrary fixed point of } |P|, 
			               &\text{ if } \overline{\text{St}} \, \sigma = \sigma.
		      \end{cases}  \\
			\text{If } y &\in \text{Lk} \, \sigma,  
			  \text{ let } \sigma(y) 
			         \text{ be an arbitrary fixed point of } \sigma.    \notag
		    \end{align}
   \end{small}

Clearly,
	\begin{align}  \label{E:w.and.mu.in.terms.of.bary.coords}
		w(y) &= \frac{1}{\sum_{v \in (\text{Lk} \, \sigma)^{(0)}} \beta_{v}(y) } 
		    \sum_{v \in (\text{Lk} \, \sigma)^{(0)}} \beta_{v}(y) v, \text{ if }
		    y \in (\overline{\text{St}} \, \sigma) \setminus \sigma,
		    \text{ and } \\
		\mu(y) &= \sum_{v \in \sigma^{(0)}} \beta_{v}(y), \text{ if }
		    y \in \overline{\text{St}} \, \sigma.    \notag
	\end{align}

It follows that if there are multiple simplices $\rho \in P$ with $\rho \subset \overline{\text{St}} \, \sigma$ having $\sigma$ as a proper face, and s.t.\ $y \in \rho$ then $\mu(y)$ and $w(y)$ are the same no matter which of these $\rho$'s we use to define them. We then have
	\begin{multline}  \label{E:wy.muy.when.y.in.Lk.or.sigma}
		 y \in \text{Lk} \, \sigma \text{ if and only if } w(y) = y  \text{ if and only if } \mu(y) = 0.  \\
		y \in \sigma \text{ if and only if } \mu(y) = 1  \text{ if and only if } \sigma(y) = y.
	\end{multline}

Since by lemma \ref{L:bary.coords.are.Lip}, $\{ \beta_{v}(y), \, v \in P^{(0)} \}$ is Lipschitz in $y$  and $\mu(y)$ and $w(y)$ can be expressed in terms of barycentric coordinates, it follows that 
	\begin{equation}   \label{E:mu.w.loc.Lip}
		\mu \text{ is Lipschitz on }  \overline{\text{St}} \, \sigma 
		     \text{ and } w \text{ is locally Lipschitz on } (\overline{\text{St}} \, \sigma) 
		           \setminus \sigma. 
	\end{equation} 

Recall \eqref{E:g.defn.when.St.sigma.ne.sigma} and \eqref{E:g.defn.when.St.=.sigma}. 
If $\overline{\text{St}} \, \sigma = \sigma$ define $\mu(y) = 1$ for every $y \in \overline{\text{St}} \, \sigma = \sigma$ and let $w(y)$ be an arbitrary fixed point of $|P|$. Then we see that \eqref{E:g.defn.when.St.sigma.ne.sigma} still makes sense if $\overline{\text{St}} \, \sigma = \sigma$. 
By \eqref{E:wy.muy.when.y.in.Lk.or.sigma} we have 
	\begin{equation*}   
		g(y) = s(y,0), \text{ if } y \in \sigma \setminus \bigl[ \mcl{C} \cap (\text{Bd} \, \sigma) \bigr].
	\end{equation*}
Thus, whether $\overline{\text{St}} \, \sigma$ equals $\sigma$  or not we have
	\begin{multline}  \label{E:general.formula.for.g}
		g(y) = g_{z}(y) = \\
			\begin{cases}
				\mu(y) \, s_{z} \bigl( \sigma(y), 1 - \mu(y) \bigr) + \bigl( 1 - \mu(y) \bigr) w(y) 
				                  \in (\overline{\text{St}} \, \sigma) \setminus \sigma, 
				   &\text{ if } y \in \overline{\text{St}} \, \sigma 
				             \setminus \sigma, \\
			       s(y,0) 
			               = \mu(y) \, s_{z} \bigl( \sigma(y), 1 - \mu(y) \bigr) + \bigl( 1 - \mu(y) \bigr) w(y) 
			                    \in \sigma, 
			                    &\text{ if } 
			                                y \in \sigma \setminus \bigl[ \mcl{C} \cap (\text{Bd} \, \sigma) \bigr], \\
				y, &\text{ if } y \in |P| \setminus (\overline{\text{St}} \, \sigma).
			\end{cases}
	\end{multline}
 (See \eqref{E:w.on.sigma,sigma(.).on.Lk}.)

Note that, by lemma \ref{L:props.of.s}(part \ref{I:s.in.Bd.iff.y.in.Bd}), \eqref{E:y.in.terms.of.sigma.w}, 
and \eqref{E:wy.muy.when.y.in.Lk.or.sigma} 
   \begin{align}  \label{E:when.g(y)=y}
      g(y) = y \text{ if: } \; &  \notag \\
        y &\in (\overline{\text{St}} \, \sigma) \setminus \bigl[ \mcl{C} \cap (\text{Bd} \, \sigma) \bigr] 
           \text{ but } \sigma(y) \in \text{Bd} \, \sigma,  \\
        y &\in \text{Lk} \, \sigma 
          \text{ (i.e., } \mu(y) = 0), \text{ or }  \notag \\
        y &\in |P| \setminus (\overline{\text{St}} \, \sigma). \notag
   \end{align}

\emph{Claim:}
	\begin{equation}  \label{E:off.sig.g.maps.simps.into.thmslvs}
		\text{If } \rho \in P \text{ then } 
		       g ( \rho \setminus \sigma) \subset \rho \setminus \sigma.
	\end{equation}
To see this, let $\rho \in P$ and $y \in \rho \setminus \sigma$.  If $y \notin \overline{\text{St}} \, \sigma$ then by \eqref{E:when.g(y)=y}, $g(y) = y \in \rho \setminus \sigma$.  So suppose $y \in (\overline{\text{St}} \, \sigma) \setminus \sigma$.  $\rho$ has a face $\xi$ s.t.\ $\xi \subset \overline{\text{St}} \, \sigma$ and $y \in \xi$.  If $y \in \text{Lk} \, \sigma$, then, again by \eqref{E:when.g(y)=y}, $g(y) = y \in \rho \setminus \sigma$.  So suppose $y \notin \text{Lk} \, \sigma$.  Then $\xi \cap \sigma \ne \varnothing$.  If $\sigma$ is not a face of $\xi$ then, by \eqref{E:when.sigma.x.in.Bd.sigma} and \eqref{E:when.g(y)=y}, we have $g(y) = y \in \rho \setminus \sigma$.  Finally, suppose $\sigma$ is a face of $\xi$ (and hence of $\rho$).  Then, by \eqref{E:s.maps.sigma.into.sigma} and \eqref{E:general.formula.for.g}, we have $g(y) \in \xi \subset \rho$.  But $y \notin \sigma$.  Therefore, by \eqref{E:wy.muy.when.y.in.Lk.or.sigma}, $\mu(y) < 1$.  So, by \eqref{E:general.formula.for.g} again, $\mu \bigl[ g(y) \bigr] < 1$.  So, by \eqref{E:wy.muy.when.y.in.Lk.or.sigma} again, $g(y) \in \rho \setminus \sigma$.  This proves the claim \eqref{E:off.sig.g.maps.simps.into.thmslvs}.  

By \eqref{E:when.sigma.x.in.Bd.sigma} and \eqref{E:when.g(y)=y}, $g$ is the identity on all simplices in $P$ that do not have $\sigma$ as a face (so in particular do not equal $\sigma$). 
Moreover, by \eqref{E:wy.muy.when.y.in.Lk.or.sigma} and lemma \ref{L:props.of.s}(part \ref {I:s.y0.=z.on.C.less.Bd}), $g(\mcl{C} \setminus  \text{Bd} \, \sigma) = \{ z \}$.  

\emph{Claim:}
	\begin{equation}  \label{E:g.is.cont.off.C.meet.Bd.sigma}
		g \text{ is a continuous map of } |P| \setminus \bigl[ \mcl{C} \cap (\text{Bd} \, \sigma) \bigr]
		\text{ into itself.}
	\end{equation}
First, we show 
$g : |P| \setminus \bigl[ \mcl{C} \cap (\text{Bd} \, \sigma) \bigr] \to |P| \setminus \bigl[ \mcl{C} \cap (\text{Bd} \, \sigma) \bigr]$. Suppose $y \in |P| \setminus \bigl[ \mcl{C} \cap (\text{Bd} \, \sigma) \bigr]$.  
If $y \notin \overline{\text{St}} \, \sigma$, 
then $g(y) = y \notin \mcl{C} \cap (\text{Bd} \, \sigma)$.  
If $y \in (\overline{\text{St}} \, \sigma) \setminus \sigma$, then, by \eqref{E:wy.muy.when.y.in.Lk.or.sigma}, $\mu(y) < 1$, 
so $g(y) \notin \sigma \supset \mcl{C} \cap (\text{Bd} \, \sigma)$.  
Suppose $y \in \sigma \setminus \bigl[ \mcl{C} \cap (\text{Bd} \, \sigma) \bigr]$.  
Then, by \eqref{E:wy.muy.when.y.in.Lk.or.sigma}, $\mu(y) = 1$, $\sigma(y) = y$, and 
by lemma \ref{L:props.of.s}(parts \ref{I:s.in.Bd.iff.y.in.Bd}, \ref{I:s(y,t).in.C.iff.y.in.C}), 
$s \bigl( \sigma(y), 1 - \mu(y) \bigr) = s(y, 0) \notin \mcl{C} \cap (\text{Bd} \, \sigma)$.  
I.e., $g : |P| \setminus \bigl[ \mcl{C} \cap (\text{Bd} \, \sigma) \bigr] \to |P| \setminus \bigl[ \mcl{C} \cap (\text{Bd} \, \sigma) \bigr]$.

By lemma \ref{L:props.of.s}(part \ref{I:s.locly.Lip.off.C.cap.Bd}),
\eqref{E:mu.w.loc.Lip}, 
and \eqref{E:sigma.(.).is.loc.Lip.off.Lk}, $g$ is continuous 
on $(\overline{\text{St}} \, \sigma) \setminus (\sigma \cup (\text{Lk} \, \sigma))$ 
(or $(\overline{\text{St}} \, \sigma) \setminus \{ \sigma \cup (\text{Lk} \, \sigma) \} = \varnothing$).  
Let $y \in \sigma \setminus \bigl[ \mcl{C} \cap (\text{Bd} \, \sigma) \bigr]$ so, by \eqref{E:sigma.ident.on.sigma}, $\sigma(y) = y$.  Let $\{ y_{m} \} \subset |P|$ be a sequence converging to $y$.  It suffices to consider two separate cases.  

Suppose $\{ y_{m} \} \subset |P| \setminus (\overline{\text{St}} \, \sigma)$. Since $P$ is finite WLOG, there exists $\rho \in P$ s.t.\ $\{ y_{m} \} \subset \rho$. Hence, if $y \in \text{Int} \, \sigma$ then, by \eqref{E:Int.rho.cuts.sigma.then.rho.in.sigma}, $\sigma$ is a face of $\rho$.  
This contradicts $\{ y_{m} \} \subset |P| \setminus (\overline{\text{St}} \, \sigma)$. 
Therefore, $y \in (\text{Bd} \, \sigma) \setminus \mcl{C}$.  
Therefore, $\sigma(y) = y \in  \text{Bd} \, \sigma$ and, by \eqref{E:when.g(y)=y}, $g(y_{m}) = y_{m} \to y = g(y)$. 

Next, suppose $\{ y_{m} \} \subset \overline{\text{St}} \, \sigma$. 
Then eventually $y_{m} \in (\text{St} \, \sigma) \setminus (\text{Lk} \, \sigma)$.  By \eqref{E:sigma.(.).is.loc.Lip.off.Lk}, \eqref{E:mu.w.loc.Lip}, and \eqref{E:wy.muy.when.y.in.Lk.or.sigma}, 
$y_{m} \to y \in \sigma \setminus \bigl[ \mcl{C} \cap (\text{Bd} \, \sigma) \bigr]$ implies $\mu(y_{m}) \to 1 = \mu(y)$ and $\sigma(y_{m}) \to \sigma(y) = y$.  Hence, eventually, $\sigma(y_{m}) \notin \mcl{C} \cap (\text{Bd} \, \sigma)$.  Thus, $g(y_{m}) \to y$ by lemma \ref{L:props.of.s} (part \ref{I:s.locly.Lip.off.C.cap.Bd}).  

It suffices, finally, to consider $y \in \text{Lk} \, \sigma$.    Let $\{ y_{m} \} \subset |P|$ be a sequence converging to $y$.  By \eqref{E:general.formula.for.g} and \eqref{E:when.g(y)=y} it suffices to prove that if 
$\{ y_{m} \} \subset (\overline{\text{St}} \, \sigma) \setminus (\text{Lk} \, \sigma)$ then $g(y_{m}) \to y$.  
By 
and \eqref{E:wy.muy.when.y.in.Lk.or.sigma}, $\mu(y_{m}) \to 0$ and $w(y_{m}) \to y$ as $m \to \infty$.  Since the range of $s$ is the bounded set $\sigma$, we have $g(y_{m}) \to y = g(y)$ as $m \to \infty$.  This completes the proof of the claim \eqref{E:g.is.cont.off.C.meet.Bd.sigma}.

We have the following.  (The proof is given below.)
	\begin{lemma}  \label{L:g.is.loc.Lip}
	$g$ is a locally Lipschitz map on $|P| \setminus \bigl[ \mcl{C} \cap (\text{Bd} \, \sigma) \bigr]$.
	\end{lemma}

Note further that, by \eqref{E:wy.muy.when.y.in.Lk.or.sigma} and \eqref{E:general.formula.for.g}, 
	\begin{equation}  \label{E:g.maps.sigma.into.itself}
		g \Bigl( \sigma \setminus \bigl[ \mcl{C} \cap (\text{Bd} \, \sigma) \bigr] \Bigr) \subset \sigma 
		       \text{ and } g \bigl( |P| \setminus \sigma \bigr) \subset |P| \setminus \sigma.
	\end{equation}
It follows from this and the definition of $\Ss'$, \eqref{E:S'.defn},  that 
	\begin{equation}  \label{E:S'.cap.sigma.=.C.cap.Bd.sigma}
		\Ss' \cap \sigma = \mcl{C} \cap (\text{Bd} \, \sigma).
	\end{equation}
This is the first sentence of point (\ref{I:S'.cap.sigma.=.C.cap.Bd}) of the lemma.  The second sentence is then immediate from \eqref{E:C.is.compact.contains.A}.
By \eqref{E:hbar.maps.A.to.C.cap.Bd.sigma}  
and lemma \ref{L:props.of.s} (part \ref{I:hbar.locally.Lip.b.Lip}) 
$\mcl{C} \cap (\text{Bd} \, \sigma)$ is a Lipschitz image of $\mcl{A} = \Ss \cap \sigma$. 
Point (\ref{I:S'.cap.sigma.=.C.cap.Bd}) then follows from lemma \ref{L:loc.Lip.image.of.null.set.is.null}.

Let $\rho \in P$ not be a face of $\sigma$.  Note that if $y \in \text{Int} \, \rho$, then by \eqref{E:Int.rho.cuts.sigma.then.rho.in.sigma}, $y \notin \mcl{C} \cap (\text{Bd} \, \sigma)$ so $g(y)$ is defined.  \emph{First, suppose $\rho$ has $\sigma$ as a proper face} and let $\omega$ be the face of $\rho$ opposite to $\sigma$.  (In particular, $\rho \subset \overline{\text{St}} \, \sigma$.) Let $y \in \text{Int} \, \rho$, then by \eqref{E:criterion.for.int.simp} and \eqref{E:w.and.mu.in.terms.of.bary.coords}, we have $w(y) \in \text{Int} \, \omega$ and, by \eqref{E:x.in.rho.intrr.iff.sigma.x.in.sigma.intrr}, we have $\sigma(y) \in \text{Int} \, \sigma$. Hence, by lemma \ref{L:props.of.s}(part \ref{I:s.in.Bd.iff.y.in.Bd}) we have $s \bigl( \sigma(y), 1 - \mu(y) \bigr) \in \text{Int} \, \sigma$.  Moreover, by \eqref{E:wy.muy.when.y.in.Lk.or.sigma}, $1 > 1 - \mu(y) > 0$.  Therefore, by \eqref{E:general.formula.for.g}, $g(y) \in \text{Int} \, \rho$.  

Conversely, let $y' \in \text{Int} \, \rho$.  We show that there exists 
$y \in \text{Int} \, \rho$ s.t.\  $y' = g(y)$.  Write $y' = \mu(y') \sigma(y') + \bigl( 1 - \mu(y') \bigr) w(y')$.  
Since $y' \in \text{Int} \, \rho$, we have, as before, $w(y') \in \text{Int} \, \omega$, $\sigma(y') \in \text{Int} \, \sigma$, and $1 > 1 - \mu(y') > 0$. By lemma \ref{L:s.hat.has.Lip.invrs}, there exists a unique $x \in \text{Int} \, \sigma$ s.t.\  
$s\bigl( x, 1 - \mu(y') \bigr) = \sigma(y')$.  Let $y = \mu(y') x + \bigl( 1 - \mu(y') \bigr) w(y')$.  Then
$y \in \text{Int} \, \rho$ and $\mu(y) = \mu(y')$, $\sigma(y) = x$, and $w(y) = w(y')$.  Thus, by \eqref{E:general.formula.for.g}, 
$y' = \mu(y') s\bigl( x, 1 - \mu(y') \bigr) + \bigl( 1 - \mu(y') \bigr) w(y) = g(y)$, as desired.  
This proves that $g(\text{Int} \, \rho) = \text{Int} \, \rho$. Now, as observed following \eqref{E:y.in.terms.of.sigma.w}, we have that $\mu(y)$ ($\mu(y')$) and $w(y)$ (resp. $w(y')$) are uniquely determined by $y \in \text{Int} \, \rho$ (respectively [resp.], $y' \in \text{Int} \, \rho$).  Therefore, by lemma \ref{L:s.hat.has.Lip.invrs}, $y \in \text{Int} \, \rho$ s.t.\ $g(y) = y'$ is uniquely determined by $y' \in \text{Int} \, \rho$.

Let $y, y'$ be as in the last paragraph. Now, $\mu(y')$, $\sigma(y')$, and $w(y')$ are uniquely determined by $y' \in \text{Int} \, \rho$ in a locally Lipschitz fashion (by \eqref{E:sigma.(.).is.loc.Lip.off.Lk} and \eqref{E:mu.w.loc.Lip}).  Hence, by \eqref{E:s.t.has.locally.Lip.inv} and  \eqref{E:comp.of.Lips.is.Lip}, 
$x = x(y') \in \text{Int} \, \sigma$ solving $s\bigl( x, 1 - \mu(y') \bigr) = \sigma(y')$ is locally Lipschitz in $y' \in \text{Int} \, \rho$.  Therefore, by \eqref{E:comp.of.Lips.is.Lip} again, we have that
$y = \mu(y') x + \bigl( 1 - \mu(y') \bigr) w(y') \in \text{Int} \, \rho$ solving $g(y) = y'$ is uniquely determined by $y' \in \text{Int} \, \rho$ in a locally Lipschitz fashion.  To sum up 
we have the following.
	\begin{multline}   \label{E:rstrction.of.g.to.tau.props}
		   g \text{ is one-to-one on }  \text{Int} \, \rho, \; g(\text{Int} \, \rho) = \text{Int} \, \rho, \\
		    \text{ and the restriction, } g \vert _{ \text{Int} \, \rho},  
		              \text{ has a locally Lipschitz inverse on } \text{Int} \, \rho. 
	\end{multline}

\emph{Next, assume $\rho \in P$ is not a face of $\sigma$ but neither is $\sigma$ a face 
of $\rho$.} If $\rho \subset \text{Lk} \, \sigma$ then by \eqref{E:when.g(y)=y}, we have $g(y) = y$ 
for $y \in \text{Int} \, \rho$. Suppose $\rho \subset  \overline{\text{St}} \, \sigma$ 
but $\rho \nsubseteq \text{Lk} \, \sigma$ and let $y \in \text{Int} \, \rho$. Then, by \eqref{E:Int.rho.cuts.sigma.then.rho.in.sigma}, $y \notin \text{Lk} \, \sigma$, so, by \eqref{E:when.sigma.x.in.Bd.sigma}, $\sigma(y) \in \text{Bd} \, \sigma$.  Therefore, by \eqref{E:when.g(y)=y} again, we have $g(y) = y$ for $y \in \text{Int} \, \rho$.
Suppose $\rho \nsubseteq  \overline{\text{St}} \, \sigma$ and let $y \in \text{Int} \, \rho$.  Then by \eqref{E:Int.rho.cuts.sigma.then.rho.in.sigma}, $y \notin  \overline{\text{St}} \, \sigma$ so, yet again 
by \eqref{E:when.g(y)=y}, we have $g(y) = y$.  Hence, \eqref{E:rstrction.of.g.to.tau.props} holds for $\rho$ s.t.\ $\rho \subset \text{Lk} \, \sigma$ or $\rho \nsubseteq  \overline{\text{St}} \, \sigma$.

It remains to prove the following. \emph{Claim:}
	\begin{equation}   \label{E:rho.not.face.of.sigma.->.g.invrs.Int.rho.=.Int.rho}
		\text{If } \rho \in P \text{ is not a face of } \sigma 
		   \text{ then } g^{-1}(\text{Int} \, \rho) = \text{Int} \, \rho.
	\end{equation}
To see this let $\rho \in P$ not be a face of $\sigma$.  By 
\eqref{E:rstrction.of.g.to.tau.props}, we have $\text{Int} \, \rho \subset g^{-1}(\text{Int} \, \rho)$.  
Let $y \in \text{Int} \, \rho$. By \eqref{E:Int.rho.cuts.sigma.then.rho.in.sigma} $y \notin \sigma$.
Suppose $x \in |P| \setminus \bigl[ \mcl{C} \cap (\text{Bd} \, \sigma) \bigr]$ and $g(x) = y$.  Then by \eqref{E:g.maps.sigma.into.itself}, $x \notin \sigma$.  Let $\tau \in P$ satisfy $x \in \text{Int} \, \tau$.  (See \eqref{E:x.in.exctly.1.simplex.intrr}.)  
Since $x \notin \sigma$, $\tau$ is not a face of $\sigma$.  Therefore, by \eqref{E:rstrction.of.g.to.tau.props}, 
$y = g(x) \in \text{Int} \, \tau$.  Thus, $(\text{Int} \, \tau) \cap (\text{Int} \, \rho) \ne \varnothing$. Hence, by (\ref{E:intersection.of.simps}'), $\tau = \rho$.  In particular, $x \in \text{Int} \, \rho$. This proves claim \eqref{E:rho.not.face.of.sigma.->.g.invrs.Int.rho.=.Int.rho}. Point (\ref{I:g.has.locally.Lip.invrs}) of the lemma follows.

We prove point (\ref{I:S'.closed.Phi'.cont.off.S'}) of the lemma, \emph{viz.}, that $\Ss'$ is closed 
and $\Phi'$ is continuous off $\Ss'$ (and locally Lipschitz, too, when appropriate). \emph{First, we prove that $\Ss'$ is closed.}  Since $\Ss$ is closed, $\Ss' \setminus \sigma$ is closed in $|P| \setminus \sigma$ by \eqref{E:S'.defn} and \eqref{E:g.maps.sigma.into.itself}.  Moreover, by \eqref{E:C.is.compact.contains.A}, 
$\mcl{C}_{z} \cap (\text{Bd} \, \sigma)$ is closed in $|P|$.  So to prove $\Ss'$ is closed it suffices to show that all limit points of $\Ss' \setminus \sigma$ in $\sigma$ lie 
in $\mcl{C}_{z} \cap (\text{Bd} \, \sigma)$.  It suffices to show that if $x \in \sigma \setminus (\mcl{C}_{z} \cap (\text{Bd} \, \sigma))$ then 
$x$ has an open neighborhood in $|P|$ disjoint from $\Ss'$.  

Suppose first that $x \in \sigma \setminus \mcl{C}_{z}$.  
We show that $x$ has an open neighborhood disjoint from $\Ss'$. 
By lemma \ref{L:props.of.s}(part \ref{I:s(y,t).in.C.iff.y.in.C}) and \eqref{E:general.formula.for.g} we have 
$g(x) = s(x,0) \in \sigma \setminus C_{z}$. But $\Ss$ is closed and, by \eqref{E:C.is.compact.contains.A}, $\Ss \cap \sigma \subset \mcl{C}_{z}$ and $\mcl{C}_{z}$ is also closed.   
I.e., $g(x)$ has a neighborhood $V$ in $|P|$ disjoint from $\Ss \cup \mcl{C}_{z}$.  
By \eqref{E:g.is.cont.off.C.meet.Bd.sigma} and \eqref{E:S'.defn}, $g^{-1}(V)$ is a relatively open subset of the open set $|P| \setminus \bigl[ \mcl{C} \cap (\text{Bd} \, \sigma) \bigr]$ and $g^{-1}(V)$ is disjoint from $\Ss'$.  Thus, $g^{-1}(V)$ is an open neighborhood of $x$ disjoint from $\Ss'$.  

Next, suppose $x \in \mcl{C} \setminus (\text{Bd} \, \sigma)$.  By \eqref{E:z.notin.A} the point $z$ has a neighborhood $U$ disjoint from $\Ss \cup (\text{Bd} \, \sigma)$.    By lemma \ref{L:g.is.loc.Lip}, the set 
$g^{-1}(U)$ is a relatively open subset of the open set 
$\subset |P| \setminus  \bigl[ \mcl{C} \cap (\text{Bd} \, \sigma) \bigr]$.  Hence, $g^{-1}(U)$ is open and, by \eqref{E:S'.defn} and \eqref{E:g.maps.sigma.into.itself}, disjoint from $\Ss' \setminus \sigma = g^{-1}\bigl(\Ss \setminus \sigma \bigr)$. Moreover, by point 
(\ref {I:S'.cap.sigma.=.C.cap.Bd}) of the lemma, $\Ss' \cap \sigma$ is disjoint from the domain of $g$. \emph{A fortiori,} $g^{-1}(U) \cap \Ss' \cap \sigma = \varnothing$.  Thus, $g^{-1}(U) \cap \Ss' = \varnothing$. But, by \eqref{E:general.formula.for.g} and lemma \ref{L:props.of.s}(\ref{I:s.y0.=z.on.C.less.Bd}),
$g(x) = s(x,0) = z \in U$. I.e., $x \in g^{-1}(U) \subset |P| \setminus \Ss'$. This completes the proof that 
$\Ss'$ is closed.

Next, we continue the proof of point (\ref{I:S'.closed.Phi'.cont.off.S'}) by \emph{showing that $\Phi'$ is defined and continuous off $\Ss'$.} Let $x_{0} \in |P| \setminus \Ss'$.  First, suppose $x_{0} \in |P| \setminus \sigma$.  By \eqref{E:g.maps.sigma.into.itself}, we have $g(x_{0}) \in |P| \setminus \sigma$.  
Then, by \eqref{E:Phi'.defn}, $\Phi'(x_{0})$ is defined except, perhaps, if $g(x_{0}) \in \Ss \setminus \sigma$, i.e., by \eqref{E:S'.defn}, except, perhaps, if $x_{0} \in \Ss'$.  But $x_{0} \in |P| \setminus \Ss'$.  Therefore, 
	\begin{equation}  \label{E:x0.notin.S'.or.sigma.gx0.notin.S}
		\text{if } x_{0} \in |P| \setminus (\Ss' \cup \sigma), 
			\text{ then } g(x_{0}) \notin \Ss \setminus \sigma 
				\text{ (so } \Phi'(x_{0}) \text{ is defined).}
	\end{equation}  

By \eqref{E:x0.notin.S'.or.sigma.gx0.notin.S} and \eqref{E:g.maps.sigma.into.itself}, if $x_{0} \in |P| \setminus (\Ss' \cup \sigma)$ then $g(x_{0}) \notin \Ss$, so $\Phi$ is defined and continuous at $g(x_{0})$, and, by \eqref{E:g.is.cont.off.C.meet.Bd.sigma}, $g$ is continuous at $x_{0}$.  But, by \eqref{E:Phi'.defn}, $\Phi'(x_{0}) = \Phi \circ g(x_{0})$.  Therefore, $\Phi'$ is defined and continuous at $x_{0}$.

Next, let 
$x_{0} \in \sigma \setminus \Ss' = \sigma \setminus \bigl( \mcl{C} \cap (\text{Bd} \, \sigma) \bigr)$ (see \eqref{E:S'.cap.sigma.=.C.cap.Bd.sigma}).   By \eqref{E:g.is.cont.off.C.meet.Bd.sigma} $g$ is continuous at $x_{0}$.  Therefore, by  \eqref{E:Phi'.defn} it suffices to show that $g(x_{0}) \notin \Ss$.  Since $x_{0} \in \sigma$ we have by \eqref{E:general.formula.for.g} $g(x_{0}) = s(x_{0}, 0)$. Suppose $x_{0} \in \sigma \setminus \mcl{C}$.  By lemma \ref{L:props.of.s}(part \ref{I:s(y,t).in.C.iff.y.in.C}) we have $g(x_{0}) = s(x_{0},0) \in \sigma \setminus \mcl{C} \subset \sigma \setminus \Ss$,
by \eqref{E:C.is.compact.contains.A}. Next, suppose $x_{0} \in (\text{Int} \, \sigma) \cap \mcl{C}$.  
By lemma \ref{L:props.of.s}(part \ref{I:s.y0.=z.on.C.less.Bd}) and \eqref{E:z.notin.A} we have 
$g(x_{0}) = s(x_{0},0) = z \notin \Ss$.  To sum up,
	\begin{equation}   \label{E:g.maps.sigma.less.S'.into.sigma.less.S}
		g(\sigma \setminus \Ss') \subset \sigma \setminus \Ss.
	\end{equation}
Thus, $\Phi$ is also continuous on $\sigma \setminus \Ss'$.

\eqref{E:x0.notin.S'.or.sigma.gx0.notin.S} and \eqref{E:g.maps.sigma.less.S'.into.sigma.less.S} tell us that $g \bigl[ |P| \setminus \Ss' \bigr] \subset |P| \setminus \Ss$. By lemma \ref{L:g.is.loc.Lip} and \eqref{E:comp.of.Lips.is.Lip}, if $\F$ (recall $\Phi : |P| \setminus \Ss \to \F$) is a metric space and $\Phi$ is locally Lipschitz on $|P| \setminus \Ss$, then $\Phi' = \Phi \circ g$ is locally Lipschitz 
on $|P| \setminus \Ss' \subset |P| \setminus \bigl[ \mcl{C}_{z} \cap (\text{Bd} \, \sigma) \bigr]$. This completes the proof of point (\ref{I:S'.closed.Phi'.cont.off.S'}).

Suppose $\rho \in P$ and $\text{Int} \, \rho \subset \Ss$. Hence, by \eqref{E:sigma.intrscts.S},  we have $\rho \ne \sigma$.  Since $\Ss$ is closed, we have, in fact, that  $\rho \subset \Ss$. 
Thus, if $\rho \subset \sigma$, then $\rho \subset \text{Bd} \, \sigma $.  By \eqref{E:S'.cap.sigma.=.C.cap.Bd.sigma}, 
we have $\Ss' \cap (\text{Bd} \, \sigma) = \mcl{C} \cap (\text{Bd} \, \sigma)$.  Hence, by \eqref{E:C.is.compact.contains.A}, we have
$\rho \subset \Ss \cap (\text{Bd} \, \sigma)  \subset \mcl{C} \cap (\text{Bd} \, \sigma) \subset \Ss'$.  
Next, assume $\rho \nsubseteq \sigma$, so by \eqref{E:Int.rho.cuts.sigma.then.rho.in.sigma}, 
$(\text{Int} \, \rho) \cap \sigma = \varnothing$.  Then by point (\ref{I:g.has.locally.Lip.invrs}) and \eqref{E:S'.defn} we have $\text{Int} \, \rho \subset \Ss'$.  This proves point (\ref{I:simps.in.S.are.also.in.S'}) of the lemma.

Let $\tau \in P$ and suppose $\tau \cap \sigma = \varnothing$.  Then $\tau \subset (\text{Lk} \, \sigma) \cup \bigl( |P| \setminus (\overline{\text{St}} \, \sigma) \bigr)$.  Hence, by \eqref{E:when.g(y)=y}, if $y \in \tau$ then $g(y) = y$.  Moreover, certainly $\tau$ is not a face of $\sigma$.  Nor is any face of $\tau$ a face of $\sigma$.  Hence, by \eqref{E:rho.not.face.of.sigma.->.g.invrs.Int.rho.=.Int.rho}, we have $g^{-1}(\tau) = \tau$.  
Therefore, if $y \in \tau$, then $y \in \Ss'$ if and only if $y \in \Ss$ and $\Phi' = \Phi$ on $\tau \setminus \Ss$.  Point (\ref{I:Phi'.and.Phi.agree.on.Lk.sigma}) of the lemma follows. 

Now we prove point \ref{I:Phi'.Phi.can.agree.off.sigma} of the lemma.  However, \emph{that point does not seem to be used anywhere!} Similarly, suppose $\overline{\text{St}} \, \sigma = \sigma$.  
Let $y \in |P| \setminus \sigma = |P| \setminus (\overline{\text{St}} \, \sigma)$.  
Then by \eqref{E:g.defn.when.St.=.sigma}
and \eqref{E:g.maps.sigma.into.itself}, $g \bigl( \{y\} \bigr) = \{y\} = g^{-1}(y)$.  Hence, by \eqref{E:S'.defn},
$\Ss' \setminus \sigma = \Ss \setminus \sigma$ and, by \eqref{E:Phi'.defn}, $\Phi'(y) = \Phi(y)$.
It then follows from \eqref{E:S'.cap.sigma.=.C.cap.Bd.sigma} and \eqref{E:C.is.compact.contains.A} that
		\[
			|P| \setminus \bigl[ (\text{Int} \, \sigma) \cup \Ss' \bigr] 
		              \subset |P| \setminus \bigl[ (\text{Int} \, \sigma) \cup \Ss \bigr].
              \]
Suppose 
$y \in  (\text{Bd} \, \sigma) \setminus \bigl[ \mcl{C}_{z} \cap (\text{Bd} \, \sigma) \bigr] 
= (\text{Bd} \, \sigma) \setminus \Ss'$.  (See \eqref{E:S'.cap.sigma.=.C.cap.Bd.sigma}.) Then, by \eqref{E:sigma.ident.on.sigma}, we have $\sigma(y) = y \in \text{Bd} \, \sigma$ so, by \eqref{E:when.g(y)=y}, again $g(y) = y$. Thus, $\Phi'(y) = \Phi(y)$ and point (\ref{I:Phi'.Phi.can.agree.off.sigma}) of the lemma is proved.

As for point (\ref{I:dont.change.sings.on.othr.n.simps}), 
let $\tau \neq \sigma$ be a simplex in $P$ of dimension no greater than $n$. \emph{Claim:} 
	  \begin{equation}   \label{E:ginvrs.tau-sigma}
		\Ss' \cap (\tau \setminus \sigma) 
		         = g^{-1}(\Ss \setminus \sigma) \cap (\tau \setminus \sigma) 
		             = g^{-1} \bigl[ \Ss \cap (\tau \setminus \sigma) \bigr].
	  \end{equation}
The first equality is immediate from \eqref{E:S'.defn}.  As for the second equality, first suppose $y \in g^{-1}(\Ss \setminus \sigma) \cap (\tau \setminus \sigma)$.  Then, there exists $x \in \Ss \setminus \sigma$ s.t.\ $x = g(y)$.  Since $y \in \tau \setminus \sigma$, we have by \eqref{E:off.sig.g.maps.simps.into.thmslvs} that $x \in \tau \setminus \sigma$.  I.e., 
$g^{-1}(\Ss \setminus \sigma) \cap (\tau \setminus \sigma) 
\subset g^{-1} \bigl[ \Ss \cap (\tau \setminus \sigma) \bigr]$.  Conversely, suppose $x \in \Ss \cap (\tau \setminus \sigma)$ and let $y \in |P| \setminus \bigl[ \mcl{C} \cap (\text{Bd} \, \sigma) \bigr]$ satisfy $g(y) = x$. Since $x \notin \sigma$, by \eqref{E:g.maps.sigma.into.itself}, we have $y \in |P| \setminus \sigma$.  By \eqref{E:x.in.exctly.1.simplex.intrr} there exists $\xi \in P$ s.t.\ $y \in \text{Int} \, \xi$.  Then $\xi$ is not a face of $\sigma$.  Since $y \in \text{Int} \, \xi$, by point (\ref{I:g.has.locally.Lip.invrs}), we have $x \in \text{Int} \, \xi$.  But $x \in \Ss \cap (\tau \setminus \sigma)$.  Hence, by \eqref{E:Int.rho.cuts.sigma.then.rho.in.sigma}, $\xi$ is a face of $\tau$.  I.e., $y \in \tau \setminus \sigma$.  I.e., $g^{-1} \bigl[ \Ss \cap (\tau \setminus \sigma) \bigr] \subset g^{-1}(\Ss \setminus \sigma) \cap (\tau \setminus \sigma)$ and the claim \eqref{E:ginvrs.tau-sigma} is proved.

If $\tau \subset  \text{Lk} \, \sigma$, then 
$\Ss' \cap (\tau \setminus \sigma) = \Ss \cap (\tau \setminus \sigma)$ and $\Phi' = \Phi$ 
on $\tau \setminus (\sigma \cup \Ss')$, by point (\ref{I:Phi'.and.Phi.agree.on.Lk.sigma}).  
If $\tau \subset \overline{\text{St}} \, \sigma$ but $\tau \nsubseteq \text{Lk} \, \sigma$, 
then $\tau \cap \sigma \ne \varnothing$.  Since $\dim \tau \leq n = \dim \sigma$ and $\tau \ne \sigma$, the simplex $\sigma$ cannot be a face of $\tau$. 
If $y \in \tau \setminus (\sigma \cup (\text{Lk} \, \sigma))$ then by \eqref{E:when.sigma.x.in.Bd.sigma}, 
$\sigma(y) \in \text{Bd} \, \sigma$, so by \eqref{E:when.g(y)=y}, we have $g(y) = y$.  
If $y \in \tau \cap (\text{Lk} \, \sigma)$ then, by \eqref{E:when.g(y)=y} again,
$g(y) = y$.  Therefore, by \eqref{E:ginvrs.tau-sigma}, we have 
$\Ss' \cap (\tau \setminus \sigma) = \Ss \cap (\tau \setminus \sigma)$  and $\Phi' = \Phi$ 
on $\tau \setminus (\sigma \cup \Ss')$.

Finally, suppose $\tau \nsubseteq \overline{\text{St}} \, \sigma$ (i.e., $\tau$ is not a face 
of any $\rho \in P$ with $\sigma \subset \rho$).  Write
	\begin{equation*}
		\Ss' \cap (\tau \setminus \sigma) 
		  = \biggl[ \Ss' \cap \Bigl( \bigl[ \tau \cap ( \overline{\text{St}} \, \sigma) \bigr] 
		           \setminus \sigma \Bigr) \biggr]  
			\cup \Bigl[ (\Ss' \cap \tau) \setminus ( \overline{\text{St}} \, \sigma) \Bigr].
	\end{equation*}
By \eqref{E:when.g(y)=y}, $g$ is the identity map 
on $(\Ss' \cap \tau) \setminus ( \overline{\text{St}} \, \sigma)$. Moreover, by \eqref{E:off.sig.g.maps.simps.into.thmslvs} and \eqref{E:g.maps.sigma.into.itself}, 
	\[
		g : (\overline{\text{St}} \, \sigma) \setminus \bigl[ \mcl{C} \cap (\text{Bd} \, \sigma) \bigr] 
			\to \overline{\text{St}} \, \sigma.
	\]
Hence, by \eqref{E:S'.defn}, we have 
$(\Ss' \cap \tau) \setminus ( \overline{\text{St}} \, \sigma) 
= (\Ss \cap \tau) \setminus ( \overline{\text{St}} \, \sigma)$ and $\Phi' = \Phi$ 
on $\tau \setminus \bigl[ \Ss \cup ( \overline{\text{St}} \, \sigma) \bigr] = \bigl[ \tau \setminus ( \overline{\text{St}} \, \sigma) \bigr] \setminus \bigl[ (\Ss \cap \tau) \setminus ( \overline{\text{St}} \, \sigma) \bigr]$.

Furthermore, by \eqref{E:intersection.of.simps}, $\tau \cap ( \overline{\text{St}} \, \sigma)$ is the union of simplices 
in $\overline{\text{St}} \, \sigma$ that are also faces of $\tau$.  
Since $\tau \nsubseteq \overline{\text{St}} \, \sigma$, none of these simplices equal $\sigma$ (otherwise $\tau$ would have $\sigma$ as a face and so would lie in $\overline{\text{St}} \, \sigma$) and the dimension of each of these simplices is less than $n$.  We have already proved that point (\ref{I:dont.change.sings.on.othr.n.simps}) of the lemma applies to such simplices.  Summing up,
	\begin{equation*}
		\Ss' \cap (\tau \setminus \sigma) 
		  = \biggl[ \Ss \cap \Bigl( \bigl[ \tau \cap ( \overline{\text{St}} \, \sigma) \bigr] 
		           \setminus \sigma \Bigr) \biggr]  
			\cup \Bigl[ (\Ss \cap \tau) \setminus ( \overline{\text{St}} \, \sigma) \Bigr]
			= \Ss \cap (\tau \setminus \sigma)
	\end{equation*} 
and $\Phi' = \Phi$ on $\tau \setminus (\sigma \cup \Ss)$. This proves point (\ref{I:dont.change.sings.on.othr.n.simps}) of the lemma.

To prove point (\ref{I:S'.bggr.S.on.n-1.simps}) first write
	\begin{equation}   \label{E:breakdown.Int.tau.cap.S'}
		(\text{Int} \, \tau) \cap \Ss'  = \bigl[ (\text{Int} \, \tau) \cap \Ss' \cap (\text{Int} \, \sigma) \bigr] 
		  \cup \bigl[ (\text{Int} \, \tau) \cap \Ss' \cap (\text{Bd} \, \sigma) \bigr] 
		    \cup \bigl[ (\text{Int} \, \tau) \cap \Ss' \setminus \sigma \bigr].
	  \end{equation}
Since $\dim \tau < n := \dim \sigma$ we have $\tau \ne \sigma$ so, by (\ref{E:intersection.of.simps}'), 
	\begin{equation}     \label{E:Int.tau.cap.Int.sigma.empty}
		(\text{Int} \, \tau) \cap \Ss' \cap (\text{Int} \, \sigma) = \varnothing
		  = (\text{Int} \, \tau) \cap \Ss \cap (\text{Int} \, \sigma).
	\end{equation}
In particular,
	\begin{equation}  \label{E:Int.tau.cap.sig.=.cap.Bd.sig}
		(\text{Int} \, \tau) \cap \Ss \cap \sigma = (\text{Int} \, \tau) \cap \Ss \cap (\text{Bd} \, \sigma).
	\end{equation}
By \eqref{E:S'.cap.sigma.=.C.cap.Bd.sigma} 
and \eqref{E:C.is.compact.contains.A} and \eqref{E:hbar.maps.A.to.C.cap.Bd.sigma},
	\begin{align}   \label{E:Int.tau.cap.S'.in.C.cap.Bd.sig}
		(\text{Int} \, \tau) \cap \Ss' \cap (\text{Bd} \, \sigma) &= (\text{Int} \, \tau) \cap \bigl[ \mcl{C}_{z} 
		  \cap   (\text{Bd} \, \sigma) \bigr]  \notag \\
		    &= (\text{Int} \, \tau) 
		          \cap \Bigl( \bigl[ \Ss \cap (\text{Bd} \, \sigma) \bigr] 
		             \cup \bigl[ \mcl{C}_{z} \cap (\text{Bd} \, \sigma) \bigr]  \Bigr) \\
		    &= (\text{Int} \, \tau) 
		          \cap \Bigl( \bigl[ \Ss \cap (\text{Bd} \, \sigma) \bigr] 
		             \cup \bar{h}_{z, \sigma}(\Ss \cap \sigma)  \Bigr)  \notag.
	\end{align}
Finally, by point (\ref{I:dont.change.sings.on.othr.n.simps}) of the lemma, we have
	\begin{equation}  \label{E:Int.tau.cap.S'.=.Int.tau.cap.S.off.sig}
		(\text{Int} \, \tau) \cap (\Ss' \setminus \sigma) 
		         = (\text{Int} \, \tau) \cap (\Ss \setminus \sigma).
	\end{equation}
Point (\ref{I:S'.bggr.S.on.n-1.simps}) of the lemma follows from applying \eqref{E:Int.tau.cap.Int.sigma.empty}, \eqref{E:Int.tau.cap.S'.in.C.cap.Bd.sig}, \eqref{E:Int.tau.cap.S'.=.Int.tau.cap.S.off.sig}, and \eqref{E:Int.tau.cap.sig.=.cap.Bd.sig} to \eqref{E:breakdown.Int.tau.cap.S'}.

To prove point (\ref{I:dont.increase.sing.dim.in.big.simps}) first note that 
if $\rho \in P$ but $\rho \nsubseteq \overline{\text{St}} \, \sigma$ (i.e., $\rho$ is not a face of any simplex in $P$ having $\sigma$ as a face; in particular $\rho$ is not a face of $\sigma$; $\rho$ cannot be a proper face 
of $\sigma$ anyway because $\dim \rho \geq n$) then 
$\Ss' \cap (\text{Int} \, \rho) = \Ss \cap (\text{Int} \, \rho)$ 
by \eqref{E:rho.not.face.of.sigma.->.g.invrs.Int.rho.=.Int.rho}
and \eqref{E:general.formula.for.g}.
Suppose $\rho \in P$ and $\rho \subset \overline{\text{St}} \, \sigma$. 
If $\rho \subset \text{Lk} \, \sigma$ then by point (\ref{I:Phi'.and.Phi.agree.on.Lk.sigma}) of the lemma, we have 
$\Ss' \cap (\text{Int} \, \tau) = \Ss \cap (\text{Int} \, \tau)$.  
So suppose $\rho \subset \overline{\text{St}} \, \sigma$ but $\rho \nsubseteq \text{Lk} \, \sigma$.  Thus, $\rho \cap \sigma \ne \varnothing$. If $\sigma(\text{Int} \, \rho) \subset \text{Bd} \, \sigma$ then by \eqref{E:rho.not.face.of.sigma.->.g.invrs.Int.rho.=.Int.rho} again and  \eqref{E:when.g(y)=y}, we have $\Ss' \cap (\text{Int} \, \rho) = \Ss \cap (\text{Int} \, \rho)$. 
Suppose $\sigma(\text{Int} \, \rho) \cap (\text{Int} \, \sigma) \ne \varnothing$.  Then by \eqref{E:when.sigma.x.in.Bd.sigma}, $\sigma$ is a face of $\rho$.  
If $\rho = \sigma$ then $\Ss' \cap (\text{Int} \, \rho) = \varnothing$ by \eqref{E:S'.cap.sigma.=.C.cap.Bd.sigma} and point (\ref{I:dont.increase.sing.dim.in.big.simps}) of the lemma holds trivially. So suppose $\sigma$ is a proper face of $\rho$.  
Then $(\text{Int} \, \rho) \cap \sigma = \varnothing$. It follows 
from \eqref{E:rho.not.face.of.sigma.->.g.invrs.Int.rho.=.Int.rho}
	\[
			\Ss' \cap (\text{Int} \, \rho) 
			        = g^{-1}( \Ss \cap (\text{Int} \, \rho)).
	\] 
By point (\ref{I:g.has.locally.Lip.invrs}) of the lemma, $g^{-1}$ is locally Lipschitz on $\text{Int} \, \rho$. Point (\ref{I:dont.increase.sing.dim.in.big.simps}) then follows from lemma \ref{L:loc.Lip.image.of.null.set.is.null} and \eqref{E:A.empty.iff.H0.A.=.0} in appendix \ref{S:Lip.Haus.meas.dim}.  

Next, we prove point (\ref{I:Int.rho.cut.S'.then.cuts.S.or.rho.is.sigma.face}). 
Suppose $s \geq 0$, $\tau \in P$, and $\Hm^{s}(\Ss' \cap (\text{Int} \, \tau)) > 0$. 
We show that either $\Hm^{s}(\Ss \cap (\text{Int} \, \tau)) > 0$ or $\tau$ is a proper face of $\sigma$ and $\Hm^{s}(\Ss \cap (\text{Int} \, \sigma)) > 0$.    
If $\dim \tau \geq n$ the claim is immediate from point (\ref{I:dont.increase.sing.dim.in.big.simps}) of the lemma.  So suppose $\tau \in P$ with $\dim \tau < n$ and suppose $\Hm^{s}(\Ss' \cap (\text{Int} \, \tau)) > 0$.  First, suppose $\tau$ is not a face of $\sigma$.  Then, 
by \eqref{E:Int.rho.cuts.sigma.then.rho.in.sigma} in appendix \ref{S:basics.of.simp.comps}, 
we have $\text{Int} \, \tau \cap \sigma = \varnothing$.  Hence, $\text{Int} \, \tau \subset \tau \setminus \sigma$.  Thus, by point (\ref{I:dont.change.sings.on.othr.n.simps}) of the lemma,
	\begin{equation}    \label{E:S'.cap.Int.tau.=.S.cap.Int.tau} 
		\Ss' \cap (\text{Int} \, \tau) = \bigl[ \Ss' \cap (\tau \setminus \sigma) \bigr] 
		                              \cap (\text{Int} \, \tau) 
				= \bigl[ \Ss \cap (\tau \setminus \sigma) \bigr] \cap (\text{Int} \, \tau) 
					= \Ss \cap (\text{Int} \, \tau).  
	\end{equation}
Hence, $\Hm^{s}(\Ss \cap (\text{Int} \, \tau)) > 0$. 

Suppose $\tau$ is a face of $\sigma$. Then it is a proper face, since $\dim \tau < n$.   I.e., $\tau \subset \text{Bd} \, \sigma$. 
By point (\ref{I:S'.cap.sigma.=.C.cap.Bd}) of lemma \ref{L:big.lemma}, \eqref{E:hbar.maps.A.to.C.cap.Bd.sigma} and \eqref{E:hbar.is.idendt.on.Bd.sigma},  we have
	\begin{align*}
		\Hm^{s} \bigl( \Ss' \cap (\text{Int} \, \tau) \bigr) 
		  &\leq \Hm^{s} \bigl[ \Ss \cap (\text{Int} \, \tau) \bigr]  
			+ \Hm^{s} \Bigl( \bar{h} \bigl[ \Ss \cap (\text{Int} \, \sigma) \bigr] 
			      \cap (\text{Int} \, \tau)  \Bigr) \\
	        &\leq \Hm^{s} \bigl[ \Ss \cap (\text{Int} \, \tau) \bigr]  
			+ \Hm^{s} \Bigl( \bar{h} \bigl[ \Ss \cap (\text{Int} \, \sigma) \bigr]  \Bigr). 
	\end{align*}
Thus, if $\Hm^{s} \bigl[ \Ss \cap (\text{Int} \, \tau) \bigr]  = 0$, we must have 
$\Hm^{s} \Bigl( \bar{h} \bigl[ \Ss \cap (\text{Int} \, \sigma) \bigr]  \Bigr) > 0$.  Therefore, by lemma \ref{L:props.of.s} part (\ref{I:hbar.locally.Lip.b.Lip}), and lemma \ref{L:loc.Lip.image.of.null.set.is.null} we have 
$\Hm^{s} \bigl[ \Ss \cap (\text{Int} \, \sigma) \bigr]   > 0$.
This completes the proof of point (\ref{I:Int.rho.cut.S'.then.cuts.S.or.rho.is.sigma.face}).

Next, point (\ref{I:dim.S'.leq.dim.S}).  Write $\Ss' = (\Ss' \cap \sigma) \cup (\Ss' \setminus \sigma)$ and $\Ss = (\Ss \cap \sigma) \cup (\Ss \setminus \sigma)$.  By \eqref{E:dim.S'.cap.sigma.le.dim.S.cap.sigma}, $\dim (\Ss' \cap \sigma) \leq \dim (\Ss \cap \sigma)$.  So, by \eqref{E:dim.of.whole.=.max.dim.of.parts} in appendix \ref{S:Lip.Haus.meas.dim}, it suffices to show $\dim (\Ss' \setminus \sigma) \leq \dim (\Ss \setminus \sigma)$.  Note that, by \eqref{E:Int.rho.cuts.sigma.then.rho.in.sigma} and \eqref{E:x.in.exctly.1.simplex.intrr}, 
	\[
		\Ss' \setminus \sigma 
			= \bigcup_{\tau \in P, \;\tau \nsubseteq \sigma} \Ss' \cap (\text{Int} \, \tau).
	\]
A similar formula holds for $\Ss$.  
Thus, by \eqref{E:dim.of.whole.=.max.dim.of.parts} again, it suffices to show that 
	\begin{equation}  \label{E:dim.S'.cap.Int.rho.leq.dim.S.cap.Int.rho}
		\dim \bigl[ \Ss' \cap (\text{Int} \, \tau) \bigr] \leq \dim \bigl[ \Ss \cap (\text{Int} \, \tau) \bigr]
	\end{equation}
for every $\tau \in P$ s.t.\ $\tau \nsubseteq \sigma$.  Let $\tau \in P$ and suppose $
\tau \nsubseteq \sigma$.  
If $\dim \tau \geq n = \dim \sigma$ then \eqref{E:dim.S'.cap.Int.rho.leq.dim.S.cap.Int.rho} holds by point (\ref{I:dont.increase.sing.dim.in.big.simps}) of the lemma.  Suppose that $\dim \tau < n = \dim \sigma$.  Then $\tau$ is not a face of $\sigma$ so \eqref{E:S'.cap.Int.tau.=.S.cap.Int.tau} implies that 
$\dim \bigl[ \Ss' \cap (\text{Int} \, \tau) \bigr] = \dim \bigl[ \Ss \cap (\text{Int} \, \tau) \bigr]$. This shows that $\dim \Ss' \leq \dim \Ss$.  The same argument shows 
$\dim \bigl( \Ss' \cap |Q| \bigr) \leq \dim \bigl( \Ss \cap |Q| \bigr)$.

We prove point (\ref{I:Phi'.rho.subset.Phi.rho}) of the lemma.  First, note that 
by \eqref{E:x.in.exctly.1.simplex.intrr} it suffices to show
	\begin{equation}  \label{E:Phi'.Int.rho.subset.Phi.Int.rho}
		\text{For every } \rho \in P \text{ we have } 
		  \Phi' \bigl( (\text{Int} \, \rho) \setminus \Ss' \bigr) 
		    \subset \Phi \bigl( (\text{Int} \, \rho) \setminus \Ss \bigr).
	\end{equation}
Let $\rho \in P$.  First, suppose $\rho \nsubseteq \sigma$ and let 
$y \in ( \text{Int} \, \rho ) \setminus \Ss'$.  Then $y \notin \sigma$ by \eqref{E:Int.rho.cuts.sigma.then.rho.in.sigma}.  Hence, by \eqref{E:g.maps.sigma.into.itself} and \eqref{E:S'.defn} $g(y) \notin \Ss$.  Moreover, by point (\ref{I:g.has.locally.Lip.invrs}), $g(y) \in \text{Int} \, \rho$.  Hence,
	\[
		\Phi'(y) = \Phi \bigl[ g(y) \bigr] \in \Phi \bigl[ (\text{Int} \, \rho) \setminus \Ss \bigr].
	\]
Now suppose $\rho \subset \sigma$ and let $y \in (\text{Int} \, \rho) \setminus \Ss'$.  First, suppose $\rho$ is a proper face of $\sigma$. Then by point (\ref{I:S'.cap.sigma.=.C.cap.Bd}) of the lemma, $y \in (\text{Int} \, \rho) \setminus \Ss$ and, by \eqref{E:sigma.ident.on.sigma}, \eqref{E:when.g(y)=y}, and \eqref{E:Phi'.defn}, $\Phi'(y) = \Phi(y)$.  \eqref{E:Phi'.Int.rho.subset.Phi.Int.rho} follows in this case. 
Suppose $\rho = \sigma$.  Then, by point (\ref{I:S'.cap.sigma.=.C.cap.Bd}) of the lemma, we have $(\text{Int} \, \rho) \cap \Ss' = \varnothing$. If $y \in \mcl{C}$, then by \eqref{E:general.formula.for.g} and lemma \ref{L:props.of.s} (part \ref{I:s.y0.=z.on.C.less.Bd}), $g(y) = z \notin \mcl{C}$.  \emph{A fortiori}, by \eqref{E:C.is.compact.contains.A}, $g(y) \in (\text{Int} \, \rho) \setminus \Ss$. Suppose $y \in \sigma \setminus \mcl{C}$.  Then by \eqref{E:general.formula.for.g} and lemma \ref{L:props.of.s} (part \ref{I:s(y,t).in.C.iff.y.in.C}), $g(y) \in \sigma \setminus \mcl{C}$ and so, again, by \eqref{E:C.is.compact.contains.A}, $g(y)  \in \sigma \setminus \Ss$.  This concludes the proof of point (\ref{I:Phi'.rho.subset.Phi.rho}) of the lemma.

This concludes the proof of lemma \ref{L:big.lemma}.
\end{proof}

\begin{proof}[Proof of lemma \ref{L:props.of.s}]
\emph{Proof of part (\ref{I:s.y0.=z.on.C.less.Bd})}  Suppose 
$y \in \mcl{C} \setminus (\text{Bd} \, \sigma)$.  By \eqref{E:y.in.C.iff.hbar.y.in.C}, $\bar{h}(y) \in \mcl{C}$ so $\bar{\Delta}(y) = 0$ and therefore, $f(y,0)$ defined by \eqref{E:f.defn} is 1.   

\emph{Proof of part (\ref{I:z.is.fixed.pt.of.s})}  Easy consequence of \eqref{E:when.bx.=1.or.0} and \eqref{E:Delta.bar.z.=diam.sig}.

\emph{Proof of part (\ref{I:s.in.Bd.iff.y.in.Bd})}  Suppose $s(y,t) \in \text{Bd} \, \sigma$.  Then by definition of $s(y,t)$, either $y \in \text{Bd} \, \sigma$ or $b(y) < 1$ and $f(y,t) = 0$.  But by definitions of $f$ and $k$, $b(y) < 1$ and $f(y,t) = 0$ are mutually contradictory.  Thus, $y \in \text{Bd} \, \sigma$.

\emph{Proof of part (\ref{I:hbar.s.=.hbar.y})}  Use property (\ref{I:z.is.fixed.pt.of.s}) .

\emph{Proof of part (\ref{I:s(y,t).in.C.iff.y.in.C})} Suppose $s(y,t) \in \mcl{C}$.  If $s(y,t) \in \text{Bd} \, \sigma$ then, by property (\ref{I:s.in.Bd.iff.y.in.Bd}), $y = s(y,t) \in \mcl{C}$.  What if $s(y,t) \notin \text{Bd} \, \sigma$?  Then $y \notin \text{Bd} \, \sigma$, by property (\ref{I:s.in.Bd.iff.y.in.Bd}) again, so $b(y) \in [0,1)$. Suppose $y \notin \mcl{C}$. By \eqref{E:C.is.compact.contains.A}, $y \ne z$ so by \eqref{E:y.in.C.iff.hbar.y.in.C} $\bar{h}(y) \notin \mcl{C}$.  Therefore, since $\mcl{C}$ is compact (\eqref{E:C.is.compact.contains.A} again), $\bar{\Delta}(y) > 0$.  Hence, 
      \[
	      k \bigl[ b(y), \bar{\Delta}(y), t \bigr] <1. 
      \]
I.e., $s(y,t) \ne z$.  Hence, by property (\ref{I:hbar.s.=.hbar.y}),  $\bar{h} \bigl[ s(y,t) \bigr] = \bar{h}(y) \notin \mcl{C}$.  But by \eqref{E:y.in.C.iff.hbar.y.in.C}, this means $s(y,t) \notin \mcl{C}$, contradiction.  It follows that $y \in \mcl{C}$.

Next, suppose $y \in \mcl{C}$.  If $y = z$, then by property (\ref{I:z.is.fixed.pt.of.s}), we have $s(y,t) = z \in \mcl{C}$.  Assume $s(y,t) \ne z$.  Then by property (\ref{I:hbar.s.=.hbar.y}),  $y \ne z$ and $\bar{h} \bigl[ s(y,t) \bigr] = \bar{h}(y) \in \mcl{C}$ so by \eqref{E:y.in.C.iff.hbar.y.in.C} again $s(y,t) \in \mcl{C}$.

\emph{Proof of part (\ref{I:y-z.h2-h1.bound})} \eqref{E:|yi-z||h2-h1|.Lip.condtn} is only interesting when
	\begin{equation}  \label{E:trivial.condns.for.Lip}
		\bar{h}(y_{2}) \ne \bar{h}(y_{1}) \text{ and either } y_{1} \ne z \text{ or } y_{2} \ne z 
		         \text{ or both}.
	\end{equation}
Assume \eqref{E:trivial.condns.for.Lip}.  Suppose $(y_{1} - z) \cdot (y_{2} - z) \leq 0$.  
	\[
		| y_{2} - y_{1} | = \sqrt{|y_{2} - z|^{2} - 2 (y_{1} - z) \cdot (y_{2} - z) + |y_{1} - z|^{2}}
			\geq |y_{i} - z |, \quad i = 1,2.
	\]
Thus, 
	\begin{equation} \label{E:|y2-z||h2-h1|.obtuse} 
		|y_{i} - z| | \bar{h}(y_{2}) - \bar{h}(y_{1}) |  \leq diam(\sigma) | y_{2} - y_{1} |,
			\text{ if } (y_{1} - z) \cdot (y_{2} - z) \leq 0.
	\end{equation}
I.e., \eqref{E:|yi-z||h2-h1|.Lip.condtn} holds if $(y_{1} - z) \cdot (y_{2} - z) \leq 0$.

So suppose 
	\begin{equation}  \label{E:acute.y1.z.y2.angle}
		(y_{1} - z) \cdot (y_{2} - z) > 0.
	\end{equation}
In particular, $y_{1} \ne z$ and $y_{2} \ne z$. Write $h_{i} := \bar{h}(y_{i})$ and note that, 
since $y_{i} \neq z$, we have $b(y_{i}) > 0$, by \eqref{E:when.bx.=1.or.0} ($i = 1,2$).  Therefore, by \eqref{E:b.hbar.defn} and \eqref{E:acute.y1.z.y2.angle} 
	\begin{equation}  \label{E:h1-z.dot.h2-z.poz}
		(h_{1} - z) \cdot (h_{2} - z) > 0
	\end{equation}
and all the points in \eqref{E:|yi-z||h2-h1|.Lip.condtn} lie on the triangle $h_{1}zh_{2}$.    

WLOG 
	\begin{equation}  \label{E:h1-z.lngr.thn.h2-z}
		|h_{1} - z| \geq |h_{2} - z|.
	\end{equation}
\emph{Claim:}  Under \eqref{E:acute.y1.z.y2.angle} the angle, $\omega$, between $z - h_{1}$ and $h_{2} - h_{1}$ is bounded away from 0 by a value depending only on $\sigma$ and 
$z \in \text{Int} \, \sigma$.  It suffices to treat the case $0 \leq \omega < \pi/2$.  Let $w$ lie on the line (unique because of \eqref{E:trivial.condns.for.Lip}) passing through $h_{1}$ and $h_{2}$.  Thus, for some $s \in \RR$, we have $w-h_{1} = s(h_{2} - h_{1})$.  Since $0 \leq \omega < \pi/2$, we have 
$(h_{2} - h_{1}) \cdot (h_{1} -z) \neq 0$.  Therefore, we may pick $s \in \RR$ s.t.\ 
$w - z \perp h_{1} - z$ . Hence, 
	\begin{multline*}
		0 - |h_{1} - z|^{2} = (h_{1} - z) \cdot \bigl[ (w - z) - (h_{1} - z) \bigr]  \\
			= (h_{1} - z) \cdot s\bigl[ (h_{2} - z) - (h_{1} - z) \bigr] 
			= s (h_{1} - z) \cdot (h_{2} - z) - s |h_{1} - z|^{2}.
	\end{multline*}
Thus,
	\[
		(s - 1)|h_{1} - z|^{2} = s (h_{1} - z) \cdot (h_{2} - z).
	\]
But $(h_{1} - z) \cdot (h_{2} - z) > 0$, by \eqref{E:h1-z.dot.h2-z.poz}.   Therefore, $s < 0$ or $s > 1$.  
If $s < 0$ then by Schwarz's inequality 
(Stoll and Wong \cite[Theorem 3.1, p.\ 79]{rrSetW68.LinearAlgebra}) and \eqref{E:h1-z.lngr.thn.h2-z}
	\[
		(s - 1)|h_{1} - z|^{2} = s (h_{1} - z) \cdot (h_{2} - z) \geq s |h_{1} - z || h_{2} - z | 
			\geq s | h_{1} - z |^{2}.
	\]
I.e., $- |h_{1} - z|^{2} \geq 0$.  This is impossible since $h_{1} \in \text{Bd} \, \sigma$ while $z \in \text{Int} \, \sigma$.  Therefore, $s > 1$.  Now,
	\begin{equation}  \label{E:h2.is.convx.combo.w.h1}
		s^{-1} w + (1 - s^{-1}) h_{1} = h_{2}.
	\end{equation}

Let $\Pi \subset \RR^{N}$ be the smallest (affine) plane containing $\sigma$.  Thus, $\dim \Pi = n$ and $h_{1}, h_{2} \in \Pi$.  Since $w = (1-s) h_{1} + s h_{2}$, we also have $w \in \Pi$.  
Since $z \in \text{Int} \, \sigma$, $r = r(z) := dist(z, \text{Bd} \, \sigma) > 0$. The open ball $B$, of radius $r$ centered at $z$ satisfies $B \cap \Pi \subset \text{Int} \, \sigma$.  
Suppose $\tan \omega < \epsilon = \epsilon(z) := r(z)/diam(\sigma) > 0$.  (So $\epsilon(z)$ is continuous in $z$.) Now, $|h_{1} - z| \leq diam(\sigma)$ and 
	\[
		\frac{r}{diam(\sigma)} > \tan \omega = \frac{|w - z|}{|h_{1} - z|} 
		   \geq  \frac{|w - z|}{diam(\sigma)}.
	\]
Thus, we must have $w \in B \cap \Pi \subset  \text{Int} \, \sigma$.  (We also observe that 
$0 < \tan \omega < + \infty$ so $0 < \omega < \pi/2$.) Hence, from \eqref{E:h2.is.convx.combo.w.h1}, we see that $h_{2} \in \text{Bd} \, \sigma$ lies on the open line segment joining $w \in  \text{Int} \, \sigma$ and $h_{1} \in \sigma$.  It easily follows that $h_{2} \in \text{Int} \, \sigma$, a contradiction.  Therefore, $\omega \geq \omega_{0} = \omega_{0}(z) := \arctan \epsilon(z)$.  This proves the claim that $\omega$ is bounded away from 0.
Note $\sin \omega \geq \sin \omega_{0} = \epsilon/\sqrt{1 + \epsilon^{2}}$.

Let $\alpha$ be the angle between $h_{1} - z$ and $h_{2} - z$.  By \eqref{E:h1-z.dot.h2-z.poz}, 
$0 \leq \alpha < \pi/2$.  
Let $K''_{1} = K''_{1}(z) := \frac{diam(\sigma) \sqrt{1 + \epsilon(z)^{2}}}{\epsilon(z)}$.  So $K''_{1}(z)$ is continuous in $z$. By the Law of Sines,
	\begin{equation}  \label{E:appl.of.Law.of.Sines}
		\frac{|h_{2} - h_{1}|}{\sin \alpha} = \frac{|h_{2} - z|}{\sin \omega} \leq K''_{1}.
	\end{equation}
Now, by \eqref{E:acute.y1.z.y2.angle}, $y_{1} \neq z$ and $y_{2} \neq z$. Thus, $\alpha$ is the angle between $y_{1} - z$ and $y_{2} - z$.  Therefore, by \eqref{E:appl.of.Law.of.Sines},
	\begin{align}  \label{E:h2-h1.radical.bound}
		|h_{2} - h_{1}| &\leq K''_{1} \sqrt{1 - \cos^{2} \alpha}  \notag \\
		       &= K''_{1} \sqrt{(1 + \cos \alpha) (1 - \cos \alpha)}  \notag \\
		  &\leq \sqrt{2} K''_{1} \sqrt{1 - \cos \alpha} \notag \\
		  &= \sqrt{2} K''_{1} \sqrt{1 - \frac{(y_{1} - z) \cdot (y_{2} - z)}{|y_{1} - z||y_{2} - z|}  } \\
		  &= \sqrt{2} K''_{1} 
		        \sqrt{1 + \frac{|y_{1} - y_{2}|^{2} - |y_{1} - z|^{2} - |y_{2} - z|^{2}}{2|y_{1} - z||y_{2} - z|}  } 
		            \notag \\
		  &= \sqrt{2} K''_{1} 
		        \sqrt{1 + \frac{|y_{1} - y_{2}|^{2}}{2|y_{1} - z||y_{2} - z|} -
		             \frac{|y_{1} - z|}{2|y_{2} - z|} - \frac{|y_{2} - z|}{2|y_{1} - z|} }. \notag 
	\end{align}
Now, for $a, b > 0$ we have $\tfrac{a}{2b} + \tfrac{b}{2a} \geq 1$.  Applying this general inequality to \eqref{E:h2-h1.radical.bound} we get
	\begin{equation*}
		|h_{2} - h_{1}| \leq \sqrt{2} K''_{1}
		        \frac{1}{\sqrt{2|y_{1} - z||y_{2} - z|}} |y_{1} - y_{2}|.
	\end{equation*}

Multiplying the extreme members of the preceding  by $\min_{i=1,2} |y_{i} - z|$ we get
	\begin{multline}  \label{E:min.|yi-z|.leq.K''.|y1-y2|}
		\Bigl( \min_{i=1,2} |y_{i} - z| \Bigr) \bigl| \bar{h}(y_{2}) - \bar{h}(y_{1}) \bigr| 
		  = |h_{2} - h_{1}|\Bigl( \min_{i=1,2} |y_{i} - z| \Bigr)  \\
		  \leq \sqrt{2} K''_{1}
		        \frac{\min_{i=1,2} |y_{i} - z|}{\sqrt{2|y_{1} - z||y_{2} - z|}} |y_{1} - y_{2}| 
		            \leq K''_{1} |y_{1} - y_{2}|.
	\end{multline}
(By definition of $K''_{1}$, we have $K''_{1} \geq diam(\sigma)$.  Therefore, by \eqref{E:|y2-z||h2-h1|.obtuse} , the same $K''_{1}$ works if $(y_{1} - z) \cdot (y_{2} - z) \leq 0$.) Let $K'' = K''(z) = diam(\sigma) + K''_{1}(z)$.  Then $K''(z)$ depends only on $\sigma$ and $z$ and is continuous in $z$.

Suppose $|y_{1} - z| \leq |y_{2} - z|$.  Then, by \eqref{E:min.|yi-z|.leq.K''.|y1-y2|}, 
\eqref{E:|yi-z||h2-h1|.Lip.condtn} holds with $i=1$.  But
	\begin{align*}
		|y_{2} - z|  | \bar{h}(y_{2}) - \bar{h}(y_{1}) |  
		 &\leq |y_{2} - y_{1}|  | \bar{h}(y_{2}) - \bar{h}(y_{1}) | 
		      + |y_{1} - z|  | \bar{h}(y_{2}) - \bar{h}(y_{1}) |   \\
		 &\leq \bigl( diam(\sigma) + K''_{1} \bigr) |y_{1} - y_{2}|  \\
		 &= K'' |y_{2} - y_{1}|.
	\end{align*}
Similarly if $|y_{2} - z| \leq |y_{1} - z|$.

\emph{Proof of part (\ref{I:hbar.locally.Lip.b.Lip})} 
As previously observed, it follows from lemma \ref{L:bary.coords.are.Lip} that $b(y)$ is Lipschitz in $y \in \sigma$. We give another proof here. 
Let $y_{1}, y_{2} \in \sigma$. WLOG $|y_{2} - z| \leq |y_{1} - z|$.  By \eqref{E:b.hbar.defn}
	\begin{equation}  \label{E:y2.-.y1.in.terms.of.bs.hs}
		y_{2} - y_{1} = \bigl[ b(y_{2}) - b(y_{1}) \bigr]  \bigl[ \bar{h}(y_{1}) - z \bigr] 
			+ b(y_{2}) \bigl[ \bar{h}(y_{2}) - \bar{h}(y_{1}) \bigr].
	\end{equation}
Now, $\bar{h}(y_{2}) \in \text{Bd} \, \sigma$ (even if $y_{2} = z$) and $dist (z, \text{Bd} \, \sigma) > 0$. Moreover, by \eqref{E:b.hbar.defn},   
	\begin{equation}   \label{E:b2.is.ratio}
		b(y_{2}) = |y_{2}-z|/ \bigl| \bar{h}(y_{2}) - z \bigr|,
	\end{equation}
even if $y_{2} = z$.  Thus, by \eqref{E:y2.-.y1.in.terms.of.bs.hs} and part (\ref{I:y-z.h2-h1.bound}), 
	\begin{align*}
		 dist(z, \text{Bd} \, \sigma) \bigl| b(y_{2}) - b(y_{1}) \bigr|  
			 &\leq \bigl| b(y_{2}) - b(y_{1}) \bigr|  \bigl| \bar{h}(y_{1}) - z \bigr| \\
			 &\leq |y_{2} - y_{1}| + b(y_{2}) \bigl| \bar{h}(y_{2}) - \bar{h}(y_{1}) \bigr| \\
			 &= |y_{2} - y_{1}| + \frac{|y_{2}-z|}{|\bar{h}(y_{2}) - z|} \bigl| \bar{h}(y_{2}) - \bar{h}(y_{1}) \bigr|  \\
			 &\leq |y_{2} - y_{1}| 
			     + dist(z, \text{Bd} \, \sigma)^{-1} |y_{2} - z| \bigl| \bar{h}(y_{2}) - \bar{h}(y_{1}) \bigr| \\
			 &\leq |y_{2} - y_{1}| 
			     + dist(z, \text{Bd} \, \sigma)^{-1} K'' |y_{2} - y_{1}|.
	\end{align*}    
Thus, \eqref{E:b.is.Lip.on.sigma} holds.

If $y \in \sigma \setminus \{ z \}$, we have, by \eqref{E:when.bx.=1.or.0}, that $b(y) > 0$ and, by \eqref{E:b.hbar.defn},
	\[
		\bar{h}(y) = b(y)^{-1} (y - z) + z.
	\]
It follows from \eqref{E:b.is.Lip.on.sigma}  and 
\eqref{E:comp.of.Lips.is.Lip} that $\bar{h} : \sigma \setminus \{ z \} \to \text{Bd} \, \sigma$ is  locally Lipschitz on $\sigma \setminus \{ z \}$.  Similarly, since $b$ is bounded away from 0 on $\mcl{A}$ (because $\mcl{A}$ is compact and disjoint from $z$; \eqref{E:A.intrscts.not.contains.sigma.int},  \eqref{E:z.notin.A}, and \eqref{E:when.bx.=1.or.0}),  $\bar{h}$ is actually Lipschitz on $\mcl{A}$. 

\emph{Proof of part (\ref{I:f=k(b,delta,t).loc.Lip})}  
We wish to show that $f$ is locally Lipschitz on $B_{z}$.  
Now, by property (\ref{I:hbar.locally.Lip.b.Lip}), $\bar{h}(y)$ is locally Lipschitz in $y \in \sigma \setminus \{ z \}$ and . Therefore, $\bar{\Delta}(y)$ is locally Lipschitz in $y \in \sigma \setminus \{ z \}$.  Moreover, $(y,t) \in B_{z}$ implies $\bigl( b(y), \bar{\Delta}(y), t \bigr) \in T$ defined in \eqref{E:defn.of.T}.  Hence, since $b(y)$ is Lipschitz in $\sigma$ it follows from \eqref{E:comp.of.Lips.is.Lip} and lemma \ref{L:k.theory} that $f$ is locally Lipschitz on $\bigl\{ (y,t) \in B_{z} : y \ne z \bigr\}$.  Therefore, to prove part (\ref{I:f=k(b,delta,t).loc.Lip}), it suffices to show that $f$ is locally Lipschitz 
in $(\text{Int} \, \sigma) \times [0,1]$.  

Let $y_{i} \in \text{Int} \, \sigma$, $t_{i} \in [0,1]$ and write $b_{i} = b(y_{i})$, $\Delta_{i} = \bar{\Delta}(y_{i})$, and $\xi_{i} = \bigl( b_{i}, \Delta_{i}, t_{i} \bigr)$ ($i=1,2$).  Since $y_{1}, y_{2} \in \text{Int} \, \sigma$ we have, by \eqref{E:when.bx.=1.or.0}, $b_{1}, b_{2} < 1$.  In fact, we may assume that for some $\epsilon > 0$, we have $b_{1}, b_{2} < 1-\epsilon$.

Now, $\xi_{1}, \xi_{2} \in T$ defined \eqref{E:defn.of.T} and, lemma \ref{L:k.theory} tells us, $T$ is convex.  Therefore, by the multivariate Mean Value Theorem (Apostol \cite[6--17, p.\ 117]{tmA57.Apostol}) and \eqref{E:gradient.of.k} there exists $\xi = (b, \Delta, t)$ in the line segment joining $\xi_{1}$ and $\xi_{2}$ s.t.\
	\begin{align*}
		\bigl| f(y_{2}, t_{2}) - f(y_{1}, t_{1}) \bigr| &= \bigl| (\xi_{2} - \xi_{1}) \cdot \nabla k(\xi) \bigr| \\
		 &= \left| \frac{(b_{2} - b_{1}) (\Delta + t)}{(1-b)^{2}} + b \frac{\Delta_{2} - \Delta_{1}}{1-b} 
		   + b \frac{t_{2} - t_{1}}{1-b}\right| \exp \left\{ - \frac{b}{1-b} (\Delta + t) \right\}.
	\end{align*}
Now, $b_{1}, b_{2} < 1-\epsilon$ implies $b < 1 - \epsilon$.  Therefore, from \eqref{E:b.is.Lip.on.sigma}, we know that there exists $K < \infty$, depending only on $\epsilon$ and $\sigma$, s.t.\ 
	\[
		\left| \frac{(b_{2} - b_{1}) (\Delta + t)}{(1-b)^{2}} \right| \leq K \bigl| (y_{2}, t_{2}) - (y_{1}, t_{1}) \bigr| 
			\text{ and } \left| b \frac{t_{2} - t_{1}}{1-b}\right| 
			          \leq K \bigl| (y_{2}, t_{2}) - (y_{1}, t_{1}) \bigr| .
	\]

So it suffices to show that, at the possible cost of increasing $K$, we have
	\begin{equation}   \label{E:Delta.frac.leq.K.y.diff}
		\left| b \frac{\Delta_{2} - \Delta_{1}}{1-b} \right|  
		  \leq K \bigl| (y_{2}, t_{2}) - (y_{1}, t_{1}) \bigr| .
	\end{equation}
Let $h_{i} = \bar{h}(y_{i})$.  By compactness of $\mcl{C}$ (\eqref{E:C.is.compact.contains.A}), there exists $w_{i} \in \mcl{C}$ s.t.\ $\Delta_{i} = |h_{i} - w_{i}|$ ($i=1,2$).  WLOG $\Delta_{2} \geq \Delta_{1}$.  Thus, 
	\begin{align*}
		| \Delta_{2} - \Delta_{1} | &= |h_{2} - w_{2}| - |h_{1} - w_{1}| \\
		   &\leq |h_{2} - w_{1}| - |h_{1} - w_{1}| \\
		   &\leq |h_{2} - w_{1} - h_{1} + w_{1}| \\
		   &= | h_{2} - h_{1} |.
	\end{align*}
Thus, as in \eqref{E:b2.is.ratio},  
	\begin{align*}
		b | \Delta_{2} - \Delta_{1} | &\leq b | h_{2} - h_{1} | \\
		  &\leq (b_{1} + b_{2}) | h_{2} - h_{1} | \\
		  &= \left( \frac{ | y_{1} - z| }{ |h_{1} - z | } +  \frac{ | y_{2} - z| }{ |h_{2} - z | } \right) 
		           | h_{2} - h_{1} | \\
		  &\leq \frac{1}{ dist( z, \text{Bd} \, \sigma) } \bigl( | y_{1} - z|  +  | y_{2} - z| \bigr) 
		           | h_{2} - h_{1} |.
	\end{align*}
Now apply \eqref{E:|yi-z||h2-h1|.Lip.condtn}.

\emph{Proof of part (\ref{I:s.locly.Lip.off.C.cap.Bd})} 
By part (\ref{I:f=k(b,delta,t).loc.Lip}) of the lemma, $f$ is locally Lipschitz on $B_{z}$.  Moreover, 
by part (\ref{I:hbar.locally.Lip.b.Lip}) of the lemma, $\bar{h}$, and hence $\bar{\Delta}$, is locally Lipschitz on $\sigma \setminus \{z\}$.  Therefore, by \eqref{E:comp.of.Lips.is.Lip}, $s(y,t)$ is locally Lipschitz on $B_{z} \setminus \bigl( \{z\} \times [0,1] \bigr)$.  

It remains to show that $s$ is locally Lipschitz in $U \times [0,1]$ for some neighborhood, $U$, of $z$.  
By part (\ref{I:f=k(b,delta,t).loc.Lip}) of the lemma again and compactness, there exists a neighborhood, $U \subset \text{Int} \, \sigma$, of $z$, s.t.\ $f$ is actually Lipschitz on $U \times [0,1]$, with Lipschitz constant $M < \infty$, say.  
We wish to show that there exists $K < \infty$ s.t.\ 
	\begin{equation}  \label{E:s.Lip.condtn.near.z}
		\bigl| s(y_{2}, t_{2}) - s(y_{1}, t_{1}) \bigr|  \leq K \bigl| (y_{2}, t_{2}) - (y_{1}, t_{1}) \bigr|,
		\quad \text{for }y_{i} \in U \text{ and }t_{i} \in [0,1] \; (i = 1,2).
	\end{equation}
Since $U \subset \text{Int} \, \sigma$, by \eqref{E:when.bx.=1.or.0}, $b(y_{i}) < 1$ and we have $(y_{i}, t_{i}) \in B_{z}$ ($i=1,2$).  Therefore, by part (\ref{I:f=k(b,delta,t).loc.Lip}) of the lemma, 
	\begin{align}  \label{E:bound.s.diff}
		\bigl| s(y_{2}, t_{2}) - s(y_{1}, t_{1}) \bigr| 
		&\leq \Bigl| \bigl(1 - f(y_{2}, t_{2}) \bigr) \bar{h}(y_{2}) 
			- \bigl(1 - f(y_{1}, t_{1}) \bigr) \bar{h}(y_{1}) \Bigr|  \notag \\
		& \qquad \qquad   + \bigl| f(y_{2}, t_{2}) 
		             - f(y_{1}, t_{1}) \bigr| |z|   \\
		&\leq \Bigl| \bigl(1 - f(y_{2}, t_{2}) \bigr) \bar{h}(y_{2}) 
			- \bigl(1 - f(y_{1}, t_{1}) \bigr) \bar{h}(y_{1}) \Bigr|  \notag \\
		& \qquad \qquad   + M |z| \bigl| (y_{2}, t_{2}) - (y_{1}, t_{1}) \bigr|  \notag \\
		&\leq \Bigl| \bigl(1 - f(y_{2}, t_{2}) \bigr) \bar{h}(y_{2}) 
			- \bigl(1 - f(y_{1}, t_{1}) \bigr) \bar{h}(y_{1}) \Bigr|  \notag \\
		& \qquad \qquad   + M D \bigl| (y_{2}, t_{2}) - (y_{1}, t_{1}) \bigr|,  \notag 
	\end{align}
where $D < \infty$ is an upper bound on the norm of the points in $\sigma$.  Now
	\begin{multline} \label{E:bound.(1-k)h.diff}
		\Bigl| \bigl(1 - f(y_{2}, t_{2}) \bigr) \bar{h}(y_{2}) 
				- \bigl(1 - f(y_{1}, t_{1}) \bigr) \bar{h}(y_{1}) \Bigr| \\
			\begin{aligned}
			    {} 
			       &\leq \bigl|1 - f(y_{2}, t_{2}) \bigr|  
				   | \bar{h}(y_{2}) - \bar{h}(y_{1}) |   \\
				 & \qquad \qquad  + \bigl| f(y_{1}, t_{1}) 
				       - f(y_{2}, t_{2}) \bigr| | \bar{h}(y_{1}) |  \\
			       &\leq \bigl|1 - f(y_{2}, t_{2}) \bigr|  
					| \bar{h}(y_{2}) - \bar{h}(y_{1}) | 
					           + M \bigl| (y_{2}, t_{2}) - (y_{1}, t_{1}) \bigr| |\bar{h}(y_{1}) | \\
			       &\leq \bigl|1 - f(y_{2}, t_{2}) \bigr|  
					| \bar{h}(y_{2}) - \bar{h}(y_{1}) | 
							+ M D \bigl| (y_{2}, t_{2}) - (y_{1}, t_{1}) \bigr|.
			\end{aligned}
	\end{multline}
Thus, to prove \eqref{E:s.Lip.condtn.near.z} it suffices to find $K' < \infty$ s.t.\
	\begin{equation}  \label{E:(1-k)|h2-h1|.Lip.condtn}
		    \bigl| \bar{h}(y_{2}) - \bar{h}(y_{1}) \bigr| 
		       \bigl|1 - f(y_{2}, t_{2})\bigr|  
		          \leq K' |y_{2} - y_{1}|.
	\end{equation}
Now $b$ is bounded away from 1 on $U$ so by lemma \ref{L:k.theory}, 
$\tfrac{\partial}{\partial \beta} k \bigl( \beta, \bar{\Delta}(y_{2}), t_{2} \bigr)$ is bounded above 
for $\beta = b(y_{2})$, $y_{2} \in U$. Therefore, by \eqref{E:f.defn}, the mean value theorem, and \eqref{E:b.hbar.defn}, 
	\begin{align}
	\bigl|1 - f(y_{2}, t_{2})\bigr|  &=
		\bigl|k \bigl[ 0, \bar{\Delta}(y_{2}), t_{2} \bigr]  
		    - k \bigl[ b(y_{2}), \bar{\Delta}(y_{2}), t_{2} \bigr] \bigr|  \\
		&\leq M' b(y_{2}) \leq M' \frac{|y_{2} - z|}{dist(z, \text{Bd} \, \sigma)},  \notag
	\end{align}
for some $M' < \infty$ valid throughout $U$. Substituting this into \eqref{E:(1-k)|h2-h1|.Lip.condtn}, \eqref{E:s.Lip.condtn.near.z} follows from part (\ref{I:y-z.h2-h1.bound}) of the lemma. 
       \end{proof}

\begin{proof}[Proof of lemma \ref{L:s.hat.has.Lip.invrs}]
Let $k$ be the function defined in \eqref{E:defn.of.k} 
So 
	\begin{equation}   \label{E:k0.=.1}
		k(0,\delta,t) = 1, \text{ for every } \delta, t \in \RR.
	\end{equation}  
From lemma \ref{L:k.theory}, we have
	\begin{equation}  \label{E:k.beta.to.1}
		\text{$k(\beta, \delta, t) \to 0$ as $\beta \uparrow 1$ 
		    for any $\delta \geq 0$ and $t > 0$.}
	\end{equation} 
    
From \eqref{E:gradient.of.k}, we observe that  
   \begin{equation*}  
      \frac{ \partial }{ \partial \beta } k(\beta, \delta, t) < 0,  \quad
        \beta \in [0,1), \; \delta \geq 0, \; t > 0.
   \end{equation*}
Hence, 
	\begin{multline}  \label{E:k.is strictly.decreasing.in.beta}
		\text{For any $t > 0$ and } \delta \geq 0,   \\
			\text{ the function } \beta \mapsto k(\beta, \delta, t) 
			\text{ is strictly decreasing in } \beta \in [0,1).
	\end{multline} 
   
Combining \eqref{E:k.is strictly.decreasing.in.beta} with \eqref{E:k0.=.1} and \eqref{E:k.beta.to.1} we see that 
	\begin{equation}  \label{E:k.maps.[0,1).onto.(0,1]}
		\text{The map  } \beta \mapsto k(\beta, \delta, t) \text{ maps $[0,1)$ onto $(0,1]$  for any }
			\delta \geq 0 \text{ and } t > 0.
	\end{equation}
 
Notice that, by definition of $k$ (\eqref {E:defn.of.k}),       
	\begin{equation}  \label{E:k.bdd.by.pwr.of.1-beta}
	      k(\beta, \delta, t) \leq \exp \left\{ - \frac{\beta \epsilon}{1 - \beta} \right\},
	      \quad \text{if } \beta \in [0,1), \; 
	        \delta \geq 0, 
	         \text{ and } 0 < \epsilon \leq t.  
	\end{equation}
Hence, since $- \log k(\beta, \delta, t) > 0$, 
	\begin{equation}  \label{E:beta.log.k.bnd}
	      \beta \leq \frac{- \log k(\beta, \delta, t)}{\epsilon - \log k(\beta, \delta, t)}, 
	      \quad \text{if } \beta \in [0,1), \; 
	        \delta \geq 0, 
	         \text{ and } 0 < \epsilon \leq t.
	\end{equation}
Note that the right hand side (RHS) of the preceding is strictly less than 1. 
From \eqref{E:k.1-beta.bnd}, we see      
	\begin{equation*} 
	      k(\beta, \delta, t) \leq (1 - \beta)^{\epsilon}, 
		    \quad \text{if } \beta \in [0,1), \; 
		        \delta \geq 0, \text{ and } 0 < \epsilon \leq t.
	\end{equation*}
Thus,
	\begin{equation}  \label{E:k.1-beta.eps.bnd}
		\epsilon^{-1} \Bigl[ 1 - k(\beta, \delta, t) \Bigr] 
		  \geq \epsilon^{-1} \bigl[ 1 - (1 - \beta)^{\epsilon} \bigr],  
		    \quad \text{if } \beta \in [0,1), \; 
		        \delta \geq 0, \text{ and } 0 < \epsilon \leq t.
	\end{equation}
For $0 \leq \beta < 1$ and $0 < \epsilon \leq 1$, let 
$j(\beta) = 1 - (1 - \beta)^{\epsilon} - \epsilon \beta$.  Then $j(0) = 0$ and 
$j'(\beta) = (1 - \beta)^{\epsilon-1} - \epsilon \geq 0$ for $0 \leq \beta < 1$ and $0 < \epsilon \leq 1$.  Hence, $j(\beta) \geq 0$ for $0 \leq \beta < 1$ and $0 < \epsilon \leq 1$.  
Thus, from \eqref{E:k.1-beta.eps.bnd}
	\begin{equation}  \label{E:k.beta.eps.bnd}
		\epsilon^{-1} \Bigl[ 1 - k(\beta, \delta, t) \Bigr] \geq \beta , 
		    \quad \text{if } \beta \in [0,1), \; 
		        \delta \geq 0, \text{ and } 0 < \epsilon \leq t \leq 1.
	\end{equation}

\emph{Claim:}   
	\begin{equation}  \label{E:s.maps.Int.sigma.x.(0,1].onto.Int.sigma}
		\text{For any } t \in (0,1] \text{ we have }
			s \bigl[ (\text{Int} \, \sigma) \times \{ t \} \bigr] = \text{Int} \, \sigma.
	\end{equation}
By \eqref{E:s.maps.sigma.into.sigma}, we have 
$s \bigl[ (\text{Int} \, \sigma) \times \{ t \} \bigr] \subset \sigma$.  
It follows from  lemma \ref{L:props.of.s} (part \ref{I:s.in.Bd.iff.y.in.Bd}) that
$s \bigl[ (\text{Int} \, \sigma) \times \{ t \} \bigr] \subset \text{Int} \, \sigma$. 

To prove the reverse inclusion, note that by lemma \ref{L:props.of.s}(part \ref{I:z.is.fixed.pt.of.s}), for every $t \in (0,1]$, $z \in s \bigl[ (\text{Int} \, \sigma) \times \{ t \} \bigr]$.  
Let $w \in (\text{Int} \, \sigma) \setminus \{ z \}$ and $t \in (0,1]$. Then, by \eqref{E:when.bx.=1.or.0}, $0 < b(w) < 1$ and by \eqref{E:b.hbar.defn}
	\[
		w = \bar{h}(w) + \bigl( 1 - b(w) \bigr) \bigl[ z - \bar{h}(w) \bigr].
	\]
By \eqref{E:k.maps.[0,1).onto.(0,1]}, we may pick $\beta \in [0, 1)$ s.t.\ 
		$k \bigl( \beta, \bar{\Delta}(w), t \bigr) = 1 - b(w)$.
Since $1 - b(w) \in (0,1)$, by \eqref{E:k0.=.1} and \eqref{E:defn.of.k}, $1 > \beta > 0$.  Let 
		$y = \bar{h}(w) + (1 - \beta) \bigl[ z - \bar{h}(w) \bigr] \in \text{Int} \, \sigma$.
Then $\bar{h}(y) = \bar{h}(w)$ (so $\bar{\Delta}(y) = \bar{\Delta}(w)$), $b(y) = \beta$, and by \eqref{E:s.in.terms.of.k},
	\[
		s(y,t) = \bar{h}(w) + \bigl( 1 - b(w) \bigr) \bigl[ z - \bar{h}(w) \bigr] = w.
	\]
This proves the claim \eqref{E:s.maps.Int.sigma.x.(0,1].onto.Int.sigma}.

Let $y \in \text{Int} \, \sigma$ and $t \in (0,1]$.  If $s(y,t) = z$ then lemma \ref{L:props.of.s}(part \ref{I:z.is.fixed.pt.of.s}) implies $y=z$ (since $t > 0$).  
Suppose $s(y,t) \ne z$.  Then, by lemma \ref{L:props.of.s}(part \ref{I:hbar.s.=.hbar.y}), 
$\bar{h}(y) = \bar{h} \bigl[ s(y,t) \bigr]$ can be determined by $s(y,t)$.  Hence, $\bar{\Delta}(y)$ and, 
by \eqref{E:s.in.terms.of.k}, $k \bigl( b(y), \bar{\Delta}(y), t \bigr)$ can be determined by $s(y,t)$.  Therefore, by \eqref{E:k.is strictly.decreasing.in.beta}, 
$b(y)$ and, hence, $y$ can be determined from $\hat{s}(y,t)$.  To sum up:
   \[
      \hat{s} \text{ is invertible on } (\text{Int} \, \sigma) \times (0,1].
   \]

By lemma \ref{L:props.of.s} 
(part \ref{I:hbar.locally.Lip.b.Lip}) and \eqref{E:when.bx.=1.or.0}, if $B_{max} \in (0,1)$, then
	\[
		W(B_{max}) := \bigl\{ y \in \sigma : b(y) < B_{max} \bigr\}
	\]
is an open neighborhood of $z$ whose closure lies in $\text{Int} \, \sigma$.  To prove that $\hat{s}$ has a locally Lipschitz inverse on $(\text{Int} \, \sigma) \times (0,1]$, it suffices to show that its inverse is \emph{Lipschitz} on $W(B_{max}) \times (\epsilon, 1]$ for arbitrary $B_{max} \in (0,1)$ and $\epsilon \in (0,1)$.  

Let $B_{max} \in (0,1)$ and $\epsilon \in (0,1)$, let $y, y' \in \text{Int} \, \sigma$, let $t, t' \in (\epsilon, 1]$ and suppose $s(y,t), s(y',t') \in W := W(B_{max})$.  Write  
	\begin{multline*}
		x = \bar{h} \bigl( s(y,t) \bigr), \, B = b \bigl( s(y,t) \bigr) 
		       = 1 - k\bigl( b(y), \bar{\Delta}(y), t \bigr), \;
		    x' = \bar{h} \bigl( s(y',t') \bigr),  \\
		\text{ and } B' = b \bigl( s(y',t') \bigr) = 1 - k\bigl( b(y'), \bar{\Delta}(y'), t' \bigr).  
	\end{multline*}
So $s(y,t) = B x + (1-B) z$ and $s(y',t') = B' x' + (1 - B') z$.  Since $s(y,t) \in W(B_{max})$, we have $B < B_{max}$.  Hence,  
	\[
		k \bigl( b(y), \bar{\Delta}(y), t \bigr) > 1 - B_{max}.
	\]
Similarly for $B'$.  Therefore, by \eqref{E:beta.log.k.bnd} and \eqref{E:k.beta.eps.bnd}, there exists 
$\beta_{max} \in (0,1)$, depending only on $B_{max}$ and $\epsilon$, s.t.\
	\begin{equation} \label{E:by.B.beta.max.eps.bnd}
		b(y) < \min \Bigl\{ \beta_{max}, \; \epsilon^{-1} B \Bigr\}.
	\end{equation}
Similarly for $b(y')$.

First, assume $s(y,t) = z$.  Then, by \eqref{E:when.bx.=1.or.0}, we have $b \bigl( s(y,t) \bigr) = 0$.  Moreover, since $t > \epsilon > 0$, by lemma \ref{L:props.of.s} (part \ref{I:z.is.fixed.pt.of.s}), $y = z$.  Hence, WLOG $y' \ne y = z$. Then, by \eqref{E:b.hbar.defn}, \eqref{E:by.B.beta.max.eps.bnd}, and \eqref{E:b.is.Lip.on.sigma} there exists $K_{3} < \infty$ s.t.\
	\begin{align*}
		|y-y'| &= |y'-z| \\
		  &= b(y') \bigl| \bar{h}(y') - z \bigr| \\
		  &\leq \epsilon^{-1} B' \; diam(\sigma) \\
		      &= \epsilon^{-1} \; diam(\sigma) \Bigl| b \bigl( s(y',t') \bigr) - 0 \Bigr| \\
		  &= \epsilon^{-1} \; diam(\sigma) \Bigl| b \bigl( s(y',t') \bigr) - b \bigl( s(y,t) \bigr) \Bigr| \\
		  &\leq K_{3} \epsilon^{-1} \; diam(\sigma) \bigl| s(y',t') - s(y,t) \bigr|.
	\end{align*}
This proves the lemma in the case where either $s(y,t) = z$ or $s(y',t') = z$.

Now assume $s(y,t), s(y', t') \in W \setminus \{ z \}$.  Then, by lemma \ref{L:props.of.s} (part \ref{I:hbar.s.=.hbar.y}), 
	\begin{equation}  \label{E:x.=.h}
	x = \bar{h} \bigl( s(y,t) \bigr) = \bar{h}(y) \text{ and } x' = \bar{h} \bigl( s(y',t') \bigr) = \bar{h}(y').
	\end{equation}
\emph{Claim:}  There exists $K_{1} < \infty$ depending only on $\sigma$, $z$, $B_{max}$, and $\epsilon$ s.t.\ 
	\begin{equation}   \label{E:b.x.diff.s.bnd}
		b(y) |x' - x| \leq K_{1} \bigl| s(y',t') - s(y,t) \bigr|.
	\end{equation}
By \eqref{E:by.B.beta.max.eps.bnd}, \eqref{E:when.bx.=1.or.0},  and \eqref{E:b.is.Lip.on.sigma}, there exists $K_{2} < \infty$ and by lemma \ref{L:props.of.s} (part \ref{I:y-z.h2-h1.bound}), there exists $K'' < \infty$ ($K_{2}$ and $K''$ depending only on $\sigma$ and $z$) s.t.\ 
	\begin{align*}
		 b(y) |x' - x| &\leq \epsilon^{-1} B | x ' - x |  \\
		 &= \epsilon^{-1} \Bigl| b \bigl( s(y,t) \bigr) - b(z) \Bigr| 
		       \Bigl| \bar{h} \bigl( s(y',t') \bigr) - \bar{h} \bigl( s(y,t) \bigr) \Bigr| \\
		 &\leq \epsilon^{-1} K_{2} \bigl| s(y,t) - z \bigr| 
		       \Bigl| \bar{h} \bigl( s(y',t') \bigr) - \bar{h} \bigl( s(y,t) \bigr) \Bigr| \\
		 &\leq \epsilon^{-1} K_{2} K'' \bigl| s(y',t') - s(y,t) \bigr|.
	\end{align*}
This proves the claim \eqref{E:b.x.diff.s.bnd}.

Next, we \emph{claim} that there exists $K_{4} < \infty$ depending only on $\sigma$, $z$, $B_{max}$, and $\epsilon$ s.t.\
	\begin{equation}  \label{E:b.s.t.diff.bound}
		\bigl| b(y') - b(y) \bigr| \leq K_{4} \Bigl( \bigl| s(y',t') - s(y,t) \bigr| + |t' - t| \Bigr).
	\end{equation}
WLOG $b(y) \geq b(y')$.  Thus, since $t, t' \in (\epsilon, 1]$, if $( \tilde{\beta}, \tilde{\delta}, \tilde{t} )$ is a point on the line joining $\bigl( b(y'), \bar{\Delta}(y'), t' \bigr)$ and $\bigl( b(y), \bar{\Delta}(y), t \bigr)$ we have $\tilde{\beta} \leq b(y)$ and $\tilde{t} > \epsilon$. By lemma \ref{L:k.theory}, the multivariate Mean Value Theorem (Apostol \cite[6--17, p.\ 117]{tmA57.Apostol}), and \eqref{E:by.B.beta.max.eps.bnd}, we may choose 
$( \tilde{\beta}, \tilde{\delta}, \tilde{t} )$  s.t.\
	\begin{multline*}
		\Bigl| k \bigl( b(y'), \bar{\Delta}(y'), t' \bigr) - k \bigl( b(y), \bar{\Delta}(y), t \bigr) \Bigr| \\
			\begin{aligned}
			  {} &= \left| \frac{\tilde{\delta} + \tilde{t}}{(1-\tilde{\beta})^{2}} \bigl( b(y') - b(y) \bigr) 
			     + \frac{\tilde{\beta}}{1-\tilde{\beta}} \bigl( \bar{\Delta}(y') - \bar{\Delta}(y) \bigr)   
			     + \frac{\tilde{\beta}}{1-\tilde{\beta}} (t' - t) \right| \\
			   &\geq \frac{\tilde{\delta} + \tilde{t}}{(1-\tilde{\beta})^{2}} \bigl| b(y') - b(y) \bigr|  
			     - \frac{\tilde{\beta}}{1-\tilde{\beta}} \bigl| \bar{\Delta}(y') - \bar{\Delta}(y) \bigr|  
			     - \frac{\tilde{\beta}}{1-\tilde{\beta}} |t' - t|   \\
			   &\geq \epsilon \bigl| b(y') - b(y) \bigr|  
			     - \frac{b(y)}{1-\beta_{max}} \bigl| \bar{\Delta}(y') - \bar{\Delta}(y) \bigr|  
			     - \frac{\beta_{max}}{1-\beta_{max}} |t' - t|.
			\end{aligned}
	\end{multline*}
Thus, 
	\begin{multline}  \label{E:b.diff.bnd}
		\bigl| b(y') - b(y) \bigr| \leq 
		\epsilon^{-1} \Bigl| k \bigl( b(y'), \bar{\Delta}(y'), t' \bigr) 
		        - k \bigl( b(y), \bar{\Delta}(y), t \bigr) \Bigr| \\
		+ \frac{b(y)}{\epsilon (1-\beta_{max})} \bigl| \bar{\Delta}(y') - \bar{\Delta}(y) \bigr|  
		+ \frac{\beta_{max}}{\epsilon (1-\beta_{max})} |t' - t|.
	\end{multline}
Now, by \eqref{E:b.is.Lip.on.sigma} there exists $K'' < \infty$ s.t.,
	\begin{multline} \label{E:k.diff.s.diff.bnd}
		\Bigl| k \bigl( b(y'), \bar{\Delta}(y'), t' \bigr) 
			        - k \bigl( b(y), \bar{\Delta}(y), t \bigr) \Bigr| = |B - B'|   \\
			        = \Bigl| b \bigl( s(y',t') \bigr) - b \bigl( s(y,t) \bigr) \Bigr| 
			                 \leq K''  \bigl| s(y',t') - s(y,t) \bigr|.
	\end{multline}
In addition, by definition of $\bar{\Delta}$ (\eqref{E:Delta.bar.defn}), example \ref{E:dist.is.Lip}, \eqref{E:x.=.h}, and \eqref{E:b.x.diff.s.bnd}, we have
	\[
		b(y) \bigl| \bar{\Delta}(y') - \bar{\Delta}(y) \bigr| 
		   \leq b(y) |x - x'| \leq K_{1} \bigl| s(y',t') - s(y,t) \bigr|.
	\]
Substituting this and \eqref{E:k.diff.s.diff.bnd} into \eqref{E:b.diff.bnd} yields the claimed \eqref{E:b.s.t.diff.bound} 

Finally, we show that there exists $K_{5} < \infty$ depending only on $\sigma$, $z$, $B_{max}$, and $\epsilon$ s.t.\ 
	\begin{equation}  \label{E:y.diff.by.s.diff.bnd}
		|y' - y| \leq K_{5} \bigl| \hat{s}(y',t') - \hat{s}(y,t) \bigr|.
	\end{equation}
To see this, argue as follows using \eqref{E:b.hbar.defn} and \eqref{E:x.=.h}
	\begin{align*}
		|y' - y| &= \Bigl| \bigl( b(y') - b(y) \bigr) \bigl( \bar{h}(y') - z \bigr) 
		   + b(y) \bigl( \bar{h}(y') - \bar{h}(y) \bigr)  \Bigr| \\
		   &\leq  \bigl| b(y') - b(y) \bigr| diam(\sigma) + b(y) \bigl| \bar{h}(y') - \bar{h}(y) \bigr| \\
		   &= diam(\sigma) \bigl| b(y') - b(y) \bigr| + b(y) \bigl| x' - x \bigr|. \\
	\end{align*}
\eqref{E:y.diff.by.s.diff.bnd} now follows from \eqref{E:b.s.t.diff.bound} and \eqref{E:b.x.diff.s.bnd}.  Since, trivially, $|t - t'| \leq \bigl| \hat{s}(y',t') - \hat{s}(y,t) \bigr|$ lemma \ref{L:s.hat.has.Lip.invrs} follows.
\end{proof}

  \begin{proof}[Proof of lemma \ref{L:g.is.loc.Lip}]
Let $x \in|P|  \setminus \bigl[ \mcl{C} \cap (\text{Bd} \, \sigma) \bigr]$ and let $\tau \in P$ satisfy 
$x \in \text{Int} \, \tau$.    
(See \eqref{E:x.in.exctly.1.simplex.intrr}.) We re-use an idea that we used in the proof of corollary \ref{C:drop.finiteness.of.P}. Write $\overline{\text{St}}_{x} := \overline{\text{St}} \, \tau$. Since $\overline{\text{St}}_{x}$ is starlike w.r.t.\ $x$ (appendix \ref{S:basics.of.simp.comps}), we have
	\begin{equation*}
		\bar{V}_{t,x} := t (\overline{\text{St}}_{x} - x) + x \subset \overline{\text{St}}_{x},
		         \quad t \in (0,1). 
	\end{equation*}
Here, the vector operations are performed point-wise.  Let
	\begin{equation*}
		V_{t,x} := t (\text{St} \, \tau - x) + x,
		      \quad t \in (0,1). 
	\end{equation*}
Note that $V_{t,x}$ only intersects simplices in $P$ that have $\tau$ as a face.  
In particular, $\tau_{t,x} \subset V_{t,x}$. If $\rho \in P$ and $\rho \subset \overline{\text{St}}_{x}$, let
	\begin{equation}   \label{E:rho.t.x.defn}
		\rho_{t,x} := t ( \rho -x) + x,
		         \quad t \in (0,1).
	\end{equation}
Then, $\rho_{t,x} \subset \overline{\text{St}}_{x}$ is a simplex, for $t \in (0,1)$
	\[
		\bigl\{ \rho_{t,x} : \rho \in P \text{ and } \, \rho \subset \overline{\text{St}}_{x} \}
	\]
is a simplicial complex, call it $P_{t,x}$, and
	\begin{equation*}
		\bar{V}_{t,x} = \bigcup_{\rho \in P; \, \rho \subset \overline{\text{St}}_{x}} \rho_{t,x}.
	\end{equation*}
Thus, $\bar{V}_{t,x}$ is the underlying space of $P_{t,x}$.

Since $\mcl{C} \cap (\text{Bd} \, \sigma)$ is closed by \eqref{E:C.is.compact.contains.A} and 
$x \notin \mcl{C} \cap (\text{Bd} \, \sigma)$,  we may pick $t_{x} \in (0,1)$ so small that 
	\begin{equation}  \label{E:Vbar.tx.doesnt.intrsct.C.capBd}
		\bar{V}_{t_{x},x} \cap \bigl[ \mcl{C} \cap (\text{Bd} \, \sigma) \bigr] = \varnothing.  
	\end{equation}
By \eqref{E:St.sigma.is.open}, we have that $V_{t_{x},x}$ is an open neighborhood of $x$. Hence, it suffices to prove that $g$ is Lipschitz on $V_{t_{x},x}$ for every 
$x \in |P|  \setminus \bigl[ \mcl{C} \cap (\text{Bd} \, \sigma) \bigr]$.

Let $y_{1}, y_{2} \in V_{t_{x},x}$.  Then, by definition of $\text{St} \, \tau$ 
(appendix \ref{S:basics.of.simp.comps}),
$y_{i} \in  \text{Int} \, \rho_{i; t_{x},x}$ for some $\rho_{i; t_{x},x} \in P_{t_{x},x}$, s.t.\ $\rho_{i} \in P$ has 
$\tau$ as a face ($i=1,2$).  
Thus,  
$\varnothing \ne \tau_{t_{x},x} \subset \rho_{1; t_{x},x} \cap \rho_{2; t_{x},x}$.  Suppose neither $\rho_{1; t_{x},x}$ nor 
$\rho_{2; t_{x},x}$ is a subset of the other. Then, by corollary \ref{C:reverse.triangle.ineq.in.simp.cmplxs}, there exists $K_{x} < \infty$, depending only 
on $\rho_{i; t_{x},x}$ ($i=1,2$), and 
$\tilde{y}_{1}, \tilde{y}_{2} \in \text{Int} \, (\rho_{1; t_{x},x} \cap \rho_{2; t_{x},x})$ s.t.\	
         \begin{equation}  \label{E:break.up.y2.y1.diff}
		|y_{2} - \tilde{y}_{2}| + |\tilde{y}_{2} - \tilde{y}_{1}| + |\tilde{y}_{1} - y_{1}| 
				     \leq K_{x} |y_{2}-y_{1}|.
	\end{equation}
Now,
		\begin{equation*}
			\bigl| g(y_{2}) - g(y_{1}) \bigr| 
			  \leq \bigl| g(y_{2}) - g(\tilde{y}_{2}) \bigr| + \bigl| g(\tilde{y}_{2}) - g(\tilde{y}_{1}) \bigr| 
			       + \bigl| g(\tilde{y}_{1}) - g(y_{1}) \bigr|.   
	  	     \end{equation*}
Therefore, by \eqref{E:break.up.y2.y1.diff}, it suffices to prove that $g$ is Lipschitz on $\rho_{t_{x},x}$ for each $\rho \in P$ having $\tau$ as a face (which means in particular $x \in \rho$). 

First, let $\rho \in P$ have $\tau$ as a face but assume $\rho \cap (\text{St} \, \sigma) = \varnothing$.  Let $y \in \rho_{t_{x},x} \subset \rho \setminus \bigl[ \mcl{C} \cap (\text{Bd} \, \sigma) \bigr]$. If $y \notin \overline{\text{St}} \, \sigma$ or $y \in \text{Lk} \, \sigma$, then, by \eqref{E:when.g(y)=y}, $g(y) = y$. Suppose $y \in (\overline{\text{St}} \, \sigma) \setminus (\text{Lk} \, \sigma)$ and let $\xi$ be the face of $\rho$ s.t.\ $y \in \text{Int} \, \xi$.  Since $\rho \cap (\text{St} \, \sigma) = \varnothing$, we have that $\sigma$ is not a face of $\xi$. Since $y \in \overline{\text{St}} \, \sigma$, by \eqref{E:Int.rho.cuts.sigma.then.rho.in.sigma}, we have $\xi \subset \overline{\text{St}} \, \sigma$.  
Therefore, $\xi \subset \overline{\text{St}} \, \sigma$, $\xi \cap \sigma \neq \varnothing$, but $\sigma$ is not a face of $\xi$.  Hence, by \eqref{E:when.sigma.x.in.Bd.sigma} and \eqref{E:when.g(y)=y} again, $g(y) = y$, since $y \notin \text{Lk} \, \sigma$.  In summary, $\rho \cap (\text{St} \, \sigma) = \varnothing$ implies $g(y) = y$ for $y \in \rho_{t_{x},x}$ so $g$ is Lipschitz on $\rho_{t_{x},x}$.

Next, let $\rho \in P$ have $\tau$ as a face (so $x \in \rho_{t_{x},x} \subset \rho$) but this time assume $\rho \cap (\text{St} \, \sigma) \ne \varnothing$.  It follows from \eqref{E:Int.rho.cuts.sigma.then.rho.in.sigma} and the definition of $\text{St} \, \sigma$ (appendix \ref{S:basics.of.simp.comps}) that $\rho$ has $\sigma$ as a face. Let $\omega \in P$ be the face 
of $\rho$ opposite $\sigma$. Let 
	\[
		U^{\sigma} := U^{\sigma, \rho} := \bigl\{ y \in \rho : \mu(y) < 3/4 \bigr\}.  
	\]
By \eqref{E:mu.w.loc.Lip}, $U^{\sigma}$ is open in $\rho$.

Define
	\begin{multline}  \label{E:ell.defn}
		   		\ell(y) :=   \\
				\begin{cases}
				\bigl[ 1 - \mu(y) \bigr]
				   +  \mu(y) \left[ \Bigl( 1 - b \bigl[ \sigma(y) \bigr] \Bigr) 
				        + b \bigl[ \sigma(y) \bigr] 
				             \frac{\bar{\Delta} \bigl[ \sigma(y) \bigr]}{diam(\sigma)}  \right], &\text{ if } 
				   y \in \rho \setminus \bigl[ (\text{Lk} \, \sigma)\cup \{ z \} \bigr] , \\
				   1, &\text{ if } y \in \rho \cap (\text{Lk} \, \sigma), \\
				   1, &\text{ if } y = z.
				\end{cases}
	\end{multline}
Thus, $\ell \geq 0$.  \emph{Claim:} 
	\begin{equation}  \label{E:el.is.cont.poz.off.C.cap.Bd}
		\ell \text{ is continuous on } \rho   
		   \text{ and } \ell(y) = 0 \text{ if and only if } y \in  \mcl{C} \cap (\text{Bd} \, \sigma).
	\end{equation}
To see this, note that \eqref{E:mu.w.loc.Lip} implies $\mu$ is continuous on $\rho$ and lemma \ref{L:props.of.s}(part \ref{I:hbar.locally.Lip.b.Lip}), \eqref{E:b.is.Lip.on.sigma},  tells us that $b$ is continuous on $\sigma$. Lemma \ref{L:props.of.s}(part \ref{I:hbar.locally.Lip.b.Lip}) also tells use that $\bar{h}(x)$ is locally Lipschitz in $x \in \sigma \setminus \{ z \}$.  Therefore, \eqref{E:Delta.bar.defn}, example \ref{E:dist.is.Lip}, and \eqref{E:C.is.compact.contains.A} imply that $\bar{\Delta}$ is continuous 
on $\sigma \setminus \{ z \}$. But  
$\bar{\Delta}/diam(\sigma) \leq 1$ is bounded (see \eqref{E:Delta.bar.z.=diam.sig}) and \eqref{E:sigma.ident.on.sigma} and\eqref{E:when.bx.=1.or.0} tells us $b \bigl[ \sigma(z) \bigr] = b(z) = 0$.  \eqref{E:sigma.(.).is.loc.Lip.off.Lk} tells us that 
$\sigma(\cdot)$ is continuous 
on $\rho \setminus (\text{Lk} \, \sigma)$.  But $\sigma(\cdot)$ is bounded and \eqref{E:wy.muy.when.y.in.Lk.or.sigma} tells us that 
$\mu(y) = 0$ on $\text{Lk} \, \sigma$ and $\mu(z) = 1$. That establishes the continuity of $\ell$. 

\eqref{E:wy.muy.when.y.in.Lk.or.sigma} tells us that $\mu(y) = 1$ if and only if $y \in \sigma$, \eqref{E:sigma.ident.on.sigma} tells us that $\sigma(y) = y$ if  $y \in \sigma$, \eqref{E:when.bx.=1.or.0} tells us that $b(x) = 1$ if and only if $x \in \text{Bd} \, \sigma$, and \eqref{E:Delta.bar.defn} and \eqref{E:C.is.compact.contains.A} tells us that $x \in \text{Bd} \, \sigma$ and $\bar{\Delta}(x) = 0$ if and only 
if $x \in \mcl{C} \cap (\text{Bd} \, \sigma)$.  Therefore, $\ell(y) = 0$ if and only if $y \in  \mcl{C} \cap (\text{Bd} \, \sigma)$. This proves the claim \eqref{E:el.is.cont.poz.off.C.cap.Bd}.

For $\epsilon \in (0,1)$, let
	\[
		U_{\sigma, \epsilon} := U_{\sigma, \epsilon, \rho} 
		    := \bigl\{ y \in \rho : 
		           \mu(y) > 1/4 \text{ and }
			        \ell(y) > \epsilon \bigr\}.
	\]
By \eqref{E:mu.w.loc.Lip}, \eqref{E:wy.muy.when.y.in.Lk.or.sigma}, and \eqref{E:el.is.cont.poz.off.C.cap.Bd}, $U_{\sigma, \epsilon}$ is open in $\rho$ and 
$U_{\sigma, \epsilon}$ contains $z \in \sigma$.  
In addition, by \eqref{E:wy.muy.when.y.in.Lk.or.sigma}, \eqref{E:mu.w.loc.Lip}, 
and \eqref{E:el.is.cont.poz.off.C.cap.Bd},
	\begin{equation}   \label{E:U.sigma.eps.disjnt.from.Lk.C.cp.Bd}
		\bar{U}_{\sigma, \epsilon} \text{ is compact and disjoint from } \text{Lk} \, \sigma 
			\text{ and } \mcl{C} \cap (\text{Bd} \, \sigma). 
	\end{equation}
Note that, since $x \in \tau \subset \rho$, we have by \eqref{E:Vbar.tx.doesnt.intrsct.C.capBd} and \eqref{E:wy.muy.when.y.in.Lk.or.sigma},
	\[
		\rho_{t_{x},x} \subset \rho \setminus \bigl[ \mcl{C} \cap (\text{Bd} \, \sigma) \bigr] 
		    =  U^{\sigma} \cup \left( \bigcup_{\epsilon > 0} U_{\sigma, \epsilon} \right)
		     \subset (\overline{\text{St}} \, \sigma) 
		            \setminus \bigl[ \mcl{C} \cap (\text{Bd} \, \sigma) \bigr].
	\]
In particular, either $x \in U^{\sigma, \rho}$ or, for some $\epsilon > 0$, 
$x \in U_{\sigma, \epsilon, \rho}$.  Therefore, by compactness of $\overline{\rho_{t_{x},x}}$
and \eqref{E:Vbar.tx.doesnt.intrsct.C.capBd}, for each of the finitely many $\rho \in P$ having $\tau$ as a face and s.t.\  $\sigma \subset \rho$, we have either $\rho_{t_{x},x} \subset U^{\sigma, \rho}$ or, for some $\epsilon > 0$, we have $\rho_{t_{x},x} \subset U_{\sigma, \epsilon, \rho}$.  Thus, it suffices to show that $g$ is Lipschitz on each $U^{\sigma, \rho}$ and on each $U_{\sigma, \epsilon, \rho}$.

Let $\rho \in P$ with $\tau \subset \rho \subset \overline{\text{St}}_{x}$ and $\sigma \subset \rho$ be fixed.  We show that $g$ is Lipschitz in $U^{\sigma}$.  The closure $\bar{U}^{\sigma}$ of $U^{\sigma, \rho}$ is compact and lies, by \eqref{E:wy.muy.when.y.in.Lk.or.sigma} and \eqref{E:mu.w.loc.Lip}, 
in $(\overline{\text{St}} \, \sigma) \setminus \sigma$. In addition, $\sigma \times [1/4, 1]$ is
a compact subset of $B_{z}$ as defined in \eqref{E:Bz.defn}.
This means, by \eqref{E:mu.w.loc.Lip} and lemma \ref{L:props.of.s}(part \ref{I:s.locly.Lip.off.C.cap.Bd}), that 
	\begin{equation}  \label{E:fns.Lip.on.U^sigma}
	    \mu  \text{ and } w \text{ are Lipschitz on } U^{\sigma}.    
		   \text{ Moreover, } 
		     s_{z} \text{ is Lipschitz on } \sigma \times [1/4,1].
	\end{equation}

Let $y_{1}, y_{2} \in  U^{\sigma}$.  Recall that $\omega$ is the face of $\rho$ opposite $\sigma$.  So $\omega \subset U^{\sigma}$ by \eqref{E:wy.muy.when.y.in.Lk.or.sigma}.  To simplify the notation, write $x_{i} = \sigma(y_{i})$, $\mu_{i} = \mu(y_{i})$, $s_{i} = s( x_{i}, 1 - \mu_{i} )$, and $w_{i} = w(y_{i}) \in \omega$ ($i=1,2$).  (See \eqref{E:w.on.sigma,sigma(.).on.Lk}.) Also, $K_{1}, K_{2}, \ldots$ will be finite and only depend on $U^{\sigma}$.  Then
	\begin{equation*}  \label{E:g2-g1.ineq}
		\bigl| g(y_{2}) - g(y_{1}) \bigr| 
		  \leq | \mu_{2} s_{2} - \mu_{1} s_{1}| + \bigl| (1-\mu_{2}) w_{2} - (1-\mu_{1}) w_{1} \bigr| 
		    \leq | \mu_{2} s_{2} - \mu_{1} s_{1}| +  K_{1} | y_{2} - y_{1} |,
	\end{equation*}
by \eqref{E:general.formula.for.g}, \eqref{E:comp.of.Lips.is.Lip}, and \eqref{E:fns.Lip.on.U^sigma}. Let $D := \max \{ |u|, u \in |P| \} < \infty \}$.  Then, also by  
\eqref{E:fns.Lip.on.U^sigma},
	\begin{align*}
		| \mu_{2} s_{2} - \mu_{1} s_{1}| &\leq | \mu_{2} s_{2} - \mu_{2} s_{1}| 
		           + | \mu_{2} s_{1} - \mu_{1} s_{1}| \\
		   &= \mu_{2} | s_{2} - s_{1} | + | \mu_{2} - \mu_{1} | |s_{1}| \\
		   &\leq \mu_{2} K_{2} \bigl| ( x_{2}, 1- \mu_{2} ) -  ( x_{1}, 1- \mu_{1} ) \bigr| +
		        D K_{3} |y_{2} - y_{1}| \\
		   &\leq \mu_{2} K_{2} \bigl( |x_{2} - x_{1}| +  | \mu_{1} - \mu_{2} | \bigr) +
		        D K_{3} |y_{2} - y_{1}| \\ 
		   &\leq \mu_{2} K_{2} |x_{2} - x_{1}| + K_{2} K_{4} |y_{2} - y_{1}| +
		        D K_{3} |y_{2} - y_{1}|.
	\end{align*}

Thus, to prove that $g$ is Lipschitz in $U^{\sigma}$ 
it suffices to find a $K < \infty$ s.t.\
	\begin{equation}   \label{E:mu2.sigma1-sigma2.Lip.bound}
		\mu_{2} |x_{2} - x_{1}| \leq K |y_{2} - y_{1}|.
	\end{equation} 
Note that $\min \{ dist(y, \omega), y \in \sigma \} > 0$.  Then, from \eqref{E:y.in.terms.of.sigma.w}, we see 
	\[
		\mu_{2} = \frac{| y_{2} - w_{2} |}{| x_{2} - w_{2} |} 
			\leq \frac{| y_{2} - w_{2} |}{\min \{ dist(x, \omega), x \in \sigma \} }.
	\]
Thus, to prove \eqref{E:mu2.sigma1-sigma2.Lip.bound}, it suffices to find $K_{5} < \infty$ s.t.\
	\begin{equation}  \label{E:y-w.sigma.bound}
		| y_{2} - w_{2} ||x_{2} - x_{1}| \leq K |y_{2} - y_{1}|.
	\end{equation}

Let 
	\[
		y_{1}' = \mu_{1} x_{1} + (1 - \mu_{1}) w_{2} \in \rho.
	\]
So, by \eqref{E:y.in.terms.of.sigma.w}, $\sigma(y_{1}') = x_{1}$ (even if $y \in \text{Lk} \, \sigma$; see \eqref{E:w.on.sigma,sigma(.).on.Lk}) and $\mu(y_{1}') = \mu_{1}$.  In particular, $y_{1}' \in U^{\sigma}$ because $y_{1} \in U^{\sigma}$. Now, by \eqref{E:fns.Lip.on.U^sigma}
we have that $w_{i} = w(y_{i})$ is Lipschitz in $y_{i} \in U^{\sigma}$ ($i=1,2$).  
Therefore, for some $K_{6} < \infty$,
	\begin{align}  \label{E:y2-y1'}
		| y_{2} - y_{1}' | &\leq | y_{2} - y_{1} | + \bigl| y_{1} - y_{1}' \bigr| \notag \\
		  &= | y_{2} - y_{1} | + (1 - \mu_{1}) | w_{1} - w_{2} | \\
		  &= | y_{2} - y_{1} | +  K_{6} | y_{2} - y_{1} |. \notag 
	\end{align}
Thus, to prove \eqref{E:y-w.sigma.bound} it suffices to show that there exists $K_{7} < \infty$ s.t.\
	\begin{equation}  \label{E:y.prime-w.sigma.bound}
		| y_{2} - w_{2} ||x_{2} - x_{1}| \leq K_{7} |y_{2} - y_{1}'|.
	\end{equation}

But we can ``trick'' lemma \ref{L:props.of.s}(part \ref{I:y-z.h2-h1.bound}) into doing this for us.  The point $w_{2}$ will play the role of $z$ in lemma \ref{L:props.of.s}(part \ref{I:y-z.h2-h1.bound}).  But $\rho$ cannot play the role of $\sigma$ in lemma \ref{L:props.of.s}(part \ref{I:y-z.h2-h1.bound}) because $z \in  \text{Int} \, \sigma$ (by \eqref{E:z.notin.A}) while $w_{2} \in  \text{Bd} \, \rho$.  So we have to replace $\rho$ by a bigger simplex that contains $w_{2}$ in its interior.  This is accomplished by the following lemma.

  \begin{lemma}  \label{L:expand.simplex}
Let $\chi$ be a simplex, let $\xi$ be a face of $\chi$, and let $\zeta$ be the face of $\chi$ opposite 
$\xi$.  Then one can construct from $\chi$ and $\xi$ a simplex $\chi'$ s.t.\ $\xi$ is a face 
of $\chi'$ and $\zeta \subset \text{Int} \, \chi'$.  In particular, $\chi \subset \chi'$.
  \end{lemma}
  \begin{proof}
Let $\hat{\chi}$ be the barycenter of $\chi$ (\eqref{E:barycenter.of.sigma}) and modify the vertices of $\chi$ as follows.
	\begin{equation}  \label{E:v'.defn}
		v' =
			\begin{cases}
				v, &\text{ if } v \in \xi^{(0)}, \\
				2( v - \hat{\chi}) + \hat{\chi} = 2 v - \hat{\chi}, &\text{ if } v \in \zeta^{(0)}.
			\end{cases}
	\end{equation}
We show that $v' (v \in \chi^{(0)}$) are geometrically independent (appendix \ref{S:basics.of.simp.comps}).  Let $c_{v} \in \RR$ ($v \in \chi^{(0)}$) satisfy
	\begin{equation}  \label{E:cv.constraints}
		\sum_{v \in \chi^{(0)}} c_{v} = 0 \text{ and } \sum_{v \in \chi^{(0)}} c_{v} v' = 0.
	\end{equation}
Let $m$ be the number of vertices in $\chi$ and let $r < m$ be the number of vertices in $\zeta$.  Then, by \eqref{E:v'.defn} and \eqref{E:cv.constraints} 
	\begin{align*}
		0 & = \sum_{\xi^{(0)}} c_{v} v  +  \sum_{\zeta^{(0)}} 2 c_{v} v   
		    - \sum_{v \in \chi^{(0)}} \left( \frac{1}{m} \sum_{u \in \zeta^{(0)}} c_{u} \right) v \\
		    &= \sum_{\xi^{(0)}} \left( c_{v} - \frac{1}{m} \sum_{\zeta^{(0)}} c_{u} \right) v  
		         +  \sum_{\zeta^{(0)}} \left( 2 c_{v} - \frac{1}{m} 
		                \sum_{\zeta^{(0)}} c_{u}  \right) v  \\
		    &= \sum_{\xi^{(0)}} \left( c_{v} + \frac{1}{m} \sum_{\xi^{(0)}} c_{u} \right) v  
		         +  \sum_{\zeta^{(0)}} \left( 2 c_{v} - \frac{1}{m} \sum_{\zeta^{(0)}} c_{u}  \right)
		                      v.
	\end{align*}
Also by \eqref{E:cv.constraints}, when we add up the coefficients in the preceding we get,
	\begin{multline*}
		\sum_{\xi^{(0)}} \left( c_{v} + \frac{1}{m} \sum_{\xi^{(0)}} c_{u} \right)   
		         +  \sum_{\zeta^{(0)}} \left( 2 c_{v} - \frac{1}{m} \sum_{\zeta^{(0)}} c_{u}  \right)  \\
			\begin{aligned}
				{} &= \sum_{\xi^{(0)}} c_{v} + \frac{m-r}{m} \sum_{\xi^{(0)}} c_{v}   
				         +  2 \sum_{\zeta^{(0)}} c_{v} - \frac{r}{m} \sum_{\zeta^{(0)}} c_{v}  \\
				   &= \frac{m-r}{m} \sum_{\xi^{(0)}} c_{v}   
				         +  \left(1 -  \frac{r}{m} \right) \sum_{\zeta^{(0)}} c_{v}  \\
				   &= 0.
			\end{aligned}
	\end{multline*}
Therefore, since the vertices, $v \in \chi^{(0)}$, are geometrically independent, we have, by \eqref{E:cv.constraints},
	\begin{equation}   \label{E:cv.relations}
		c_{v} = -\frac{1}{m}\sum_{\xi^{(0)}} c_{u} 
		          = \frac{1}{m} \sum_{\zeta^{(0)}} c_{u}, \; (v \in \xi^{(0)}) \text{ and }
			c_{v} = \frac{1}{2m} \sum_{\zeta^{(0)}} c_{u}, \; (v \in \zeta^{(0)}).
	\end{equation}
But $\sum_{v \in \chi^{(0)}} c_{v} = 0$, by \eqref{E:cv.constraints}.  Hence, \eqref{E:cv.relations} implies that a non-zero multiple of $\sum_{u \in \zeta^{(0)}} c_{u}$ is 0.  Therefore, by \eqref{E:cv.relations} again, $c_{v} \propto \sum_{u \in \zeta^{(0)}} c_{u} = 0$ for every $v \in \chi^{(0)}$.  This proves the geometric independence of $v'$ ($v \in \chi^{(0)}$).

Let $\chi'$ be the simplex with vertex set $\{ v', v \in \chi^{(0)} \}$.  Then by \eqref{E:v'.defn},
	\begin{equation} \label{E:v.in.trms.of.v'.rho}
		v = \tfrac{1}{2} v' + \tfrac{1}{2} \hat{\chi}, \; \text{ if } v \in \zeta^{(0)}.
	\end{equation}
It easily follows that 
	\[
		\hat{\chi} = \frac{2}{2m-r} \sum_{v \in \xi^{(0)}} v' 
		  + \frac{1}{2m-r} \sum_{v \in \zeta^{(0)}} v'.
	\]
In particular, $\hat{\chi}$ is an interior point of $\chi'$.  Hence, by \eqref{E:v.in.trms.of.v'.rho}, each $v \in \zeta^{(0)}$ is an interior point of $\chi'$.  Hence, $\zeta \subset \text{Int} \, \chi'$.  Since $\hat{\chi} \in \text{Int} \, \chi$, \eqref{E:v'.defn} and \eqref{E:v.in.trms.of.v'.rho} imply that $v \in \chi'$ for every $v \in \chi^{(0)}$.  Finally, every vertex of $\xi$ is a vertex of $\chi'$ by \eqref{E:v'.defn}.  Therefore, $\xi$ is a face of $\chi'$.
  \end{proof}

\emph{Proof of lemma \ref{L:g.is.loc.Lip}, continued:} Apply lemma \ref{L:expand.simplex} with 
$\chi = \rho$ and $\xi = \sigma$, so $\zeta = \omega$. (Recall that $\omega$ is the face 
of $\rho$ opposite $\sigma$.)  Thus, there is a simplex $\rho'$ s.t.\ 
$\omega \subset \text{Int} \, \rho'$, $\rho \subset \rho'$, and $\sigma$ is a face of $\rho'$.   In particular, $w_{2} \in \text{Int} \, \rho'$ and $y_{1}', y_{2} \in \rho'$.  By \eqref{E:b.hbar.defn}, \eqref{E:y.in.terms.of.sigma.w}, 
and lemma \ref{L:props.of.s}(part \ref{I:y-z.h2-h1.bound}),
	\begin{equation*}
		|y_{2} - w_{2}| |x_{2} - x_{1}| 
			= |y_{2} - w_{2}| \bigl| \bar{h}_{w_{2},\rho'}(y_{2}) - \bar{h}_{w_{2},\rho'}(y_{1}') \bigr|
				\leq K''(w_{2}) |y_{2} - y_{1}'|.
	\end{equation*}
Now, $K''(w_{2})$ is continuous in $w_{2}$, by lemma \ref{L:props.of.s}(part \ref{I:y-z.h2-h1.bound}).  Therefore, by \eqref{E:fns.Lip.on.U^sigma}, $K''(w_{2}) = K''\bigl[ w(y_{2}) \bigr]$ is bounded in $y_{2} \in U^{\sigma}$. I.e., \eqref{E:y.prime-w.sigma.bound} holds.

Next, let $\epsilon \in (0,1)$ and consider $U_{\sigma, \epsilon} = U_{\sigma, \epsilon, \rho}$.   
By \eqref{E:U.sigma.eps.disjnt.from.Lk.C.cp.Bd}, \eqref{E:sigma.(.).is.loc.Lip.off.Lk}, and\eqref{E:mu.w.loc.Lip}, and  we have
	\begin{equation}  \label{E:sig.mu.Lip.on.U.sig.eps}
		\sigma(\cdot) \text{ and } \mu \text{ are Lipschitz on } \bar{U}_{\sigma, \epsilon}.
	\end{equation}
Moreover, the map $\sigma(\cdot) \times (1-\mu)$ maps $\bar{U}_{\sigma, \epsilon}$ onto a compact subset of $B_{z}$. (See \eqref{E:Bz.defn}, \eqref{E:sigma.ident.on.sigma}, and \eqref{E:wy.muy.when.y.in.Lk.or.sigma}.)  Hence, 
by lemma \ref{L:props.of.s}(part \ref{I:s.locly.Lip.off.C.cap.Bd}),  
	\begin{equation} \label{E:s.Lip.on.sig.1-mu.image}
		s \text{ is Lipschitz on } \bigl[ \sigma(\cdot) \times (1-\mu) \bigr](\bar{U}_{\sigma, \epsilon}).
	\end{equation}
Let $y_{1}, y_{2} \in U_{\sigma, \epsilon}$. Write $x_{i} = \sigma(y_{i})$, $\mu_{i} = \mu(y_{i})$, $s_{i} = s( x_{i}, 1 - \mu_{i} )$, and $w_{i} = w(y_{i})$ ($i=1,2)$. WLOG $\mu_{2} \geq \mu_{1}$ 

Recall $D := \max \{ |u|, u \in |P| \} < \infty \}$. Then we have, by \eqref{E:general.formula.for.g}, \eqref{E:s.Lip.on.sig.1-mu.image},
and \eqref{E:sig.mu.Lip.on.U.sig.eps}, 
	\begin{align}  \label{E:gy2.gy1.diff.bnd}
		\bigl| g(y_{2}) - g(y_{1})| &\leq |\mu_{2} s_{2} - \mu_{1} s_{1} | 
		   + \bigl| (1-\mu_{2}) w_{2} - (1-\mu_{1}) w_{1}  \bigr|  \notag \\
		   &\leq \bigl| \mu_{2} ( s_{2} - s_{1} ) \bigr| + |\mu_{2} s_{1} - \mu_{1} s_{1} |
		   + (1-\mu_{2}) | w_{2} - w_{1} | + | \mu_{1} - \mu_{2} ||w_{1}|  \notag \\
		   &= \mu_{2} | s_{2} - s_{1} | + |\mu_{2}  - \mu_{1} | |s_{1}|
		   + | \mu_{1} - \mu_{2} ||w_{1}| + (1-\mu_{2}) | w_{2} - w_{1} |  \\
		   &\leq K_{1} \bigl( |x_{2} - x_{1}| + | \mu_{1} - \mu_{2} | \bigr) 
		     + K_{2} D | y_{2} - y_{1} |  \notag \\
		     & \qquad \qquad + K_{2} D | y_{2} - y_{1} | + (1-\mu_{2}) | w_{2} - w_{1} |.  \notag
	\end{align}
By \eqref{E:sig.mu.Lip.on.U.sig.eps} again, 
	\[
		K_{1} \bigl( |x_{2} - x_{1}| + | \mu_{1} - \mu_{2} | \bigr)
			 \leq K_{1} \bigl( K_{3} | y_{2} - y_{1} | + K_{2} | y_{2} - y_{1} | \bigr).
	\]
So, by \eqref{E:gy2.gy1.diff.bnd},  to prove that $g$ is Lipschitz on $\bar{U}_{\sigma, \epsilon}$
it suffices to find $K < \infty$ s.t.\ 
	\begin{equation}  \label{E:1-mu.w2-w1.leq.K.y2-y1}
		(1-\mu_{2}) | w_{2} - w_{1} | \leq K | y_{2} - y_{1} |.
	\end{equation}
By \eqref{E:y.in.terms.of.sigma.w},
	\[
		(1-\mu_{2})(w_{2} - x_{2}) = y_{2} - x_{2}.
	\] 
Recall that $\omega$ is the face of $\rho$ opposite $\sigma$. I.e.
	\begin{equation}  \label{E:mu2.in.terms.of.y2.sig2}
		1-\mu_{2} = \frac{|y_{2} - x_{2}|}{|w_{2} - x_{2}|} 
		   \leq \frac{|y_{2} - x_{2}|}{\min \{ |w - x| : w \in \omega, x \in \sigma \}}.
	 \end{equation}

Let 
	\[
		y_{1}' = \mu_{1} x_{2} + (1-\mu_{1}) w_{1} \in \rho.
	\]
Let $\Delta_{i} = \bar{\Delta}(x_{i})$ ($i=1,2$).  Now, by \eqref{E:y.in.terms.of.sigma.w}, 
$\sigma(y_{1}') = x_{2}$ and $\mu(y_{1}') = \mu_{1} > 1/4$.   
Hence, $\bar{\Delta} \bigl[ \sigma(y_{1}') \bigr] = \Delta_{2}$ and 
$b \bigl[ \sigma(y_{1}') \bigr] = b \bigl[ \sigma(y_{2}) \bigr]$.  

Now, by \eqref{E:wy.muy.when.y.in.Lk.or.sigma}, $y_{2} \in U_{\sigma, \epsilon}$ implies 
$y_{2} \notin \text{Lk} \, \sigma$.  Suppose $y_{2} = z \in \sigma$.  Then, by \eqref{E:sigma.ident.on.sigma}, $\sigma(y_{1}') = \sigma(y_{2}) = z$.  Therefore, by \eqref{E:when.bx.=1.or.0}, $b \bigl[ \sigma(y_{1}') \bigr] = 0$ so, by \eqref{E:ell.defn} 
(whether $y_{1}' = z$, $y_{1}' \in \text{Lk} \, \sigma$, or neither), we have $\ell(y_{1}') = 1$ 
so $y_{1}' \in U_{\sigma, \epsilon}$ if $y_{2} = z$.   

Suppose $y_{2} \ne z$.  Then we have, by definition of $U_{\sigma, \epsilon}$ and the facts that 
$y_{2} \notin \text{Lk} \, \sigma$ and $\mu_{2} \geq \mu_{1}$, 
	\begin{align*}
		\epsilon < \ell(y_{2}) 
		&= (1 - \mu_{2})
				   +  \mu_{2} \left[ \bigl( 1 - b (x_{2}) \Bigr) 
				        + b ( x_{2}) 
				             \frac{\Delta_{2}}{diam(\sigma)}  \right] \\
		&= 1 - \mu_{2} \, b(x_{2}) \left( 1 -  \frac{\Delta_{2}}{diam(\sigma)} \right) \\
		&\leq 1 - \mu_{1} \, b(x_{2}) \left( 1 -  \frac{\Delta_{2}}{diam(\sigma)} \right) \\
		   &= \ell(y_{1}'),
	\end{align*}
since $\Delta_{2}/diam(\sigma) \leq 1$ and, by \eqref{E:wy.muy.when.y.in.Lk.or.sigma}, $\mu_{1} = \mu(y_{1}') = 0$ if $y_{1} \in \text{Lk} \, \sigma$.  (See \eqref{E:Delta.bar.defn} and \eqref{E:Delta.bar.z.=diam.sig}.) Therefore, again $y_{1}' \in U_{\sigma, \epsilon}$.  Now by \eqref{E:y.in.terms.of.sigma.w} and \eqref{E:sig.mu.Lip.on.U.sig.eps},
	\begin{align}
		|y_{2} - y_{1}'| &\leq |y_{2} - y_{1}| + |y_{1} - y_{1}'|  \notag \\
		  &= |y_{2} - y_{1}| +  \mu_{1} | x_{1} - x_{2} | \\
		  & \leq (1 + K_{3}) |y_{2} - y_{1}|. \notag 
	\end{align}
Hence, by \eqref{E:mu2.in.terms.of.y2.sig2}, to prove \eqref{E:1-mu.w2-w1.leq.K.y2-y1} it suffices to find $K_{4} < \infty$ s.t.\ 
	\begin{equation} \label{E:y2-sigma2.w2-w1.leq.K.y2-y1'}
		|y_{2} - x_{2}| | w_{2} - w_{1} | \leq K_{4} | y_{2} - y_{1}' |.
	\end{equation}

To prove \eqref{E:y2-sigma2.w2-w1.leq.K.y2-y1'}, apply lemma \ref{L:expand.simplex} 
with $\chi = \rho$ and $\xi = \omega$, so $\zeta = \sigma$ to enlarge $\rho$ to a simplex $\rho'$ (not in $P$) s.t.\ $\sigma \subset \text{Int} \, \rho'$ while $\omega$ remains a face of $\rho'$.  Then apply \eqref{E:y.in.terms.of.sigma.w}, \eqref{E:b.hbar.defn} (identifying $1-\mu$ and $b$), and lemma \ref{L:props.of.s}(part \ref{I:y-z.h2-h1.bound}) as follows.
	\begin{equation*}
		|y_{2} - x_{2}| | w_{2} - w_{1} | 
			= |y_{2} - x_{2}| 
				   \bigl| \bar{h}_{x_{2}, \rho'}(y_{2}) - \bar{h}_{x_{2}, \rho'}(y_{1}') \bigr|
				\leq K''(x_{2}) | y_{2} - y_{1}' |.
	\end{equation*}
Now, $K''(x_{2})$ is continuous in $x_{2}$, by lemma \ref{L:props.of.s}(part \ref{I:y-z.h2-h1.bound}).  Therefore, by \eqref{E:sig.mu.Lip.on.U.sig.eps}, $K''(x_{2}) = K''\bigl[ \sigma(y_{2}) \bigr]$ is bounded in $y_{2} \in U_{\sigma, \epsilon}$. I.e., \eqref{E:y2-sigma2.w2-w1.leq.K.y2-y1'} holds. This completes the proof that $g$ is Lipschitz on $U_{\sigma, \epsilon}$.  Lemma \ref{L:g.is.loc.Lip} is proved.
   \end{proof}

The following was adapted from \cite[Appendix A]{spE.fact.anal.long}.

   \begin{lemma} \label{L:Eigen.cont.}
      If $M$ is a symmetric $q \times  q$ (real) matrix ($q$, a given positive integer), let 
      $\Lambda (M) =$ \linebreak $(\lambda_{1}(M), \ldots, \lambda _{q}(M))$, where 
      $\lambda _{1}(M) \geq \ldots \geq \lambda _{q}(M)$ are the eigenvalues of $M$. Let also $\| M \|$ be the Frobenius norm (Blum \emph{et al} \cite[p.\ 203]{lBfCmSsS98.realcompute}) $\| M \| = \sqrt{trace \, M M^{T}}$.  Then $\Lambda$ 
      is a continuous function (w.r.t. $\| \cdot \|$).  Moreover, if $N$ and $M_{1}, M_{2},\ldots$ are all symmetric $q \times  q$ (real) matrices s.t.\ 
      $M_{j} \rightarrow N$ (w.r.t.\ $\| \cdot \|$, i.e., entrywise) as 
      $j \rightarrow \infty $, let $Q_{j}^{q \times q}$ be a matrix whose rows comprise an orthonormal
	     basis of $\mathbb{R}^{q}$ consisting of eigenvectors of $M_{j}$.  Then there is a subsequence $j(n)$
      s.t.\ $Q_{j(n)}$ converges to a matrix whose rows comprise a basis of $\mathbb{R}^{q}$ consisting
      of unit eigenvectors of $N$.
   \end{lemma}

 \begin{proof}
  Let $N$ and $M_{1}, M_{2},\ldots $ all be symmetric $q \times  q$ matrices.  
  Suppose $M_{j} \rightarrow N$ as $j \rightarrow \infty$.  
Let   $\mu _{i} = \lambda _{i}(N)$ ($i = 1, \ldots, q$).  Since   $M_{j} \rightarrow N $, $\{ M_{j} \}$ is bounded.  Hence, $\bigl\{ (\lambda _{1}(M_{j}), \ldots, \lambda _{q}(M_{j})) \bigr\}$ is bounded (Marcus and Minc \cite[1.3.1, pp.\ 140--141]{mMhM64.MatrixThyIneq}).  For each $n$, let   $v_{n1}, \ldots, v_{nq} \in \RR^{q}$ be orthonormal eigenvectors of $M_{n}$ corresponding to $\lambda _{1}(M_{n}), \ldots, \lambda _{q}(M_{n})$, resp.  
In particular, $v_{n1}, \ldots, v_{nq}$ span $\RR^{q}$.  Let $S^{q -1}$ be the $(q-1)$-sphere
	\[
		S^{q -1} = \{ x \in \RR^{q} : |x| = 1 \}.
	\]
Then $v_{n1}, \ldots, v_{nq} \in S^{q-1}$. By compactness of $S^{q -1}$ we may choose a subsequence, $\{ j(n) \}$, s.t.\  $v_{j(n)i}$ converges to some $v_{i} \in \mathbb{R}^{q}$ 
and $\lambda _{i}(M_{j(n)})$ converges to some $\nu_{i} \in \mathbb{R}$ as $n \rightarrow \infty$ 
($i = 1, \ldots, q$).  Then $v_{1}, \ldots, v_{q}$ are orthonormal eigenvectors of $N$ with eigenvalues 
$\nu _{1} \geq \ldots \geq \nu _{q}$. 
Hence, $\{ \nu _{1}, \ldots, \nu _{q} \} \subset \{ \mu _{1}, \ldots, \mu _{q} \}$. 
Note that $v_{1}, \ldots, v_{q}$ must span $\RR^{q}$.
  
We show that 
	 \begin{equation}  \label{E:nus.=.mus}
		\nu_{i} = \mu_{i} \quad (i=1, \ldots, q).
	\end{equation}
 This is obvious if $q = 1$.  Suppose \eqref{E:nus.=.mus} always holds for $q = m \geq 1$ and suppose $q = m+1$. Suppose some $\mu_{i} \notin \{ \nu _{1}, \ldots, \nu _{q} \}$ 
 and let $w_{i} \in \RR^{q} \setminus \{ 0 \}$ be an eigenvector of $N$ corresponding to $\mu_{i}$.  Then by Stoll and Wong \cite[Theorem 4.1, p.\ 207]{rrSetW68.LinearAlgebra}, 
 we have that $w_{i} \perp v_{1}, \ldots, v_{q}$.  But $v_{1}, \ldots, v_{q}$ span $\RR^{q}$.  
  Hence, $w_{i} = 0$, contradiction.  Thus, for some $j_{i}$ we have $\mu_{i} = \nu_{j_{i}}$.  
  
Let $T : \RR^{m+1} \to \RR^{m+1}$ be the linear operator associated with $N$.  Let $V \subset \RR^{q}$ be the orthogonal complement of the span of $v_{j_{i}} = w_{i}$ and let $N_{i}$ be a matrix of the restriction $T \vert_{V}$.  Then $V$ is spanned by $v_{j}$ with $j \ne j_{i}$ and $\mu_{i'}$ with $i' \ne i$ are the eigenvalues of $N_{i}$.  Then by the induction hypothesis $j_{i} = i$ 
and $\nu_{i'} = \mu_{i'}$ for $i \ne i'$, etc.  
Thus, $\{ \mu _{1}, \ldots, \mu _{q} \} = \{ \nu _{1}, \ldots, \nu _{q} \}$.  
  
 The preceding argument obviously goes through it we had started with a subsequence 
 of $M_{1}, M_{2}, \ldots$.  Thus, any subsequence has a further subsequence 
 s.t.\ $\Lambda (M_{n})$ converges to $\Lambda (N)$ along that subsequence of a subsequence.  
  It follows that $\Lambda (M_{n}) \to \Lambda (N)$.  I.e., $\Lambda$ is continuous.
  \end{proof}
 
\section{Basics of simplicial complexes}  \label{S:basics.of.simp.comps}
This appendix presents some of the material in Munkres \cite{jrM84}, mostly from pages 2 -- 11, 83, and 371.  (See also Rourke and Sanderson \cite{cpRbjS72.PiecewiseLinearTopol}.)  Let $N$ be a positive integer and let $n \in \{0, \ldots, N \}$. Points $v(0), \ldots, v(n)$ in $\RR^{N}$ are ``geometrically independent'' (or are in ``general position'') if $v(1) - v(0), \ldots, v(n) - v(0)$ are linearly independent.  Equivalently, $v(0), \ldots, v(n)$ are geometrically independent if and only if
	\[
		\sum_{i=1}^{n} t_{i} = 0 \text{ and } \sum_{i=1}^{n} t_{i} v(i) = 0
	\]  
together imply $t_{0} = \cdots t_{n} = 0$. If $v(0), \ldots, v(n) \in \RR^{N}$ are geometrically independent then they are the vertices of the ``simplex''
	\[
		\sigma = \{ \beta_{0} v(0) + \cdots + \beta_{n} v(n) : \beta_{0}, \ldots \beta_{n} \geq 0 
			\text{ and } \beta_{0} + \cdots + \beta_{n} = 1 \}.
	\]
We say that $\sigma$ is ``spanned'' by $v(0), \ldots, v(n)$ and $n$ is the ``dimension'' of $\sigma$. (Sometimes we call $\sigma$ a ``$n$-simplex'' and write $\dim \sigma = n$.) Note that $\sigma$ is convex and compact.  Indeed, it is the convex hull of $\{ v(0), \ldots, v(n) \}$.  Thus, every $y \in \sigma$ can be expressed uniquely (and continuously) in ``barycentric coordinates''
   \[
      y = \sum_{v \text{ is a vertex in } \sigma} \beta_{v}(y) \, v,
   \]
where the $\beta_{v}(y)$'s are nonnegative and sum to 1.

The simplex $\sigma$ lies on the plane
	\begin{equation}  \label{E:formla.for.affine.plane}
		\Pi = \{ \beta_{0} v(0) + \cdots + \beta_{n} v(n) : \beta_{0} + \cdots + \beta_{n} = 1 \}.
	\end{equation}
I.e., the definition of $\Pi$ is like that of $\sigma$ except the non-negativity requirement is dropped.  Note that $\Pi$ need not include the origin of $\RR^{N}$.  $\Pi$ is the smallest plane containing 
$\sigma$. The dimension of $\Pi$ is $n$.  Any subset of $\{ v(0), \ldots, v(n) \}$ is geometrically independent and the simplex spanned by that subset is a ``face'' of $\sigma$.  So $\sigma$ is a face of itself and a vertex of $\sigma$ is also a face of $\sigma$  (and so is ?).  A ``proper'' face of $\sigma$ is a face of $\sigma$ different from $\sigma$. 

Let $\sigma$ be a simplex spanned by geometrically independent points $v(0), \ldots, v(n)$. If $J \subsetneqq \{ 0, \ldots, n \}$ is nonempty, let $\tau$ be  the proper face of $\sigma$ spanned by $\{ v(j), \, j \in J \}$.  The face ``opposite'' $\tau$ is the span, $\omega$, of $\{ v(j), \, j \notin J \}$ (Munkres \cite[p.\ 5 and Exercise 4, p.\ 7]{jrM84}). E.g., $\tau$ might be a vertex of $\sigma$.  Thus,  $\tau$ consists of those $y \in \sigma$ s.t.\ $\beta_{v}(y) = 0$ for all vertices $v \notin \tau$ and $\omega$ consists of those $y \in \sigma$ s.t.\ $\beta_{v}(y) = 0$ for all vertices $v \in \tau$.  

The union of all proper faces of $\sigma$ is the ``boundary'' of $\sigma$, denoted $\text{Bd} \, \sigma$.  The ``(simplicial) interior'' of $\sigma$ (as a simplex) is the set 
$\text{Int} \, \sigma := \sigma \setminus (\text{Bd} \, \sigma)$, where ``$\setminus$'' indicates set-theoretic subtraction.  If 
	\begin{equation}  \label{E:criterion.for.int.simp}
		y = \sum_{j=0}^{n} \beta_{j}(y) \, v(j), 
		      \text{ then } y \in \text{Int} \, \sigma \text{ if and only if $\beta_{j}(y) > 0$ 
		          for all }  j = 0, \ldots, n.
	\end{equation}
Thus, the interior of $\sigma$ as a simplex is in general different from its (usually empty) interior as a subspace of $\RR^{N}$.  In fact, the interior (as a simplex) of a 0-dimensional simplex (a single point) is the point itself.  But $\sigma$ is the topological closure of $\text{Int} \, \sigma$ and $\text{Int} \, \sigma$ is the relative interior of $\sigma$ as a subset of $\Pi$ defined by \eqref{E:formla.for.affine.plane}.

The following lemma (Munkres \cite[lemma 1.1, p.\ 6]{jrM84}) about convex sets is handy.

\begin{lemma}  \label{L:rays.intersect.bndry.in.one.pt}
Let $U$ be a bounded, convex, open set in some affine space (e.g., a Euclidean space).  Let $w \in U$.  Then each ray emanating from $w$ intersects the boundary of $U$ in precisely one point.
\end{lemma}

Let $v(0), \ldots, v(n) \in \RR^{n}$ be the vertices of $\sigma$.  Then
	\begin{equation}  \label{E:barycenter.of.sigma}
      		\hat{\sigma} = \frac{1}{n+1} \sum_{j=0}^{n} v(j)
	\end{equation}
is the ``barycenter'' of $\sigma$ (Munkres \cite[p.\ 85]{jrM84}). Munkres \cite[p.\ 90]{jrM66} defines the 
``radius'', $r(\sigma)$, of $\sigma$ to be the minimum distance from $\hat{\sigma}$ to $\text{Bd} \, \sigma$.  He defines the ``thickness'' of the simplex $\sigma$ to be
$t(\sigma) := r(\sigma)/diam(\sigma)$. Here, ``$diam(\sigma)$'' is the diameter of $\sigma$, i.e., the length of the longest edge of $\sigma$.

A ``simplicial complex'', $P$, in $\RR^{N}$ is a collection of simplices in $\RR^{N}$ s.t.\
	\begin{equation}   \label{E:complex.contains.all.faces}
		\text{Every face of a simplex in $P$ is in $P$.}
	\end{equation}
and		     
	\begin{equation}   \label{E:intersection.of.simps}
		\text{The intersection of any two simplices in $P$ is a face of each of them.}  
	\end{equation}
It turns out that an equivalent definition of simplicial complex is obtained by replacing condition \eqref{E:intersection.of.simps} by the following. 
	\begin{equation}  \tag{\ref{E:intersection.of.simps}'}
		\text{Every pair of distinct simplices in $P$ have disjoint interiors.}
	\end{equation}  
It follows that
	\begin{equation}   \label{E:Int.rho.cuts.sigma.then.rho.in.sigma}
		\text{If } \rho, \sigma \text{ are elements of a simplicial complex and } 
		       (\text{Int} \, \rho) \cap \sigma \ne \varnothing \text{ then } 
			         \rho \text{ is a face of } \sigma.
	\end{equation}
(\emph{Proof:} $(\text{Int} \, \rho) \cap \sigma$ lies in some face of $\sigma$, e.g., in $\sigma$ itself.  Let $\tau$ be the smallest face of $\sigma$ (in terms of inclusion) containing 
$(\text{Int} \, \rho) \cap \sigma$.  Suppose $\rho \ne \tau$.  Then by (\ref{E:intersection.of.simps}') $(\text{Int} \, \rho) \cap \tau$ lies in some proper face of $\tau$.  But $(\text{Int} \, \rho) \cap \tau  = \bigl[ (\text{Int} \, \rho) \cap \sigma \bigr] \cap \tau = (\text{Int} \, \rho) \cap \sigma$, since $(\text{Int} \, \rho) \cap \sigma \subset \tau \subset \sigma$.  I.e., $(\text{Int} \, \rho) \cap \sigma$ lies in a proper face of $\tau$. That contradicts the minimality of $\tau$.  Therefore, $\rho =  \tau$.)

A simplicial complex, $P$|, is ``finite'' if it is finite as a set (of simplices). The ``dimension'' of a simplicial complex is 
	\[
		\dim P = \max \{ \dim \sigma : \sigma \in P \}
	\] 
(Munkres \cite[p.\ 14]{jrM84}).  (So infinite dimensional simplicial complexes are possible.)  In the following assume $P$ is a non-empty simplicial complex.

A subset, $L$, of $P$ is a ``subcomplex'' of $P$ if $L$ is a simplicial complex in its own right.  The collection, $P^{(q)}$, of all simplices in $P$ of dimension at most $q \geq 0$ is a subcomplex, called the ``$q$-skeleton'' of $P$.  In particular, $P^{(0)}$ is the set of all vertices of simplices in $P$.  The ``polytope'' or ``underlying space'' of $P$, denoted by $|P|$, is just the union of the simplices in $P$.  If $P$ is finite, i.e., consists of finitely many simplices, then $|P|$ is assigned the relative topology it inherits from $\RR^{N}$.  In general, a subset, $X$, of $|P|$ is closed (open) if and only if 
$X \cap \sigma$ is closed (resp. open) in $\sigma$ for every $\sigma \in P$.  Call this topology the ``polytope topology'' on $|P|$.  A space that equals $|P|$ for some simplicial complex, $P$, is called a ``polyhedron''.  By Munkres \cite[Lemma 2.5, p.\ 10]{jrM84}, $|P|$ is compact if and only if $P$ is finite.  If $X$ is a topological space, then a ``triangulation'' of $X$ is a simplicial complex, $P$, and a homeomorphism $f : |P| \to X$.

Let $P$ be a finite simplicial complex of positive dimension.  As in Munkres \cite[p.\ 10]{jrM84}, define ``barycentric coordinates'' on $|P|$ as follows.
First, note that
	\begin{equation} \label{E:x.in.exctly.1.simplex.intrr}
		\text{If $x \in |P|$ then there is exactly one simplex 
		    $\tau \in P$ s.t.\ $x \in \text{Int} \, \tau$.}
	\end{equation}
(To see this, note that since $P$ is finite, there is a smallest simplex (w.r.t.\ inclusion order), $\tau$, in $P$ containing $x$.  Clearly, $x \in \text{Int} \, \tau$.  By (\ref{E:intersection.of.simps}') this implies $\tau$ is unique.)  Let $\tau^{(0)}$ be the set of vertices of $\tau$.  
Then, by \eqref{E:criterion.for.int.simp}, there exist strictly positive numbers 
$\beta_{v}(x), \; v \in \tau^{(0)}$ that sum to 1 and satisfy
	\[
		x = \sum_{v \in \tau^{(0)}} \beta_{v}(x) v.
	\]
Since $v \in \tau ^{(0)}$ are geometrically independent, $\beta_{v}(x)$, $v \in \tau ^{(0)}$, are unique.  
If $v \in P^{(0)}$ is not a vertex of $\tau$ define $\beta_{v}(x) = 0$.  Thus, 
	\[
		x = \sum_{v \in P^{(0)}} \beta_{v}(x) v, \quad \quad x \in |P|.
	\]
The entries in $\bigl\{ \beta_{v}(x), \; v \in P^{(0)} \bigr\}$ are the ``barycentric coordinates'' of $x$. For each $v \in P^{(0)}$ the function $\beta_{v}$ is continuous on $|P|$ (Munkres \cite[p.\ 10]{jrM84}).  If $P$ is finite, we have the following.  (See below for the proof.)

	\begin{lemma}  \label{L:bary.coords.are.Lip}
Let $P$ be a finite simplicial complex.  Then the vector-valued function 
$\boldsymbol{\beta} : x \mapsto \bigl\{ \beta_{v}(x), \; v \in P^{(0)} \bigr\}$ is Lipschitz in $x \in |P|$ (w.r.t.\ the obvious Euclidean metrics; see appendix \ref{S:Lip.Haus.meas.dim}).
	\end{lemma}
In the course of proving this lemma, the following useful fact emerges.

\begin{corly}   \label{C:reverse.triangle.ineq.in.simp.cmplxs}
Let $P$ be a finite simplicial complex.  There exists $K < \infty$ s.t.\ the following holds.  Let $\rho, \tau \in P$ satisfy $\rho \cap \tau \ne \varnothing$, but suppose neither simplex is a subset of the other.  If $x \in \text{Int} \, \rho$ and $y \in \text{Int} \, \tau$ then there exist 
$\tilde{x}, \tilde{y} \in  \text{Int} \, (\rho \cap \tau)$ s.t.\ 
	\[
		|x - \tilde{x}| + |\tilde{x} - \tilde{y}| + |\tilde{y} - y| \leq K |x-y|.
	\]
\end{corly}

Let $\sigma \in P$ and let
   \[
      \overline{\text{St}} \, \sigma 
         = \bigcup_{\sigma \subset \omega \in P} \omega.
   \]
$\overline{\text{St}} \, \sigma$ is the ``closed star'' of $\sigma$ (Munkres \cite[p.\ 371]{jrM84}).  By \eqref{E:intersection.of.simps} $\overline{\text{St}} \, \sigma$ is the union of all simplices in $P$ having 
$\sigma$ as a face. In particular, $\sigma \subset \overline{\text{St}} \, \sigma$.  
Let $\text{Lk} \, \sigma$ be the union of all simplices in $\overline{\text{St}} \, \sigma$ that do \emph{not} intersect $\sigma$. $\text{Lk} \, \sigma$ is the ``link'' of $\sigma$. 
The simplices in $\text{Lk} \, \sigma$ will be faces of simplices in $\overline{\text{St}} \, \sigma$ that also have $\sigma$ as a face. We may have $\overline{\text{St}} \, \sigma = \sigma$, which implies 
$\text{Lk} \, \sigma = \varnothing$.  This can happen, e.g., if $\dim \sigma = \dim P$. If $\rho \in P$, 
$\sigma$ is a proper face of $\rho$, and $\omega$ is the face of $\rho$ opposite $\sigma$, 
then $\omega \subset \text{Lk} \, \sigma$.  Thus, $\overline{\text{St}} \, \sigma = \sigma$ 
if and only if $\text{Lk} \, \sigma = \varnothing$.

Let the ``star'', $\text{St} \, \sigma$ of $\sigma$ be the union of the interiors of all simplices of $P$ having $\sigma$ as a face (Munkres \cite[p.\ 371]{jrM84}). (If $ \overline{\text{St}} \, \sigma = \sigma$, then $\text{St} \, \sigma = \text{Int} \, \sigma$.) We have
	\begin{multline}  \label{E:St.sigma.is.open}
	        \text{St} \, \sigma 
	             = \bigl\{ y \in |P| : \beta_{v}(y) > 0 \text{ for every } v \in \sigma^{(0)} \bigr\}
	               \text{ so  } \text{St} \, \sigma \text{ is open in } |P|. \\
	                \text{Moreover, } \text{Int} \, \sigma \subset \text{St} \, \sigma, 
	                       (\text{St} \, \sigma) \cap (\text{Lk} \, \sigma) = \varnothing, \text{ and } 
	                          (\text{St} \, \sigma) \cap (\text{Bd} \, \sigma) = \varnothing.
	\end{multline}
(\emph{Proof:}  $\rho \in P$ has $\sigma$ as a face if and only if $\sigma^{(0)} \subset \rho^{(0)}$.  But $x = \sum_{v \in \rho^{(0)}} \beta_{v}(x) \in \text{Int} \, \rho$ if and only if $\beta_{v}(x) > 0$ for every 
$v \in \rho^{(0)}$.  Hence, if $x \in \text{Int} \, \rho$ and $\rho$ has $\sigma$ as a face then 
$\beta_{v}(x) > 0$ for every $v \in \sigma^{(0)}$.  Conversely, suppose $x \in |P|$ and 
$\beta_{v}(x) > 0$ for every $v \in \sigma^{(0)}$.  Then obviously, if $\rho$ is the simplex in $P$ with 
$x \in \text{Int} \, \rho$, we have $\sigma^{(0)} \subset \rho^{(0)}$ so $\rho \in P$ has $\sigma$ as a face.  Thus, $x \in \text{Int} \, \rho \subset \text{St} \, \sigma$.  
In particular, $\text{Int} \, \sigma \subset \text{St} \, \sigma$.  Since 
$\beta_{v}$ ($v \in \sigma^{(0)}$) are continuous, continuous, if follows that $\text{St} \, \sigma$ is open.  Moreover, if $x \in (\text{Lk} \, \sigma) \cup (\text{Bd} \, \sigma)$ then $\beta_{v}(x) = 0$ 
for some $v \in \sigma^{(0)}$.  Hence, neither $\text{Lk} \, \sigma$ nor $\text{Bd} \, \sigma$ intersects 
$\text{St} \, \sigma$.)

Let $\sigma \in P$.  Observe that, true to their names, both $\overline{\text{St}} \, \sigma$ and 
$\text{St} \, \sigma$ are ``starlike'' w.r.t.\ any $x \in \text{Int} \, \sigma$.  
I.e., if $y \in \overline{\text{St}} \, \sigma$ then the line segment joining $x$ and $y$ lies entirely 
in $\overline{\text{St}} \, \sigma$.  The same goes for $y \in \text{St} \, \sigma$.  \emph{Claim:} $|P|$ is locally arcwise connected (Massey \cite[p.\ 56]{wsM67.Masseey}).  To see this, let $x \in |P|$ and let $\sigma$ be the unique simplex in $P$ s.t.\ $x \in \text{Int} \, \sigma$.  (See \eqref{E:x.in.exctly.1.simplex.intrr}.)  $\text{St} \, \sigma$ is an open neighborhood of $x$.  Let $r > 0$ be so small that the open ball $B_{r}(x)$, of radius $r$ centered at $x$ satisfies $B_{r}(x) \cap |P| \subset \text{St} \, \sigma$.  If $y, z \in B_{r}(x)  \cap |P|$, then the line segments joining $y$ to $x$ and $x$ to $z$ also lie in $B_{r}(x)  \cap |P|$.  I.e., $B_{r}(x)  \cap |P|$ is path connected.  This proves the claim. Thus, if $|P|$ is connected it is also arcwise connected.

The following result will be helpful in proving corollary \ref{C:drop.finiteness.of.P}.
	  \begin{lemma}  \label{L:local.finiteness.and.compactness}
 Let $P$ be a simplicial complex lying in a finite dimensional Euclidean space, $\RR^{N}$.  Suppose every $x \in |P|$ has a neighborhood, open in $\RR^{N}$, intersecting only finitely many simplices in $P$.  
 Then the following hold.
		\begin{enumerate}
		\renewcommand{\theenumi}{\roman{enumi}}
		\item $P$ is ``locally finite'':  Each $v \in P^{(0)}$ belongs to only finitely many simplices 
		     in $P$.  \label{I:P.locally.finite}
		\item $|P|$ is locally compact.  \label{I:local.compactness}
		\item $|P|$ is a subspace of $\RR^{N}$.  I.e., the polytope topology of $|P|$ coincides 
		      with the topology that $|P|$ inherits from $\RR^{N}$.  \label{I:|P|.is.subspace}
		\end{enumerate}
	  \end{lemma}
  \begin{proof}
Suppose $|P| \subset \RR^{N}$ and every $x \in |P|$ has a neighborhood intersecting only finitely many simplices in $P$. Let $v \in P^{(0)}$.  Then $v$ has a neighborhood $U$ that intersects only finitely many simplices in $P$. If $\sigma \in P$ and $v \in \sigma^{(0)}$, then $v \in \sigma \cap U$. I.e., $U$ intersects 
$\sigma$.  Therefore, $v$ is a vertex of only finitely many simplices in $P$. This proves (\ref{I:P.locally.finite}).  
(See Munkres \cite[p.\ 11]{jrM84}.)

By item (\ref{I:P.locally.finite}) and Munkres \cite[Lemma 2.6, p.\ 11]{jrM84} we have that $|P|$ is locally compact.  And by Munkres \cite[Exercise 9, p.\ 14]{jrM84}, the space $|P|$ is a subspace of $\RR^{N}$.
  \end{proof}  

A simplicial complex $P'$ in $\RR^{N}$ is a ``subdivision'' of $P$ (Munkres \cite[p.\ 83]{jrM84}) if:
	\begin{enumerate}
		\item Each simplex in $P'$ is contained in a simplex of $P$.
		\item Each simplex in $P$ equals the union of finitely many simplices in $P'$.
	\end{enumerate}
In particular, a subdivision of a finite complex is finite.
	
\begin{proof}[Proof of lemma \ref{L:bary.coords.are.Lip}]
Let $x, y \in |P|$.    Since $P$ is a finite complex there exists $\delta_{1} > 0$ s.t.\ if $\rho, \tau \in P$ are disjoint then $dist(\rho, \tau) > 2 \delta_{1}$. 
Let $\rho$ ($\tau$) be the unique simplex in $P$ s.t.\ $x \in \text{Int} \, \rho$ (respectively [resp.], $y \in \text{Int} \, \tau$; see \eqref{E:x.in.exctly.1.simplex.intrr}).  Therefore, if $\rho$ and $\tau$ are disjoint then the Euclidean length $|x-y|$ is bounded below by $2 \delta_{1}$.  Moreover, $\bigl| \boldsymbol{\beta}(z) \bigr| \leq 1$ for every $z \in |P|$ since the components of $\boldsymbol{\beta}(x)$ are nonnegative and sum to 1.  Thus,
	\begin{equation}  \label{E:x.y.more.than.delta1.apart}
		\bigl| \boldsymbol{\beta}(x) - \boldsymbol{\beta}(y) \bigr| \leq (1/\delta_{1}) |x - y| 
		    \text{ if $x$ and $y$ lie in disjoint simplicies.}
	\end{equation}

So assume $\rho \cap \tau \ne \varnothing$.  In fact, first consider the behavior of $\beta$ on a single simplex, $\rho$ in $P$.  (This covers the case where $\tau\subset \rho$ or \emph{vice versa}.) Suppose $\rho$ is an $n$-simplex, so $\rho$ has $n+1$ vertices $v(0), \ldots, v(n)$.  If $n=0$, i.e., $\rho$ is a single point, then $\beta$ is trivially Lipschitz on $\rho$.  So suppose $n > 0$. We show that $\beta$ is Lipschitz on $\rho$.   
We can assume $|P| \subset \RR^{N}$ for some $N \geq n$.  
Let $V^{(n+1) \times N}$ be the matrix whose $i^{th}$ row is $v(i-1)$ ($i = 1, \ldots, n+1$).  (Use superscripts to indicate matrix dimension.)  Let $V_{0}^{n \times N}$ be the matrix whose $i^{th}$ row is $v(i) - v(0)$ ($i = 1, \ldots, n$).  Let 
$1_{n}^{n \times 1}$ be the column vector $(1, \ldots, 1)^{T}$.  Thus,
	\begin{equation}  \label{E:V0.from.V}
		(-1_{n} \; \;  I_{n} ) V = V_{0},
	\end{equation}	
where $I_{n}$ is the $n \times n$ identity matrix.

The vertices of $\rho$ are geometrically independent so $V_{0}$ has full rank $n$.  This means $V_{0} V_{0}^{T}$ is invertible.  But by \eqref{E:V0.from.V} $(-1_{n} \; \;  I_{n} ) V V_{0}^{T} = V_{0} V_{0}^{T}$.  Therefore, $W^{(n+1) \times n} := V V_{0}^{T}$ has rank $n$.  This implies 
that the vector $1_{n+1}^{(n+1) \times 1} = (1, \ldots, 1)^{T}$ is not in the column space 
of $W^{(n+1) \times n}$.  For suppose for some column vector $\alpha$ we have $W \alpha = 1_{n+1}$.  Then $\alpha \ne 0$ and from \eqref{E:V0.from.V} and the fact that $V_{0} V_{0}^{T}$ is nonsingular we have
	\[
		0 \ne V_{0} V_{0}^{T} \alpha = (-1_{n} \; I_{n} ) W \alpha = (-1_{n} \; I_{n} ) 1_{n+1} = 0.
	\]
Therefore, $(W, \; 1_{n+1})$ is invertible.

For $x \in \rho$, let $\bigl( \boldsymbol{\beta}^{\rho}(x) \bigr)^{1 \times (n+1)}$ be the row vector 
$\bigl( \beta_{v(0)}(x), \ldots, \beta_{v(n)}(x) \bigr)$.  Think of $x \in \RR^{N}$ as a row vector.  Then we have $x = \boldsymbol{\beta}^{\rho}(x) V$ and 
$1 = \boldsymbol{\beta}^{\rho}(x) 1_{n+1}$.  Therefore,
	\[
		(x V_{0}^{T}, \; 1) = \boldsymbol{\beta}^{\rho}(x) (W, \; 1_{n+1})^{(n+1) \times (n+1)}.
	\]
But we have just observed that $U^{(n+1) \times (n+1)} := (W, \; 1_{n+1})$ is invertible.  Therefore,
	\[
		\boldsymbol{\beta}^{\rho}(x) = (x V_{0}^{T}, \; 1) U^{-1}.
	\]
Hence, $\boldsymbol{\beta}^{\rho}$ is affine on $\rho$.  Therefore, $\boldsymbol{\beta}^{\rho}$ and, hence, $\boldsymbol{\beta}$ is Lipschitz on $\rho$.  Since $P$ is a finite complex there is $K < \infty$ that works as a Lipschitz constant for every simplex in $P$.  I.e.,
	\begin{equation}  \label{E:beta.unif.Lip.on.each.simplex}
		\bigl| \boldsymbol{\beta}(x) - \boldsymbol{\beta}(x')  \bigr| \leq K |x - x'| 
		\text{ for every } x, x' \in \rho \text{ for every } \rho \in P.
	\end{equation}

It remains to tackle the case 
	\begin{multline}  \label{E:rho.intersect.but.aren't.subsets}
		x \in \text{Int} \, \rho \text{ and } y \in \text{Int} \, \tau ; \; \rho, \tau \in P; \;  \\
		     \rho \cap \tau \ne \varnothing \text{ but } \rho \text{ is not a subset of } 
		       \tau \text{ and } \tau \text{ is not a subset of } \rho.
	\end{multline}
$\rho \cap \tau \ne \varnothing$ but $\rho$ is not a subset of $\tau$ and $\tau$ is not a subset of $\rho$.  In this case, by \eqref{E:Int.rho.cuts.sigma.then.rho.in.sigma},  
$(\text{Int} \, \rho) \cap (\text{Int} \, \tau) = \varnothing$. We handle this case by reducing it to the last case.  By \eqref{E:intersection.of.simps}, 
$\rho \cap \tau$ is a simplex, a proper face of both $\rho$ and $\tau$. Let $\xi$ be the face of $\rho$ opposite $\rho \cap \tau$ 
and let $\omega$ be the face of $\tau$ opposite $\rho \cap \tau$.  Let $x \in \text{Int} \, \rho$ and $y \in \text{Int} \, \tau$.  

\emph{Claim:} There is a unique $z_{0} = z_{0}(x) \in \xi$ s.t.\ the line passing through $x$ and $z_{0}$ intersects $ \text{Int} \, (\rho \cap \tau)$.  Given $z \in \xi$, the line, $L(z) = L(z,x)$, passing through $z$ and $x$ is unique since $x \in \text{Int} \, \rho$ implies $x \notin \xi$.  Let $v(0), \ldots, v(n)$ be the vertices of $\rho$ and, renumbering if necessary, we may assume $v(0), \ldots, v(m)$ are the vertices of $\rho \cap \tau$ for some $m = 0, \ldots, n-1$.  Then $v(m+1), \ldots, v(n)$ are the vertices of $\xi$.  Let $z \in \xi$ and write
	\[
		z = \sum_{i=m+1}^{n} \mu_{i} v(i),
	\]
where the $\mu_{i}$'s are nonnegative and sum to 1.

First, we prove there is at most one $z \in \xi$ s.t.\ $L(z) \cap \rho \cap \tau \ne \varnothing$. Suppose 
$L(z)$ intersects $\rho \cap \tau$ at $\tilde{x} = \sum_{i=0}^{m} \mu_{i} v(i)$.  Then for some $t \in \RR$ with $t \ne 1$ we have
	\begin{multline}  \label{E:xtilde.is.convx.combo.of.x.z}
		\tilde{x} = \sum_{i=0}^{m} \mu_{i} v(i) 
			= t  \sum_{i=0}^{n} \beta_{v(i)}(x) v(i) + (1-t) \sum_{i=m+1}^{n} \mu_{i} v(i)  \\
			= \sum_{i=0}^{m} t \, \beta_{v(i)}(x) v(i) 
			+ \sum_{i=m+1}^{n} \bigl[ t \, \beta_{v(i)}(x) - (t-1) \mu_{i} \bigr] v(i).
	\end{multline}
Then by geometric independence of $v(0), \ldots, v(n)$ we have
	\begin{equation}  \label{E:mu.t.beta}
		\mu_{i} = t  \beta_{v(i)}(x), \quad i = 0, \ldots, m  \quad \text{ and } \quad
			\mu_{i} = \frac{t}{t-1} \beta_{v(i)}(x), \quad i = m+1, \ldots, n.
	\end{equation}
Let $b = \sum_{i=0}^{m} \beta_{v(i)}(x)$.  Since $x \in \text{Int} \, \rho$, we have $b \in (0,1)$.  From \eqref{E:mu.t.beta} and the fact that $\sum_{i=0}^{m} \mu_{i} = 1$ we see $t = 1/b > 1$.  In particular, $z$ and $\tilde{x}$ are unique if they exist.  If it exists, denote that $z$ by $z_{0}$.

Next, we prove existence of $z_{0}$.  Let $t = 1/b$.  Then it is easy to see that if $\mu_{0}, \ldots, \mu_{n}$ are defined by \eqref{E:mu.t.beta} then 
	\[
	\sum_{i=0}^{m} \mu_{i} = 1 = \sum_{i=m+1}^{n} \mu_{i}.
	\]
Hence, $z_{0} := \sum_{i=m+1}^{n} \mu_{i} v(i) \in \xi$ and 
$\tilde{x} := \sum_{i=0}^{m} \mu_{i} v(i) \in \rho \cap \tau$ and \eqref{E:xtilde.is.convx.combo.of.x.z} holds.  Since $x \in  \text{Int} \, \rho$, we have $\beta_{v(i)}(x) > 0$ for $i = 1, \ldots, n$.  Therefore, $\mu_{i} > 0$ for $i = 1, \ldots, m$.  Thus, $\tilde{x} \in \text{Int} \, (\rho \cap \tau)$.  I.e., $z_{0} \in \xi$, $x$, and $\tilde{x} \in \text{Int} \, (\rho \cap \tau)$ lie on the same line. This proves the claim. Define $\tilde{y} \in \rho \cap \tau$ similarly.  It has similar  properties.

The idea behind the rest of the proof is to first show that 
	\begin{equation}  \label{E:x.y.tilde.bckwrds.tringle.ineq}
		| x-\tilde{x} |+ |\tilde{x} - \tilde{y}| + | \tilde{y} - y | \leq K' \, |x - y|, 
	\end{equation}
where $K' = K'(\rho, \tau) < \infty$ depends only on $\rho$ and $\tau$, not on $x$ or $y$.  Notice that $x$ and $\tilde{x}$ lie in the same simplex in $P$, \emph{viz.} $\rho$. Similarly, $\tilde{x}$ and $\tilde{y}$ both lie in $\rho \cap \tau \in P$.  The points $\tilde{y}$ and $y$ also lie in the same simplex in $P$. So we may apply \eqref{E:beta.unif.Lip.on.each.simplex} to each term 
in $| x-\tilde{x} |+ |\tilde{x} - \tilde{y}| + | \tilde{y} - y |$ and then maximize $K'(\rho, \tau)$ over appropriate pairs $\rho, \tau \in P$.

The simplex $\rho \cap \tau$ lies on a unique plane, $\Pi_{\rho \cap \tau}$,  of minimum dimension.  (See \eqref{E:formla.for.affine.plane}.)  ($\Pi_{\rho \cap \tau}$ might not pass through the origin.) So, e.g., if $\rho \cap \tau$ is a single point $v$ (i.e., $\rho \cap \tau$ 0-dimensional) 
then $\Pi_{\rho \cap \tau} = \{ v \}$.  Now, $x \in \text{Int} \, \rho$ so $x \notin \Pi_{\rho \cap \tau}$.
Let $\hat{x} \in \Pi_{\rho \cap \tau}$ be the orthogonal projection of $x$ onto $\Pi_{\rho \cap \tau}$, i.e., 
$\hat{x}$ is the closest point of $\Pi_{\rho \cap \tau}$ to $x$. Note that $\hat{x}$ may not lie 
in $\rho \cap \tau$.  Define $\hat{y}$ similarly. Let $x_{0}$ be an arbitrary point 
in $\text{Int} \, (\rho \cap \tau)$. E.g., $x_{0}$ might be the barycenter of $\rho \cap \tau$.  (See \eqref{E:barycenter.of.sigma}.)  In any case, $x_{0}$ need only depend on $\rho \cap \tau$, not on $x$ or $y$.Let 
	\begin{equation}  \label{E:y0.=.x0}
		y_{0} := x_{0}.
	\end{equation}   
Then by \eqref{E:Int.rho.cuts.sigma.then.rho.in.sigma}, there exists $r > 0$ s.t.\ the distance 
from $x_{0} = y_{0}$ to any face of $\rho$ or $\tau$ that does not itself have $\rho \cap \tau$ as a face is at least $2r$.  We may assume $r$ only depends on $\rho \cap \tau$, not on $x$ or $y$.

\emph{Claim:} 
	\begin{equation}  \label{E:xdot.in.rho.ydot.in.tau}
		\dot{x} := x_{0} + |x - \hat{x}|^{-1} r (x - \hat{x}) \in \rho \text{ and }
			\dot{y} := y_{0} + |y - \hat{y}|^{-1} r (y - \hat{y}) \in \tau.  
	\end{equation}
First, note that  
	\begin{equation}  \label{E:x0.+.t.x.-.xhat.in.rho}
		\text{for $t > 0$ sufficiently small, } x_{0} + t (x - \hat{x}) \in \text{Int} \, \rho.
	\end{equation}
To see this, observe that by \eqref{E:formla.for.affine.plane} we can write
	\[
		\hat{x} = \sum_{i=0}^{m} \zeta_{i} v(i),
	\]
where $v(0), \ldots, v(m)$ are the vertices of $\rho \cap \tau$; $\zeta_{0}, \ldots, \zeta_{m} \in \RR$; and $\zeta_{0} + \cdots + \zeta_{m} = 1$.  (But the $\zeta_{i}$'s do not have to be nonnegative.)  Moreover, since $x_{0}$ is an interior point of $\rho \cap \tau$ we have
	\[
		\beta_{v(i)}(x_{0}) > 0, \text{ for } i = 0, \ldots, m, 
			\text{ but } \beta_{v(i)}(x_{0}) = 0 \text{ for } i = m+1, \ldots, n.
	\]
Let $t > 0$.  Then
	\begin{equation}  \label{E:x0.+.t.x.-.xhat.in.terms.of.vs}
		x_{0} + t (x - \hat{x}) 
			= \sum_{i=0}^{m} \bigl( \beta_{v(i)}(x_{0}) - t \zeta_{i} + t \beta_{v(i)}(x) \bigr) v(i) 
			+ t \sum_{i=m+1}^{n} \beta_{v(i)}(x) v(i).
	\end{equation}
Since $\beta_{v(i)}(x_{0}) > 0$ for $i = 0, \ldots, m$, for $t > 0$ sufficiently small $\beta_{v(i)}(x_{0}) - t \zeta_{i} > 0$ for $i = 0, \ldots, m$.  So certainly $\beta_{v(i)}(x_{0}) - t \zeta_{i} + t \beta_{v(i)}(x) > 0$ for $i = 0, \ldots, m$.  I.e., the coefficients in \eqref{E:x0.+.t.x.-.xhat.in.terms.of.vs} are all strictly positive.  Finally, the sum of the coefficients satisfies
	\begin{align*}
		\sum_{i=0}^{m} \bigl( \beta_{v(i)}(x_{0}) - t \zeta_{i} + t \beta_{v(i)}(x) \bigr) 
				+ t \sum_{i=m+1}^{n} \beta_{v(i)}(x)
			&= \sum_{i=0}^{m} \beta_{v(i)}(x_{0}) - t \sum_{i=0}^{m} \zeta_{i} 
				+ t \sum_{i=0}^{n} \beta_{v(i)}(x) \\
			&= 1 - t + t \\
			&= 1.
	\end{align*}
That completes the proof of \eqref{E:x0.+.t.x.-.xhat.in.rho}.

Now suppose $\dot{x}$ defined by \eqref{E:xdot.in.rho.ydot.in.tau} does \emph{not} lie in $\rho$.  Let $\Pi_{\rho}$ be the smallest plane in $\RR^{N}$ containing $\rho$.  So $\Pi_{\rho \cap \tau} \subset \Pi_{\rho}$.  By \eqref{E:formla.for.affine.plane}, we have
	\[
		\Pi_{\rho} = \left\{ \sum_{i=0}^{n} \gamma_{i} v(i) : \sum_{i=0}^{n} \gamma_{i} = 1 \right\}
		         = \left\{ v(0) + \sum_{i=1}^{n} \gamma_{i} \bigl( v(i) - v(0) \bigr) : 
		            \gamma_{1}, \ldots, \gamma_{n} \in \RR \right\},
	\]
where $v(0), \ldots, v(n)$ are the vertices of $\rho$.  Since $v(1) - v(0), \ldots, v(n) - v(0)$ are linearly independent, the map that takes a point 
$\sum_{i=0}^{n} \gamma_{i} v(i) \in \Pi_{\rho}$ to the vector $\gamma_{0}, \ldots, \gamma_{n}$ is well-defined and continuous.  Now $x_{0} \in \rho \cap \tau \subset \Pi_{\rho}$, $x \in \rho \subset \Pi_{\rho}$, and $\hat{x} \in \Pi_{\rho \cap \tau} \subset \Pi_{\rho}$. Moreover, the coefficients of $x_{0}$, $x$, and $\hat{x}$ in the expression for $\dot{x}$ in \eqref{E:xdot.in.rho.ydot.in.tau}, \emph{viz.}, 1, $r/|x - \hat{x}|$, and $-r/|x-\hat{x}|$ sum to 1.  It follows that  $\dot{x} \in \Pi_{\rho}$.  
Hence, we can write $\dot{x} = \sum_{i=0}^{n} \zeta_{i} v(i)$ with $\zeta_{0} + \cdots + \zeta_{n} = 1$.  

Let $S$ be the line segment joining $x_{0}$ and $\dot{x}$.  I.e., 
	\begin{equation}  \label{E:line.segment.S.defn}
		S = \bigl\{ x_{0} + t(x - \hat{x}) : 0 \leq t \leq r/|x-\hat{x}| \bigr\}.
	\end{equation}
By \eqref{E:x0.+.t.x.-.xhat.in.rho} for some $t \in (0, r/|x-\hat{x}|)$ we have
	\begin{equation}   \label{E:x.prime.defn}
		x' := x_{0} + t(x-\hat{x}) \in (\text{Int} \, \rho) \cap S.
	\end{equation}
Since $x' \in \text{Int} \, \rho$, the coefficients in the representation of $x'$ as a linear combination of 
\linebreak $v(0), \ldots, v(n)$ must all be strictly positive.  
Since by assumption $\dot{x} \notin \rho$, one or more of the coefficients, $\zeta_{0}, \ldots, \zeta_{n}$, of $v(0), \ldots, v(n)$ for $\dot{x}$ must be strictly negative.  Therefore, somewhere between $x'$ and $\dot{x}$ the segment $S$ must cross the boundary $\text{Bd} \, \rho$.  
Let $w \in \text{Bd} \, \rho$ be the point of intersection. Thus, for some $s \in (t, r/|x-\hat{x}|)$ we have 
	\begin{equation}  \label{E:w.in.terms.of.x0.x.xhat}
		w = x_{0} + s (x - \hat{x}).
	\end{equation}
	
Let $\omega$ be the, necessarily proper, face of $\rho$ s.t.\ $w \in \text{Int} \, \omega$.  (See \eqref{E:x.in.exctly.1.simplex.intrr}.) Now, $\rho \cap \tau$ cannot be a face of $\omega$.  For suppose  $\rho \cap \tau \subset \omega$.  Note that $w \ne x_{0}$, because otherwise $s (x - \hat{x}) = 0$ in \eqref{E:w.in.terms.of.x0.x.xhat}, an impossibility since $x \ne \hat{x}$ and $s > 0$.  Hence, under the assumption that $\rho \cap \tau \subset \omega$ the segment $S$ contains two distinct points of $\omega$, \emph{viz.}, $x_{0} \in \rho \cap \tau$ and $w$.  As a proper face of $\rho$, the simplex 
$\omega$ is defined by the vanishing of some set of barycentric coordinates.  Thus, there exists a nonempty proper subset 
$J$ of $\{0, \ldots, n  \}$ s.t.\
	\[
		\omega = \left\{ \sum_{j = 0}^{n} \beta_{j} v(j) : \beta_{j} \geq 0 \; (j = 0, \ldots, n), 
			\beta_{j} = 0 \text{ if } j \in J, \text{ and } \sum_{j=0}^{n} \beta_{j} = 1 \right\}.
	\]
Since $x, \hat{x} \in \Pi_{\rho}$, for some $\gamma_{0}, \ldots, \gamma_{n} \in \RR$ we have
	\[
		x - \hat{x} =  \sum_{j=0}^{n} \gamma_{j} v(j), 
			\text{ where } \sum_{j=0}^{n} \gamma_{j} = 0.
	\]
Under the hypothesis that $\rho \cap \tau \subset \omega$, we have $w, x_{0} \in \omega$.  In particular, we have $\beta_{v(j)}(x_{0}) = 0$ for $j \in J$.  It follows from \eqref{E:w.in.terms.of.x0.x.xhat} that $\gamma_{j} = 0$ if $j \in J$. Hence, by \eqref{E:line.segment.S.defn} for every $x'' \in S \subset \Pi_{\rho}$ we can write (uniquely)
	\[
		x'' = \sum_{j \in J^{c}} \alpha_{j} v(j), \text{ where } \sum_{j \in J^{c}} \alpha_{j} = 1.
	\]
(Here, $J^{c} = \{ j = 0, \ldots, n : j \notin J \}$.)  In particular, $S \cap (\text{Int} \, \rho) = \varnothing$.  
But by \eqref{E:x.prime.defn}, $x' \in S \cap (\text{Int} \, \rho)$. Contradiction.  This proves $\rho \cap \tau$ cannot be a face of $\omega$.  

Since $\rho \cap \tau$ is not a face of $\omega$, by choice of $r > 0$ the distance 
from $x_{0}$ to $\omega$ is at least $2r$.  Since $\omega$ lies between $x_{0}$ and $\dot{x}$ 
along $S$ we have by \eqref{E:xdot.in.rho.ydot.in.tau} 
	\[
		r = | \dot{x} - x_{0} | \geq 2r > 0.
	\]
This contradiction proves the claim \eqref{E:xdot.in.rho.ydot.in.tau}.

\emph{Claim:  The angle between $x-\hat{x}$ and $y-\hat{y}$ is bounded away from 0.}  I.e., there exists $\gamma \in (0,1)$ independent of $x \in \text{Int} \, \rho$ and $y \in \text{Int} \, \tau$ (i.e., $\gamma$ only depends on $\rho$ and $\tau$) s.t.\
	\begin{equation}  \label{E:angle.tween.x-xhat.y-yhat.ain't.0}
		(x-\hat{x}) \cdot (y-\hat{y}) \leq \gamma |x-\hat{x}||y-\hat{y}|,
	\end{equation}
where, as usual, ``$\cdot$'' indicates the usual Euclidean inner product.
Suppose \eqref{E:angle.tween.x-xhat.y-yhat.ain't.0} is false.  Then there exist sequences 
$\{ x_{n} \} \subset \text{Int} \, \rho$, $\{ y_{n} \} \subset \text{Int} \, \tau$ s.t.\ 
	\[
		\frac{(x_{n}-\hat{x}_{n}) \cdot (y_{n}-\hat{y}_{n})}{|x_{n}-\hat{x}_{n}||y_{n}-\hat{y}_{n}|} 
		           \to 1,
	\]
where $\hat{x}_{n}$ ($\hat{y}_{n}$) is the orthogonal projection of $x_{n}$ (resp. $y_{n}$) onto $\Pi_{\rho \cap \tau}$. Define $\dot{x}_{n}$ as in \eqref{E:xdot.in.rho.ydot.in.tau} with $x$ and $\hat{x}$ replaced by $x_{n}$ and $\hat{x}_{n}$, resp. Define $\dot{y}_{n}$ similarly.   By definition of $\dot{x}_{n}$ and $\hat{x}_{n}$ the vector $\dot{x}_{n} - x_{0}$ has length $r > 0$ and is orthogonal to $\Pi_{\rho \cap \tau}$.  Ditto for $\dot{y}_{n} - y_{0}$.  But $x_{0} \in \rho \cap \tau  \subset \Pi_{\rho \cap \tau}$.  Hence, $dist(\dot{x}_{n}, \rho \cap \tau) \geq r$.  Moreover, by \eqref{E:xdot.in.rho.ydot.in.tau}, $\dot{x}_{n} \in \rho$.  
Similarly, $dist(\dot{y}_{n}, \rho \cap \tau) \geq r$ and $\dot{y}_{n} \in \tau$.
Therefore, by compactness of $\rho$ and $\tau$, we may assume $\dot{x}_{n} \to \dot{x}_{\infty} \in \rho$ and $\dot{y}_{n} \to \dot{y}_{\infty} \in \tau$.  We must have $| \dot{x}_{\infty} - x_{0} | = r$, $| \dot{y}_{\infty} - y_{0} | = r$, $dist(\dot{x}_{\infty}, \rho \cap \tau) \geq r$, and $dist(\dot{y}_{\infty}, \rho \cap \tau) \geq r$. 
In particular, 
	\begin{equation}  \label{E:x.dot.infty.in.rho.less.rho.cap.tau}
		\dot{x}_{\infty} \in \rho \setminus (\rho \cap \tau) \text{ and }
			\dot{y}_{\infty} \in \tau \setminus (\rho \cap \tau)
	\end{equation}

Now, by definition of $\{ x_{n} \}$, $\{ y_{n} \}$, $\{ \dot{x}_{n} \}$, and $\{ \dot{y}_{n} \}$, we have
	\begin{equation*}
		(\dot{x}_{n} - x_{0}) \cdot (\dot{y}_{n} - y_{0})  
			= r^{2} \frac{(x_{n}-\hat{x}_{n}) \cdot (y_{n}-\hat{y}_{n})}
					{|x_{n}-\hat{x}_{n}||y_{n}-\hat{y}_{n}|}
			         \to r^{2} = | \dot{x}_{\infty} - x_{0} | | \dot{y}_{\infty} - y_{0} | \text{ as } n \to \infty.
	\end{equation*} 
But, 
	\[
		 (\dot{x}_{n} - x_{0}) \cdot (\dot{y}_{n} - y_{0})
			  \to (\dot{x}_{\infty} - x_{0}) \cdot (\dot{y}_{\infty} - y_{0}) \text{ as } n \to \infty.
	\]
This means $\dot{x}_{\infty} - x_{0}$ and $\dot{y}_{\infty} - y_{0}$ are positive multiples of each other.  But $\dot{x}_{\infty} - x_{0}$ and $\dot{y}_{\infty} - y_{0}$ have the same length $r$.  
Hence, $\dot{x}_{\infty} - x_{0} = \dot{y}_{\infty} - y_{0}$. However, by \eqref{E:y0.=.x0}, $y_{0} = x_{0}$. Therefore, $\dot{x}_{\infty} = \dot{y}_{\infty}$. In particular, 
$\dot{x}_{\infty}, \dot{y}_{\infty} \in \rho \cap \tau$.  This contradicts \eqref{E:x.dot.infty.in.rho.less.rho.cap.tau}.  The claim \eqref{E:angle.tween.x-xhat.y-yhat.ain't.0} follows.

By definition of $\hat{x}$ and $\hat{y}$ and \eqref{E:angle.tween.x-xhat.y-yhat.ain't.0}, we have 
	\begin{align}  \label{E:expand.length.of.x-y.using.hats}
		|x - y|^{2} &= \bigl| (x-\hat{x}) + (\hat{x} - \hat{y}) + (\hat{y} - y) \bigr|^{2} \notag \\
		  &= | x-\hat{x} |^{2} + | \hat{x} - \hat{y} |^{2} - 2 (x-\hat{x})  \cdot (y-\hat{y}) + | \hat{y} - y |^{2}            
		             \notag \\
		  &\geq | x-\hat{x} |^{2} + | \hat{x} - \hat{y} |^{2} - 2 \gamma |x-\hat{x}||y-\hat{y}| 
		        + | y  - \hat{y} |^{2} \\
		  &= (1-\gamma) \bigl( | x-\hat{x} |^{2} + | y  - \hat{y} |^{2} \bigr) + | \hat{x} - \hat{y} |^{2} + 
		         \gamma \bigl( | x-\hat{x} | - | y  - \hat{y} | \bigr)^{2} \notag \\
		  &\geq (1-\gamma) \bigl( | x-\hat{x} |^{2} + | y  - \hat{y} |^{2}+ | \hat{x} - \hat{y} |^{2}  \bigr).  
		             \notag
	\end{align}
Recall the inequality
	\begin{equation}  \label{E:2a.squared.etc.ineq}
		2 a^{2} + 2 b^{2} \geq (a + b)^{2}, \quad (a,b \in \RR).
	\end{equation}
Applying this twice to \eqref{E:expand.length.of.x-y.using.hats} we get
	\begin{align*}
		|x - y|^{2} &\geq \frac{1-\gamma}{4} 
	                      \Bigl(2  \bigl[ | x-\hat{x} | + | y  - \hat{y} | \bigr]^{2}+ 4| \hat{x} - \hat{y} |^{2}  \Bigr) \\
	                 &\geq \frac{1-\gamma}{4} 
	                      \Bigl(2  \bigl[ | x-\hat{x} | + | y  - \hat{y} | \bigr]^{2}+ 2| \hat{x} - \hat{y} |^{2}  \Bigr) \\
	                  &\geq \frac{1-\gamma}{4} 
	                      \bigl[ | x-\hat{x} | + | y  - \hat{y} | + | \hat{x} - \hat{y} | \bigr]^{2}. \\
	\end{align*}
We conclude
	\begin{equation}  \label{E:|x-y|.dominates.mult.of.sum.of.tilde.pieces}
		\frac{2}{\sqrt{1-\gamma}} |x - y| 
		\geq  | x-\hat{x} |+ | \hat{x} - \hat{y} | + | y  - \hat{y} |  ,
			\text{ for } x \in \text{Int} \, \rho, \; y \in \text{Int} \, \tau.
	\end{equation}

\emph{Claim: The angle, $\theta$, between $x - \tilde{x}$ and $\Pi_{\rho \cap \tau}$ is bounded away from 0.}   Since $\hat{x}$ is the orthogonal projection of $x$ onto $\Pi_{\rho \cap \tau}$, we have 
that $\theta$ is the angle between $x - \tilde{x}$ and $\hat{x} - \tilde{x}$ 
and $\sin \theta =  |x - \hat{x}|/|x - \tilde{x}|$.  By definition of $\hat{x}$, $|x - \tilde{x}|/|x - \hat{x}| \ge 1$.  Therefore, $\theta$ being bounded away from 0 is equivalent to
	\begin{multline}  \label{E:x.minus.xtilde.over.x.minus.xhat.bdd}
		1/\sin \theta = |x - \tilde{x}|/|x - \hat{x}| \text{ is bounded above by some }  \\
		    \alpha \in (1, \infty)  \text{ independent of } x \in \text{Int} \, \rho.
	\end{multline}
And similarly for $y$, $\tilde{y}$, and $\hat{y}$.

If $z \in \xi$ (the face of $\rho$ opposite $\rho \cap \tau$), let $\hat{z}$ denote the orthogonal projection of $z$ onto $\Pi_{\rho \cap \tau}$.  Recall that $\tilde{x}$, $x$, and $z_{0}$ lie on the same line.  Taking orthogonal projections, we see that $\tilde{x}$, $\hat{x}$, and $\hat{z}_{0}$ lie on the same line in $\Pi_{\rho \cap \tau}$. Therefore, by similarity of triangles\footnote{To see all this analytically, 
let $c = |x-\tilde{x}|/|z_{0} - \tilde{x}|$.  ($|z_{0} - \tilde{x}| > 0$, since $\tilde{x} \in \rho \cap \tau$ 
and $z_{0} \in \xi$, the face opposite $\rho \cap \tau$.)  Then 
	\begin{equation} \label{E:x.in.terms.of.c.z0.xtilde}
		x = c(z_{0} - \tilde{x}) + \tilde{x},
	\end{equation}
since $x$ lies on the line segment joining $z_{0}$ and $\tilde{x}$.  
Let $\ddot{x} = c(\hat{z}_{0} - \tilde{x}) + \tilde{x}$.  Then $\ddot{x}$ lies on the line joining $\tilde{x}$ and $\hat{z}_{0}$.  (In particular, $\ddot{x} \in \Pi_{\rho \cap \tau}$.)  But it is easy to see 
from \eqref{E:x.in.terms.of.c.z0.xtilde} that $x - \ddot{x} = c(z_{0} - \hat{z}_{0}) \perp \Pi_{\rho \cap \tau}$.  I.e., 
	\begin{equation}  \label{E:xhat.=.xdoubledot}
		\hat{x} = \ddot{x} = c(\hat{z}_{0} - \tilde{x}) + \tilde{x}.
	\end{equation}
Thus, $z_{0} - \tilde{x}$, $x - \tilde{x}$, $\hat{x} - \tilde{x}$, and $\hat{z}_{0} - \tilde{x}$ lie in the subspace spanned by $z_{0} - \tilde{x}$ and $\hat{z}_{0} - \tilde{x}$ and
		\[
			\frac{|x - \tilde{x}|}{|x - \hat{x}|} 
			  = \frac{\Bigl| \bigl[ c(z_{0} - \tilde{x}) + \tilde{x} \bigr] - \tilde{x} \Bigr|}
			            {\Bigl| \bigl[ c(z_{0} - \tilde{x}) + \tilde{x} \bigr] 
			                     - \bigl[ c(\hat{z}_{0} - \tilde{x}) + \tilde{x} \bigr]  \Bigr|} 
			  = \frac{|z_{0} - \tilde{x}|}{|z_{0} - \hat{z}_{0}|}
		\]
by \eqref{E:x.in.terms.of.c.z0.xtilde} and \eqref{E:xhat.=.xdoubledot}.},  
		\[
			\frac{|x - \tilde{x}|}{|x - \hat{x}|} 
			  = \frac{|z_{0} - \tilde{x}|}{|z_{0} - \hat{z}_{0}|}.
		\]
But since $\xi$ and 
$\rho \cap \tau$ are disjoint and compact, $|z - \hat{z}|$ is bounded below and $|z - w|$ is bounded above in $(z,w) \in \xi \times (\rho \cap \tau)$.  The claim
\eqref{E:x.minus.xtilde.over.x.minus.xhat.bdd} follows.  Of course, the same thing goes for $y$ and we may assume the same $\alpha$ works for both $\rho$ and $\tau$.  

It follows from \eqref{E:x.minus.xtilde.over.x.minus.xhat.bdd} and the Pythagorean theorem that 
	\begin{equation}  \label{E:xtilde.minus.xhat.bound}
		|\tilde{x} - \hat{x}| \leq \sqrt{\alpha^{2} -1} \;   |x - \hat{x}| < \alpha |x - \hat{x}|. 
		   \text{ Similarly for } y, \tilde{y}, \text{ and } \hat{y}.
	\end{equation}
Consequently,  
	\begin{multline*}
		|\tilde{x} - \tilde{y}| \leq |\tilde{x} - \hat{x}| + |\hat{x} - \hat{y}| + |\hat{y} - \tilde{y}|
			\leq \alpha |x - \hat{x}| + |\hat{x} - \hat{y}| + \alpha |y - \hat{y}| \\
			\leq \alpha |x - \hat{x}| + 2 \alpha |\hat{x} - \hat{y}| + \alpha |y - \hat{y}|,
	\end{multline*}
since $\alpha > 1$.  Hence,
	\[
		|\hat{x} - \hat{y}| \geq \frac{1}{2 \alpha} |\tilde{x} - \tilde{y}| - \frac{1}{2} |x - \hat{x}| 
			- \frac{1}{2} |y - \hat{y}|.
	\]
Substituting this into \eqref{E:|x-y|.dominates.mult.of.sum.of.tilde.pieces} we get
	\begin{equation*}
		\frac{2}{\sqrt{1-\gamma}} |x - y| 
			\geq  \frac{1}{2} | x-\hat{x} |+ \frac{1}{2 \alpha} |\tilde{x} - \tilde{y}| 
				+ \frac{1}{2}| y  - \hat{y} |.
	\end{equation*}
Therefore, by \eqref{E:x.minus.xtilde.over.x.minus.xhat.bdd} again,
	\begin{equation}
		\frac{2}{\sqrt{1-\gamma}} |x - y| 
			\geq  \frac{1}{2 \alpha}  \bigl( | x-\tilde{x} |+ |\tilde{x} - \tilde{y}| + | y - \tilde{y} | \bigr).
	\end{equation}
I.e., if \eqref{E:rho.intersect.but.aren't.subsets} holds
	\begin{equation}  \label{E:mult.of.x.minus.y.dominates.x.minus.xtilde.etc}
		\frac{4 \alpha}{\sqrt{1-\gamma}} |x - y| 
			\geq  | x-\tilde{x} |+ |\tilde{x} - \tilde{y}| + | y - \tilde{y} |.
	\end{equation}

Let $K' = K'(\rho, \tau) := \tfrac{4 \alpha}{\sqrt{1-\gamma}}$.  
Then \eqref{E:x.y.tilde.bckwrds.tringle.ineq} holds.  (Maximizing over all appropriate $\rho, \tau \in P$ yields corollary \ref{C:reverse.triangle.ineq.in.simp.cmplxs}.) \eqref{E:mult.of.x.minus.y.dominates.x.minus.xtilde.etc} 
and \eqref{E:beta.unif.Lip.on.each.simplex} together imply 
	\begin{align*}
		\bigl| \boldsymbol{\beta}(x) - \boldsymbol{\beta}(y)  \bigr| 
			& \leq \bigl| \boldsymbol{\beta}(x) - \boldsymbol{\beta}(\tilde{x})  \bigr|
				+ \bigl| \boldsymbol{\beta}(\tilde{x}) - \boldsymbol{\beta}(\tilde{y})  \bigr|
				+ \bigl| \boldsymbol{\beta}(\tilde{y}) - \boldsymbol{\beta}(y)  \bigr| \\
			&\leq K \bigl( | x-\tilde{x} |+ | \tilde{x} - \tilde{y} | + | y - \tilde{y} |  \bigr) \\
			&\leq K K'(\rho, \tau) |x - y|.
	\end{align*}
Now maximize over all $\rho, \tau \in P$.  This completes the proof.
\end{proof}

\section{Lipschitz maps and Hausdorff measure and dimension}   \label{S:Lip.Haus.meas.dim}
Hausdorff dimension (Giaquinta \emph{et al} \cite[p.\ 14, Volume I]{mGgMjS98.cart.currents} and Falconer \cite[p.\ 28]{kF90}) is defined as follows.  First, we define Hausdorff measure (Giaquinta \emph{et al} \cite[p.\ 13, Volume I]{mGgMjS98.cart.currents}, 
Hardt and Simon \cite[p.\ 9]{rHlS86.GMT}, and Federer \cite[2.10.2. p.\ 171]{hF69}).  Let $s \geq 0$.  If $s$ is an integer, let $\omega_{s}$ denote the volume of the unit ball in $\RR^{s}$:
	\begin{equation}  \label{E:vol.of.unit.ball}
		\omega_{s} = \frac{\Gamma(1/2)^{s}}{\Gamma \left( \tfrac{s}{2} + 1 \right)},
	\end{equation}
where $\Gamma$ is Euler's gamma function (Federer \cite[pp.\ 135, 251]{hF69}).  If $s$ is not an integer, then $\omega_{s}$ can be any convenient positive constant.  Federer uses \eqref{E:vol.of.unit.ball} for any $s \geq 0$. Let $X$ be a metric space with metric $d_{X}$. For any subset $A$ of $X$ and $\delta > 0$ first define	
	\begin{equation}  \label{E:Hs.delta.defn}
		\Hm^{s}_{\delta}( A ) 
		      = \omega_{s} \inf \left\{ \sum_{j} \left( \frac{diam(C_{j})}{2} \right)^{s} \right\}.
	\end{equation}
Here, ``$diam$'' is diameter (w.r.t.\ $d_{X}$) and the infimum is taken over all (at most) countable collections 
$\{ C_{j} \}$ of subsets of $X$ with $A \subset \bigcup_{j} C_{j}$ and 
$diam \, C_{j} < \delta$.  (Thus, $\Hm^{s}_{\delta}(\varnothing) = 0$ since an empty cover covers $\varnothing$ and an empty sum is 0.  If $X$ is second countable, it follows from Lindel\"of's theorem, Simmons \cite[Theorem A, p.\ 100]{gfS63}, that for any $\delta > 0$, such a countable cover exists. 
Otherwise, $\Hm^{s}_{\delta}( A ) = + \infty$.) We may assume that the covering sets $C_{j}$ are all open or that they are closed (Federer \cite[2.10.2, p.\ 171]{hF69}).  The $s$-dimensional Hausdorff measure of $A$ is then
	\begin{equation}  \label{E:Hs.defn}
		\Hm^{s}( A ) = \lim_{\delta \downarrow 0} \Hm^{s}_{\delta}( A ) 
			= \sup_{\delta > 0} \Hm^{s}_{\delta}( A ).
	\end{equation}
Note that $\Hm^{0}( A )$ is the cardinality of $A$ if it is finite.  
Otherwise, $\Hm^{0}( A ) = +\infty$.  At the other extreme,
	\begin{equation}  \label{E:A.empty.iff.H0.A.=.0}
		A = \varnothing \text{ if and only if } \Hm^{0}( A ) = 0.
	\end{equation}

For every $s \geq 0$, $\Hm^{s}$ is an outer measure on $X$ and the Borel subsets of $X$ are $\Hm^{s}$-measurable (Federer \cite[2.10.2, p.\ 171]{hF69} and Hardt and Simon \cite[p.\ 10]{rHlS86.GMT}).  Note that if $X$ is a subset of a Euclidean space (and inherits the Euclidean metric) and we rescale $X$ by multiplying each vector in $X$ by $\lambda > 0$, then for every $A \subset X$ the measure $\Hm^{s}( A )$ will be replaced by $\lambda^{s} \, \Hm^{s}( A )$.

For $A \subset X$ nonempty there will be a number $s_{0} \in [0, + \infty]$ s.t.\ $0 \leq s < s_{0}$ implies $\Hm^{s}( A ) = + \infty$ and $s > s_{0}$ implies 
$\Hm^{s}( A ) = 0$.  That number $s_{0}$ is the ``Hausdorff dimension'', $\dim A$, of $A$ (Falconer \cite[p.\ 28]{kF90}).  (In particular, $\dim \varnothing = 0$.  In appendix \ref{S:basics.of.simp.comps} we already defined $\dim \sigma$, where $\sigma$ is a simplex, and $\dim P$, where $P$ is a simplicial complex.  These dimensions are the same as the respective Hausdorff dimensions, at least if $P$ is finite.) But $\Hm^{s}( A ) = 0$ is a stronger statement than $\dim A \leq s$ 
(Falconer \cite[p.\ 29]{kF90}).  It is easy to see that 
	\begin{multline}  \label{E:dim.of.whole.=.max.dim.of.parts}
		\text{if } A \text{ is a finite union of Borel measurable sets } 
		      A_{1}, \ldots, A_{k}    \\
			\text{ then } \dim A = \max \{ \dim A_{1}, \ldots, \dim A_{k} \}
	\end{multline}
(Falconer \cite[p.\ 29]{kF90}).

Let $Y$ be a metric space with metric $d_{Y}$ and let $f : X \to Y$.  Recall that $f$ is 
\linebreak
``Lipschitz(ian)'' (Giaquinta \emph{et al} \cite[p.\ 202, Volume I]{mGgMjS98.cart.currents}, Falconer \cite[p.\ 8]{kF90}, Federer \cite[pp.\ 63 -- 64]{hF69}) if there exists $K < \infty$ (called a ``Lipschitz constant'' for $f$) s.t.\
	\[
		d_{Y} \bigl[ f(x), f(y) \bigr] \leq K \, d_{X}(x, y), \quad \text{ for every } x, y \in X.
	\]

\begin{example}   \label{E:dist.is.Lip}
If $\Ss \subset X$ is compact then the function $y \mapsto dist(y, \Ss) \in \RR$ is Lipschitz with Lipschitz constant 1.
\end{example}

Further recall the following.  Let $k = 1, 2, \ldots $ and let $\mcl{L}^{k}$ denote $k$-dimensional Lebesgue measure.  Suppose $T$ is a linear operator on $\RR^{k}$ and $v \in \RR^{k}$.  Then by Rudin \cite[Theorems 8.26(a) and 8.28, pp.\ 173--174]{wR66.realcmplx} if $A \subset \RR^{k}$ is Borel measurable then $T(A) + v$ is Lebesgue measurable and
	\begin{equation}  \label{E:Leb.meas.of.affine.trans}
		\mcl{L}^{k} \bigl[ T(A) + v \bigr] = |\det T| \, \mcl{L}^{k}(A).
	\end{equation}
This motivates the following basic fact about Hausdorff measure and dimension
(Falconer \cite[p.\ 28]{kF90}, Hardt and Simon \cite[1.3, p.\ 11]{rHlS86.GMT}).  Let $f : X \to Y$ be Lipschitz with Lipschitz constant $K$.  Then for $s \geq 0$,
   \begin{equation}  \label{E:Lip.magnification.of.Hm}
      \Hm^{s} \bigl[ f(X) \bigr] \leq K^{s} \, \Hm^{s}(X). 
          \text{ Therefore, } \dim f(X) \leq \dim X.  
   \end{equation}

$f : X \to Y$ is ``locally Lipschitz'' (Federer \cite[pp.\ 64]{hF69}) if each $x \in X$ has a neighborhood, $V$, s.t.\ the restriction $f \vert_{V}$ is Lipschitz.  So any Lipschitz map is locally Lipschitz and, conversely, any locally Lipschitz function on $X$ is Lipschitz on any compact subset of $X$.  Moreover,
	\begin{equation}  \label{E:comp.of.Lips.is.Lip}
		\text{The composition of (locally) Lipschitz maps is (resp., locally) Lipschitz.}
	\end{equation}
An easy consequence of \eqref{E:Lip.magnification.of.Hm} 
is the following.

	\begin{lemma}  \label{L:loc.Lip.image.of.null.set.is.null}
Let $X \subset \RR^{k}$.  Suppose $f : X \to Y$ is locally Lipschitz.  If $s \geq 0$ and 
$\Hm^{s}(X) = 0$, then $\Hm^{s} \bigl[ f(X) \bigr] = 0$.  In particular, $\dim f(X) \leq \dim X$.  We also have $\Hm^{0} \bigl[ f(X) \bigr] \leq \Hm^{0}(X)$.
	\end{lemma}
  \begin{proof}
By Lindel\"of's theorem (Simmons \cite[Theorem A, p.\ 100]{gfS63}) $X$ can be partitioned into a countable number of disjoint Borel sets $A_{1}, A_{2}, \ldots$ on each of which $f$ is Lipschitz with respective Lipschitz constant $K_{i}$. By  \eqref{E:Lip.magnification.of.Hm}, we have
	\begin{equation*}
		\Hm^{s} \bigl[ f(X) \bigr] \leq \sum_{i} \Hm^{s} \bigl[ f(A_{i}) \bigr] 
			\leq \sum_{i} K_{i}^{s} \Hm^{s}(A_{i}).
	\end{equation*}
  \end{proof}

Another generalization of \eqref{E:Lip.magnification.of.Hm} is the following.

\begin{lemma}   \label{L:bound.on.Haus.meas.h.A}
Let $k$ and $m$ be positive integers.  Let $U \subset \RR^{k}$ be open and suppose $h = (h_{1}, \ldots, h_{m}) : U \to \RR^{m}$ is continuously differentiable.  For $x = (x_{1}, \ldots, x_{k}) \in U$, let $Dh(x)$ be the $m \times k$ Jacobian matrix
   \[
      Dh(x) = \left( \frac{\partial h_{i}(y)}{\partial y_{j}} \right)_{y = x}.
   \]
At each $x \in U$, let $\lambda(x)^{2}$ be the largest eigenvalue of $Dh(x)^{T} Dh(x)$ (with $\lambda(x) \geq 0$; ``${}^{T}$'' indicates matrix transposition.) Then $\lambda$ is continuous.  Furthermore, let $a \geq 0$ and let $A \subset U$ be Borel with $\Hm^{a}(A) < \infty$.   Then 
   \begin{equation}  \label{E:lwr.bnd.on.meas.of.image}
      \Hm^{a} \bigl[ h(A) \bigr]  \leq \int_{A} \lambda(x)^{a} \, \Hm^{a}(dx).
   \end{equation}
\end{lemma}

\begin{proof} 
By lemma \ref{L:Eigen.cont.} and  and continuity of $Dh$, $\lambda$ is continuous.  Let $\epsilon > 0$.  
Since $\lambda$ is continuous, by Lindel\"of's theorem (Simmons \cite[Theorem A, p.\ 100]{gfS63}), 
there exists an at most countable cover, $C_{1}, C_{2}, \ldots $, of $U$ by open convex sets with the property
   \begin{equation}   \label{E:narrow.rnge.lambda.on.Ci}
         x, x' \in C_{i} \Rightarrow \bigl| \lambda(x)^{a} - \lambda(x')^{a} \bigr| < \epsilon, 
            \quad (i = 1, 2, \ldots).
   \end{equation}

For each $i = 1, 2, \ldots$ let $\Lambda_{i} = \sup_{x \in C_{i}} \lambda(x)$.  
We prove the \emph{claim:} on each $C_{i}$, the function $h$ is Lipschitz with Lipschitz constant 
$\Lambda_{i}$. 
(See Giaquinta \emph{et al} \cite[Theorem 2, p.\ 202, Vol. I]{mGgMjS98.cart.currents}.)  
Let $x,y \in C_{i}$.  Think of $x,y$ as row vectors. Since $C_{i}$ is open and convex there is an open interval $I \supset [0,1]$ s.t.\ for every $u \in I$ we have $(1-u)x + u y \in C_{i}$.  
The function $f: u \mapsto h \bigl[ (1-u)x + u y \bigr] \in \RR^{m}$ ($u \in I$) is defined and differentiable.  It defines an arc in $\RR^{m}$.  By the area formula (Hardt and Simon \cite[p.\ 13]{rHlS86.GMT})
   \begin{equation}   \label{E:bound.incr.by.arc.lngth}
      \bigl| h(y) - h(x) \bigr| \leq \text{length of arc } f = \int_{0}^{1} \bigl| f'(u) \bigr| \, du 
        = \int_{0}^{1} \Bigl| Dh[ (1-u)x + u y \bigr] (y - x)^{T} \Bigr| \, du.
   \end{equation}
Let $u \in [0,1]$ and let $w = (1-u)x + u y \in C_{i} \subset U \subset \RR^{k}$.  Let 
$\lambda^{2}_{1} \geq \lambda^{2}_{2} \geq \cdots \geq \lambda^{2}_{k} \geq 0$ be the eigenvalues of $Dh(w)^{T} Dh(w)$, so 
$\lambda^{2}(w) = \lambda^{2}_{1}$.  
Let $z_{1}, \ldots, z_{k} \in \RR^{k}$ be corresponding orthonormal eigenvectors, thought of as row vectors.  Write
$y - x = \sum_{j=1}^{k} \alpha_{j} z_{j}$.  Then
\begin{align*}
\bigl| Dh(w) (y - x)^{T} \bigr|^{2} &= (y - x)  \bigl( Dh(w)^{T} Dh(w) \bigr) (y - x)^{T}  \\
  &= \left( \sum_{j=1}^{k} \alpha_{j} z_{j} \right) \bigl( Dh(w)^{T} Dh(w) \bigr) \left( \sum_{j=1}^{k} \alpha_{j} z_{j}^{T} \right)   \\
  &= \left( \sum_{j=1}^{k} \alpha_{j} z_{j} \right) \left( \sum_{j=1}^{k} \alpha_{j} \lambda^{2}_{j} z_{j}^{T} \right)   \\
  &= \sum_{j=1}^{k} \lambda^{2}_{j} \alpha_{j}^{2}  \\
  &\leq \lambda^{2}(w) \sum_{j=1}^{k} \alpha_{j}^{2} \\
  &\leq \Lambda_{i}^{2} | y - x |^{2}.
\end{align*}
I.e., 
	\[
		\bigl| Dh(w) (y - x)^{T} \bigr| \leq \Lambda_{i} | y - x |.
	\]
Substituting this into \eqref{E:bound.incr.by.arc.lngth} proves the claim.

Let $A_{1} = A \cap C_{1}$.  Having defined $A_{1}, \ldots, A_{n}$, let
   \[
      A_{n+1} = (A \cap C_{n+1}) \setminus \left( \bigcup_{i=1}^{n} A_{i} \right).
   \]
Then $A_{1}, A_{2}, \ldots$ is a Borel partition of $A$.
By \eqref{E:Lip.magnification.of.Hm} and \eqref{E:narrow.rnge.lambda.on.Ci},
   \[
      \Hm^{a} \bigl[ h(A) \bigr] \leq \sum_{i} \Hm^{a} \bigl[ h(A_{i}) \bigr] 
         \leq \sum_{i} \Lambda_{i}^{a} \Hm^{a} (A_{i}) 
         \leq \int_{A} \lambda(x)^{a} \Hm^{a}(dx) + \Hm(A) \epsilon.
   \]
Since $\epsilon > 0$ is arbitrary and $\Hm^{a}(A) < \infty$, the lemma follows.
\end{proof}

The following fact is useful.  For the proof see the preceding proof of lemma \ref{L:bound.on.Haus.meas.h.A}.

\begin{corly}  \label{C:cont.diff.=.loc.Lip}
Let $k, m > 0$ be integers.  Let $U \subset \RR^{k}$ be open and let $h : U \to \RR^{m}$ be continuously differentiable.  Then $h$ is locally Lipschitz.  In particular, if $A \subset U$ is compact then $h$ is Lipschitz on $A$.
\end{corly}



\end{document}